\documentclass[reqno, 11pt]{article}
\usepackage[utf8]{inputenc}   
\usepackage[T1]{fontenc}      

\usepackage{amsmath,amsthm} 
\usepackage{amssymb,mathrsfs}
\usepackage{amsfonts}
\usepackage{euscript}
\usepackage{upgreek}
\usepackage{nicefrac}
\usepackage{mwe}
\usepackage{multicol}
\usepackage{geometry}
\usepackage{authblk}
\usepackage[numbers]{natbib}
\usepackage{graphicx}
\usepackage[caption = false]{subfig} 
\usepackage{color}
\usepackage[colorlinks = true, citecolor = blue]{hyperref}
\usepackage{algpseudocode,algorithm,algorithmicx}
\algnewcommand{\Inputs}[1]{%
  \State \textbf{Inputs:}
  \Statex \hspace*{\algorithmicindent}\parbox[t]{.8\linewidth}{\raggedright #1}
}
\algnewcommand{\Initialize}[1]{%
  \State \textbf{Initialize:}
  \Statex \hspace*{\algorithmicindent}\parbox[t]{.8\linewidth}{\raggedright #1}
}
\algnewcommand{\Outputs}[1]{%
  \State \textbf{Outputs:}
  \Statex \hspace*{\algorithmicindent}\parbox[t]{.8\linewidth}{\raggedright #1}
}

\usepackage{stmaryrd}

\usepackage[inline]{enumitem}
\usepackage{url}
\usepackage{graphicx} 
\usepackage{lmodern} 
\usepackage{xcolor}
\usepackage{bbm}

\usepackage{ifthen}
\usepackage{xargs}

\usepackage[textwidth=1.8cm]{todonotes}

\usepackage{aliascnt}
\usepackage{cleveref}
\usepackage{autonum}
\makeatletter
\newtheorem{theorem}{Theorem}
\crefname{theorem}{theorem}{Theorems}
\Crefname{Theorem}{Theorem}{Theorems}

\newtheorem*{lemma_nonumber*}{Lemma}

\newaliascnt{lemma}{theorem}
\newtheorem{lemma}[lemma]{Lemma}
\aliascntresetthe{lemma}
\crefname{lemma}{lemma}{lemmas}
\Crefname{Lemma}{Lemma}{Lemmas}

\newaliascnt{corollary}{theorem}

\aliascntresetthe{corollary}
\crefname{corollary}{corollary}{corollaries}
\Crefname{Corollary}{Corollary}{Corollaries}

\newaliascnt{proposition}{theorem}
\newtheorem{proposition}[proposition]{Proposition}
\aliascntresetthe{proposition}
\crefname{proposition}{proposition}{propositions}
\Crefname{Proposition}{Proposition}{Propositions}

\newaliascnt{definition}{theorem}
\newtheorem{definition}[definition]{Definition}
\aliascntresetthe{definition}
\crefname{definition}{definition}{definitions}
\Crefname{Definition}{Definition}{Definitions}

\newaliascnt{remark}{theorem}

\aliascntresetthe{remark}
\crefname{remark}{remark}{remarks}
\Crefname{Remark}{Remark}{Remarks}

\crefname{example}{example}{examples}
\Crefname{Example}{Example}{Examples}

\crefname{figure}{figure}{figures}
\Crefname{Figure}{Figure}{Figures}

\crefformat{assumption}{{\textbf{A}}#2#1#3}

\newtheorem{assumptionA}{\textbf{A}\hspace{-3pt}}
\crefformat{assumptionA}{{\textbf{A}}#2#1#3}
\Crefname{assumptionA}{\textbf{A}\hspace{-3pt}}{\textbf{A}\hspace{-3pt}}
\crefname{assumptionA}{\textbf{A}}{\textbf{A}}

\newtheorem{assumptionB}{\textbf{B}\hspace{-3pt}}
\crefformat{assumptionB}{{\textbf{B}}#2#1#3}
\Crefname{assumptionB}{\textbf{B}\hspace{-3pt}}{\textbf{B}\hspace{-3pt}}
\crefname{assumptionB}{\textbf{B}}{\textbf{B}}

\newtheorem{assumptionH}{\textbf{H}\hspace{-3pt}}
\crefformat{assumptionH}{{\textbf{H}}#2#1#3}
\Crefname{assumptionH}{\textbf{H}\hspace{-3pt}}{\textbf{H}\hspace{-3pt}}
\crefname{assumptionH}{\textbf{H}}{\textbf{H}}

\Crefname{assumptionL}{\textbf{L}\hspace{-3pt}}{\textbf{L}\hspace{-3pt}}
\crefname{assumptionL}{\textbf{L}}{\textbf{L}}

\Crefname{assumptionQ}{\textbf{Q}\hspace{-3pt}}{\textbf{Q}\hspace{-3pt}}
\crefname{assumptionQ}{\textbf{Q}}{\textbf{Q}}

\usepackage{tikz}
\usepackage{pgfplots}
\usepackage{pgfkeys}
\usetikzlibrary{spy}
\usetikzlibrary[patterns] 
\usetikzlibrary{calc,decorations.pathreplacing} 
\usetikzlibrary{shadows.blur} 
\usetikzlibrary{shapes.symbols}

\usepackage{pgffor}

\def\Vpow{\mathscr{V}_{\pow}}
\def\Vpowdeux{\mathscr{V}_{2}}
\def\VR{V_R}
\def\WR{W_R}
\def\covmat{\mathbf{C}}
\def\borne{\EuScript{M}}
\def\calJ{\mathbf{j}}
\def\affinop{\mathbb{A}}
\def\rmi{\mathrm{i}}
\def\hash{\sharp}
\def\dimp{p}
\def\dim{d}
\def\fourier{\mathcal{F}}
\def\RLipFp{3 R \LipF p}

\def\Imag{\mathfrak{I}}
\def\Re{\mathfrak{R}}
\def\LipgradF{\mathtt{B}}
\def\LipF{\mathtt{M}}
\def\pow{m}

\newcommand{\etaa}[1]{{\eta}^{(a)}_{#1}}
\newcommand{\etab}[1]{{\eta}^{(b)}_{#1}}
\newcommand{\etac}[1]{{\eta}^{(c)}_{#1}}
\newcommand{\etad}[1]{{\eta}^{(d)}_{#1}}

\newcommand{\dw}[1]{\mathrm{d}_{\WR}(#1)}

\def\tA{\tilde{A}}
\def\Rtheta{R_{\msk}}
\def\tbfX{\mathbf{X}}
\def\bfj{\mathbf{j}}
\def\tbfB{\mathbf{B}}
\def\no{\mathrm{n}_0}
\def\discrete{\mathrm{d}}
\def\Tg{T_{\gamma, \theta}}
\def\rmC{\mathrm{C}}
\def\Jcont{\calJ_{\mathrm{content}}}
\def\xcont{x_{\mathrm{content}}}
\def\xstyle{x_{\mathrm{style}}}
\def\tFcol{\tilde{F}_{\mathrm{color}}}
\def\Fcol{F_{\mathrm{color}}}
\newcommand{\kernellin}[1]{\mathrm{Ker}( #1 )}

\def\vgg19{$\mathtt{VGG19}$}
\def\nrmse{\mathrm{N}}
\def\NRMSE{$\mathrm{NRMSE}$}

\def\Pker{\mathrm{P}}
\def\Kker{\mathrm{K}}
\newcommand{\Reg}{r}
\def\distance{\mathbf{d}}
\newcommandx{\wasserstein}[3][1=\distance,3=]{\mathbf{W}_{#1}^{#3}\left(#2\right)}
\newcommandx{\wassersteinLigne}[3][1=\distance,3=]{\mathbf{W}_{#1}^{#3}(#2)}
\newcommandx{\wassersteinD}[1][1=\distance]{\mathbf{W}_{#1}}
\newcommandx{\wassersteinDLigne}[1][1=\distance]{\mathbf{W}_{#1}}
\def\scrG{\mathscr{G}}

\newcommand{\cball}[2]{\bar{\operatorname{B}}(#1,#2)}

\newcommand{\MAP}{\mathrm{MAP}}

\newcommand{\convex}{\mathcal{E}}
\newcommand{\tup}[1]{\textup{#1}}
\newcommand{\val}{v}
\newcommand{\Lag}{\mathcal{L}}
\newcommand{\primal}{$\mathrm{(P)}$}
\newcommand{\dual}{$\mathrm{(Q)}$}

\newcommand{\Palpha}{\mathcal{P}_{\upalpha}}
\newcommand{\PalphaF}{\mathcal{P}_{\upalpha}^F}
\newcommand{\argmin}{\mathrm{argmin}}

\newcommand{\without}[1]{\backslash \left \lbrace #1 \right \rbrace}

\newcommand{\X}{\mathrm{X}}
\newcommand{\noise}{\mathrm{Z}}

\newcommand{\ULA}{$\operatorname{ULA}$}

\newcommand{\tpoissfix}[1]{\hat{F}_{\gamma,\theta_{#1}}}

\newcommand{\tkernelfix}[1]{\mathrm{K}_{\gamma, \theta_{#1}}}
\newcommand{\poissp}[2]{\hat{F}_{#1,#2}}

\newcommand{\step}{\left\lceil 1/\gamma \right\rceil}

\newcommandx{\norm}[2][1=]{\ifthenelse{\equal{#1}{}}{\left\Vert #2 \right\Vert}{\left\Vert #2 \right\Vert^{#1}}}
\newcommandx{\normLigne}[2][1=]{\ifthenelse{\equal{#1}{}}{\Vert #2 \Vert}{\Vert #2\Vert^{#1}}}

\newcommand*{\figuretitle}[1]{%
  \textit{\textbf{#1.}}
}

\def\bfDd{\mathbf{D}_{\mathrm{d}}}
\def\bfDc{\mathbf{D}_{\mathrm{c}}}
\def\Psibf{\boldsymbol{\Psi}}
\def\Lambdabf{\boldsymbol{\Lambda}}

\def\mtt{\mathtt{m}}
\def\mttun{\mathtt{m}_1}

\def\msi{\mathsf{I}}
\def\msa{\mathsf{A}}

\def\msk{\mathsf{K}}

\def\msn{\mathsf{N}}

\def\msn{\mathsf{N}}
\def\msu{\mathsf{U}}
\def\msv{\mathsf{V}}

\def\msx{\mathsf{X}}

\def\mcbb{\mathcal{B}}  
\newcommand{\mcb}[1]{\mathcal{B}(#1)}

\def\mcf{\mathcal{F}}
\def\mcg{\mathcal{G}}

\def\rset{\mathbb{R}}

\def\cset{\mathbb{C}}

\def\nset{\mathbb{N}}
\def\nsets{\mathbb{N}^*}

\def\Zset{\mathbb{Z}}

\def\rml{\mathrm{L}}

\def\rmd{\mathrm{d}}

\def\rme{\mathrm{e}}

\def\mrc{\mathrm{C}}

\def\rmc{\mathrm{C}}

\newcommand{\R}{\mathbb R}
\newcommand{\Z}{\mathrm Z}

\newcommand{\A}{\mathcal A}

\newcommandx{\functionspace}[2][1=+]{\mathbb{F}_{#1}(#2)}
\newcommandx{\relativeentropydef}[1][1=]{H_{#1}}
\newcommandx{\relativeentropy}[2][1=]{\relativeentropy[#1](#2)}

\newcommandx{\VarDeux}[3][3=]{\operatorname{Var}^{#3}_{#1}\left\{#2 \right\}}

\newcommand{\1}{\mathbbm{1}}

\newcommand{\LeftEqNo}{\let\veqno\@@leqno}

\newcommand{\floor}[1]{\left\lfloor #1 \right\rfloor}
\newcommand{\ceil}[1]{\left\lceil #1 \right\rceil}
\newcommand{\ceilLigne}[1]{\lceil #1 \rceil}

\newcommand{\N}{\ensuremath{\mathbb{N}}}

\newcommand{\PE}{\mathbb{E}}
\newcommand{\PP}{\mathbb{P}}

\newcommand{\abs}[1]{\left\vert #1 \right\vert}
\newcommand{\absLigne}[1]{\vert #1 \vert}
\newcommand{\tvnorm}[1]{\| #1 \|_{\mathrm{TV}}}

\newcommandx{\Vnorm}[2][1=V]{\| #2 \|_{#1}}
\newcommandx{\VnormEq}[2][1=V]{\left\| #2 \right\|_{#1}}

\newcommand{\parenthese}[1]{\left(#1 \right)}
\newcommand{\parentheseLigne}[1]{(#1 )}
\newcommand{\parentheseDeux}[1]{\left[ #1 \right]}
\newcommand{\parentheseDeuxLigne}[1]{[ #1 ]}
\newcommand{\defEns}[1]{\left\lbrace #1 \right\rbrace }
\newcommand{\defEnsLigne}[1]{\lbrace #1 \rbrace }

\newcommand{\ps}[2]{\left\langle#1,#2 \right\rangle}

\newcommandx\probaMarkovTilde[2][2=]
{\ifthenelse{\equal{#2}{}}{{\widetilde{\mathbb{P}}_{#1}}}{\widetilde{\mathbb{P}}_{#1}\left[ #2\right]}}

\newcommand{\expe}[1]{\PE \left[ #1 \right]}

\newcommand{\expeLigne}[1]{\PE [ #1 ]}

\newcommand{\bigO}{\ensuremath{\mathcal O}}

\newcommand{\plusinfty}{+\infty}

\newcommand\numberthis{\addtocounter{equation}{1}\tag{\theequation}}

\def\ie{\textit{i.e.}}
\def\eg{\textit{e.g.}}

\def\eqsp{\;}
\newcommand{\coint}[1]{\left[#1\right)}
\newcommand{\ocint}[1]{\left(#1\right]}
\newcommand{\ooint}[1]{\left(#1\right)}
\newcommand{\ccint}[1]{\left[#1\right]}

\newcommand{\oointLigne}[1]{(#1)}

\renewcommand{\iint}[2]{\left\lbrace #1,\ldots,#2\right\rbrace}

\newcommandx{\weight}[2][2=n]{\omega_{#1,#2}^N}

\newcommand{\ball}[2]{\operatorname{B}(#1,#2)}

\def\as{almost surely}

\newcommandx\sequence[3][2=,3=]
{\ifthenelse{\equal{#3}{}}{\ensuremath{\{ #1_{#2}\}}}{\ensuremath{\{ #1_{#2}, \eqsp #2 \in #3 \}}}}
\newcommandx\sequenceD[3][2=,3=]
{\ifthenelse{\equal{#3}{}}{\ensuremath{\{ #1_{#2}\}}}{\ensuremath{( #1)_{ #2 \in #3} }}}

\newcommandx{\sequencen}[2][2=n\in\N]{\ensuremath{\{ #1_n, \eqsp #2 \}}}
\newcommandx\sequenceDouble[4][3=,4=]
{\ifthenelse{\equal{#3}{}}{\ensuremath{\{ (#1_{#3},#2_{#3}) \}}}{\ensuremath{\{  (#1_{#3},#2_{#3}), \eqsp #3 \in #4 \}}}}
\newcommandx{\sequencenDouble}[3][3=n\in\N]{\ensuremath{\{ (#1_{n},#2_{n}), \eqsp #3 \}}}

\newcommand{\wrt}{w.r.t.}

\def\eg{e.g.}

\newcommand{\opnorm}[1]{{\left\vert\kern-0.25ex\left\vert\kern-0.25ex\left\vert #1 
    \right\vert\kern-0.25ex\right\vert\kern-0.25ex\right\vert}}

\def\Lip{\mathtt{N}}

\def\generator{\mathcal{A}}

\def\Id{\operatorname{Id}}

\newcommandx{\CPE}[3][1=]{{\mathbb E}_{#1}\left[\left. #2 \middle \vert #3 \right. \right]} 
\newcommandx{\CPVar}[3][1=]{\mathrm{Var}^{#3}_{#1}\left\{ #2 \right\}}
\newcommand{\CPP}[3][]
{\ifthenelse{\equal{#1}{}}{{\mathbb P}\left(\left. #2 \, \right| #3 \right)}{{\mathbb P}_{#1}\left(\left. #2 \, \right | #3 \right)}}

\newcommandx{\osc}[2][1=]{\mathrm{osc}_{#1}(#2)}

\def\Id{\operatorname{Id}}

\def\n{\mathrm{n}}

\def\domain{\mathrm{D}}

\def\Gammabf{\mathbf{\Gamma}}

\def\transpose{\operatorname{T}}

\def\bgamma{\bar{\gamma}}

\def\tX{\tilde{\mathrm{X}}}

\def\tE{\tilde{E}}

\def\sign{\operatorname{sign}}

\newcommand{\ensemble}[2]{\left\{#1\,:\eqsp #2\right\}}
\newcommand{\ensembleLigne}[2]{\{#1\,:\eqsp #2\}}

\def\rmD{\mathrm{D}}

\def\mrc{\mathrm{C}}

\def\diag{\Delta_{\rset^d}}

\newcommand\coupling[2]{\Gamma(\mu,\nu)}

\newcommand{\complementary}{\mathrm{c}}

\def\Leb{\mathrm{Leb}}

\def\interior{\mathrm{int}}

\def\vareps{\varepsilon}

\def\Phibf{\mathbf{\Phi}}

\newcommandx{\KL}[2]{\mathrm{KL}\left( #1 | #2 \right)}
\def\vol{\operatorname{Vol}}

\def\tFcolmean{\tFcol^{\mathrm{m}}}
\def\tFcolvar{\tFcol^{\mathrm{cov}}}
\def\Fcolmean{\Fcol^{\mathrm{m}}}
\def\Fcolvar{\Fcol^{\mathrm{cov}}}
\def\rank{\mathrm{rank}}
\def\Span{\mathrm{span}}

\title{Maximum entropy methods for texture synthesis: theory and practice.}

\author[1]{Valentin De Bortoli}
\author[1]{Agn\`{e}s Desolneux}
\author[1]{Alain Durmus}
\author[2]{Bruno Galerne}
\author[3]{Arthur Leclaire}

\affil[1]{Centre de mathématiques et de leurs applications, CNRS, ENS
  Paris-Saclay, Université Paris-Saclay, 94235, Cachan cedex, France.}
\affil[2]{Institut Denis Poisson, Universit\'{e} d'Orléans, Universit\'{e} de Tours, CNRS}
\affil[3]{Univ. Bordeaux, IMB, Bordeaux INP, CNRS, UMR 5251.
  F 33-400, TALENCE. FRANCE.}

\begin{document}
\maketitle


\begin{abstract}
  Recent years have seen the rise of convolutional neural network techniques in
  exemplar-based image synthesis.  These methods often rely on the minimization
  of some variational formulation on the image space for which the minimizers
  are assumed to be the solutions of the synthesis problem. In this paper we
  investigate, both theoretically and experimentally, another framework to deal
  with this problem using an alternate sampling/minimization scheme. First, we
  use results from information geometry to assess that our method yields a
  probability measure which has maximum entropy under some constraints in
  expectation. Then, we turn to the analysis of our method and we show, using
  recent results from the Markov chain literature, that its error can be
  explicitly bounded with constants which depend polynomially in the dimension
  even in the non-convex setting. This includes the case where the constraints are
  defined via a differentiable neural network.  Finally, we present an extensive
  experimental study of the model, including a comparison with state-of-the-art
  methods and an extension to style transfer.
\end{abstract}

 \maketitle
  \section{Introduction}
\label{sec:introduction}

Understanding texture formation is a crucial step towards a global theory of the
human visual system as texture is an important perceptual cue. The more specific
problem of exemplar-based texture synthesis arises in computer graphics where it
is often desirable to be able to generate new large natural textures which
look like an input image. This application highlights the need of a
mathematically sound model for texture generation as our only criterion for
evaluating the performance of algorithms is via human inspection.  

Two main approaches have been proposed in the literature: the patch-based
methods \cite{efros1999texture, efros2001image,kwatra2003graphcut,
  levina2006texture, raad2015conditional, kwatra2005texture,
  han2006fast,kaspar2015self, galerne2018texture} and the parametric ones
\cite{vanwijk1991spotnoise, galerne2011random, gagalowicz1986model,cano1988hierarchical,heeger1995pyramid, zhu1998filters, gatys2015texture, johnson2016perceptual,ulyanov2016texture,ulyanov2017improved}. In the present paper, we are interested in the theoretical and visual
properties of information-based parametric models. More precisely, we consider
maximum entropy models. Indeed, the maximum entropy approach has the appealing
property that the trade-off between innovation (maximizing the entropy) and the
visual similarity with the input (geometrical or statistical feature
constraints) is explicitly embedded in the model. There exist two main
approaches for these maximum entropy formulation, the microcanonical model in
which the constraints must be met almost surely and the macrocanonical model in
which the constraints must be met in expectation
\cite{bruna2018multiscale}. Both share connections with statistical physics. In
\cite{bruna2018multiscale} the authors address the convergence of usual sampling
scheme for the microcanonical model. One key contribution of the present paper is the
derivation of a similar result for the macrocanonical model.

Contrary to the microcanonical model, the distribution of any macrocanonical
model is a Gibbs measure, \ie \ the exponential distribution of the features up
to a scalar product with some parameters \cite{jaynes1957info}.  Our first
contribution is to give explicit conditions on the features ensuring the
existence of such a macrocanonical model, extending results from information
geometry
\cite{csiszar1996maxent,csiszar1984sanov,csiszar1975divergence,dupuis1997weak}.

Even if such a Gibbs measure exists we are facing two issues: 1) finding the
optimal parameters, 2) sampling from the associated Gibbs measure. The first
challenge can in fact be seen as the dual formulation of the maximum entropy
problem under constraints and corresponds to the minimization of a convex
functional over an open susbet of $\rset^p$ with $p \in \nset$.  Therefore it is
natural to consider gradient based method in order to find such parameters and
this approach was considered in the seminal work of \cite{zhu1998filters} which
was the first to consider macrocanonical models in the context of image
processing. However, the gradient of this functional is the expectation of the
features with respect to the Gibbs measure.  In this context, we turn ourselves
to the Stochastic Approximation (SA) literature
\cite{robbins1951stochastic,chen1987convergence}. More precisely, we use the
Markov Chain Monte Carlo (MCMC) SA methodology proposed in
\cite{debortoli2018souk} and referred as Stochastic Optimization with Unadjusted
Langevin Algorithm (SOUL) which relies on the Langevin algorithm. This MCMC
method has received a lot of attention in the recent years
\cite{dalalyan2017theoretical,dalalyan2017further,durmus2017unadjusted} since it
exhibits desirable convergence properties and has been extensively used in
machine learning applications \cite{welling2011bayesian, simsekli2016stochastic,
  patterson2013stochastic, ma2015complete, ahn2012bayesian}. Note that a similar
methodology to SOUL was already used in a texture synthesis context in
\cite{lu2015learning, debortoli2019macro}.

Our second contribution is to establish the convergence of the methodology
proposed in \cite{debortoli2018souk} in the context of macrocanonical texture
synthesis and improve existing results on the dependency with respect to the
hyperparameters in this specific case. In particular, the dependency in the
dimension is polynomial even in the non-convex setting. This is in accordance
with similar results for the convergence of diffusion processes with respect to
the Kantorovitch-Rubinstein distance which are known to be optimal
\cite{eberle2016reflection}.

The paper is organized as follows. Notations and image descriptors are
introduced in \Cref{sec:notations}. Related work on maximum entropy methods is
discussed in \Cref{sec:previous_work} In \Cref{sec:max_def}, we give a
mathematical presentation of the microcanonical model and the macrocanonical
model. In \Cref{sec:maxim-entr-prob} we extend results from information geometry
to the context of exemplar-based texture synthesis. A Bayesian interpretation
and some examples are given in \Cref{sec:exampl-bayes-appr}. We then turn to the
proposed algorithm for sampling macrocanonical models. The SOUL algorithm is
exposed in \Cref{sec:stoch-unadj-lang} and the convergence results applied to
our settings are presented in \Cref{sec:main-results}. In
\Cref{sec:links-with-macr} we investigate the links between the microcanonical
model and the macrocanonical one. Experiments are presented in
\Cref{sec:experiments}. First we show that the SOUL algorithm numerically solves
the problem at hand, considering a toy circular Gaussian case in
\Cref{sec:peri-gauss-mod}.  After discussing the parameters of the algorithm we
turn to the challenging application of texture synthesis in
\Cref{sec:neur-netw-feat}. We study the advantages and the limitations of
macrocanonical models and compare our visual results with existing algorithms.
We conclude this section by presenting an extension of the studied framework to
style transfer. Proofs and additional results are gathered in a supplementary
document.


  \section{Notation and feature models}
\label{sec:notations}

\subsection{Notation}
\label{sec:notations-1}
Let $d, p \in \nset$. The complement of a set $\msa \subset \rset^d$, is denoted by $\msa^{\complementary}$.
For any $\msa \in \rset^d$, we denote $\interior(\msa)$ its interior, $\Delta_{\msa} = \ensembleLigne{(x,x) \in \rset^{d} \times \rset^d}{ x \in \msa}$ the diagonal of $\msa$, $\mathcal{B}(\rset^d)$ the Borel $\sigma$-field of
$\rset^d$, $\functionspace[]{\rset^d, \rset^p}$ the set of all $\rset^p$-valued Borel measurable
functions on $\rset^d$. If $p=1$, we write $\functionspace[]{\rset^d, \rset^p} = \functionspace[]{\rset^d}$ and define for $f \in \functionspace[]{\rset^d}$,
\begin{equation}
  \| f \|_{\infty} = \inf \ensemble{t \geq 0}{\Leb(\ensembleLigne{x \in \rset^{\dim}}{ \abs{f(x)} > t}) = 0 } \eqsp ,
\end{equation}
where $\Leb$ is the Lebesgue measure over $(\rset^d,\mcb{\rset^d})$.  An open
ball of $\rset^{\dim}$ for the Euclidean distance with center
$x_0 \in \rset^{\dim}$ and radius $r >0$ is denoted $\ball{x_0}{\rset^{\dim}}$.
For $\mu$ a probability measure on $(\rset^d, \mathcal{B}(\rset^d))$ and
$f \in \functionspace[]{\rset^d, \rset^p}$, a $\mu$-integrable function, denote by
$\mu(f)$ the integral of $f$ with respect to (\wrt)~$\mu$, \ie
\begin{equation}
  \mu(f) = \int_{\rset^{\dim}} f(x) \rmd \mu(x) \eqsp .
\end{equation}<
If $f = \1_{\msa}$ for some measurable set $\msa$ then we denote
$\mu(\1_{\msa}) = \mu(\msa)$.  Let $f \in \functionspace[]{\rset^d}$ then for
any probability measure $\mu$ on $(\rset^d, \mathcal{B}(\rset^d))$ we denote by
$f_{\hash}\mu$ the pushforward measure of $\mu$ by $f$.

Let $\msu$ be an open set of $\rset^d$. We denote by $\rmc^{k}(\msu, \rset^p)$
the set of $\rset^p$-valued $k$-continuously differentiable functions. The
differential of $f \in \rmc^{k}(\msu, \rset^p)$ is denoted $\rmd f$ and its
Jacobian matrix $\rmD f$. Let $\rmc^k(\msu)$ stand for $\rmc^k(\msu,\rset)$.
Let $f : \msu \to \rset$, we denote by $\nabla f$, the gradient of $f$ if it
exists. $f$ is said to be $\mtt$-strongly convex with $\mtt\geq 0$ if for all
$x,y \in \rset^d$ and $t \in \ccint{0,1}$,
\begin{equation}
f(t x + (1-t) y) \leq t f(x)  + (1-t) f(y) -(\mtt/2)t(1-t)  \norm[2]{x-y}  \eqsp.
\end{equation}
We recall that if $f : \msu \to \rset$ is twice differentiable at point
$a \in \rset^d$, its Laplacian is given by
$\Delta f(a) = \sum_{i=1}^d \partial^2 f(a) /\partial x_i^2$.  For any
$\msa \subset \rset^d$, we denote by $\partial \msa$ the boundary of $\msa$ and
$\vol(\msa) = \Leb(\msa)$  For any $\upalpha >0$, let $\Palpha$ be the set of
probability measures over $\mcb{\rset^d}$ such that
$\int_{\rset^d} \norm{x}^{\upalpha} \rmd \pi (x) <+\infty$.  Let
$(\Omega,\mcg,\PP)$ be a probability space, and
$\rml^2(\Omega,\mcg) = \{X \, : \, \text{$X$ is a real-valued random variable on
  $\Omega$ such $\expeLigne{X^2} < \plusinfty$}\}$.  Let $\mu,\nu$ be two
probability measures on $(\rset^d, \mcbb(\rset^d))$. We write $\mu \ll \nu$ if
$\mu$ is absolutely continuous \wrt~$\nu$ and $\rmd \mu / \rmd \nu$ an
associated density. The Kullback-Leibler divergence, or relative entropy, of
$\mu$ from $\nu$ is defined by
\begin{equation}
  \KL{\mu}{\nu} = 
  \begin{cases}
    \int_{\rset^d} \frac{\rmd \mu}{\rmd \nu}(x) \log \parenthese{\frac{\rmd \mu}{\rmd \nu} (x)} \rmd \nu (x) \eqsp & \text{if } \mu \ll \nu \eqsp ,\\
\plusinfty & \text{ otherwise} \eqsp.
  \end{cases}
\end{equation}
If $\mu$ and $\nu$ are probability measures, the relative entropy takes values
in $\ccint{0,+\infty}$.  We take the convention that $\prod_{k=p}^n = 1$ and
$\sum_{k=p}^n =0$ for $n,p \in \nset$, $n< p$. If $x, y \in \cset^d$ with
$d \in \nset$ we define the periodic convolution between $x$ and $y$ and denote
$z = x * y \in \cset^d$, the element $z$ such that for any
$i \in \{0, \dots, d-1\}$, $z(i) = \sum_{k=0}^{d-1} x(k)y(i-k)$, where $x$ and
$y$ are extended over $\Zset$ by periodicity. We also denote
$\check{x} \in \cset^{\dim}$ such that for any $i \in \{0, \dots, d-1\}$,
$\check{x}(i) = x(-i)$ and $x$ is extended over $\Zset$ by periodicity. For any
$x \in \cset^d$, $\fourier(x)$ (respectively $\fourier^{-1}(x) \in \cset^d$)
stands for the Fourier transform (respectively the inverse Fourier transform),
defined for any $j \in \{0, \dots, d-1 \}$ by
\begin{equation}
\fourier(x)(j) = \sum_{k=0}^{d-1} x(j) \rme^{-2\rmi \pi j k / d} \eqsp , \qquad \fourier^{-1}(x)(j) = d^{-1}\sum_{k=0}^{d-1} x(j) \rme^{2\rmi\pi j k / d} \eqsp .
\end{equation}
Note that we have $\fourier^{-1}(\fourier(x)) = x$. For any
$z \in \cset$, we denote by $\Re(z)$ the real part of $z$ and by $\Imag(z)$ its
imaginary part.  We denote by $\affinop_{n_2,n_1}(\R)$ the vector space of
affine operators from $\rset^{n_1}$ to $\rset^{n_2}$ and for any
$A \in \affinop_{n_2, n_1}$, $\tilde{A}$ is the linear part of $A$. Finally,
$\mathrm{S}_d(\rset)$ is the space of $d \times d$ real symmetric matrices.
\subsection{Image descriptors}
\label{sec:feature-models}

In this work we are interested in sampling probability distributions derived
from image models.  Let $x_0 \in \rset^{\dim}$ be an exemplar image and consider
a set of constraints associated with some image descriptors
$F: \ \rset^{\dim} \to \rset^{\dimp}$. Assume that $F(x_0) = 0$, this can always
been achieved upon subtracting $F(x_0)$ to the original features.  The
constraints on the target distribution are then given by $F=0$ almost surely or
in expectation.

In this section, we review some of the popular possibilities for the choice of the
function $F$.  In the literature, many different approaches such as Gaussian
models, \ie \ mean and correlation features
\cite{vanwijk1991spotnoise,galerne2011random}, wavelet-based descriptors
\cite{heeger1995pyramid,portilla2000parametric,peyre2010grouplet,duits2007image}
or convolutional neural network features (CNN)
\cite{gatys2015texture,ustyuzhaninov2016texture,jetchev2016texture} have been
proposed to come up with visually satisfying image descriptors.

In our study we will focus on two sets of features:
\begin{enumerate*}[label=(\roman*)]
\item Gaussian features ; 
\item CNN features. 
\end{enumerate*}
Gaussian features have the mathematical advantage of defining a strongly convex
model, therefore allowing for strong convergence results to apply. However
Gaussian textures do not exhibit sharp edges and lack long-range structures and
as a consequence richer models should be investigated in order to obtain
visually satisfying images. Similarly to
\cite{gatys2015texture,ustyuzhaninov2016texture,lu2015learning} we consider
features derived from a pretrained CNN.  It has been observed that these
features are efficient for describing a large variety of natural
images. However, these improvements over the Gaussian model come at a high
computational price. First, the features we end up with are no longer
convex. Second, the dimension of the associated parameter space is usually
high. An experimental investigation of the behavior of our proposed algorithm
for these two sets of features is conducted in \Cref{sec:experiments}. We now
describe precisely these two models.

\paragraph{Gaussian features}
Let $x_0 \in \rset^d$ and consider $F(x) = x * \check{x} - x_0 *
\check{x}_0$. In the Fourier domain, we have for any $i \in \{1, \dots, d\}$,
$\fourier(F(x))(i) = \abs{\fourier(x)(i)}^2 -
\abs{\fourier(x_0)(i)}^2$. Therefore, if $F(x) = 0$, $x$ has same power
spectrum, \ie \ same autocorrelation, as $x_0$, namely $\fourier(x)$ has the
same modulus as $\fourier(x_0)$. However the equation $F(x) = 0$ gives no
information on the phases of $\fourier(x)$.

\paragraph{Neural network features}
A Neural network is a series of affine operations (usually convolutions) followed at each step by a pointwise non-linearity.
We define
 \begin{equation}
   \label{eq:def_cnn}
   (A_j^k)_{j \in \lbrace 1, \dots, M\rbrace, k \in \{1, \dots, c_j\}} \in \prod_{j=1}^M \affinop_{n_j, c_{j-1} \times n_{j-1}}(\R)^{c_j} \eqsp , \quad  (n_j, c_j)_{j \in \lbrace 0,\dots, M \rbrace} \in \N^{M+1} \times \N^{M+1} \eqsp ,
 \end{equation}
 with $M \in \N$, $n_0=d$ and $c_0 = 1$. For each $j \in \{ 0, \dots, M\}$, $n_j$ is the \emph{dimension} of the $j$-th layer and $c_j$ is the \emph{number of channels} of the $j$-th layer, and for any $k \in \{1, \dots, c_j\}$, $A_j^k \in \affinop_{n_j, c_{j-1} \times n_{j-1}}$.  Namely for any layer $j \in \{1, \dots, M\}$ and channel $k \in \{1, \dots, c_j \}$, $A_j^k: \ \rset^{c_{j-1} \times n_{j-1}} \to \rset^{n_{j}}$ is the affine operator which maps the $(j-1)$-th layer to the $j$-th layer and channel $k$ before the non-linear operation. With our notations, the $0$-th layer corresponds to the original image.
 We recall that for any $j \in \{1, \dots, M\}$ and $k \in \{1, \dots, c_j\}$, $\tilde{A}_j^k$ denotes the linear part of $A_j^k$.
 We also define for any $j \in \{1, \dots, M\}$ and $x \in \rset^{c_{j-1} \times n_{j-1}}$, $A_j(x) = (A_j^k(x))_{k \in \{1, \dots, c_j\}}$, \ie \ $A_j \in \affinop_{c_j \times n_j, c_{j-1} \times n_{j-1}}$ is the affine operator which maps the $j-1$-th layer to the $j$-th layer before the non-linear operation.
 
 Let $\varphi: \ \rset \to \rset$ be a measurable function. By a slight abuse of notation we denote for any $d \in \nset$ and $x \in \rset^{\dim}$, $\varphi(x) = (\varphi(x(0)), \dots, \varphi(x(d-1)))$. We assume that $\varphi$ satisfies the following conditions:
 \begin{enumerate}[label=(\alph*)]
 \item for any $\dim \in \nset$, there exists $C_{\dim, \varphi} \geq 0$ such that for any $x \in \rset^d$, $\norm{\varphi(x)} \leq C_{\dim, \varphi}(1 + \norm{x}) \eqsp ,$
 \item $\varphi$ is non-decreasing, 
 \item $\lim_{t \to +\infty} \varphi(t) = +\infty \eqsp .$
 \end{enumerate}
 We define for any $j \in \lbrace 1, \dots,M \rbrace$ and $k \in \{1, \dots, c_j\}$, the $(j, k)$-th \emph{layer-channel feature} $\scrG_j^k: \rset^d \to \rset^{n_{j}}$, for any $x \in \rset^d$, by 
\begin{equation}
  \scrG_j^k (x) = \left( \varphi \circ A_j^k \circ \varphi \circ A_{j-1} \circ \dots \circ \varphi \circ A_1 \right) (x) \eqsp , \qquad \scrG_0(x) = x \eqsp . \label{eq:features}
\end{equation}
For any layer $j \in \{ 1, \dots, M\}$ and channel $k \in \{1, \dots, c_j \}$, $\scrG_j^k(x)$ is the neural network response of $x$ at layer $j$ and channel $k$. We also define for any $j \in \{1, \dots, M\}$, the $j$-th \textit{layer feature} $\scrG_j: \ \rset^{\dim} \to \rset^{c_j \times n_j}$, for any $x \in \rset^{\dim}$, by $\scrG_j(x) = (\scrG_j^k(x))_{k \in \{ 1, \dots, c_j \}}$.  Let $\calJ \subset \{1 , \dots, M\}$ then we can define $F(x) \in \rset^{p}$ for any $x \in \rset^d$ by
\begin{equation}
  \label{eq:neural_network}
  F(x) = \left( \overline{\scrG}_{j}^k(x)-\overline{\scrG}_{j}^k(x_0)\right)_{j \in \calJ, k \in \{1, \dots, c_{j}\}} \eqsp , \qquad \overline{\scrG}_{j}^k(x) = n_{j}^{-1}\sum_{\ell=1}^{n_{j}} \scrG_{j}^k(x)(\ell) \eqsp , \qquad p = \sum_{j \in \calJ} c_{j} \eqsp .\end{equation}
A few remarks are in order regarding the dimension of the associated parameter space. 
In our applications, we will use the \vgg19 \ convolutional neural network \cite{simonyan2014vgg}, see \Cref{sec:structure_of_vgg19} for details on the structure of \vgg19 \ network.
Note that since \vgg19 \ is a convolutional neural network, \ie \ the linear part of the affine operators is given by a convolutional operator, and since we average the neural network response, the output dimension $p$ is independent of the input dimension $d$.
Selecting the layers $\bfj =\{1, 3, 6, 8, 11, 13, 15, 24, 26, 31\}$ we have $p = 2688 \approx 10^3$. Usually we will consider images of size at least $512 \times 512$ for which $d = 262144 \approx 10^5$. Therefore the features described by \eqref{eq:neural_network} performs a dimension reduction. In \cite{gatys2015texture} similar image
descriptors are considered but Gram matrices are used instead of $(\overline{\scrG_j})_{j \in \calJ, k \in \{1, \dots, c_j\}}$. This leads to a parameter space with dimension $352256 \approx 10^5$, see \cite{raadcisa:hal-01553841}. 


  \section{Statistical texture models}
\label{sec:maxim-entr-meas}

\subsection{Previous work}
\label{sec:previous_work}

As emphasized in \Cref{sec:introduction}, there have been two main approaches to
address the exemplar-based texture synthesis problem. First, non-parametric
methods produce an output image following a statistical process, \eg \ a Markov
random field \cite{efros1999texture, chellappa1985texture,
  purks1977visual}. These methods do not require an explicit texture model. Most
of these algorithms are based on patch information, see the review paper
\cite{raadcisa:hal-01553841}.  Indeed, in order to update the current image, the
patches of the input texture are rearranged in order to generate a new element
(a pixel or a block of pixels) which is locally coherent with the pre-existing
structure. The seminal work of \cite{efros1999texture} paved the way for the use
of such methods and was later extended in
\cite{efros2001image,kwatra2003graphcut} to handle blockwise updates instead of
pixelwise ones.  The statistical model of \cite{efros1999texture} was analyzed
in \cite{levina2006texture} in which the authors reformulate the original
algorithm as a bootstrap scheme. Since these methods duplicate some part of the
input image in order to sample the new image, their innovation capability is
limited as they might suffer from a copy-paste problem. Some recent patch-based
methods deal with these issues by introducing randomness either in the update
\cite{raad2015conditional}. In \cite{galerne2018texture}, starting from a random
microtexture initial image, the patches are rearranged using optimal transport.
\cite{kwatra2005texture} reformulates a patch-based synthesis algorithm as an
optimization procedure, therefore yielding a global texture model defined by the
patch information.  This model was later extended in
\cite{han2006fast,kaspar2015self}.

For the seconde type of approaches, \ie \ parametrics ones, in the early work of
\cite{fournier1982computer}, textures were described as fractional Brownian
motions. It was later noticed in \cite{vanwijk1991spotnoise} that a large class
of textures could be generated using spot noise models whose normalized limit for a large
number of spots is a Gaussian random field with a circulant covariance matrix
\cite{galerne2011random}.  In this works the underlying image model is Gaussian
and the pixel distribution of the output image has, in expectation, the same
moment of order 1 and 2 as the input texture.  It was shown that these
algorithms are able to sample a large class of textures.  Different features are
designed in \cite{perlin1985image}. All the images in this class share the
property that they do not exhibit salient spatial structures implying that the
knowledge of the second-order moments was not enough to reproduce natural
images. In \cite{gagalowicz1986model,cano1988hierarchical,heeger1995pyramid} the
authors remark that structured textures could be obtained using hand-selected
features. These ideas are extended in \cite{portilla2000parametric} by designing
a bank of wavelet features. In addition the authors rewrite parametric
exemplar-based texture synthesis algorithms as maximum entropy problems under
constraints. Similarly, in the seminal paper \cite{zhu1998filters}, the authors
derive the first texture synthesis methodology based on a maximum entropy
approach with constraints in expectation. Later, replacing wavelet features in
\cite{portilla2000parametric} by convolutional neural network features,
\cite{gatys2015texture} was able to obtain striking visual results. Combining
these neural network features with pixel-based constraints yields improvements
of the original work of \cite{gatys2015texture}. For instance, in
\cite{liu2016texture} spectral constraints are added to the convolutional neural
network ones.
\cite{johnson2016perceptual,ulyanov2016texture,ulyanov2017improved} design
parametric methods relying on convolutional neural network features. The
sampling procedure does not depend on any gradient-based method but instead is
performed in a feed-forward manner. More recently, many papers investigate the
use of Generative Adversarial Networks (GAN) \cite{goodfellow2014generative} in
texture synthesis giving promising results
\cite{jetchev2016texture,bergmann2017learning,li2016precomputed,zhou2018non}. In GAN \cite{jetchev2016texture,goodfellow2014generative} the structure constraint is encoded in the loss
on the generator, \ie \ the samples must look like the input image. The innovation
constraint is encoded in the loss on the discriminator, \ie \ the samples must be
diverse enough for the discrimination task to be hard. This formulation can be rewritten as an
optimization problem on probability measures \cite{arjovsky2017wasserstein, biau2018gans}.
In this case, the target distribution on natural textures is the solution of the training problem.
Note that in this case, though the synthesis is performed in a feed-forward manner, the neural network must
be retrained when presented a new class of textures. In the following approaches, the generation is not feed-forward
but the natural texture distribution is designed so that the same features are generically used for all textures.

\subsection{Maximum entropy probability measures}
\label{sec:max_def}

We now define microcanonical and macrocanonical models as introduced by \cite{bruna2018multiscale}.
Let $\mu$ be a probability measure over $(\rset^d, \mcb{\rset^d})$ with $d \in \nset$. The measure $\mu$ will be referred to as the \emph{reference probability measure}. Let $F : \ \rset^d \to \rset^p$ with $p \in \nset$ be a measurable mapping. This mapping will be referred to as the \emph{statistical constraints} of the model. A discussion on the choice of $F$ was conducted in \Cref{sec:feature-models}. From now on we assume that we observe an input texture $x_0$ such that $F(x_0) = 0$. Once again, note that this can always been achieved upon subtracting $F(x_0)$ to the original features. Given a set of features and a target image, we are now interested in the probability distributions which have maximum entropy (innovation constraint) and such that the features are equal to the ones of the target image (structure constraint) \as . In order for the model to be well-defined we replace the maximization of the entropy by a minimization of the Kullback-Leibler divergence $\KL{\cdot}{\mu}$ where $\mu$ is the reference probability measure.
This methodology is called the \emph{microcanonical model}, see \cite{bruna2018multiscale}.

\begin{definition}
  \label{def:micro}
  A probability measure $\pi$ is a \emph{microcanonical model} with respect to the constraints $F$ and the reference measure $\mu$, if $\pi(\{ F =0 \}) = 1$ and if for any probability distribution $\nu$ which satisfies the previous assumption we have $\KL{\pi}{\mu} \leq \KL{\nu}{\mu}$.
\end{definition}

This model  was already considered in \cite{gagalowicz1986model, cano1988hierarchical}. Steerable pyramids constraints are used in \cite{heeger1995pyramid}. In \cite{portilla2000parametric} more than 700 cross-correlations, autocorrelations and statistical moments of wavelet coefficients are selected in order 
to design the statistical constraint features. With the rise of convolutional neural networks, striking visual results were obtained by \cite{gatys2015texture}. This line of work has then been extended by several authors, adding other constraints such as spectrum projection~\cite{liu2016texture}. However the microcanonical model distribution is untractable for most statistical constraints. Therefore in order to sample from this distribution the following heuristics is used: 1) sample a white noise image; 2) perform a gradient descent on the square norm of the features. We shall see in \Cref{sec:links-with-macr} that this algorithm, while providing satisfying visual results for neural network constraints, does not sample from the microcanonical distribution. We now consider a relaxation of the \emph{microcanonical model}: the \emph{macrocanonical model}, see \cite{bruna2018multiscale}. Given a set of features and a target image, we are now interested in the probability distributions which have maximum entropy (innovation constraint) and such that expectation of the features is equal to the features of the target image (structure constraint).

\begin{definition}
  \label{def:macro}
  A probability measure $\pi$ is a \emph{macrocanonical model} with respect to the constraints $F$ and the reference measure $\mu$, if $F$ is $\pi$-integrable, $\pi(F) := \int_{\rset^d} F(x) \rmd \pi (x) = 0$ and if for any probability distribution $\nu$ which satisfies the previous assumption we have $\KL{\pi}{\mu} \leq \KL{\nu}{\mu}$.  
\end{definition}
This maximum entropy model was introduced by \cite{jaynes1957info} and used in texture synthesis by \cite{zhu1998filters}.
Under some assumptions the macrocanonical model can be written in a closed-form. In \cite{zhu1998filters}, the authors used a dictionary of quantiles for various filters, linear and non-linear, as statistical constraint features. This work has been extended, using convolutional neural network features in \cite{lu2015learning}. In \cite{bruna2018multiscale} conditions under which the macrocanonical and the microcanonical models coincide when the dimension of the image space grows towards infinity are identified.

In the next section we study the existence and uniqueness of the macrocanonical model introducing a dual problem and using the point of view of information geometry.

\subsection{Existence, uniqueness and dual formulation}
\label{sec:maxim-entr-prob}
Considering the model defined by \Cref{def:macro} it is natural to ask the following questions:
\begin{enumerate}[label= (\alph*)]
\item When does a macrocanonical model exist? Can we identify explicit conditions for its existence? \label{item:questiona}
\item If such a model exists, is it unique? \label{item:questionb}
\item Can we find closed forms for the probability distribution functions of  macrocanonical models? \label{item:questionc}
\end{enumerate}
We will answer positively to \ref{item:questionb}.
In the case where the problem is non degenerate, \ie \ there exists a probability measure $\nu$ such that $\KL{\nu}{\mu} < +\infty$ and $\nu(F) = 0$,
then the macrocanonical model exists and is given by a parametric measure, answering both \ref{item:questiona} and \ref{item:questionc}. However, checking that the problem is indeed non
degenerate can be as hard as finding the macrocanonical model.

To show the existence of a macrocanonical model we give a dual, convex and finite dimensional formulation. This problem is then solved under the following conditions on $F$ and $\mu$. Let $\upalpha >0$ and $\upbeta >0$:

\begin{assumptionA}[$\upalpha$]
  \label{assum:sub_holder}
 $F$ is continuous and there exists $C_{\upalpha} \geq 0$ with
  $\sup_{x \in \rset^d} \defEns{ \norm{F(x)}(1 + \norm{x}^{\upalpha})^{-1}} \leq C_{\upalpha} <  +\infty$.
\end{assumptionA}
\begin{assumptionA}[$\upbeta$]
  \label{assum:weak}
   There exists $\eta > 0$, such that $\int_{\rset^d} \exp [ \eta \norm{x}^{\upbeta} ] \rmd \mu(x) < +\infty$.
 \end{assumptionA}
Let $\Palpha$ be the set of probability measures over $(\rset^d, \mcb{\rset^d})$ such that $\int_{\rset^d} \norm{x}^{\upalpha} \rmd \pi(x) <+\infty$. We define the set of \emph{admissible probability measures} by $\PalphaF = \ensemble{\pi \in \Palpha}{\pi(F) = 0}$ and consider the following problem:
\begin{align}
  \label{eq:primal}
  \text{minimize } \KL{\pi}{\mu}  \ 
  \text{subject to } \pi \in \PalphaF \eqsp . \tag{P}
\end{align}
We denote $\val(\mathrm{P}) = \inf_{\PalphaF} \KL{\pi}{\mu}$.  It is clear that
any solution of \primal \ is a macrocanonical model with respect to the
constraints $F$ and the reference measure $\mu$.  First, we assert that if the
solution of \primal \ exists and is non-degenerate, \ie \
$\val(\mathrm{P}) < +\infty$, then it is unique.  Let $\pi_1^{\star}$ and
$\pi_2^{\star}$ be two solutions of \primal \ with $\val(\mathrm{P}) < +\infty$
and $\phi(t) = t \log(t)$, defined on $\coint{0,+\infty}$ with the convention
that $\phi(0) = 0$. Since $\val(\mathrm{P}) < +\infty$ we have that
$\pi_1^{\star} \ll \mu$ and $\pi_2^{\star} \ll \mu$. Since $\PalphaF$ is convex,
$\pi^{\star}$ defined for any $x \in \rset^d$ by
$\frac{\rmd \pi^{\star}}{\rmd \mu}(x) = \parenthese{\frac{\rmd
    \pi_1^{\star}}{\rmd \mu}(x) + \frac{\rmd \pi_2^{\star}}{\rmd \mu}(x)}/2$
belongs to $\PalphaF$. Using that $\phi$ is strictly convex we have
  \begin{equation}
    2\val(\mathrm{P}) = \KL{\pi_1^{\star}}{\mu} + \KL{\pi_2^{\star}}{\mu}  = \int_{\rset^d} \defEns{\phi\parenthese{\frac{\rmd \pi_1^{\star}}{\rmd \mu}(x)} + \phi\parenthese{\frac{\rmd \pi_1^{\star}}{\rmd \mu}(x)}} \rmd \mu(x) \geq 2\KL{\pi^{\star}}{\mu} \eqsp ,
  \end{equation}
  with equality if and only if for $\mu$ almost every $x \in \rset^d$ we have $\frac{\rmd \pi_1^{\star}}{\rmd \mu}(x) = \frac{\rmd \pi_2^{\star}}{\rmd \mu}(x)$. As a consequence, since $\KL{\pi^{\star}}{\mu} = \val(\mathrm{P})$,  $\pi_1^{\star} = \pi_2^{\star}$.

We now introduce another optimization problem.
\begin{align}
  \label{eq:dual}
  &\text{maximize } \inf_{\pi \in \Palpha}\defEnsLigne{ \KL{\pi}{\mu} + \langle \theta, \pi(F) \rangle } \ 
  \text{subject to } \theta \in \Theta_F \eqsp , \tag{Q}
\end{align}
with
\begin{equation}
\label{eq:def_theta_F}
  \Theta_F = \ensemble{\theta \in \rset^p}{\int_{\rset^d} \exp\parentheseDeux{-\langle \theta, F(x) \rangle} \rmd \mu (x) < +\infty} \eqsp . \end{equation}
Using Hölder's inequality, one can show that $\Theta_F$ is convex. In addition, if \Cref{assum:sub_holder}($\upalpha$) and \Cref{assum:weak}($\upalpha$) hold with $\upalpha >0$ then $\cball{0}{\eta / C_{\upalpha}} \subset \interior(\Theta_F)$.  Similarly to \primal , we denote $\val(\mathrm{Q}) = \sup_{\Theta_F} \inf_{\pi \in \Palpha}\defEnsLigne{ \KL{\pi}{\mu} + \langle \theta, \pi(F) \rangle }$. 
Introducing the Lagrangian $\Lag(\pi, \theta) = \KL{\pi}{\mu} + \langle \theta, \pi(F) \rangle$ defined over $\Palpha \times \Theta_F$, we have
\begin{equation}
  \label{eq:lagrangian}
  \val(\mathrm{Q}) = \sup_{\Theta_F} \inf_{\Palpha} \Lag \leq  \inf_{\Palpha} \sup_{\Theta_F} \Lag \leq \val(\mathrm{P}) \eqsp .
\end{equation}
We denote $\mathrm{d}(\mathrm{P}, \mathrm{Q})$ the \emph{duality gap} $\mathrm{d}(\mathrm{P}, \mathrm{Q}) = \val(\mathrm{P}) - \val(\mathrm{Q})$ with the convention that $\infty - \infty = 0$.
Let the log-partition function $L: \ \Theta_F \to \rset$ be given for any $\theta \in \Theta_F$ by
\begin{equation}
  \label{eq:log_partition}
  L(\theta) = \log \parentheseDeux{ \int_{\rset^d} \exp[-\langle \theta, F(x) \rangle] \rmd \mu(x) } \eqsp .
\end{equation}
We also define for any $\theta \in \Theta_F$, the probability measure $\pi_{\theta}$ whose density with respect to $\mu$ is given for any $x \in \rset^{\dim}$ by
\begin{equation}
  \label{eq:gibbs_measure}
  \frac{\rmd \pi_{\theta}}{\rmd \mu}(x) = \exp[- \langle \theta, F(x) \rangle - L(\theta)] \eqsp .
\end{equation}
Using \Cref{prop:var_dual},   \dual \ is equivalent to 
\begin{align}
  &\text{minimize } L(\theta) \  \text{subject to } \theta \in \Theta_F \eqsp . \tag{Q'}
\end{align}
More precisely $\theta^{\star}$ is a solution of \dual \ if and only if it is a solution of $(\mathrm{Q'})$.
Assuming that $\rset^{\dim}$ is replaced by $\msx$, a finite set, in \primal , that $\mu$ is the counting measure on $\msx$ and identifying the probability distributions over $\msx$ and the $|\msx|$-dimensional simplex, we obtain that the primal problem \primal \ can be rewritten as a finite dimensional convex optimization problem under linear constraints. In this case $\Theta_F = \rset^p$ and we have that \dual \ is the Lagrangian dual problem of \primal , where $\theta$ are the Lagrange multipliers.
In this setting and under some identifiability conditions, the solution of \primal \ is given by some Gibbs measure, \ie \ for any $x \in \msx$, $\pi_{\theta^{\star}}(x) = \exp\parentheseDeux{- \langle \theta^{\star}, F(x) \rangle - L(\theta^{\star})}$ with $\theta^{\star} \in \rset^p$, see \cite{mumford2010pattern}. In what follows we investigate how the results obtained in the discrete setting can be extended to the general case. The next proposition is an extension of \cite[Theorem 3.1]{csiszar1975divergence} in the case where $F$ is not bounded, which characterizes the solutions of \primal . Under \Cref{assum:weak}($\upalpha'$) with $\upalpha' > \upalpha$ the existence of a solution of \primal \ is ensured as soon as the set of admissible probability
measures is not empty.
\begin{proposition}
  \label{prop:primal_parametric}
  Assume \tup{\Cref{assum:sub_holder}($\upalpha$)} with $\upalpha \geq 0$.  The following holds.
  \begin{enumerate}[leftmargin=0.5cm, label= (\alph*)]
  \item  If there exists a solution  of \primal \ such that $v(\mathrm{P}) < +\infty$ then there exists  $\theta^{\star} \in \rset^p$ such that the solution of \primal \ is given by $\pi_{\theta^{\star}}$ defined for $\mu$ almost any $x \in \rset^d$ by
  \begin{equation}
    \label{eq:primal_solution}
    \begin{cases} \frac{\rmd \pi_{\theta^{\star}}}{\rmd \mu}(x) =  \left. \exp\parentheseDeux{- \langle \theta^{\star}, F(x) \rangle} \middle/ \int_{\rset^d \backslash \msn} \exp\parentheseDeux{- \langle \theta^{\star}, F(y) \rangle} \rmd \mu(y) \right. & \text{if } x \notin \msn \eqsp ,\\
       \frac{\rmd \pi_{\theta^{\star}}}{\rmd \mu}(x) = 0 & \text{otherwise} \eqsp ,
      \end{cases}
  \end{equation}
  with $\msn \in \mcb{\rset^d}$ such that for all $\pi \in \PalphaF$ with $\KL{\pi}{\mu} < +\infty$, $\pi(\msn) = 0$. If there exists $\pi \in \PalphaF$ with $\KL{\pi}{\mu} < +\infty$ such that $\mu \ll \pi$ then $\theta^{\star}$ is a solution of \dual , $\pi_{\theta^{\star}}$ given by \eqref{eq:gibbs_measure} is a solution of \primal \  and $v(\mathrm{Q}) = v(\mathrm{P})$, where $\val(\mathrm{Q})$ and $\val(\mathrm{P})$ are given in \eqref{eq:lagrangian}.

\item   Assume \tup{\Cref{assum:weak}($\upalpha'$)} with $\upalpha' > \upalpha$.  There exists $\pi_{\theta^{\star}}$ solution of \primal \ with $\val(\mathrm{P}) < +\infty$ if and only if there exists $\pi \in \PalphaF$ such that $\KL{\pi}{\mu} < +\infty$.

  \end{enumerate}

\end{proposition}

\begin{proof}
The proof is postponed to \Cref{prop:primal_parametric_proof}.
\end{proof}
Note that this result was extended to a large class of divergences, see
\cite[Theorem 3.8]{amari2000methods}.  In \cite{csiszar1975divergence,
  csiszar1984sanov, csiszar1996maxent} different sufficient conditions for
solving \primal \ are derived. In particular, in \cite[Theorem
3.3]{csiszar1975divergence}, the existence of a solution to \primal \ is shown
if $F$ is measurable, \Cref{assum:weak}($\upalpha$) holds and there exists
$\msu \subset \rset^{p}$ open with $0 \in \msu$ such that for any $a \in \msu$
there exists $\nu_a$ with $\KL{\nu_a}{\mu} < +\infty$ and $\nu_a(F)= a$.  In
\cite[Theorem 3]{topose1979information}, the authors show similar results in the
case where the Kullback-Leibler divergence is given by the true entropy.  In
practice it is difficult to check the conditions of
\Cref{prop:primal_parametric}.  Indeed, even if \Cref{assum:weak}($\upalpha'$)
holds with $\upalpha ' > \upalpha$ it is not trivial to find an element
$\pi \in \PalphaF$ such that $\KL{\pi}{\mu} <+\infty$.  We now turn to the dual
formulation \dual \ which is easier to deal with since it is a finite
dimensional and convex optimization problem. Under
\tup{\Cref{assum:weak}}($\upalpha$), any stationary point of the log-partition
function yields a solution of the primal formulation.

\begin{proposition}
  \label{prop:existence_P}
  Assume \tup{\Cref{assum:sub_holder}($\upalpha$)} and
  \tup{\Cref{assum:weak}($\upalpha$)} with $\upalpha \geq 0$.  Then,
  $L \in \rmC^{\infty}(\interior(\Theta_F))$, where $L$ is given in
  \eqref{eq:log_partition}, and for any $\theta \in \interior(\Theta_F)$, $\nabla L(\theta) = \pi_{\theta}(F)$
  with $\pi_{\theta}$ given by
  \eqref{eq:gibbs_measure}. In addition, if there exists
  $\theta^{\star} \in \interior(\Theta_F)$ such that
  $\nabla L(\theta^{\star}) = 0$, then $\pi_{\theta^{\star}}$ is the solution of
  \primal .
\end{proposition}

\begin{proof}
  The proof is postponed to \Cref{prop:existence_P_proof}.
\end{proof}
A similar result was derived in \cite{jupp1983note}, in the case where $\mu$ is no longer a probability measure but only sigma-finite. The study of the log-partition function $L$ is fundamental in information theory. Its main properties are summarized in \cite[Chapter 9]{barndorff2014information} and in \cite{csizar1999MEM} the authors show that the log-partition is a special case of a more general information theoretical model, where the Kullback-Leibler divergence is replaced by another notion of entropy.
In fact, the log-partition function  is a convex function, hence any stationary point is a global minimizer. We exploit this fact in the following proposition.
\begin{proposition}
  \label{prop:existence_Q}
  Assume \tup{\Cref{assum:sub_holder}($\upalpha$)}, \tup{\Cref{assum:weak}($\upalpha'$)} with $\upalpha \geq 0$, $\upalpha' > \upalpha$. 
  \begin{enumerate}[leftmargin=0.5cm, label= (\alph*)]
  \item \label{item:item_a_info} If for any $\theta \in \rset^p$ with $\| \theta \| = 1$, we have
    \begin{equation} \label{eq:identifiability} \mu\parenthese{\ensemble{x \in \rset^d}{\langle F(x), \theta\rangle < 0}} >0 \eqsp , \end{equation}
     then $\theta^{\star}$, solution of \dual , exists and $\pi_{\theta^{\star}}$ given by \eqref{eq:gibbs_measure}
 is the solution of \primal . 
\item \label{item:item_b_info}  In particular, \eqref{eq:identifiability} is satisfied if $\mu(\msa)>0$ for any non-empty  open set $\msa \subset \rset^d$, $F$ is continuous and there exists $x \in F^{-1}(\{0\})$ such that $F$ is differentiable at $x$ and we have $\det( \rmD F(x) \rmD F(x)^{\transpose}) > 0$.
\end{enumerate}
\end{proposition}
\begin{proof}
The proof is postponed to \Cref{prop:existence_Q:proof}.
\end{proof}

In \Cref{sec:gaussian-features}, we apply \Cref{prop:existence_Q} for some example of  functions $F$ of the form \eqref{eq:neural_network}. However, note that it does not apply to the case where $\varphi$ is the rectified linear unit (RELU), \ie \ for any $t \in \rset$, $\varphi(t) = \max(0, t)$. Still we give a similar result, valid for RELU, in \Cref{prop:neural_network}. In \cite[Theorem 3.5]{ishwar2005existence} the authors derive analogous results in the case where $\mu$ is the Lebesgue measure, \ie \ in the case where the Kullback-Leibler is replaced by the true entropy.

We are also able to show that the condition \eqref{eq:identifiability} is almost necessary in the next result.

\begin{proposition}
  \label{prop:necessary_Q}
  Assume \tup{\Cref{assum:sub_holder}($\upalpha$)}, \tup{\Cref{assum:weak}($\upalpha'$)} with $\upalpha \geq 0$, $\upalpha' > \upalpha$ and that there exists $\theta \in \rset^p$ with $\| \theta \| = 1$ such that
  \begin{equation} \mu\parenthese{\ensemble{x \in \rset^d}{\langle F(x), \theta\rangle \leq 0}} =0 \eqsp . \end{equation}
  Then $\updelta_{x_0}$ solves \primal \ and $v(\mathrm{P}) = +\infty$.
\end{proposition}

\begin{proof}
  The proof is postponed to \Cref{prop:necessary_Q:proof}.
\end{proof}

We have seen that, under assumptions on the reference distribution $\mu$ and the statistical constraints $F$, the macrocanonical model is
the distribution $\pi_{\theta^{\star}}$ with the following parametric form: for any $x \in \rset^d$, $(\rmd \pi_{\theta^{\star}} / \rmd \mu)(x) = \exp\parentheseDeux{-\langle \theta^{\star}, F(x) \rangle -L(\theta^{\star})}$, with $\theta^{\star}$ which satisfies $\theta^{\star} \in \argmin_{\theta \in\Theta_F} L(\theta)$. These exponential families can also be retrieved in the following Bayesian framework. Assume that the likelihood of texture images  associated with parameters $\theta \in \rset^p$ is given for any $x \in \rset^{\dim}$ by 
\begin{equation} p(x | \theta) = \left. \exp\parentheseDeux{-\langle \theta, F(x) \rangle} \middle/ \int_{\rset^d} \exp\parentheseDeux{-\langle \theta, F(y) \rangle} \rmd \mu(y)\right. \eqsp.\end{equation}
Assume that $x_0$ is a sample from this distribution and that $F(x_0) = 0$. Using Bayes' formula we obtain that for any $\theta \in \rset^p$
\begin{equation}
  p(\theta | x_0) = \frac{p(x_0 | \theta)p(\theta)}{p(x_0)} \eqsp \propto \left. p(\theta) \middle/ \int_{\rset^d} \exp\parentheseDeux{-\langle \theta, F(y) \rangle} \rmd \mu(y)\right. \eqsp .
\end{equation}
In order to compute the maximum a posteriori estimator $\theta_{\MAP}$ we need to set a prior distribution $p(\theta)$.
Choosing the non-informative improper prior $p(\theta) = \1_{\Theta_F}(\theta)$ we get that
\begin{equation}
  \label{eq:egalite_map_entropy}
  \theta_{\MAP} \in \argmin_{\theta \in \Theta_F} \log \parentheseDeux{ \int_{\rset^d} \exp\parentheseDeux{-\langle \theta, F(x) \rangle} \rmd \mu(x)} \eqsp ,
\end{equation}
which corresponds to the dual formulation \dual . However, other prior distributions could be considered, yielding hierarchical Bayesian models \cite{vidal2019empirical}.

\subsection{Illustrative examples}
\label{sec:exampl-bayes-appr}
\paragraph{Gaussian features}
\label{sec:gaussian-features}
We consider $F: \ \rset^{\dim} \to \rset^{\dim}$, \ie \ $\dim = \dimp$, given for any $x \in \rset^{\dim}$ by $F(x) = x * \check{x} - x_0 * \check{x}_0$ and $\mu$ a Gaussian probability measure with zero mean and diagonal covariance matrix with diagonal coefficient $\sigma^2$ with $\sigma >0$. Assume in addition that for any $\ell \in \{0, \dots, d-1 \rbrace$, $\fourier(x_0)(\ell) \neq 0$. We have that, \Cref{assum:sub_holder}($\upalpha$) and \Cref{assum:weak}($\upalpha$) hold with $\upalpha =2$ and $\eta \in \ooint{0, (2\sigma^2)^{-1}}$. Using the Fourier-Plancherel formula we get that for any $x \in \rset^d$ and $\theta = (\theta(0), \dots, \theta(d-1)) \in \rset^d$,
\begin{align}
  \langle \theta, F(x) \rangle + \norm{x}^2/(2 \sigma^2) &= d^{-1} \parentheseDeux{ \langle \fourier(\theta), \abs{\fourier(x)}^2- \abs{\fourier(x_0)}^2 \rangle + \norm{\fourier(x)}^2/(2 \sigma^2)} \\ &= \sum_{i=0}^{d-1} \defEns{d^{-1}(\fourier(\theta)(i) + (2\sigma^2)^{-1})\abs{\fourier(x)(i)}^2} - d^{-1} \langle \fourier(\theta),  \abs{\fourier(x_0)}^2 \rangle \eqsp .
\end{align}
This implies that, $\Theta_F = \fourier^{-1}\parentheseDeuxLigne{\Re^{-1}\defEnsLigne{\ooint{-(2\sigma^2)^{-1},+\infty}^d}} \cap \rset^d$. In addition, for any $\theta \in \Theta_F$ with $\fourier(\theta) \in \rset^d$ and $X$ distributed according to $\pi_{\theta}$, we obtain that $\fourier(X)$ is a $d$-dimensional complex Gaussian random variable on $\fourier(\rset^d)$ with zero mean and diagonal covariance matrix with diagonal coefficients given by $d(\fourier(\theta) + (2\sigma^2)^{-1})^{-1}/2$. Similarly, we obtain that $X$ distributed according to $\pi_{\theta}$ is a $d$-dimensional Gaussian random variable with zero mean and invertible circulant covariance matrix $\covmat_{\theta} \in \mathrm{S}_d(\rset)$ whose inverse is given for any $i,j \in \{0, \dots, d-1\}$ by
\begin{equation}
  \label{eq:cov}
  \covmat_{\theta}^{-1}(i,j) = 2\theta(i-j) + \sigma^{-2} \eqsp .
\end{equation}
Note that in this case, since $\fourier(\theta) \in \rset^{\dim}$, $\theta = \check{\theta}$. Let
\begin{equation} \label{eq:theta_star_hat} \fourier(\theta^{\star}) = (d \absLigne{\fourier(x_0)}^{-2} - \sigma^{-2})/2 \in \fourier(\Theta_F) \eqsp . \end{equation}
In this case, for any $X$ distributed according to $\pi_{\theta^{\star}}$ we obtain that $\fourier(X)$ is a $d$-dimensional Gaussian random variable with zero mean and diagonal covariance matrix with diagonal coefficients given by $\abs{\fourier(x_0)}^2$. Therefore we have
\begin{equation}
  \expe{X * \check{X}} = \fourier^{-1}\parenthese{\expe{\abs{\fourier(X)}^2}} = \fourier^{-1}(\abs{\fourier(x_0)}^2) = x_0 * \check{x}_0 \eqsp ,
\end{equation}
which implies that $\pi_{\theta^{\star}}(F) = 0$ and that $\theta^{\star} = \fourier^{-1}( d \absLigne{\fourier(x_0)}^{-2} - \sigma^{-2})/2$ is a solution of \dual , since $\theta^{\star} \in \Theta_F$. Using \Cref{prop:existence_P} we get that $\pi_{\theta^{\star}}$ is a solution of \primal . Therefore the solution of \primal \ is the Gaussian probability measure with zero mean which satisfies  for any $m, n \in \lbrace 0, \dots, d-1 \rbrace$, $\expe{X(m)X(n)} = d^{-1} (x_0 * \check{x}_0)(m-n)$. This distribution is invariant by spatial translation, \ie \ denoting $\tau: \ \rset^{\dim} \to \rset^{\dim}$, defined for any $x \in \rset^{\dim}$ and $i \in \{1, \dots, \dim - 1\}$ by $\tau(x)(i) = x(i+1)$ and $\tau(x)(\dim-1) = x(0)$, we have for any $\msa \in \mcb{\rset^{\dim}}$
  \begin{equation}
    \label{eq:translation_invariance}
     \pi^{\star}(\msa) = \pi^{\star}(\tau(\msa)) \eqsp .
  \end{equation}
Note that the distribution of the random variable $X$ is the same as the one of $d^{-1/2} (x_0 * Z)$ with $Z$ a $d$-dimensional Gaussian random variable with zero mean and covariance matrix identity, see \cite{galerne2011random}.

\paragraph{CNN Features}
We now turn to the case where $F$ is given by \eqref{eq:neural_network} for
$\bfj \subset \{1,\ldots,M\}$ and the reference measure $\mu$ is a Gaussian
probability measure with zero mean and symmetric positive covariance
matrix. \Cref{assum:sub_holder}(1) and \Cref{assum:weak}($\upalpha'$) hold for
any $\upalpha' \in \coint{0, 2}$. Note that in the case, where $\varphi$ is
differentiable, the results of \Cref{prop:existence_P} hold assuming that there
exists a point $x \in \rset^d$ such that $F(x) = F(x_0)$ and $\rmd F(x)$ is
surjective.  In the case where for all $t \in \rset$, $\varphi(t) = \max(0, t)$
we can define a certificate ensuring the existence of a macrocanonical model.
We now introduce some preliminary notations. 
Let $(A_j)_{j \in \{1, \dots, M\}}$  be given by \eqref{eq:def_cnn} and $(\scrG_j)_{j \in \{1, \dots, M\}}$ be given by \eqref{eq:neural_network}. For any $j \in \{1, \dots, M\}$, let $A_{j, +} \in \affinop_{c_j \times n_j, c_{j-1} \times n_{j-1}} (\rset)$ defined by
\begin{equation}
  \label{eq:linear_plus}
  A_{j, +} = D_j(x_0) A_j \eqsp ,
\end{equation}
 with  $D_j(x_0) \in \rset^{c_j \times n_j} \times \rset^{c_j \times n_j}$ a diagonal matrix such that for any $i \in \{1, \dots, c_j \times n_j\}$, $D_j(x_0)(i,i) = \1_{\ooint{0,+\infty}}(\scrG_j(x_0)(i))$. Note that we have for any $j \in \{1, \dots, M\}$, $\tilde{A}_{j,+} =  D_j(x_0) \tilde{A}_j$. Let $(\tilde{v}_{j,k})_{j \in \calJ, k \in \{1, \dots, c_j \}}$ such that for any $j, j' \in \calJ$ and $k, k' \in \{1, \dots, c_j \}$, $\tilde{v}_{j, k} \in \rset^{c_j \times n_j}$ and
 \begin{equation}
    \begin{aligned}
      &\tilde{v}_{j,c}(j', c') = 0 \eqsp , \text{if } c' \neq c \eqsp , \\
      &\tilde{v}_{j,c}(j', c') = n_{j}^{-1} \eqsp  \text{otherwise} \eqsp .
      \end{aligned}
    \end{equation}
Note that for any $x \in \rset^{\dim}$, $F(x) = (\tilde{v}_{j,k}^{\transpose} \scrG_j(x))_{j \in \calJ, k \in \{1, \dots, c_j\}}$.
In addition, define for any $j \in \calJ$ and $k \in \{1, \dots, c_j \}$
  \begin{equation}
    v_{j,k} = \tilde{A}_{1,+}^{\transpose} \dots \tilde{A}_{j,+}^{\transpose} \tilde{v}_{j,k} \eqsp .
  \end{equation}
  
Note that for any $j \in \calJ$ and $k \in \{1, \dots, c_j\}$, $v_{j,k} \in \rset^{\dim}$. The next proposition highlights the role of $(v_{j,k})_{j \in \calJ, k \in \{1, \dots, c_j \}}$ as a certificate for the existence of the macrocanonical model.
\begin{proposition}
  \label{prop:neural_network}
  Assume that \tup{\Cref{assum:weak}($\upalpha$)} holds with $\upalpha > 1$, that $\mu(\msa) >0$ for every open set $\msa \subset \rset^{\dim}$ with $\msa \neq \emptyset$ and that $F$ is given by \eqref{eq:neural_network} with $\varphi(t) = \max(0,t)$ for any $t \in \rset$, $x_0 \in \rset^d$ and $\bfj \subset \{1,\ldots,M\}$.
Moreover assume that for any $j \in \{1, \dots, M\}$ and $k \in \{1, \dots, c_j \times n_j \}$
\begin{equation}
  \label{eq:condition_x0}
    e_{k,j}^{\transpose} A_j \scrG_{j-1}(x_0) \neq 0 \eqsp ,
  \end{equation}
  where for any $j \in \{1, \dots, M\}$, $(e_{k,j})_{k \in \{1, \dots, c_j \times n_j\}}$ is the canonical basis of $\rset^{c_j \times n_j}$. In addition, assume that the family $(v_{j,c})_{j \in \calJ, c \in \{1, \dots, c_j \}}$ is linearly independent. Then there exists a solution $\theta^{\star}$ to \dual \ and $\pi_{\theta^{\star}}$ defined given by \eqref{eq:gibbs_measure} is the solution of \primal .
\end{proposition}
\begin{proof}
  The proof is postponed to \Cref{prop:neural_network:proof}.
\end{proof}

Note that since $(v_{j,c})_{j \in \calJ, c \in \{1, \dots, c_j\}}$ has a closed form, the independence condition in \Cref{prop:neural_network} can be explicitly checked given a trained neural network. The proof of \Cref{prop:neural_network} exploits the fact that neural networks are locally linear under mild assumptions. Finally, since the family $(v_{j,c})_{j \in \calJ, c \in \{1, \dots, c_j \}}$ has cardinality $p = \sum_{j \in \calJ} c_j$, \Cref{prop:neural_network} never applies when $p > d$.

Finally, note that here we propose to study Gaussian features and CNN features but other choices are possible such
as wavelet features or scattering transform features \cite{mallat2012group, bruna2013invariant}.

  \section{Sampling from macrocanonical models}
\label{sec:sampl-from-macr}

In this section, our objective is twofold. First, we want to find a sequence $(\theta_n)_{n \in \nset}$ which converges \as \ to $\theta^{\star}$, the solution of \dual . Second, we aim at sampling from the macrocanonical model $\pi_{\theta^{\star}}$ defined by \eqref{eq:gibbs_measure}. We present a Stochastic Approximation (SA) algorithm addressing simultaneously these two problems in \Cref{sec:stoch-unadj-lang}. Our main result are summarized in \Cref{sec:main-results}. In \Cref{sec:links-with-macr} we draw a qualitative link between macrocanonical and microcanonical models.

\subsection{The Stochastic Optimization via Unadjusted Langevin method}
\label{sec:stoch-unadj-lang}
In this section, we introduce a SA algorithm to minimize $L$. In \Cref{sec:stoch-appr}, we recall some basic facts on convex convex stochastic optimization and in \Cref{sec:soul-algorithm} we present the Stochastic Optimization with Unadjusted Langevin (SOUL) Algorithm applied to the maximum entropy problem.

\subsubsection{Stochastic approximation}
\label{sec:stoch-appr}

Let $\msk \subset \interior(\Theta_F)$ be a non-empty compact convex set such that $\msk \cap \argmin_{\Theta_F} L \neq \emptyset$ with $L$ the log-partition given in \eqref{eq:log_partition}. Since $\theta \mapsto L(\theta)$ is a convex mapping we obtain that the sequence $(\tilde{\theta}_n)_{n \in \nset}$ defined by $\tilde{\theta}_0 \in \msk$ and for any $n \in \nset$, $\tilde{\theta}_{n+1} = \Pi_{\msk} [ \tilde{\theta}_n - \delta \nabla L(\tilde{\theta}_n)]$ where $\delta >0$ is a stepsize and $\Pi_{\msk}$ is the projection onto $\msk$, converges under mild assumptions to $\theta^{\star} \in \argmin_{\Theta_F} L$, since $L$ is convex, see \cite{nesterov2013introductory}. However, for any $\theta \in \Theta_F$, $\nabla L(\theta) = \pi_{\theta}(F)$ and evaluating this quantity is generally unfeasible. In order to approximate this high-dimensional integral, nested Laplace approximations were recently introduced \cite{schelldorfer2014lasso,ogden2015sequential}. However, it is difficult to give upper-bounds on the bias of these approximations. In what follows, we rely on Monte-Carlo approximations for which we derive explicit upper-bounds on the bias. More precisely, assuming that it is possible to sample from $\pi_{\theta}$ then $\nabla L(\theta)$ can be approximated by $N^{-1} \sum_{k=1}^N F(\X_k)$, where $(\X_k)_{k \in \{1, \dots, N \}}$ are independently sampled from $\pi_{\theta}$. In most of our applications it is not feasible to sample directly from $\pi_{\theta}$, but we can construct Markov chains for which $\pi_{\theta}$ is an invariant probability measure. Then, under assumptions and using classical Markov chain theory, it can be shown that $N^{-1} \sum_{k=1}^N F(\X_k)$ converges \as \ to $\pi_{\theta}(F)$, \cite[Theorem 11.3.1]{douc2018markov}. Such examples of Markov chains include the Metropolis-Hastings algorithm \cite{hastings1970monte}, which uses a rejection step. In a high-dimensional setting, the acceptance ratio can be extremely low and the proposed new iterate is then always discarded. Hence, we focus on Markov chains without rejection step. In this scenario, 
$\pi_{\theta}$ is not an invariant measure of the Markov chain in general. However, for an appropriate choice of Markov chain, the bias between its actual invariant probability measure and the target probability measure $\pi_{\theta}$ can be explicitly controlled.

\subsubsection{SOUL algorithm}
\label{sec:soul-algorithm}
First, we consider some regularity assumption on the measure $\mu$ with respect to the Lebesgue measure.
\begin{assumptionB}
  \label{assum:equi_meas}
  $\mu \ll \Leb$ and its Radon-Nikodym density \wrt \ to the Lebesgue measure is given for almost every
  $x \in \rset^d$ by $\exp[-r(x)] / \int_{\rset^d} \exp[-r(y) ] \rmd y$, where $r : \ \rset^{\dim} \to \rset$ is measurable.
\end{assumptionB}
Let $\theta = (\theta(0), \dots, \theta(p-1)) \in \rset^{p}$ and consider the overdamped unadjusted Langevin algorithm, called \ULA \ in \cite{roberts:tweedie:1996}, defined by $(\tX_n)_{n \in \nset}$ with $\tX_0 = x \in \rset^d$ and for any $n \in \nset$
\begin{equation}
  \label{eq:ula}
  \tX_{n+1} = \tX_n - \gamma \parenthese{\sum_{i=1}^p \theta(i) \nabla F_i(\tX_n) + \nabla \Reg(\tX_n)}  + \sqrt{2 \gamma} \noise_{n+1} \eqsp ,
\end{equation}
where $\gamma >0$ is a stepsize and $(\noise_n)_{n \in \nset}$ is a sequence of independent $d$-dimensional Gaussian random variables with zero mean and identity covariance matrix. This algorithm is the Euler-Maruyama discretization of the overdamped Langevin stochastic differential equation \cite{durmus2017fast} for which $\pi_{\theta}$ is the invariant probability measure. The study of the geometric convergence of this Markov chain under various metrics was conducted in \cite{durmus2017unadjusted,durmus2017fast,dalalyan2017theoretical}. In \cite{debortoli2018souk} a SA scheme, the Stochastic Optimization with Unadjusted Langevin (SOUL) Algorithm, is proposed in order to construct a sequence $(\theta_n)_{n \in \nset}$ such that $(\theta_n)_{n \in \nset}$ converges \as \ and in $\mathrm{L}^1$ to some $\theta^{\star} \in \argmin_{\Theta_F \cap \msk} L$.
Let $\theta_0 \in \msk$ and $\X_0^0 \in \rset^d$. For any $n \in \nset$ and $k \in \lbrace 0, \dots, m_n - 1 \rbrace$ we define
\begin{equation}
  \label{eq:souk}
  \begin{aligned}
    &\X_{k+1}^n = \X_k^n  - \gamma_n \parenthese{\sum_{i=1}^p \theta_n(i) \nabla F_i(\X_k^n) + \nabla \Reg(\X_k^n)} + \sqrt{2 \gamma_n} \noise_{k+1}^n \quad \text{and } \X_0^n = \X_{n-1}^{m_{n-1}} \eqsp ;     \\
    &\theta_{n+1} = \Pi_{\msk} \parentheseDeux{ \theta_n + \delta_{n+1} m_{n}^{-1} \sum_{k=1}^{m_n} F(\X_k^n)} \eqsp , 
  \end{aligned}
\end{equation}
where $(\delta_n)_{n \in \nsets}$ and $(\gamma_n)_{n \in \nset}$ are sequence of positive stepsizes and the sequence  $(\noise_k^n)_{n \in \nset, k \in \lbrace 1, \dots, m_n \rbrace}$ is a sequence of independent $d$-dimensional Gaussian random variables with zero mean and covariance identity. By convention, $\X_{-1}^{m_{-1}} = \X_0^0$.  The condition $\X_0^n = \X_{n-1}^{m_n-1}$ for all $n \in \nset$ is referred to as a warm-start condition.

To illustrate the expected behavior of the proposed SOUL algorithm \eqref{eq:souk}, we consider the toy example where $x_0 \in \rset$, $F(x) = x^2 - 4$ and $F(x_0) = 0$. In this case the maximum entropy distribution is given by the Gaussian distribution with zero mean and variance $4$, see \cite[Section 4.6.2]{mumford2010pattern}. One shows that, the optimal weight $\theta^{\star}$ is given by $\theta^{\star} = 1/8$. In \Cref{fig:gauss}, we experimentally check the convergence of $(\theta_n)_{n \in \nset}$. We set $r(x) = 0$ and observe that the sequence $(\theta_n)_{n \in \nset}$ as well as the sequence $(\bar{\theta}_n)_{n \in \nset}$ define for any $n \in \nset$ by
\begin{equation}
  \label{eq:avg}
  \bar{\theta}_n = 0 \ \text{if } n < N \eqsp , \qquad \bar{\theta}_n = \left . \sum_{k=N}^n \delta_k \theta_k \middle / \sum_{k=N}^n \delta_k  \right . \ \text{otherwise} \eqsp ,
\end{equation}
where $N \in \nset$ is a fixed parameter, converge to $\theta^{\star}$.
\begin{figure}[h!]
  \centering
    \subfloat[]{
      \begin{tikzpicture}[spy using outlines={rectangle, yellow,magnification=2.2, connect spies}, scale = 0.75]
  \begin{axis}[grid=major,no markers,domain=-5:5,enlargelimits=false]
    \input{./data/gaussian_simple_weight.tex}
    \input{./data/gaussian_simple_weight_avg.tex}
    \input{./data/gaussian_simple_true_theta.tex}
    \addlegendentry{$(\theta_n)_{n \in \nset}$}
    \addlegendentry{$(\bar{\theta}_n)_{n \in \nset}$}
    \addlegendentry{$1/8$}
    \coordinate (spypoint) at (320, 80);
    \coordinate (spyviewer) at (280, 20);
    \spy[width=1.3cm,height=1.3cm] on (spypoint) in node [fill=white] at (spyviewer);
    \end{axis}
  \end{tikzpicture}} \hfill
\subfloat[]{\includegraphics[width=0.5\linewidth]{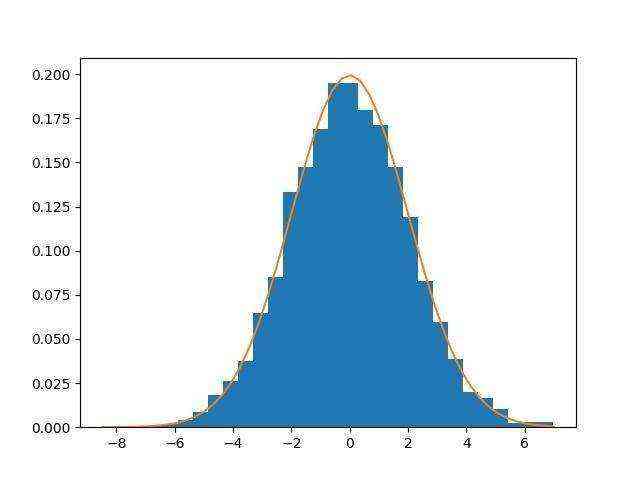}}
\caption{\figuretitle{Variance estimation} In (a), the sequence of parameters $(\theta_n)_{n \in \nset}$ (blue curve) and the sequence of average parameters $(\bar{\theta}_n)_{n \in \nset}$ (red curve) converge to the optimal value $\theta^{\star} = 1/8$. In (b) we illustrate empirically the convergence of the sequence $(\X_0^n)_{n \in \nset}$ to the Gaussian distribution with zero mean and variance $4$ (orange curve) by plotting its histogram. In this example $N = 0$, $\delta_n = 0.1 \times n^{-0.7}$, $\gamma_n = 0.1 \times n^{-0.3}$ and $m_n = 10 \times \lceil n^{0.6} \rceil$.}
  \label{fig:gauss}
\end{figure}
We now state our main results on the dependency on the dimension in the explicit error in SOUL.

\subsection{Main results}
\label{sec:main-results}

In \cite{debortoli2018souk}, the convergence of the sequence
$(\theta_n)_{n \in \nset}$ is studied under general assumptions. In this
section, we complement these results in our setting. In particular, we show that
the error bound in $\mathrm{L}^1$ norm between $L(\theta_n)$ and
$L(\theta^{\star})$ is upper bounded by a constant which depends polynomially in
the dimension $d$.  Let $\upalpha \geq 1$, we consider the following assumption:

\begin{assumptionB}[$\upalpha$]
  \label{assum:existence_compact}
There exists $\msk \subset \rset^p$ such that:
\begin{enumerate}[label=(\alph*), leftmargin=1cm]
\item \label{item:a_compact} $\msk$ is a non-empty convex compact set with $\msk \subset \interior(\Theta_F)$
  and we denote $\Rtheta >0$ such that $\msk \subset \cball{0}{\Rtheta}$ ;
\item $F$ is differentiable and there exists $\LipF \geq 0$ such that for any
  $x,y \in \rset^d$
  \begin{equation}
    \norm{F(x) - F(y)} \leq \LipF (1 + \norm{x}^{\upalpha-1} + \norm{y}^{\upalpha-1})\norm{x - y} \eqsp ;
  \end{equation}
\item \label{item:c_existence} there exists $\theta^{\star} \in \msk$ solution of \dual .
\end{enumerate}
\end{assumptionB}
Under \Cref{assum:existence_compact}($\upalpha$) and
\tup{\Cref{assum:weak}($\upalpha$)} with $\upalpha \geq 1$, a solution of
\primal \ exists and is given by \eqref{eq:gibbs_measure}, see
\Cref{prop:existence_P}. In addition, $L$ is differentiable on $\msk$ and we
show in \Cref{prop:existence_P} that
$\nabla L \in \rmc^1(\rset^{\dim}, \rset^{\dim})$, hence Lipschitz continuous
over $\msk$ with constant $\sup_{\theta \in \msk} \norm{\nabla^2 L(\theta)}$.
Conditions under which
\Cref{assum:existence_compact}($\upalpha$)-\ref{item:c_existence} is satisfied
are given in \Cref{prop:existence_Q}.

Condition \Cref{assum:equi_meas} implies that the density of $\pi_{\theta}$, see \eqref{eq:gibbs_measure}, with respect to the Lebesgue measure, is given for any $x \in \rset^{\dim}$ by  $(\rmd \pi_{\theta} / \rmd \Leb)(x) =  \exp[-U(\theta, x)] / \int_{\rset^d} \exp[-U(\theta, y)]\rmd y$ with $U$ defined for any $\theta \in \msk$ and $x \in \rset^d$ by
\begin{equation}
  \label{eq:potential} U(\theta, x) = \langle \theta, F(x) \rangle + r(x) \eqsp .
\end{equation}
The mapping $U: \ \msk \times \rset^{\dim} \to \rset$ is referred to as the
potential function. Consider the following assumption on $U$.
\begin{assumptionB}
  \label{assum:curv_reg}
  There exist $U_i: \ \msk \times \rset^d \to \rset$ with $i \in \lbrace 1,2 \rbrace$ such that for any $\theta \in \msk$ and $x \in \rset^{\dim}$ $U(\theta, x) = U_1(\theta, x) + U_2(\theta, x)$. In addition,
  \begin{enumerate}[label=(\alph*), leftmargin=1cm]
  \item there exists $\Lip \geq 0$ such that for any $i \in \{1, 2\}$, $x \mapsto U_i(\theta, x)$ is continuously differentiable and for any $x, y \in \rset^d$, $\norm{\nabla_x U_i(\theta, x) - \nabla_x U_i(\theta, y)} \leq \Lip \norm{x- y }$ ;
  \item there exists $\mtt_1 >0$ and $x^{\star} \in \rset^d$ such that for any $\theta \in \msk$, $U_1(\theta, \cdot)$ is $\mtt_1$-strongly convex and $x^{\star} \in \argmin_{x \in \rset^{\dim}} U_1(\theta, x)$ ; \label{item:strong_convex}
  \item \label{item:bounded_grad} there exists $\borne \geq 0$ such that for any $\theta \in \msk$ and $x \in \rset^d$, $\| \nabla_x U_2(\theta, x) \| \leq \borne$ ;   \end{enumerate} 
\end{assumptionB}
We can relax the assumption that for any $\theta \in \msk$,
$x^{\star} \in \argmin_{x \in \rset^{\dim}} U_1(\theta, x)$ by the following:
there exists $R \geq 0$ such that for any $\theta \in \msk$, there exists
$x_{\theta}^{\star} \in \argmin_{x \in \rset^{\dim}} U_1(\theta, x)$ and
$x_{\theta}^{\star} \in \cball{0}{R}$. But for the sake of simplicity we do not
consider this assumption. The general assumption \Cref{assum:curv_reg} is satisfied for both the
Gaussian features and the CNN features introduced in
\Cref{sec:feature-models}. Indeed, if the features are Gaussian and the
reference measure is Gaussian we recall that
$\Theta_F =
\mathcal{F}^{-1}\parentheseDeuxLigne{\Re^{-1}\defEnsLigne{\ooint{-(2\sigma^2)^{-1},+\infty}^d}}
\cap \rset^d$ containing $\theta^{\star}$ with $\theta^{\star}$ given in
\eqref{eq:theta_star_hat}, see \Cref{sec:gaussian-features}. Then,
$x \mapsto U(\theta, x)$ is a definite positive quadratic form associated with
the symmetric matrix $\mathbf{C}_{\theta}$, see \eqref{eq:cov}. Setting $\Lip$
and $\mtt$ respectively the largest and lowest eigenvalues of
$\mathbf{C}_{\theta}$ over $\msk$ we obtain that \Cref{assum:curv_reg} is
satisfied with $U_1 = U$ and $U_2 =0$.

In the case of CNN features, if the reference measure is a Gaussian distribution with zero mean and  invertible covariance matrix $\covmat$, we obtain that for any $\theta \in \rset^p$ and $x \in \rset^d$
\begin{equation}
  \label{eq:potential_CNN}
  U(\theta, x) = \langle \theta, F(x) \rangle + x^{\transpose} \covmat^{-1} x  / 2 \eqsp .
\end{equation}
If in addition, $\varphi$ is differentiable with Lipschitz derivative and for any $t \in \rset$, $\sup_{t \in \rset} \abs{\varphi'(t)} < +\infty$, we have that \Cref{assum:curv_reg} is satisfied with for any $\theta \in \msk$ and $x \in \rset^d$
\begin{equation}
 U_1(\theta, x) = x^{\transpose} \covmat^{-1} x/2 \eqsp , \qquad  U_2(\theta, x) = \langle \theta, F(x) \rangle  \eqsp .
\end{equation}
In particular the fact that $U_2$ is gradient-Lipschitz and Lipschitz is ensured by a straightforward recursion since for any $f \in \rmC^1(\rset^{\dim_3}, \rset^{\dim_2})$ and $g \in \rmC^1(\rset^{\dim_2}, \rset^{\dim_1})$, $x \mapsto \rmd (g \circ f)(x)$ and $g \circ f$ Lipschitz if $f, g, \rmd f$ and $\rmd g$ are Lipschitz. Note that the differentiability assumption is not met in classical convolutional neural networks such as \vgg19 . Therefore, in all of our experiments we replace the max-pooling operator by a mean-pooling operator and the RELU function by a Continuously Differentiable Exponential Linear Unit (CELU), see \cite{barron2017continuously}.

We now state our main results in the case where $U$ is a strongly convex potential, \ie \ $U_2 = 0$.
\label{sec:strongly-conv-case}
\begin{theorem}
  \label{thm:cvx}
Let $\upalpha \geq 1$.  Assume \tup{\Cref{assum:weak}($\upalpha$)}, \tup{\Cref{assum:equi_meas}},
  \tup{\Cref{assum:existence_compact}($\upalpha$)}, \tup{\Cref{assum:curv_reg}}
  with $U_2 = 0$. Let $(\gamma_n)_{n \in \nset}$,
  $(\delta_n)_{n \in \nset}$ be sequences of non-increasing positive real
  numbers and $(m_n)_{n \in \nset}$ a sequence of positive integers satisfying
  $\delta_n < 1/(\sup_{\theta \in \msk} \normLigne{\nabla^2 L(\theta)})$ and
  $\gamma_n < \min(\mttun / (2 \Lip^2), 1/2)$ for any $n \in \nset$. Then, there
  exists $(E_n)_{n \in \N}$ such that for any $n \in \nsets$
    \begin{equation}
    \expe{  \defEns{\left. \sum_{k=1}^n \delta_k L(\theta_k) \middle/ \sum_{k=1}^n \delta_k \right. } - \min_{\msk} L  }\leq  \left. E_n \middle/  \left( \sum_{k=1}^n \delta_k \right) \right. \eqsp ,
  \end{equation}
  with for any $n \in \nsets$, 
  \begin{enumerate}[label=(\alph*), leftmargin=1.5cm]
  \item \label{item:a} if $m_n = m_0$ for all $n \in \nset$ and $\sup_{n \in \nset} \abs{\delta_{n+1} - \delta_n} \delta_n^{-2} < +\infty$
    \begin{align}
      &E_n = C_1 (1 + d^\varpi) \left( 1 + \sum_{k=0}^{n-1} \delta_{k+1} \gamma_k^{1/2} + \sum_{k=0}^{n-1} \delta_{k+1} \gamma_{k+1}^{-5/2} (\gamma_k - \gamma_{k+1})^{1/2} \right . \\ & \qquad \qquad \qquad \qquad \left . + \sum_{k=0}^{n-1} \delta_{k+1}^2 / \gamma_k^{3/2} + \delta_{n+1} / \gamma_n \right)  \eqsp ;
  \end{align}
  \item \label{item:b} otherwise
    \begin{align}
      &E_n = C_2 (1 + d^\varpi) \left( 1 + \sum_{k=0}^{n-1} \delta_{k+1} \gamma_k^{1/2} + \sum_{k=0}^{n-1} \delta_{k+1} / (m_k \gamma_k) \right. \\ & \qquad \qquad \qquad \qquad\left. + \sum_{k=0}^{n-1} \delta_{k+1}^2 \gamma_k + \sum_{k=0}^{n-1} \delta_{k+1}^2 / (m_k \gamma_k)^2 \right ) \eqsp ,
    \end{align}
  \end{enumerate}
      with $C_1, C_2, \varpi \geq 0 $ which do not depend on the dimension $d$.
\end{theorem}

\begin{proof}
  The proof is postponed to \Cref{thm:cvx:proof}.
\end{proof}
It should be noted that \Cref{thm:cvx}-\ref{item:b} applies even if for any $n \in \nset$, $\X_0^n$ is distributed according to $\nu$ with $\nu \in \Palpha$, \ie \ the warm-start procedure can be avoided.

In the case where for any $n \in \nset$, $m_n = m_0$, $\gamma_n = \gamma_0$ and $\lim_{n \to +\infty} \delta_n = 0$ with $\sum_{k=0}^{+\infty} \delta_n = +\infty$, we obtain using \cite[Problem 80, Part I]{polya1998problem} that, $\lim_{n \to +\infty} \sum_{k=0}^n \delta_k^2 / \sum_{k=0}^n \delta_k = 0$. Therefore, using \Cref{thm:cvx}-\ref{item:a} we get that
\begin{equation}
\limsup_{n \to +\infty} \expe{L(\bar{\theta}_n)} - \min_{\msk} L \leq C_1 (1+d^{\varpi})\gamma_0^{1/2} \eqsp ,
\end{equation}
with $\bar{\theta}_n = \sum_{k=1}^n \delta_k \theta_k / \sum_{k=1}^n \delta_k$ and using \Cref{thm:cvx}-\ref{item:b} we get 
\begin{equation}
\limsup_{n \to +\infty} \expe{L(\bar{\theta}_n)} - \min_{\msk} L\leq C_2 (1 + d^{\varpi}) \parentheseDeux{ \gamma_0^{1/2} + (m_0 \gamma_0)^{-1}} \eqsp .
\end{equation}
Therefore, the minimum of $L$ can be reached with arbitrary precision. Note that the constants $C_1, C_2$ do not have the same dependency with respect to the problem parameters and that $C_2$ is usually better than $C_1$.

We now state our main results in the case where the potential is not convex anymore.
We consider the following additional regularity assumption on $F$.
\begin{assumptionB}
  \label{assum:f_grad_lip}
  $F \in \rmc^1(\rset^d, \rset^p)$ and there exists $\LipgradF \geq 0$ such that for any $x,y \in \rset^d$
  \begin{equation}
    \norm{\rmd F(x) - \rmd F(y)} \leq \LipgradF \norm{x - y} \eqsp .
  \end{equation}
\end{assumptionB}

\begin{theorem}
  \label{thm:non_cvx}
Assume \tup{\Cref{assum:weak}($1$)}, \tup{\Cref{assum:equi_meas}}, \tup{\Cref{assum:existence_compact}($1$)}, \tup{\Cref{assum:curv_reg}} and \tup{\Cref{assum:f_grad_lip}}. Let $(\gamma_n)_{n \in \nset}$, $(\delta_n)_{n \in \nset}$ be sequences of non-increasing positive real numbers and $(m_n)_{n \in \nset}$ a sequence of positive integers satisfying $\delta_n < 1/(\sup_{\theta \in \msk} \normLigne{\nabla^2 L(\theta)})$ and $\gamma_n < \min(\mtt_1 / (8 \Lip^2),  1/2)$ for any $n \in \nset$. Then, there exists $(E_n)_{n \in \N}$ such that for any $n \in \nsets$
    \begin{equation}
    \expe{  \defEns{\left. \sum_{k=1}^n \delta_k L(\theta_k) \middle/ \sum_{k=1}^n \delta_k \right. } - \min_{\msk} L  }\leq  \left. E_n \middle/  \left( \sum_{k=1}^n \delta_k \right) \right. \eqsp ,
  \end{equation}
  with for any $n \in \nsets$, 
  \begin{enumerate}[label=(\alph*), leftmargin=1.5cm]
  \item  \label{item:non_cvx_a} if $m_n = m_0$, $\gamma_n = \gamma_0$ for all $n \in \nset$ and $\sup_{n \in \nset} \abs{\delta_{n+1} - \delta_n} \delta_n^{-2} < +\infty$
    \begin{align}
      E_n = C_1 (1 + d^{\varpi}) \parenthese{1 + \sum_{k=0}^{n-1} \delta_{k+1}\gamma_0^{1/2}  + \sum_{k=0}^{n-1} \delta_{k+1}^2/\gamma_0 + \delta_{n} / \gamma_0 } \eqsp ;
  \end{align}
  \item \label{item:non_cvx_b} else
    \begin{equation}
      E_n = C_2 (1 + d^{\varpi}) \parenthese{1 + \sum_{k=0}^{n-1} \delta_{k+1} \gamma_k^{1/2} + \sum_{k=0}^{n-1} \delta_{k+1} / (m_k \gamma_k) + \sum_{k=0}^{n-1} \delta_{k+1}^2} \eqsp ,
    \end{equation}
  \end{enumerate}
    with $C_1, C_2, \varpi \geq 0 $ which do not depend on the dimension $d$.  
\end{theorem}

\begin{proof}
  The proof is postponed to \Cref{thm:non_cvx:proof}
\end{proof}

The discussion conducted after \Cref{thm:cvx} is still valid in this case.
Also, note that the results of \Cref{thm:cvx}-\ref{item:b} are covered by \Cref{thm:non_cvx}-\ref{item:non_cvx_b} under \Cref{assum:f_grad_lip}. However, in \Cref{thm:non_cvx}-\ref{item:non_cvx_a}, $\gamma_n = \gamma_0$ for all $n \in \nset$ whereas in \Cref{thm:cvx}-\ref{item:a}, $(\gamma_n)_{n \in \nset}$ is any arbitrary non-increasing sequence of positive numbers.
\Cref{thm:non_cvx} follows from more general results derived in \Cref{thm:cv_soul_non_cvx_1} and \Cref{thm:cv_soul_non_cvx_2}.

\subsection{Links with macrocanonical models}
\label{sec:links-with-macr}

In this section, we present qualitative results on the microcanonical model
and the asymptotic behavior of the macrocanonical model for specific geometrical
constraints. We start by recalling a result concerning the convergence of the sampler
of the microcanonical model from \cite[Theorem 3.7-(i)]{bruna2018multiscale}.

Let $\nu_0 \ll \Leb$ be an initial probability measure. We consider the sequence of probability measures $(\nu_n)_{n \in \nset}$ defined by the following recursion: for any $n \in \nset$,
\begin{equation}
  \label{eq:pushforward}
  \nu_{n+1} = \Phi_{\hash}(\nu_n) \eqsp ,
\end{equation}
where $\Phi : \ \rset^d \to \rset^d$ is defined for any $x \in \rset^d$ by $\Phi(x) = x - \gamma \rmd F(x)^{\transpose} F(x)$, with $\gamma >0$ a stepsize. Namely, for all $n \in \nset$, $\nu_n$ is the pushforward measure of $\nu_0$ by $n$ steps of the gradient descent for the the loss function $x \mapsto \normLigne{F(x)}^2$. 

\begin{theorem}[\cite{bruna2018multiscale}]
  \label{theo:mallat}
  Let $F \in \mrc^2(\rset^d, \rset^p)$ such that for any compact set $\msk \subset \rset^p$, $F^{-1}(\msk)$
  is compact. Assume \tup{\Cref{assum:f_grad_lip}} and that for any
  $x \in F^{-1}(\{ 0 \})$, $\det(\rmD F(x) \rmD F(x)^{\transpose}) > 0$. In addition, assume that $F$ satisfies the strict saddle property, \ie \ defining $
  M_v \in \mathrm{S}_{d}(\rset)$ for any $v = (v_1, \dots, v_p) \in \kernellin{\rmd
    F(x)^{\transpose}} \without{0}$, by
    \begin{equation}
      \label{eq:strict_saddle}
      M_v = \sum_{k=1}^p v_k \nabla^2F_k(x) + \rmD F(x)^{\transpose}\rmD F(x) \eqsp ,
    \end{equation}
  $M_v$ admits at least one negative eigenvalue. Then, if $\gamma \in \oointLigne{0, \LipgradF^{-1}}$ and $\X_0$ is distributed according to $\nu_0$, $(\X_n)_{n \in \nset}$ defined for any $n \in \nset$ by $\X_{n+1} = \Phi(\X_n)$ converges \as \ to a random variable $\X_{\infty}$. In addition, $\nu_{\infty}(F^{-1}(\{0\})^{\complementary}) = 0$.
\end{theorem}
Let $\msa = F^{-1}(\{ 0 \})$. If $\msa$ is compact, the microcanonical model,
see \Cref{def:micro}, associated with the reference measure $\Leb$ and the
constraints $F$, is given by the uniform distribution over $\msa$, denoted
$\nu_{\msa}$.  If $\nu_{\infty}$ were the microcanonical model associated with
$F$ then we should have $\nu_{\infty} = \nu_{\msa}$. However, as illustrated in
\Cref{fig:micros_scheme}, $\nu_{\infty}$ strongly depends on the initial
probability measure $\nu_0$.  Let $(\Phi_n)_{n \in \nset}$ be defined by
$\Phi_0= \Id$ and the following recursion: $\Phi_{n+1} = \Phi \circ \Phi_n$.
Note that $\X_n$ in \Cref{theo:mallat} is distributed according to
$(\Phi_n)_{\hash}(\nu_0)$.  It is shown in the proof of \cite[Theorem
3.7]{bruna2018multiscale} that $\Phi_{\infty} = \lim_{n} \Phi_n$ is well-defined. Since for any
$n \in \nset$, $\Phi_n$ is Lebesgue measurable so is $\Phi_{\infty}$. Therefore
we have that $\nu_{\infty} = (\Phi_{\infty})_{\hash}(\nu_0)$.

\begin{figure}
  \centering
  \subfloat[]{\includegraphics[width=0.45\linewidth]{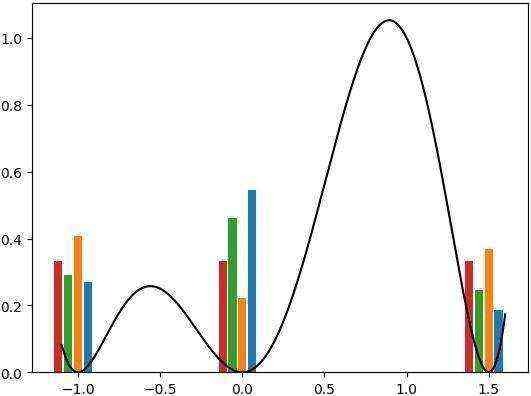}} \hfill 
  \subfloat[]{\includegraphics[width=0.45\linewidth]{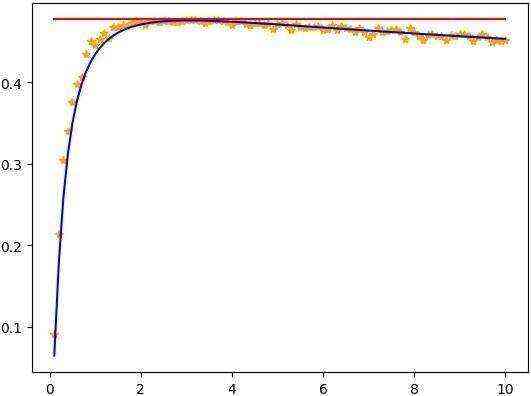}} \hfill
  \caption{\figuretitle{Microcanonical sampling scheme} In this one-dimensional toy example the features are given by $F(x) = (x+1)^2x^2(x-1.5)^2$ (black curve). The microcanonical model associated with these features is the uniform distribution over $\{-1, 0, 1.5\}$ (red bars). In (a), we plot, the distribution $(\Phi_{\infty})_{\hash}(\nu_0)$ for different initial distributions $\nu_0$, standard Gaussian (blue bars), uniform over $[-3,3]$ (orange bars) and standard Cauchy (green bars). The distribution $(\Phi_{\infty})_{\hash}(\nu_0)$ is approximated by sampling $10^3$ points according to $\nu_0$ and performing the recursion associated with \eqref{eq:pushforward} for these points for $10^4$ iterations. None of the initial distribution yields a distribution $(\Phi_{\infty})_{\hash}(\nu_0)$ which is the uniform distribution. Let $\nu_0$ be a Gaussian distribution with zero mean and variance $\sigma^2$ with $\sigma >0$. In (b), we show the dependency of the entropy of $(\Phi_{\infty})_{\hash}(\nu_0)$ with respect to $\sigma^2$ (orange points). The distribution $(\Phi_{\infty})_{\hash}(\nu_0)$ is approximated by sampling $10^3$ points according to $\nu_0$ and performing the recursion associated with \eqref{eq:pushforward} for these points for $10^3$ iterations. Then, we compute its entropy and show that it is close to the one given by numerical integration (blue curve). We also plot the entropy upper-bound $\log_{10}(3)$ (red curve) given by the uniform distribution on $\{-1, 0, 1.5\}$.}  \label{fig:micros_scheme}
\end{figure}

In what follows we prove that considering specific constraint functions $f_{\vareps}$, one can
construct an explicit probability measure $\pi_{\infty}$ such that $\pi_{\infty}$ is supported on $F^{-1}(\{0 \})$ and $\pi_{\infty}$ is the limit of macrocanonical models associated with $f_{\vareps}$ and some reference probability measure $\mu$.
Let $\vareps >0$. We define $f_{\vareps}: \ \rset^d \to \rset$ such that for any $x \in \rset^d$, $f_{\vareps}(x) = \norm{F(x)}^2 - \vareps$. We denote $\pi_{\vareps}$ the macrocanonical model, see \Cref{def:macro}, associated with $f_{\vareps}$ when it exists.

\begin{proposition}
  \label{prop:limit}
  Assume \tup{\Cref{assum:weak}($2$)} and that for any non-empty open set $\msa \subset \rset^d$, $\mu(\msa) >0$. Let $F$ be given by \eqref{eq:neural_network}, assume that $1 \in \calJ$
  and that there exists $k \in \{1, \dots, c_1\}$ such that for any $x \in \rset^{\dim}$ with $x \neq 0$ there exists $\ell \in \{1, \dots, n_1\}$ with $e_{\ell}^{\transpose} \tilde{A}_1^k x >0$.  Then there exists $\vareps_0 >0$ such that for any $\vareps \in \ooint{0, \vareps_0}$, $\pi_{\vareps}$ exists. In addition, the following propositions hold:
 \begin{enumerate}[label=(\alph*), leftmargin=1.5cm]    
  \item \label{prop:limit:item:a}  Assume 
   that  $\mu(F^{-1}(\{ 0 \})) > 0$ then $\lim_{\vareps \to 0} \pi_{\vareps} = \pi_{\infty}$, with for any $x \in \rset^d$
    \begin{equation}
      \frac{\rmd \pi_{\infty}}{\rmd \mu}(x) = \frac{\1_{F^{-1}(\{0\})}(x)}{\mu(F^{-1}(\{ 0 \}))} \eqsp .
    \end{equation}
  \item \label{prop:limit:item:b} Assume that $F^{-1}(\{ 0 \}) = \{ x_1, \dots, x_K \}$ with $(x_i)_{i \in \{ 1, \dots, K \}} \in (\rset^d)^K$, $K \in \nsets$, $\varphi \in \rmc^3(\rset)$, $x \mapsto (\rmd \mu /\rmd \Leb)(x)$ is continuous and $(\rmd \mu /\rmd \Leb)^{-1}(F^{-1}(\{ 0 \})) \neq \{ 0 \}$. Let $H(x) = \nabla^2 (\norm{F(\cdot)}^2)(x)$ and assume that for any $x \in F^{-1}(\{0 \})$, $\det H(x) \neq 0$. Then $\lim_{\vareps \to 0} \pi_{\vareps} = \pi_{\infty}$ with
    \begin{equation}
      \pi_{\infty} = \sum_{i=1}^K \frac{\frac{\rmd \mu}{\rmd \Leb}(x_i) \det(H(x_i))}{\sum_{j=1}^K \frac{\rmd \mu}{\rmd \Leb}(x_j) \det(H(x_j))} \updelta_{x_i} \eqsp .
    \end{equation}
  \item \label{prop:limit:item:c} Assume that $F^{-1}(\{ 0 \})$ is a smooth compact manifold, $\varphi \in \rmc^3( \rset)$, $x \mapsto (\rmd \mu /\rmd \Leb) (x)$ is continuous and $(\rmd \mu /\rmd \Leb )^{-1}(F^{-1}(\{ 0 \})) \neq 0$. Let $H(x) = \nabla^2 (\norm{F(\cdot)}^2)(x)$ and assume that for any $x \in F^{-1}(\{0 \})$, $\det H(x) \neq 0$. Then $\lim_{\vareps \to 0} \pi_{\vareps} = \pi_{\infty}$ with for any $x \in \rset^d$
    \begin{equation}
      \frac{\rmd \pi_{\infty}}{\rmd \mathcal{H}}(x) = \frac{\1_{F^{-1}(\{0\})}(x) \frac{\rmd \mu}{\rmd \Leb}(x) \det H(x)}{\int_{F^{-1}(\{0\})}\frac{\rmd \mu}{\rmd \Leb}(y) \det H(y) \rmd \mathcal{H}(y)} \eqsp ,
    \end{equation}
where $\mathcal{H}$ is the intrisic measure on $F^{-1}(\{0 \})$, see \cite[Chapter 6]{boothby1986introduction}.
\end{enumerate}
\end{proposition}

\begin{proof}
  The proof is postponed to \Cref{prop:limit_proof}
\end{proof}
In \Cref{prop:limit}
, if $F^{-1}(\{0\}) \subset \cball{0}{R}$, with $R >0$ then letting $\mu$ such that for any $x \in \rset^d$, $\frac{\rmd \mu}{\rmd \Leb}(x) = \1_{\cball{0}{R}}(x) / \Leb(\cball{0}{R})$, we have  $\pi_{\infty} \ll \Leb$ and for any $x \in \rset^d$
\begin{equation}
  \frac{\rmd \pi_{\infty}}{\rmd \Leb}(x) = \parenthese{\frac{\rmd \pi_{\infty}}{\rmd \mu}(x)}\parenthese{ \frac{\rmd \mu}{\rmd \Leb}(x)} = \frac{\1_{F^{-1}(\{0\})}(x)}{\mu(F^{-1}(\{0\})) \Leb(\cball{0}{R})} = \frac{\1_{F^{-1}(\{0\})}(x)}{\Leb(F^{-1}(\{0\})} \eqsp ,
\end{equation}
hence $\pi_{\infty}$ is a microcanonical model. 
Finally, note that the conclusions of \Cref{prop:limit}-\ref{prop:limit:item:a}-\ref{prop:limit:item:b} do not hold if  $d > p$, since in this case $\det(H(x)) = 0$ for any $x \in \rset^d$ such that $F(x) = 0$.


  \section{Experiments}
\label{sec:experiments}

In this section, we assess the computational efficiency of SOUL algorithm \eqref{eq:souk} for texture synthesis.
Variants of the original methodology are presented in \Cref{sec:addit-exper}.

\subsection{Periodic Gaussian model}
\label{sec:peri-gauss-mod}

First, we consider the toy problem of periodic Gaussian texture synthesis, see
\Cref{sec:feature-models} for details.  Note that the presentation of the model
was conducted for one dimensional signals. The extension of our findings two
dimensional signals is straightforward upon replacing the one dimensional
Fourier transform by its two dimensional counterpart.  We recall that in this
case the macrocanonical model is explicit and given by a measure
$\pi_{\theta^{\star}}$ which is the probability distribution of
$\X = d^{-1/2} (x_0 * \Z)$ where $\Z$ is a standard $d$-dimensional Gaussian
random variable, see \Cref{sec:gaussian-features}.  This model was introduced in
the context of computer graphics in \cite{vanwijk1991spotnoise} and its
mathematical study was conducted in
\cite{galerne2016gaussian,galerne2011random,galerne2014texton}.

\subsubsection{Empirical convergence}
\label{sec:empir-conv}

We consider a $8 \times 8$ image, denoted $x_0$, corrupted by some noise, so that $\fourier(x_0)$ is non-zero everywhere
on the $8 \times 8$ grid. The reference measure $\mu$ is a Gaussian distribution with zero mean and diagonal covariance matrix with diagonal coefficients given by $\sigma^2$. In this setting, $d = p = 64$ and using \eqref{eq:theta_star_hat} we have
\begin{equation}
\theta^{\star} = \fourier^{-1}{(d \absLigne{\fourier(x_0)}^{-2} - \sigma^{-2})/2} \eqsp .
\end{equation}
Using the spatial translation invariance property of $\pi_{\theta^{\star}}$, see \eqref{eq:translation_invariance}, we identify four configurations which are equally likely to be sampled by $\pi_{\theta^{\star}}$, see \Cref{fig:target_gaussian}.
\begin{figure}[h]
  \centering
  \subfloat[]{\includegraphics[width=0.18\linewidth]{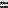}} \hfill 
  \subfloat[]{\includegraphics[width=0.18\linewidth]{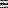}} \hfill
  \subfloat[]{\includegraphics[width=0.18\linewidth]{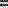}} \hfill
  \subfloat[]{\includegraphics[width=0.18\linewidth]{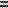}} \hfill
  \subfloat[]{\includegraphics[width=0.18\linewidth]{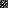}}  
  \caption{\figuretitle{Exemplar images and parameters} The exemplar image $x_0$ is shown in (a). Translated versions of this image, which are equally likely to be sampled by $\pi_{\theta^{\star}}$ are presented in (b), (c) and (d). The target parameter $\theta^{\star}$ is shown in (d).}
  \label{fig:target_gaussian}
\end{figure}

The images $(\X_0^n)_{n \in \nset}$ generated by the SOUL algorithm \eqref{eq:souk} are approximate samples of $\pi_{\theta^{\star}}$ for $n$ large enough. The configurations identified in \Cref{fig:target_gaussian} are recovered during one run of the algorithm, see \Cref{fig:adsn_sequence}. 
A video of the evolution of the sequence $(\X_0^n)_{n \in \nset}$ is available at \url{https://vdeborto.github.io/publication/texture_soul/}.
\begin{figure}[h]
  \centering
  \subfloat[$n=0$]{\includegraphics[width=0.2\linewidth]{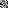}} \hfill
  \subfloat[$n = 58000$]{\includegraphics[width=0.2\linewidth]{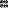}} \hfill
  \subfloat[$n = 74000$]{\includegraphics[width=0.2\linewidth]{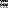}} \hfill  
  \subfloat[$n = 100000$]{\includegraphics[width=0.2\linewidth]{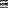}} \hfill  
  \caption{\figuretitle{Sequence of images} The initialization (a) of the algorithm is some white noise, \ie \ the realization of a standard Gaussian random variable on the $8 \times 8$ grid. We then show some selected samples (b)-(d) of the sequence generated with fixed parameters $\delta_n = 10^{-1}$, $\gamma_n = 10^{-4}$ and $m_n = 1$. Note that these samples are visually close to the ones presented in \Cref{fig:target_gaussian}.}
  \label{fig:adsn_sequence}
\end{figure}

The main theoretical results in \Cref{thm:cvx} deal with the error between $L(\bar{\theta}_n)_{n \in \nset}$ and $\argmin_{\theta \in \Theta}L(\theta)$, where $(\bar{\theta}_n)_{n \in \nset}$ is given by \eqref{eq:avg}. Selecting fixed parameters, $\gamma_n = 10^{-4}$, $\delta_n = 10^{-1}$ and $m_n = 1$ we observe the convergence of the sequence $(\theta_n)_{n \in \nset}$ towards a biased estimate of $\theta^{\star}$. The Normalized Root Mean Square Error (\NRMSE ) defined for any $n \in \nset$ by
\begin{equation}
  \label{eq:nrmse}
  \nrmse(\theta_n) = \norm{\theta_n - \theta^{\star}}_2 / \norm{\theta^{\star}}_2 \eqsp ,
\end{equation}
is upper bounded by 0.2 for $n \geq 4 \times 10^4$, see \Cref{fig:error_decreasing}. In \Cref{fig:parameter_comp}, we show that this error level yields satisfactory parameters from a visual point of view. We highlight that $\fourier(\theta^{\star})$ corresponds to the precision matrix (up to a constant factor) of the Gaussian model under study.

\begin{figure}[h]
  \centering
\subfloat[]{\begin{tikzpicture}[scale = 0.75]
  \begin{axis}[grid=major,no markers,domain=-5:5,enlargelimits=false,ymin=0]
    \input{./data/ADSN_error_fixed.tex}
    \input{./data/ADSN_avg_error_fixed.tex}
    \addlegendentry{$(\theta_n)_{n \in \nset}$}
    \addlegendentry{$(\bar{\theta}_n)_{n \in \nset}$}    
    \end{axis}
  \end{tikzpicture}} \qquad \qquad 
\subfloat[]{\includegraphics[width=0.4\linewidth]{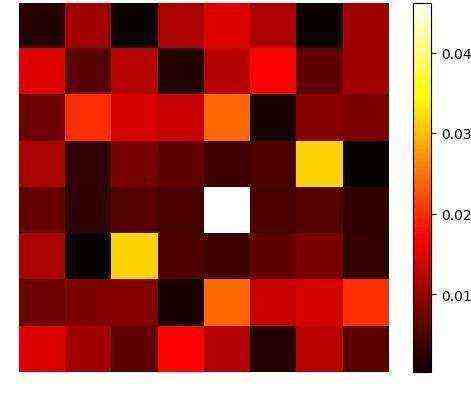}}
\caption{\figuretitle{Convergence of the parameters} We recall that the parameters are initialized with $\theta_0 = 0$ and that $\gamma_n = 10^{-4}$, $\delta_n = 10^{-1}$ and $m_n = 1$. The \NRMSE \ error in (a) rapidly decreases before oscillating around $0.1$ (blue curve). A similar smoothed behaviour can be observed for the averaged sequence $(\bar{\theta}_n)_{n \in \nset}$ (red curve). The heatmap of the \NRMSE , \ie \ a pixel $i \in \{0, \dots, 7\}^2$ in (b) corresponds to $(\theta_n(i) - \theta^{\star}(i))^2 / \norm{\theta^{\star}}^2$.}
\label{fig:error_decreasing}
\end{figure}

\captionsetup[subfigure]{labelformat=empty}
\begin{figure}[h]
  \centering
  \subfloat[$\theta^{\star}$]{\includegraphics[width=0.24\linewidth]{./data/ADSN_weight_target.jpg}} \hfill  
  \subfloat[$\theta_n$]{\includegraphics[width=0.24\linewidth]{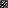}} \hfill
  \subfloat[$\fourier(\theta^{\star})$]{\includegraphics[width=0.24\linewidth]{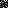}} \hfill    
  \subfloat[$\fourier(\theta_n)$]{\includegraphics[width=0.24\linewidth]{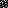}} \hfill
  \caption{\figuretitle{Visual evaluation of parameters} We display the target parameters $\theta^{\star}$ and the parameters obtained after $10^6$ iterations of the algorithm. Similarly we display the discrete Fourier transform of the target parameters $\fourier(\theta^{\star})$ and the Fourier transform of the parameters after $10^6$ iterations. There is no visual difference between $\theta^{\star}$ and $\theta_n$.}
  \label{fig:parameter_comp}
\end{figure}
\captionsetup[subfigure]{labelformat=parens}

The previous experiment suggests to set $\gamma_n =\gamma >0$, $\delta_n = \delta >0$ and $m_n = m \in \nsets$, at least for a burnin period. When the behavior of the sequence $(\theta_n)_{n \in \nset}$ becomes oscillatory, the setting can be changed in order to obtain a better approximation of $\theta^{\star}$. We investigate the long-time behavior after a burnin period of $N = 5\times 10^4$ iterations with $m_n = 1$, $\gamma_n = 10^{-4}$ and $\delta_n= 10^{-1}$. After this period we set $m_n = \ceil{\bar{n}^a}$, $\gamma_n = \bar{n}^{-b}$ and $\delta_n = \bar{n}^{-c}$ with $\bar{n} = n - N +1$, $a, b, c >0$. We observe that the \NRMSE \ error decreases from $0.1$ to $0.06$ for appropriate choices of rates, see  \Cref{fig:heatmap_langevin}. Nevertheless, this improvement comes at a cost since the number of Markov chain iterations is no longer equal to the number of iterations $n$ and grows as $\ceil{\bar{n}^a}$. 

\begin{figure}[h]
  \centering
  \hfill
  \subfloat[]{\includegraphics[width=0.2\linewidth]{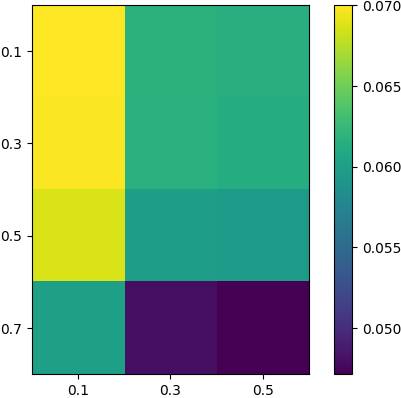}} \hfill
  \subfloat[]{\includegraphics[width=0.2\linewidth]{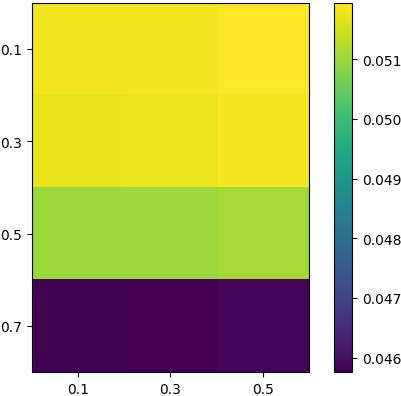}} \hfill
  \subfloat[]{\includegraphics[width=0.2\linewidth]{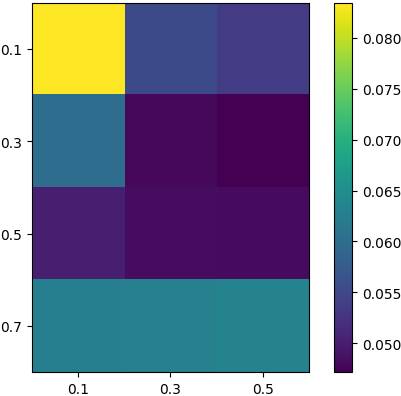}} \hfill
  \subfloat[]{\includegraphics[width=0.2\linewidth]{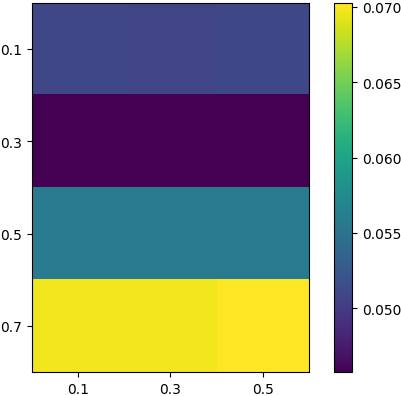}} \hfill
  \hfill
  
  \centering
  \hfill
  \subfloat[]{\includegraphics[width=0.25\linewidth]{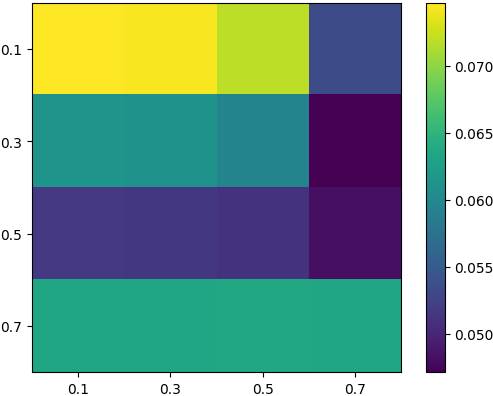}} \hfill
  \subfloat[]{\includegraphics[width=0.25\linewidth]{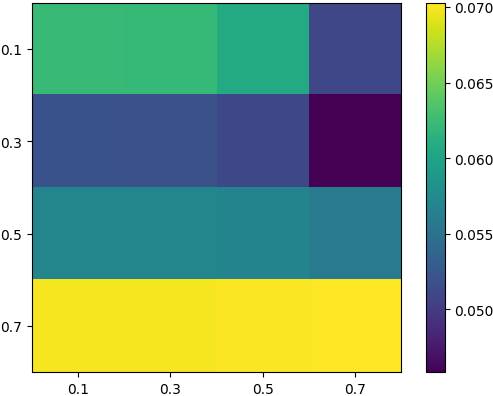}} \hfill
  \subfloat[]{\begin{tikzpicture}[scale = 0.5]
  \begin{axis}[grid=major,no markers,domain=-5:5,enlargelimits=false,ymin=0]
    \input{./data/long_run_ADSN/weight_avg.tex}
    \addlegendentry{$(\bar{\theta}_n)_{n \in \nset}$}
    \end{axis}
  \end{tikzpicture}} \hfill
\caption{\figuretitle{Evolution of the error}
  In (a) and (b) we present the heatmap of the \NRMSE \ error between $\theta_{2\times 10^5}$ and $\theta^{\star}$ in (a) and $\bar{\theta}_{2 \times 10^5}$ and $\theta^{\star}$ in (b), given different values of $b, c >0$ where  $\gamma_n = 10^{-4}\times\bar{n}^{-b}$ and $m_n=\ceil{\bar{n}^c}$ with $\delta_n = 10^{-1} \times\bar{n}^{-0.3}$ and $\bar{n} = n - N +1$ with $N = 5 \times 10^4$. On the $y$-axis in (a) and (b) we represent the different values for parameter $b$ and on the $x$-axis the different values for parameter $c$. 
  Similarly, in (c) and (d) we present the heatmap of the \NRMSE \ error between $\theta_{2\times 10^5}$ and $\theta^{\star}$ in (c) and $\bar{\theta}_{2 \times 10^5}$ and $\theta^{\star}$ in (d), given different values of $a, c >0$ where  $\delta_n = 10^{-1}\times \bar{n}^{-a}$ and $m_n=\ceil{\bar{n}^c}$ with $\gamma_n = 10^{-4} \times \bar{n}^{-0.7}$. On the $y$-axis in (c) and (d) we represent the different values for parameter $a$ and on the $x$-axis the different values for parameter $c$.
  In (e) and (f) we present the heatmap of the \NRMSE \ error between $\theta_{2\times 10^5}$ and $\theta^{\star}$ in (e) and $\bar{\theta}_{2 \times 10^5}$ and $\theta^{\star}$ in (f), given different values of $b, c >0$ where  $\delta_n = 10^{-1}\bar{n}^{-a}$ and $\gamma_n=\ceil{\bar{n}^{-b}}$ with $m_n = \ceil{\bar{n}^{0.5}}$. On the $y$-axis in (e) and (f) we represent the different values for parameter $a$ and on the $x$-axis the different values for parameter $b$. A plot of the \NRMSE \ for the averaged sequence is presented  in (g) with $a = 0.3$ and  $c= 0.7$.}
\label{fig:heatmap_langevin}
\end{figure}
The previous comments along with  \Cref{fig:heatmap_langevin} suggest to set fix hyperparameters with $m_n = 1$ for all $n \in \nset$. This is a good strategy to obtain acceptable approximations of the target parameter $\theta^{\star}$ in reasonable time. However, the sampled images move slowly between the acceptable configurations of \Cref{fig:target_gaussian}. Increasing the fixed batch size, \ie \ increasing $m_n$, for instance setting $m_n = 10^2$ for all $n \in \nset$ we obtain more innovation in the chain $(\X_0^n)_{n \in \nset}$. Namely, for the same number of Markov chain iterations the chain $(\X_0^n)_{n \in \nset}$ visits more different acceptable configurations for $m_n = 10^2$ than for $m_n =1$, see \Cref{fig:batch_size}. However, if $m_n = 10^2$, the \NRMSE \ error of the sequence $(\theta_n)_{n \in \nset}$ has a lower decrease rate than if $m_n =1$. 

\captionsetup[subfigure]{labelformat=empty}
\begin{figure}[h]
  \centering
  \subfloat[ $\mathrm{NRMSE} = 0.73$ \newline $n = 200$]{\includegraphics[width=0.2\linewidth]{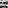}} \hfill
  \subfloat[ $\mathrm{NRMSE} = 0.63$ \newline $n = 800$]{\includegraphics[width=0.2\linewidth]{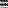}} \hfill
  \subfloat[ $\mathrm{NRMSE} = 0.60$ \newline $n = 1000$]{\includegraphics[width=0.2\linewidth]{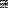}} \hfill
  \subfloat[ $\mathrm{NRMSE} = 0.47$ \newline $n = 2800$]{\includegraphics[width=0.2\linewidth]{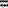}} \hfill  
  \caption{\figuretitle{Larger batch sizes improve visual quality} The algorithm with $\delta_n = 10^{-1}$, $\gamma_n = 10^{-4}$ and $m_n = 10^2$ produces more diverse samples than the ones obtained with the same algorithm and $m_n=1$, see \Cref{fig:target_gaussian}. Note that the \NRMSE \ errors $0.73, 0.63, 0.60$ and $0.47$ are still high.}
  \label{fig:batch_size}
\end{figure}
\captionsetup[subfigure]{labelformat=parens}

Therefore, the hyperparameters of the algorithm should be adapted for the problem at hand.
If we are interested in finding  $\theta^{\star}$ then fixed settings for a burnin period
followed by an eventual run of the algorithm with increasing batch size and decreasing stepsizes is recommended. However, if we are concerned with the innovation of the sequence $(\X_0^n)_{n \in \nset}$
then larger batch sizes, not necessarily increasing, are recommended.
In what follows we experimentally assess some generalizations of the SOUL algorithm.

\subsection{Neural network features}
\label{sec:neur-netw-feat}

\subsubsection{Spatially averaged CNN features}
\label{sec:cnn-mean-features}
We now investigate the case where the features are given by a convolutional neural network, see \Cref{sec:feature-models}.
We briefly recall that this model is similar to the one introduced by \cite{gatys2015texture} for microcanonical models but
instead of considering Gram matrices for different layers of a convolutional neural network, we consider the means of different layers and channels for the same convolutional neural network to build the features.
In our experiments we fix $\msk = \ccint{-10^4, 10^4}^d$.

\paragraph{Model hyperparameters}

In the proposed model a few hyperparameters must be selected. First, a convolutional neural
network architecture is to be chosen. We use the \vgg19 \ model since it has been
highlighted by \cite{ustyuzhaninov2016texture,gatys2015texture} that such an architecture is well
suited for the task of texture synthesis. In \cite{gatys2015texture} the neural network is pretrained
on a classification task, see \cite{simonyan2014vgg}. We first assess that this pretraining is a crucial
step in our model in \Cref{fig:pretraining}. Indeed, if for each convolutional layer $\ell$ and channel $c$, the pretrained filters are replaced by filters with weights given by a Gaussian random variable which has same mean and same variance as the pretrained filters then no visually satisfying results are obtained.

\captionsetup[subfigure]{labelformat=empty}
\begin{figure}[h]
  \centering
  \subfloat[output (with pretraining)]{\includegraphics[width=0.3\linewidth]{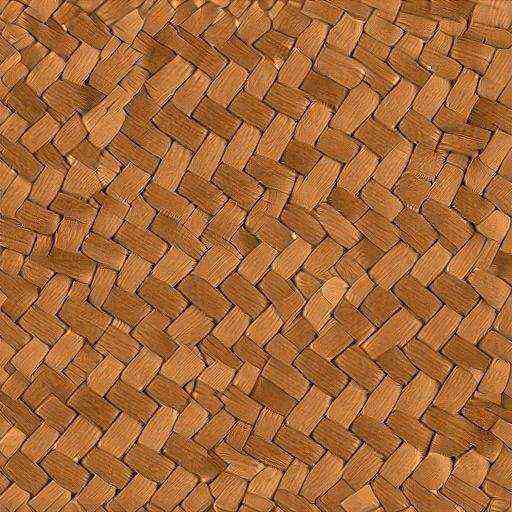}} \hfill 
  \subfloat[output (no pretraining)]{\includegraphics[width=0.3\linewidth]{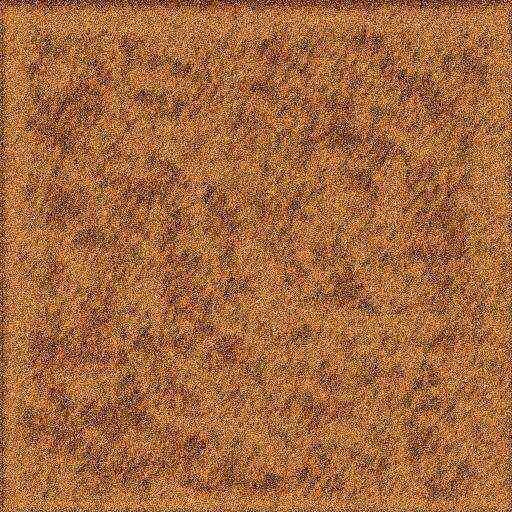}} \hfill
  \subfloat[exemplar image]{\includegraphics[width=0.3\linewidth]{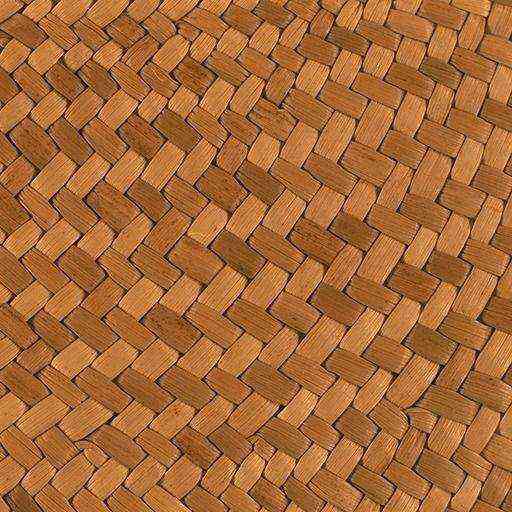}} \hfill     
  \caption{\figuretitle{Influence of the pretraining} The exemplar image on the right is a $512 \times 512$ color image. We present the output of the SOUL algorithm on this image after $10^4$ iterations. The hyperparameters are fixed as follows: $\delta_n = 10^{-3}$, $\gamma_n = 10^{-5}$ and $m_n = 1$.}
  \label{fig:pretraining}
\end{figure}
\captionsetup[subfigure]{labelformat=parens}

Another hyperparameter of the model is the set $\calJ$ of layers we consider to build our features, in \eqref{eq:neural_network}. We consider three settings: \begin{enumerate*}[label=(\roman*)]
\item shallow network; \label{item:shallow}
\item deep network; \label{item:deep}
\item full network. \label{item:full}
\end{enumerate*}
The structure of the network is recalled in \Cref{sec:structure_of_vgg19}. 
In \ref{item:shallow} we set $\calJ = \{1, 3, 6, 8, 11, 13\}$, in \ref{item:deep} we set $\calJ = \{15, 24, 26, 31\}$ and in \ref{item:full} we set $\calJ = \{1, 3, 6, 8, 11, 13, 15, 24, 26, 31\}$. Note that in the restricted models \ref{item:shallow} and \ref{item:deep} the dimension of the parameter space is reduced to $p=896$ respectively $p=1792$, whereas in the full model $p = 2688$. The influence of $\calJ$ is visually investigated in \Cref{fig:layers}. In what follows we consider the full CNN model given by \ref{item:full} in order to be able to synthesize a wide variety of texture images.

\captionsetup[subfigure]{labelformat=empty}
\begin{figure}[h]
  \centering
  \begin{tikzpicture}
    \node[inner sep=0pt] (name1) at (0, 0)
    {\includegraphics[width=.18\textwidth]{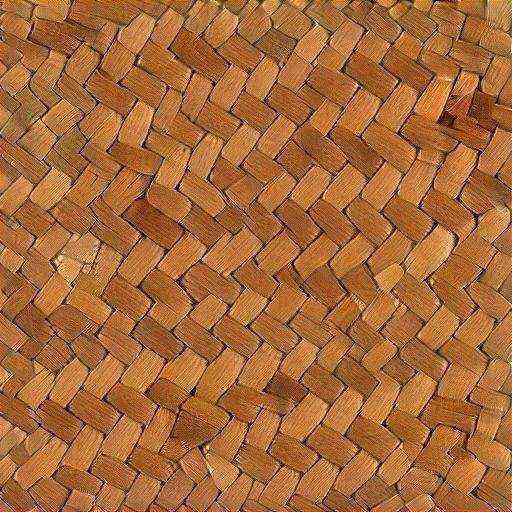}};
\node[inner sep=0pt] (name2) at (3.5, 0) 
{\includegraphics[width=.18\textwidth]{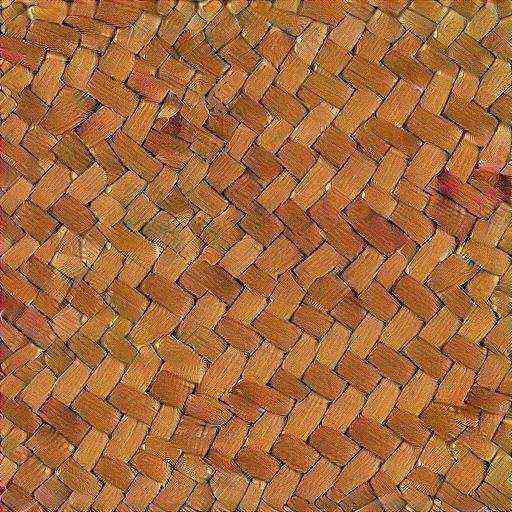}};
\node[inner sep=0pt] (name3) at (7, 0) 
    {\includegraphics[width=.18\textwidth]{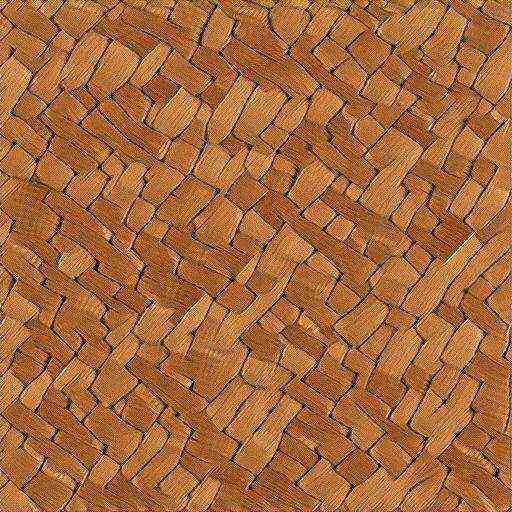}};
\node[inner sep=0pt] (name4) at (12, 0) 
{\includegraphics[width=.18\textwidth]{./data/wood.jpg}};

\node[inner sep=0pt] (name1) at (0, -3)
    {\includegraphics[width=.18\textwidth]{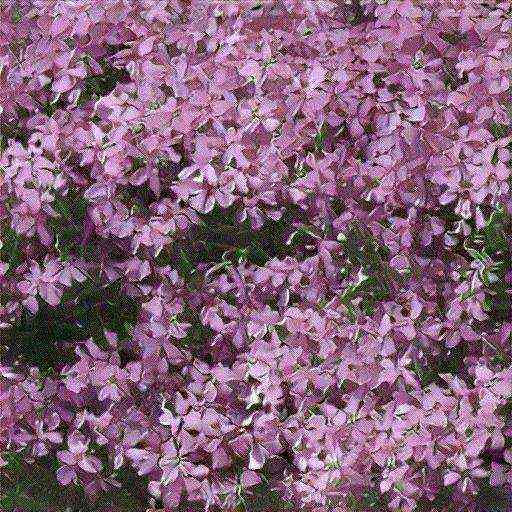}};
\node[inner sep=0pt] (name2) at (3.5, -3) 
{\includegraphics[width=.18\textwidth]{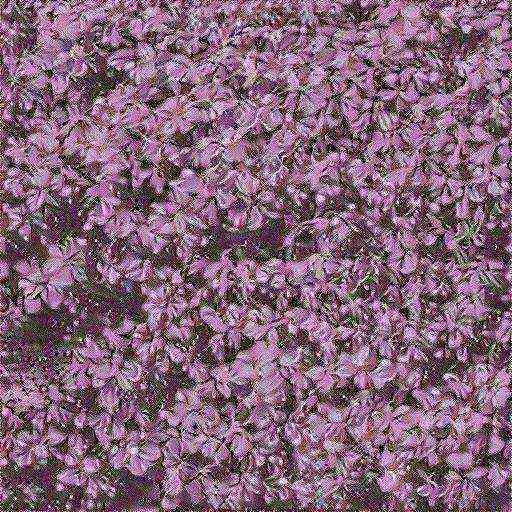}};
\node[inner sep=0pt] (name3) at (7, -3) 
    {\includegraphics[width=.18\textwidth]{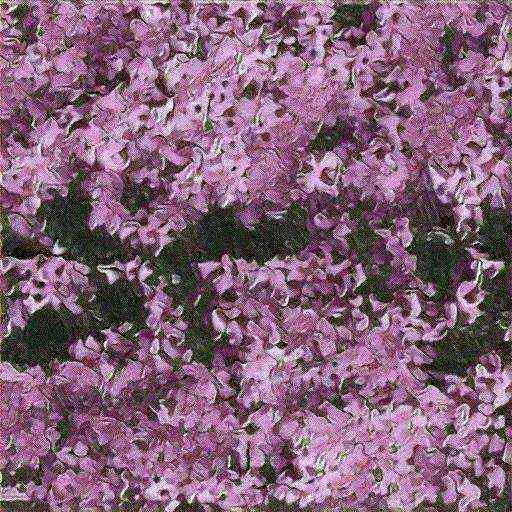}};
\node[inner sep=0pt] (name4) at (12, -3) 
    {\includegraphics[width=.18\textwidth]{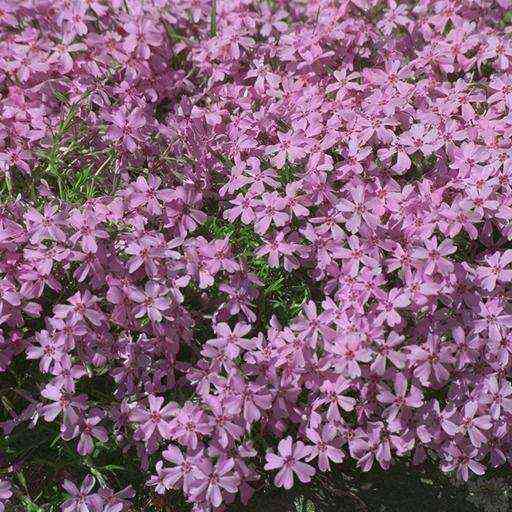}};

\node[label=below:\small full CNN, inner sep=0pt] (name1) at (0, -6)
    {\includegraphics[width=.18\textwidth]{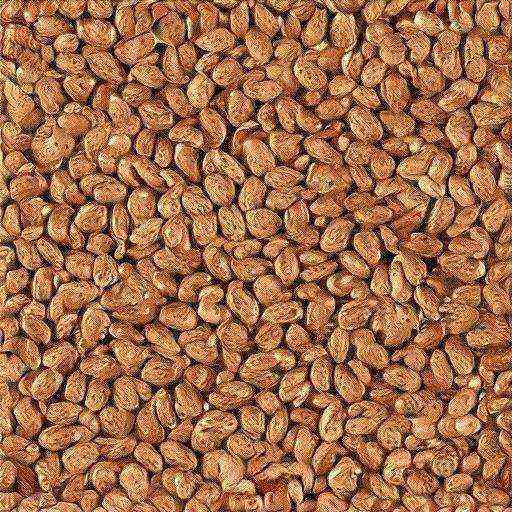}};
\node[label=below:\small deep CNN, inner sep=0pt] (name2) at (3.5, -6) 
{\includegraphics[width=.18\textwidth]{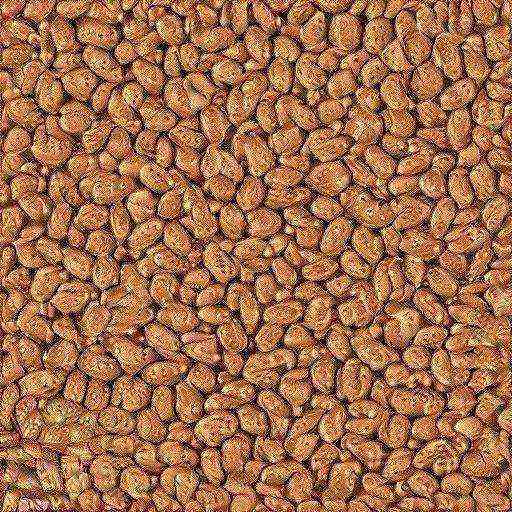}};
\node[label=below:\small shallow CNN, inner sep=0pt] (name3) at (7, -6) 
    {\includegraphics[width=.18\textwidth]{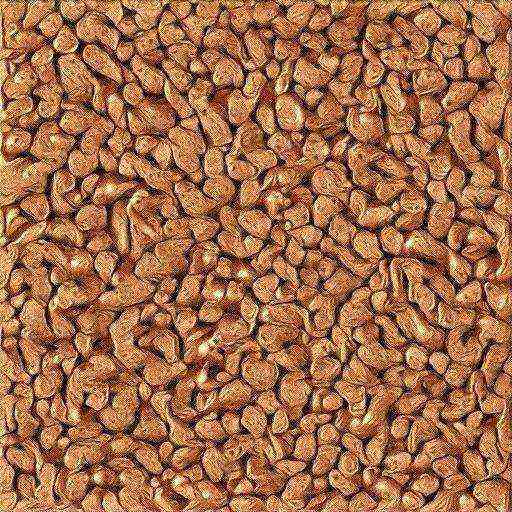}};
\node[label=below:\small exemplar image, inner sep=0pt] (name4) at (12, -6) 
{\includegraphics[width=.18\textwidth]{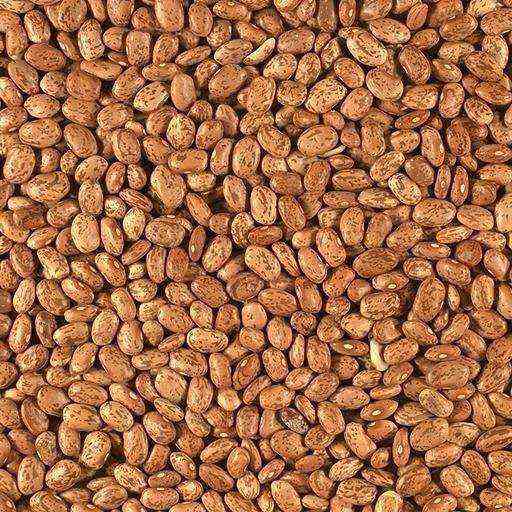}};

\draw [dashed, thick] (9.5,1.25) -- (9.5,-7.25); 
\end{tikzpicture}
\caption{\figuretitle{Influence of $\calJ$} As expected the best visual results of SOUL after $10^4$ iterations are obtained with the full CNN setting. The local structure and some details (the petals of the flowers, the form of the beans) are lost when using the shallow CNN setting. On the other hand, using only the deep part of the CNN is not suitable for texture with strong low frequency components. For instance in the flower image, almost
  no grass is retrieved when using the deep CNN setting. The hyperparameters are fixed as follows: $\delta_n = 10^{-3}$, $\gamma_n = 10^{-5}$ and $m_n = 1$.}
\label{fig:layers}
\end{figure}
\captionsetup[subfigure]{labelformat=parens}
It has been observed, in the case of microcanonical model, that using only CNN based features
is not sufficient to describe all the textures. For instance in \cite{liu2016texture}, the authors
propose to add spectrum constraints in order to reimpose some spatial arrangement in the images.
Similarly we can combine our neural network features with pixel-based features. In order to
impose some color statistics we set $\Fcolmean: \rset^d \to \rset^3$ and $\Fcolvar: \rset^d \to \mathcal{S}_3(\rset)$ defined for any $i \in \lbrace 1, 2, 3 \rbrace$ by
\begin{equation}
  \label{eq:color}
  \begin{aligned}
  &\tFcolmean(x)(i) = D^{-1} \sum_{k=1}^{D} x_i(k) \eqsp , \\
  &\tFcolvar(x) = D^{-1} \left(
    \begin{matrix} & x_1 - \tFcolmean(x)(1)
      \\ &x_2 - \tFcolmean(x)(2)
    \\ &x_3 - \tFcolmean(x)(3)\end{matrix} \right) \left(
    \begin{matrix} & x_1 - \tFcolmean(x)(1)
      \\ &x_2 - \tFcolmean(x)(2)
      \\ &x_3 - \tFcolmean(x)(3)\end{matrix} \right)^{\transpose} \eqsp , \\
  &\Fcolmean(x) = \tFcolmean(x) - \tFcolmean(x_0) \eqsp , \qquad \Fcolvar(x) = \tFcolvar(x) - \tFcolvar(x_0) \eqsp .
  \end{aligned}
\end{equation}
where $d = 3 D$ and $x = (x_1, x_2, x_3)$ where $x_i$ corresponds to the $i$-th color channel
of $x$. These features add $9$ more parameters to the model. We refer to this model as the
CNN + color features. Doing so the color statistics are imposed in expectation. It is also
natural to ask that all the produced images have exactly the same color statistics as the exemplar image, \ie \
that the equality holds \as . This procedure can be implemented by reimposing at each Langevin step the
mean and the color covariance matrix of the images. We call this model CNN + color projection. The effect of imposing,
in expectation or \as , the color constraints is investigated in \Cref{fig:color} and we observe that the proposed modifications do reimpose the color statistics of order 1 and 2.

\captionsetup[subfigure]{labelformat=empty}
\begin{figure}[h]
  \centering
  \begin{tikzpicture}
    \node[inner sep=0pt] (name1) at (0, 0)
    {\includegraphics[width=.18\textwidth]{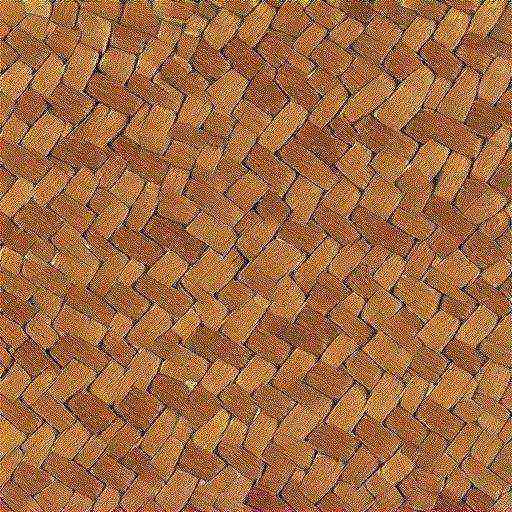}};
\node[inner sep=0pt] (name2) at (3.5, 0) 
{\includegraphics[width=.18\textwidth]{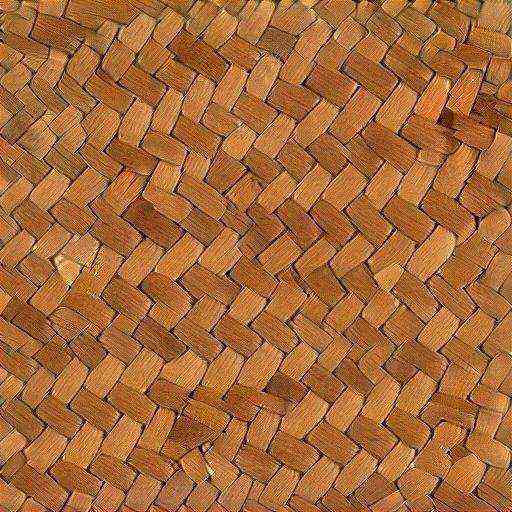}};
\node[inner sep=0pt] (name3) at (7, 0) 
    {\includegraphics[width=.18\textwidth]{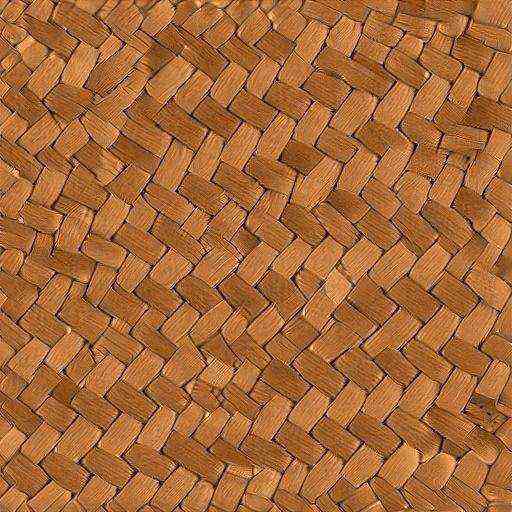}};
\node[inner sep=0pt] (name4) at (12, 0) 
{\includegraphics[width=.18\textwidth]{./data/wood.jpg}};

\node[inner sep=0pt] (name1) at (0, -3)
    {\includegraphics[width=.18\textwidth]{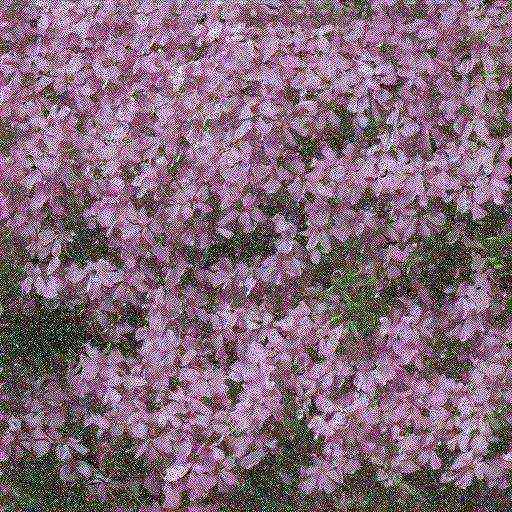}};
\node[inner sep=0pt] (name2) at (3.5, -3) 
{\includegraphics[width=.18\textwidth]{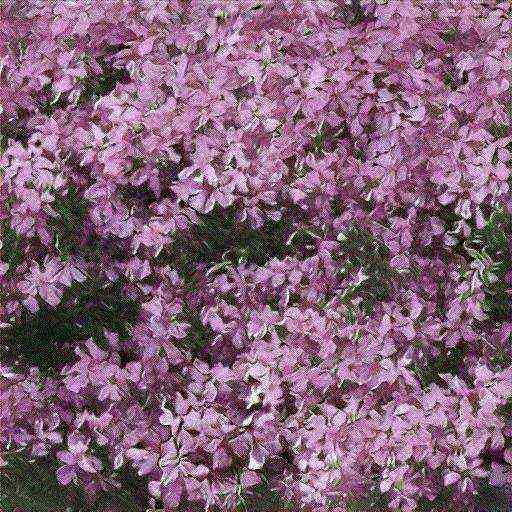}};
\node[inner sep=0pt] (name3) at (7, -3) 
    {\includegraphics[width=.18\textwidth]{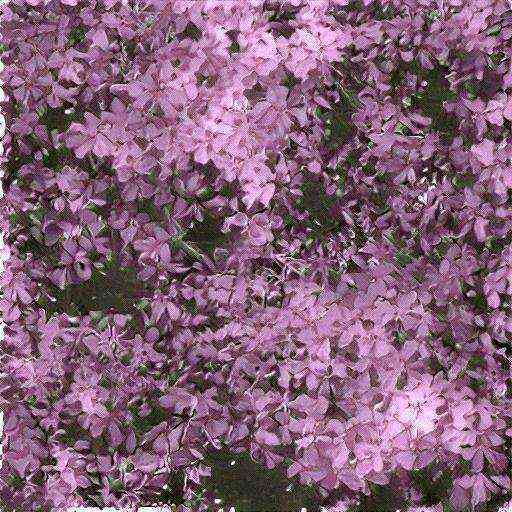}};
\node[inner sep=0pt] (name4) at (12, -3) 
    {\includegraphics[width=.18\textwidth]{./data/flower.jpg}};

\node[label=below: \small $\mathrm{CNN}$, inner sep=0pt] (name1) at (0, -6)
    {\includegraphics[width=.18\textwidth]{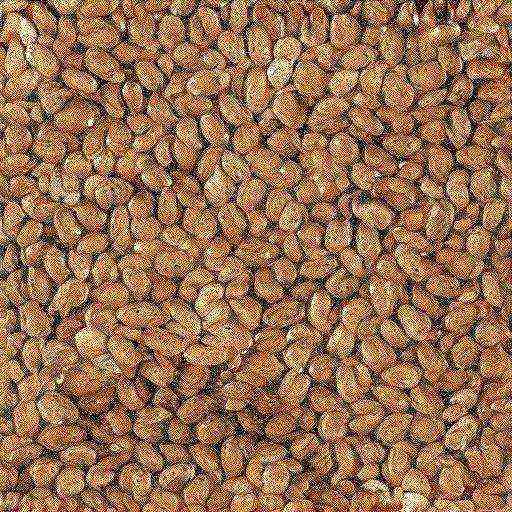}};
\node[label=below: \small $\mathrm{CNN + color}$, inner sep=0pt] (name2) at (3.5, -6) 
{\includegraphics[width=.18\textwidth]{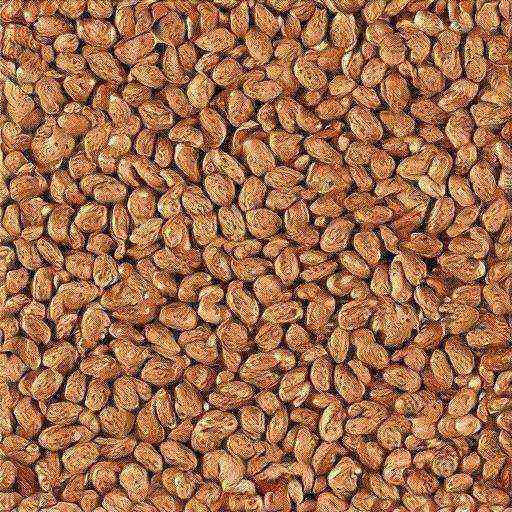}};
\node[label=below: \small $\mathrm{CNN + color \ projection}$, inner sep=0pt] (name3) at (7, -6) 
    {\includegraphics[width=.18\textwidth]{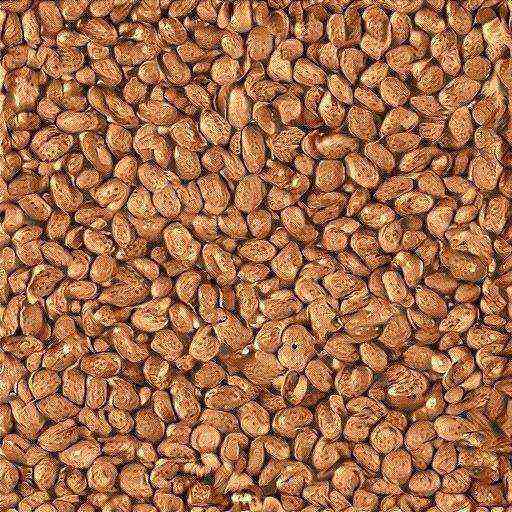}};
\node[label=below: \small $\mathrm{exemplar \ image}$, inner sep=0pt] (name4) at (12, -6) 
{\includegraphics[width=.18\textwidth]{./data/coffee.jpg}};

\draw [dashed, thick] (9.5,1.25) -- (9.5,-7.25); 
\end{tikzpicture}
\caption{\figuretitle{Color models} The SOUL algorithm with CNN features yields
  images with less contrast than the exemplar images. To address this issue we
  either introduce color features in the model (CNN + color), see
  \eqref{eq:color}, or reimpose the mean and color covariance of the image after
  each Langevin iteration (CNN + color projection). The results are similar for
  both methods. The hyperparameters are fixed as follows: $\delta_n = 10^{-3}$,
  $\gamma_n = 10^{-5}$ and $m_n = 1$.}
\label{fig:color}
\end{figure}
\captionsetup[subfigure]{labelformat=parens}

\paragraph{Behavior of the parameter sequence}

We now study the behavior of the sequence $(\theta_n)_{n \in \nset}$. 
In \Cref{fig:non_cv} we present the evolution
of $(\theta_n)_{n \in \nset}$ for some layers in $\calJ$ and three channels for each layer. The sequence $(\theta_n)_{n \in \nset}$ does not converge, even though we observe some stabilization of the averaged sequences.
The reasons for the failure of the convergence are twofold. First, in all our settings we fix the hyperparameters as follows: $\delta_n = 10^{-3}$, $\gamma_n = 10^{-5}$ and $m_n = 1$ but run only $10^5$ iterations. Considering a continuous Langevin dynamics, the images we observe correspond to a time $T = 10^5\times \gamma_n =~1$ of the evolution. Increasing the stepsize $\gamma_n$ is not an option since it yields diverging sequences of images. Second, the chain is slowly mixing and therefore it is hard to produce entirely different, yet visually coherent, samples with one run of SOUL. As a consequence, the Markov chain is prevented from efficiently exploring the image space.

\captionsetup[subfigure]{labelformat=empty}
\begin{figure}[h]
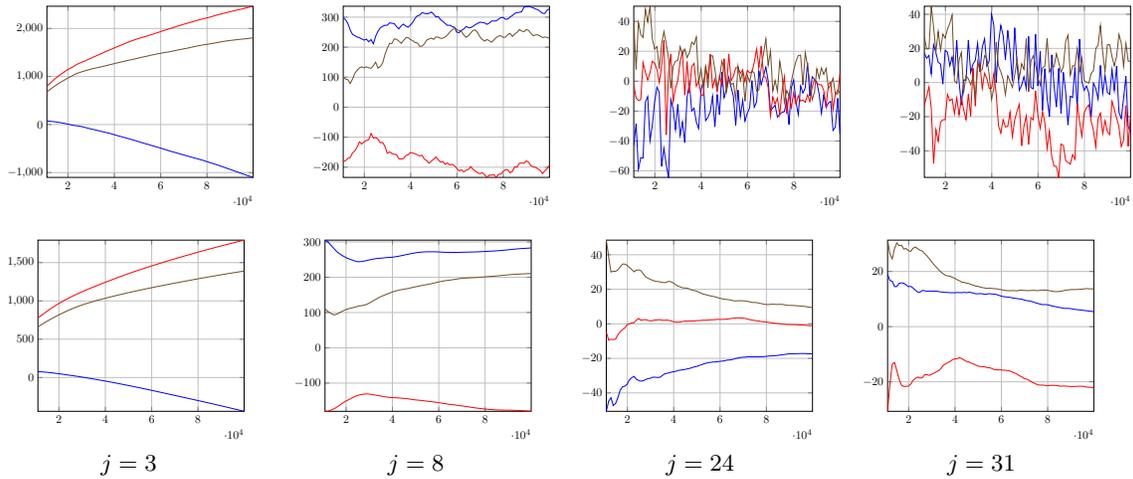

  \centering
  \subfloat{\begin{tikzpicture}[scale = 0.4]
  \begin{axis}[grid=major,no markers,domain=-5:5,enlargelimits=false]
    \input{./data/sweet_w_03_00.tex}
    \input{./data/sweet_w_03_10.tex}
    \input{./data/sweet_w_03_20.tex}
    \end{axis}
\end{tikzpicture}

} \hfill
  \subfloat{\begin{tikzpicture}[scale = 0.4]
  \begin{axis}[grid=major,no markers,domain=-5:5,enlargelimits=false]
    \input{./data/sweet_w_08_00.tex}
    \input{./data/sweet_w_08_10.tex}
    \input{./data/sweet_w_08_20.tex}
    \end{axis}
\end{tikzpicture}

} \hfill
  \subfloat{\begin{tikzpicture}[scale = 0.4]
  \begin{axis}[grid=major,no markers,domain=-5:5,enlargelimits=false]
    \input{./data/sweet_w_24_00.tex}
    \input{./data/sweet_w_24_10.tex}
    \input{./data/sweet_w_24_20.tex}
    \end{axis}
\end{tikzpicture}

} \hfill
  \subfloat{\begin{tikzpicture}[scale = 0.4]
  \begin{axis}[grid=major,no markers,domain=-5:5,enlargelimits=false]
    \input{./data/sweet_w_31_00.tex}
    \input{./data/sweet_w_31_10.tex}
    \input{./data/sweet_w_31_20.tex}
    \end{axis}
\end{tikzpicture}

} \hfill

  \subfloat[$j = 3$]{\begin{tikzpicture}[scale = 0.4]
  \begin{axis}[grid=major,no markers,domain=-5:5,enlargelimits=false]
    \input{./data/sweet_w_avg_03_00.tex}
    \input{./data/sweet_w_avg_03_10.tex}
    \input{./data/sweet_w_avg_03_20.tex}
    \end{axis}
\end{tikzpicture}

} \hfill
  \subfloat[$j = 8$]{\begin{tikzpicture}[scale = 0.4]
  \begin{axis}[grid=major,no markers,domain=-5:5,enlargelimits=false]
    \input{./data/sweet_w_avg_08_00.tex}
    \input{./data/sweet_w_avg_08_10.tex}
    \input{./data/sweet_w_avg_08_20.tex}
    \end{axis}
\end{tikzpicture}

} \hfill
  \subfloat[$j = 24$]{\begin{tikzpicture}[scale = 0.4]
  \begin{axis}[grid=major,no markers,domain=-5:5,enlargelimits=false]
    \input{./data/sweet_w_avg_24_00.tex}
    \input{./data/sweet_w_avg_24_10.tex}
    \input{./data/sweet_w_avg_24_20.tex}
    \end{axis}
\end{tikzpicture}

} \hfill
  \subfloat[$j = 31$]{\begin{tikzpicture}[scale = 0.4]
  \begin{axis}[grid=major,no markers,domain=-5:5,enlargelimits=false]
    \input{./data/sweet_w_avg_31_00.tex}
    \input{./data/sweet_w_avg_31_10.tex}
    \input{./data/sweet_w_avg_31_20.tex}
    \end{axis}
\end{tikzpicture}

} \hfill
  \caption{\figuretitle{Non convergence of the weights} For the layers corresponding to $j=3, 8, 24$ and $31$ we study, on three channels ($k=10, 20 $ and $30$), the behavior of the sequence $(\theta_n(i_{k,j}))_{n \in \nset}$ (first row) and the averaged sequence $(\bar{\theta}_n(i_{k,j}))_{n \in \nset}$ (second row), where $i_{k,j}$ is the index corresponding to layer $j$ and channel $k$. These sequences have not converged yet, although the averaged sequence seems to stabilize for some layers, \ie \ some values of $j$.}
  \label{fig:non_cv}
\end{figure}
\captionsetup[subfigure]{labelformat=parens}

It appears that the algorithm produces good visual results even though the parameter sequence is not stable. 

\paragraph{Arbitrary size synthesis}
We assess in \Cref{fig:bigger} that contrary to the algorithm proposed in \cite{lu2015learning}, our implementation can produce arbitrary large images from one input. Indeed, if for any $j \in \{1, \dots, M\}$ and $k \in \{1, \dots, c_j\}$, $\tilde{A}_j^k$ in \eqref{eq:def_cnn} is given by a convolutional operator, \eqref{eq:neural_network} can be defined for any $d \in \nset$. The number of features does not depend on the size of the image but only on the number of layers selected in the network, since we average the neural network response in \eqref{eq:neural_network}.
\begin{figure}[h]
  \centering
  \hfill
  \subfloat[]{\includegraphics[width=0.4\linewidth]{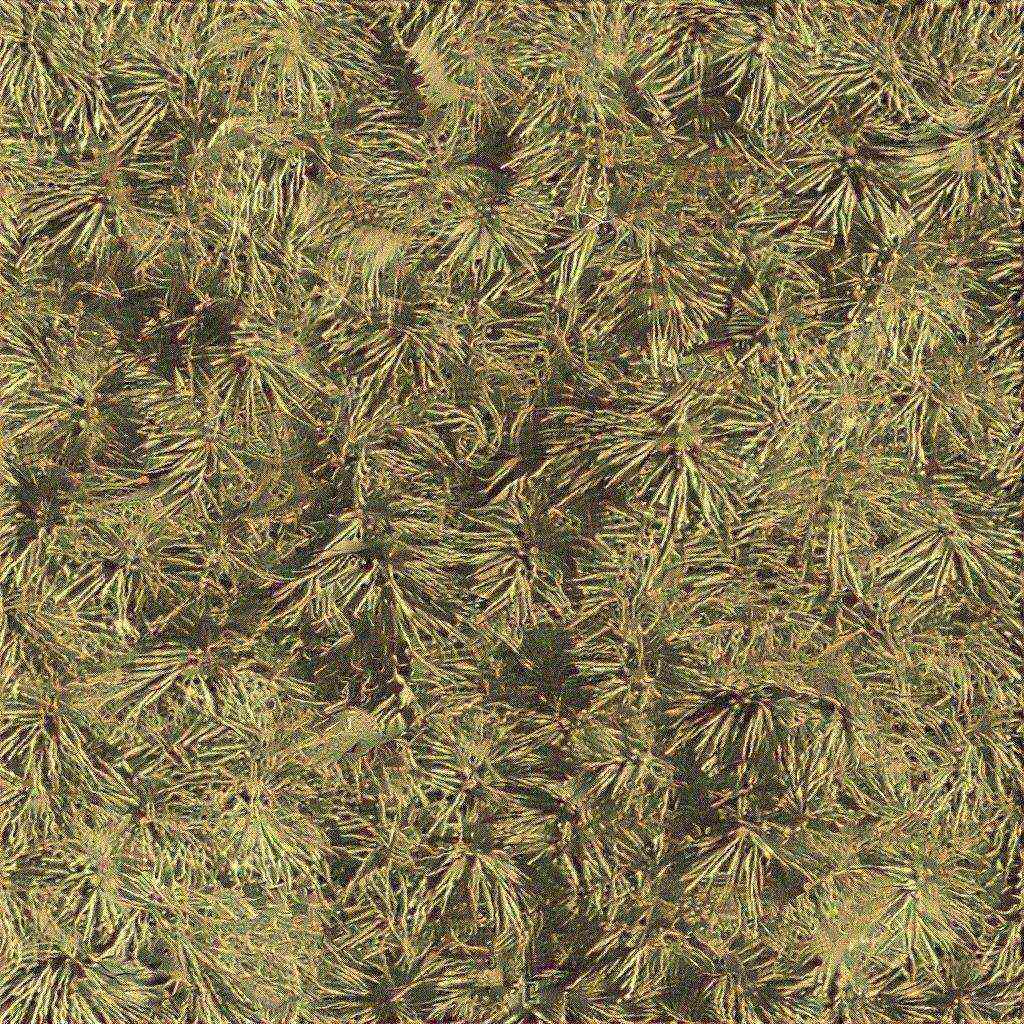}} \hfill 
  \subfloat[]{\includegraphics[width=0.2\linewidth]{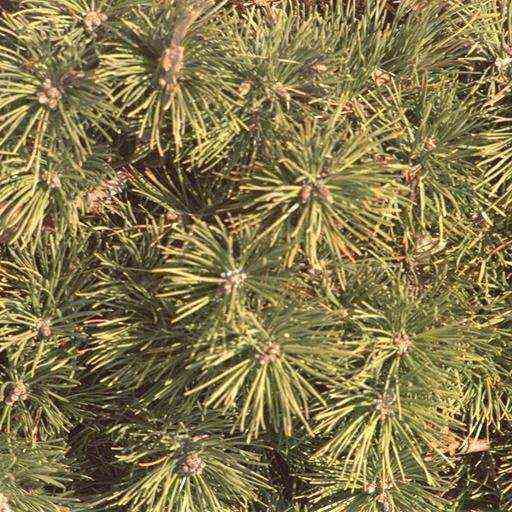}} \hfill
  \caption{\figuretitle{Arbitrary size synthesis} a texture of size
    $1024 \times 1024$ (a) is generated from an exemplar texture of size
    $512 \times 512$ (b). The hyperparameters are fixed as follows:
    $\delta_n = 10^{-3}$, $\gamma_n = 10^{-5}$ and $m_n = 1$.}
  \label{fig:bigger}
\end{figure}

\subsubsection{Comparison with existing methods}
\label{sec:real-texture-results}

In this section we compare the proposed algorithm with several state of the art examplar-based texture synthesis
methods. We use the CNN + color version of the features and set $\delta_n = 10^{-3}$, $\gamma_n = 10^{-5}$ and $m_n = 1$.
The algorithm is run for $10^4$ iterations for each image. For each comparison we systematically include the results obtained with the methodology proposed in \cite{gatys2015texture}, which is a microcanonical methodology using Gram matrices computed on neural network outputs as features.

First, we consider the Portilla-Simoncelli algorithm \cite{portilla2000parametric}, see \Cref{fig:simoncelli}, which is a microcanonical based methodology and does not rely on neural network features, see \Cref{fig:simoncelli}. Our algorithm and the one from \cite{gatys2015texture} provide visually satisfying results, whereas the method from \cite{portilla2000parametric} fails to produce realistic images.

We test our algorithm on texture images which do not exhibit salient spatial
structures. \Cref{fig:jetchev} shows the results obtained using the Generative
Adversarial Network approach proposed \cite{jetchev2016texture} in
\Cref{fig:jetchev}. It was already noted in \cite[Figure
26]{raadcisa:hal-01553841} that this generative fails to produce high quality
image in this case. On the other hand, our algorithm and the one from \cite{gatys2015texture}
yield good visual results.

We also compare our algorithm to the one of \cite{lu2015learning} in \Cref{fig:deepframe}. In \cite{lu2015learning}, the authors
propose a similar macrocanonical methodology but do not consider more than one convolutional neural network layer to build their features.

Another experiment on highly regular textures, comparing our algorithm with the ones of \cite{liu2016texture} and \cite{gonthier2019high}, is presented in \Cref{sec:highly-regul-text}.

\begin{figure}[h]
  \centering
  \resizebox{.7\linewidth}{!}{
  \begin{tikzpicture}
    \node[inner sep=0pt] (name1) at (0, 0)
    {\includegraphics[width=.18\textwidth]{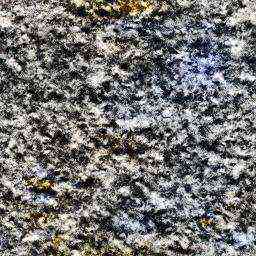}};
\node[inner sep=0pt] (name2) at (3.5, 0) 
{\includegraphics[width=.18\textwidth]{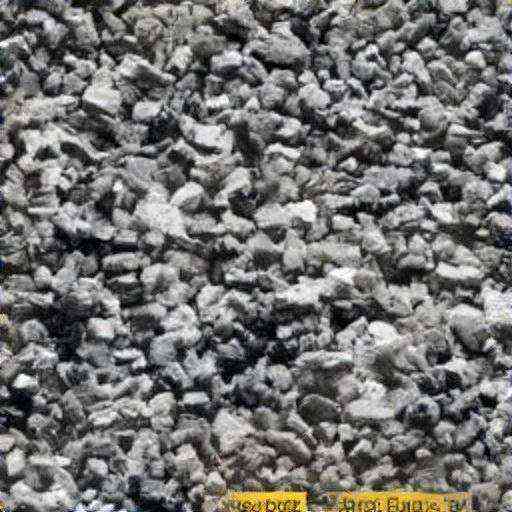}};
\node[inner sep=0pt] (name3) at (7, 0) 
    {\includegraphics[width=.18\textwidth]{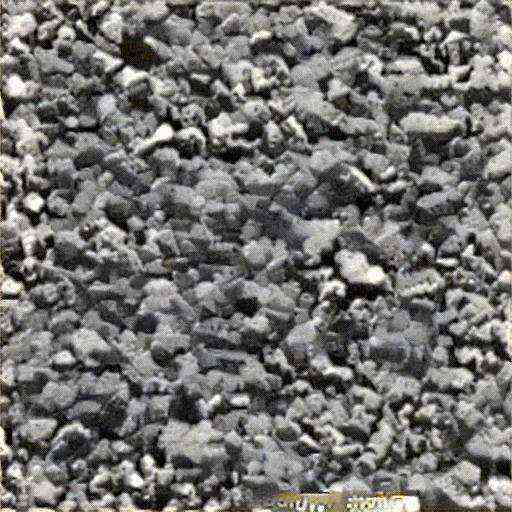}};
\node[inner sep=0pt] (name4) at (12, 0) 
{\includegraphics[width=.18\textwidth]{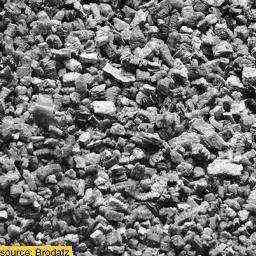}};

\node[inner sep=0pt] (name1) at (0, -3)
    {\includegraphics[width=.18\textwidth]{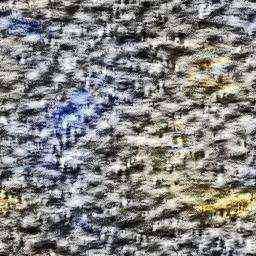}};
\node[inner sep=0pt] (name2) at (3.5, -3) 
{\includegraphics[width=.18\textwidth]{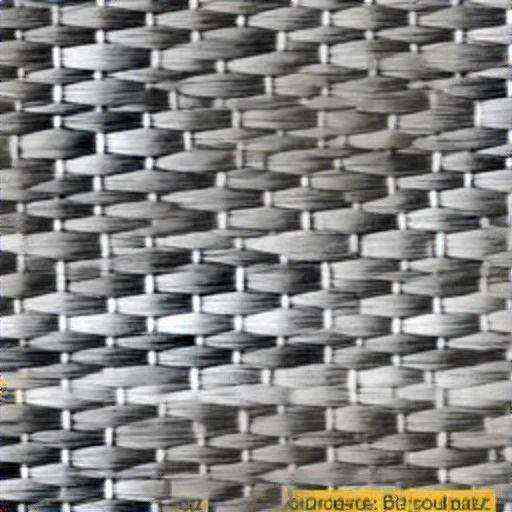}};
\node[inner sep=0pt] (name3) at (7, -3) 
    {\includegraphics[width=.18\textwidth]{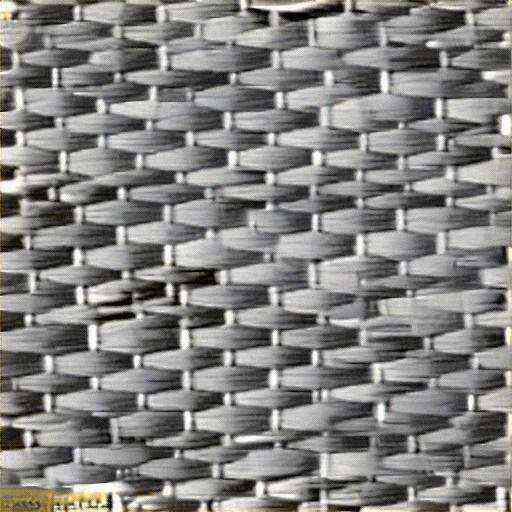}};
\node[inner sep=0pt] (name4) at (12, -3) 
    {\includegraphics[width=.18\textwidth]{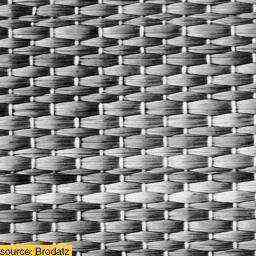}};

\node[label=below:\small Portilla-Simoncelli , inner sep=0pt] (name1) at (0, -6)
    {\includegraphics[width=.18\textwidth]{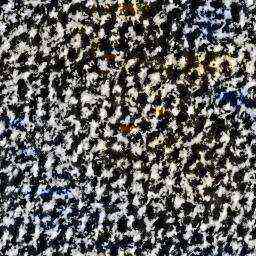}};
\node[label=below:\small Gatys, inner sep=0pt] (name2) at (3.5, -6) 
{\includegraphics[width=.18\textwidth]{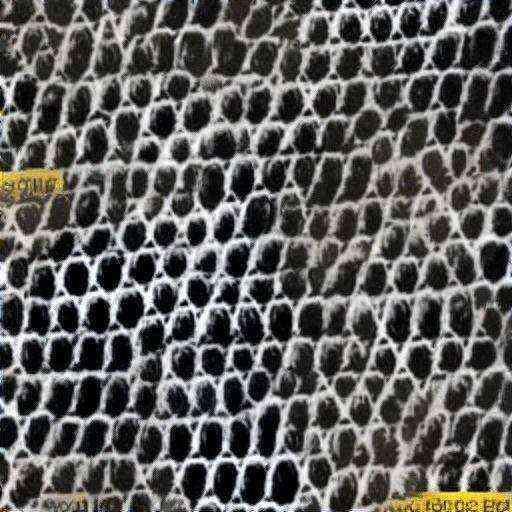}};
\node[label=below:\small ours , inner sep=0pt] (name3) at (7, -6) 
    {\includegraphics[width=.18\textwidth]{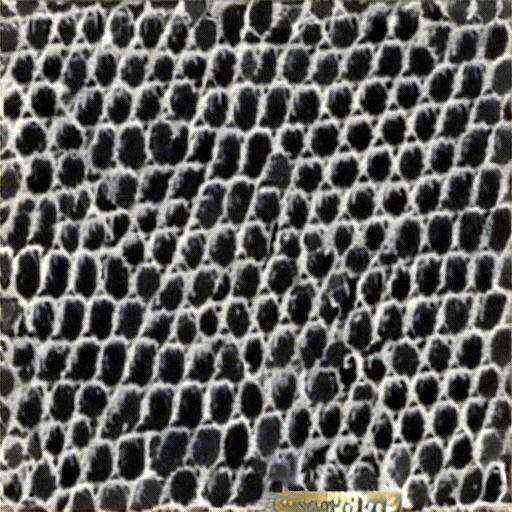}};
\node[label=below:\small exemplar image $x_0$, inner sep=0pt] (name4) at (12, -6) 
{\includegraphics[width=.18\textwidth]{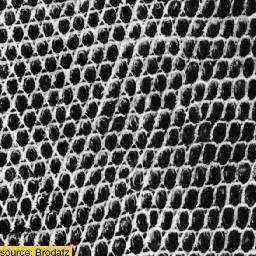}};

\draw [dashed, thick] (9.5,1.25) -- (9.5,-7.25); 
\end{tikzpicture}
}
\label{fig:simoncelli}
\caption{\figuretitle{Comparison with \cite{portilla2000parametric}} The images presented in the column ``Portilla-Simoncelli'' are synthesized with the algorithm introduced in \cite{portilla2000parametric}, the ones presented in the column ``Gatys'' are generated with \cite{gatys2015texture} and the third column contains our results. }
\end{figure}

\begin{figure}[h]
  \centering
  \resizebox{.7\linewidth}{!}{
  \begin{tikzpicture}
    \node[inner sep=0pt] (name1) at (0, 0)
    {\includegraphics[width=.18\textwidth]{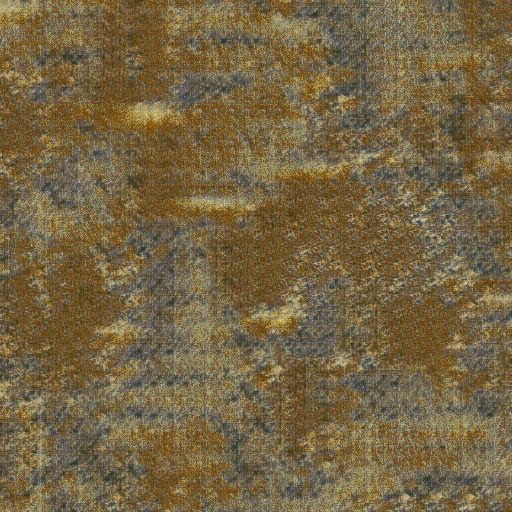}};
\node[inner sep=0pt] (name2) at (3.5, 0) 
{\includegraphics[width=.18\textwidth]{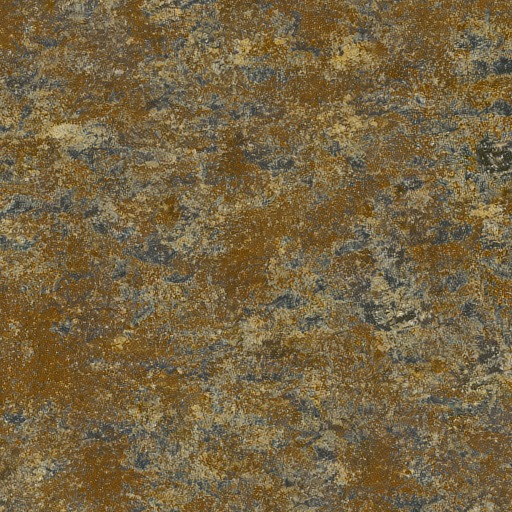}};
\node[inner sep=0pt] (name3) at (7, 0) 
    {\includegraphics[width=.18\textwidth]{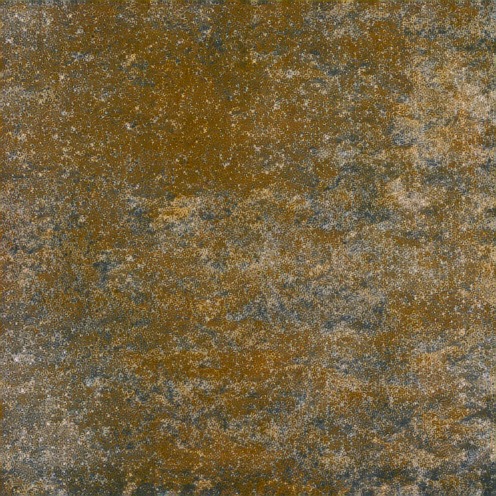}};
\node[inner sep=0pt] (name4) at (12, 0) 
{\includegraphics[width=.18\textwidth]{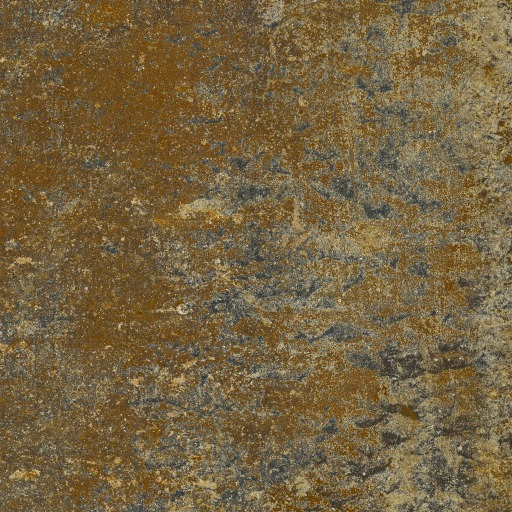}};

\node[inner sep=0pt] (name1) at (0, -3)
    {\includegraphics[width=.18\textwidth]{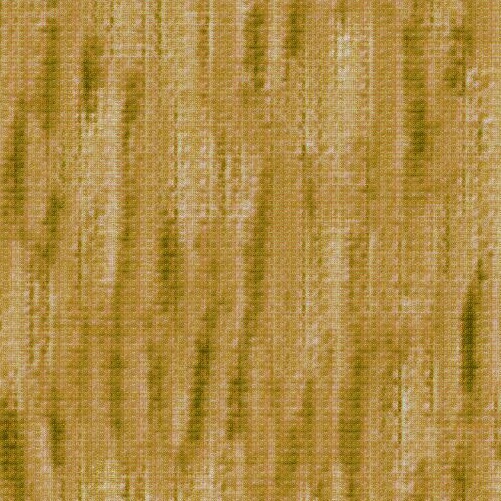}};
\node[inner sep=0pt] (name2) at (3.5, -3) 
{\includegraphics[width=.18\textwidth]{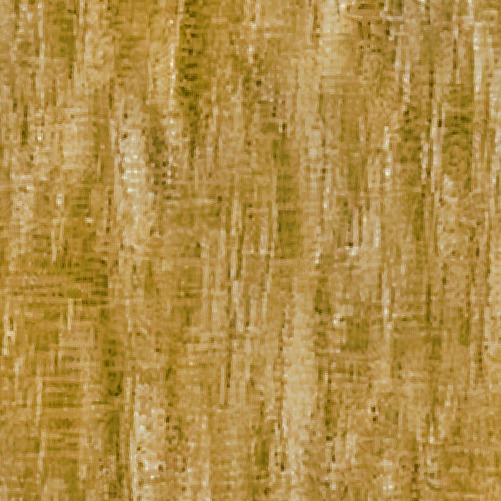}};
\node[inner sep=0pt] (name3) at (7, -3) 
    {\includegraphics[width=.18\textwidth]{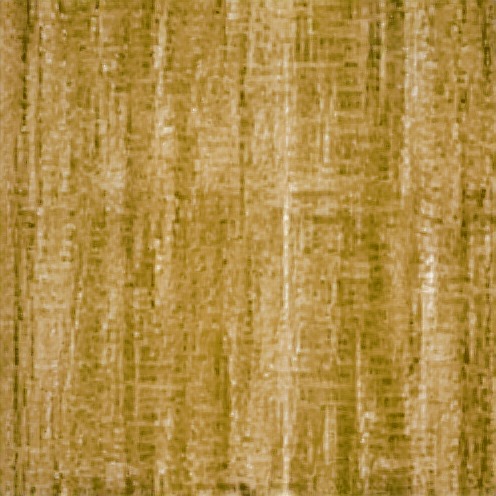}};
\node[inner sep=0pt] (name4) at (12, -3) 
    {\includegraphics[width=.18\textwidth]{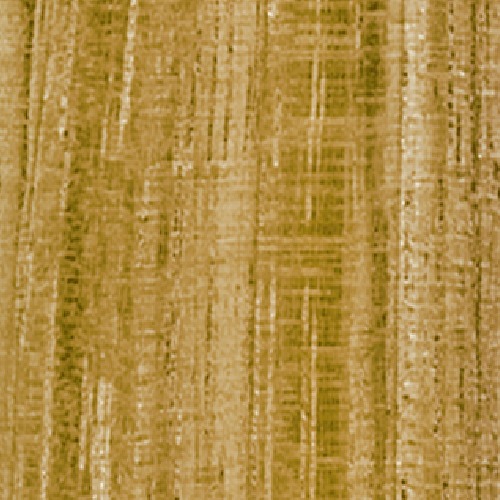}};

\node[label=below:\small Jetchev-Bergmann-Vollgraf , inner sep=0pt] (name1) at (0, -6)
    {\begin{tikzpicture}[spy using outlines={rectangle, yellow,magnification=4, connect spies}]
  \node {\pgfimage[interpolate=true,width=.18\linewidth]{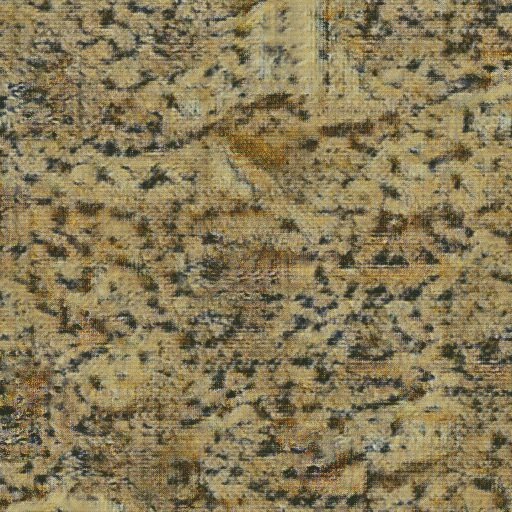}};
 \coordinate (spypoint) at (-1.1, 0.75);
 \coordinate (spyviewer) at (0.5,-0.5);
 \spy[width=1.5cm,height=1.5cm] on (spypoint) in node [fill=white] at (spyviewer);
\end{tikzpicture}};
\node[label=below:\small Gatys, inner sep=0pt] (name2) at (3.5, -6) 
{\begin{tikzpicture}[spy using outlines={rectangle, yellow,magnification=4, connect spies}]
  \node {\pgfimage[interpolate=true,width=.18\linewidth]{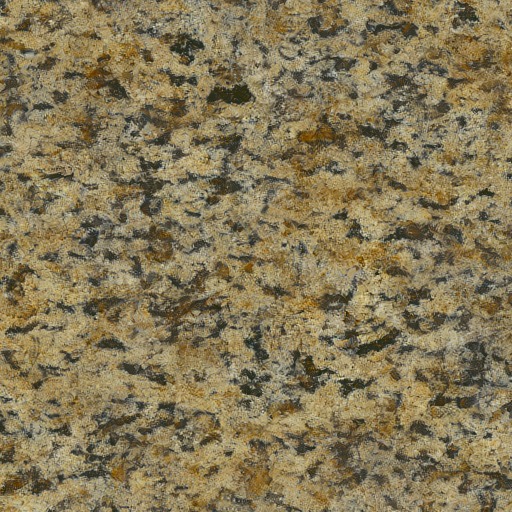}};
 \coordinate (spypoint) at (-1.1, 0.75);
 \coordinate (spyviewer) at (0.5,-0.5);
 \spy[width=1.5cm,height=1.5cm] on (spypoint) in node [fill=white] at (spyviewer);
\end{tikzpicture}};
\node[label=below:\small ours , inner sep=0pt] (name3) at (7, -6) 
{\begin{tikzpicture}[spy using outlines={rectangle, yellow,magnification=4, connect spies}]
  \node {\pgfimage[interpolate=true,width=.18\linewidth]{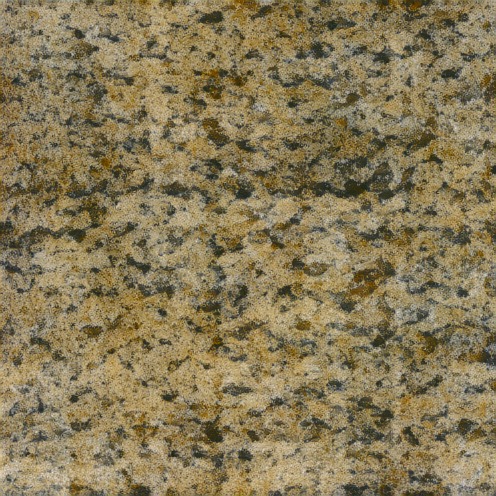}};
 \coordinate (spypoint) at (-1.1, 0.75);
 \coordinate (spyviewer) at (0.5,-0.5);
 \spy[width=1.5cm,height=1.5cm] on (spypoint) in node [fill=white] at (spyviewer);
\end{tikzpicture}};
\node[label=below:\small exemplar image $x_0$, inner sep=0pt] (name4) at (12, -6) 
{\begin{tikzpicture}[spy using outlines={rectangle, yellow,magnification=4, connect spies}]
  \node {\pgfimage[interpolate=true,width=.18\linewidth]{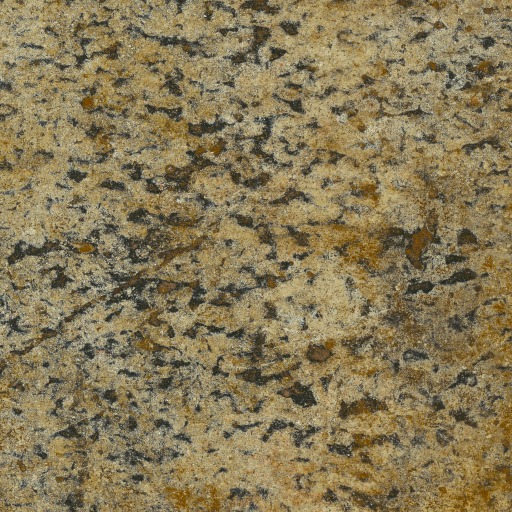}};
 \coordinate (spypoint) at (-1.1, 0.75);
 \coordinate (spyviewer) at (0.5,-0.5);
 \spy[width=1.5cm,height=1.5cm] on (spypoint) in node [fill=white] at (spyviewer);
\end{tikzpicture}};

\draw [dashed, thick] (9.5,1.25) -- (9.5,-7.25); 
\end{tikzpicture}
}
\label{fig:jetchev}
\caption{\figuretitle{Comparison with \cite{jetchev2016texture}} The images presented in the column ``Jetchev-Bergmann-Vollgraf'' are synthesized with the algorithm introduced in \cite{jetchev2016texture}, the ones presented in the column ``Gatys'' are generated with \cite{gatys2015texture} and the third column contains our results.}
\end{figure}
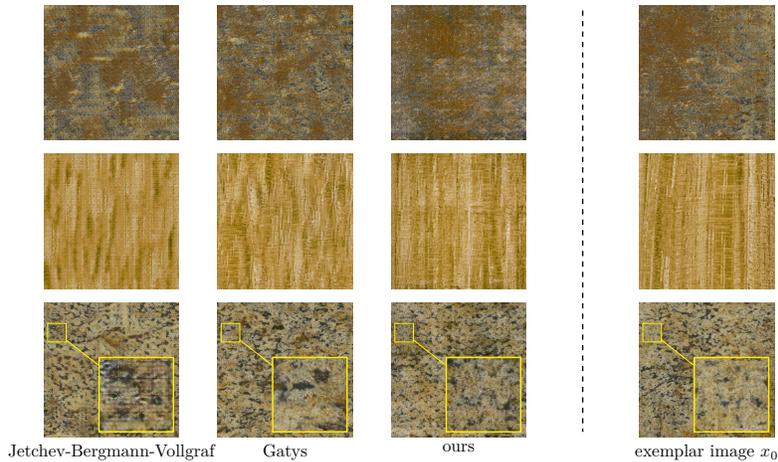

\begin{figure}[h]
  \centering
  \resizebox{.7\linewidth}{!}{
  \begin{tikzpicture}
    \node[inner sep=0pt] (name1) at (0, 0)
    {\includegraphics[width=.18\textwidth]{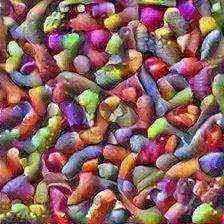}};
\node[inner sep=0pt] (name2) at (3.5, 0) 
{\includegraphics[width=.18\textwidth]{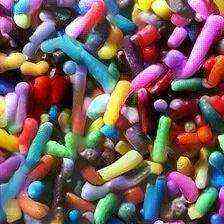}};
\node[inner sep=0pt] (name3) at (7, 0) 
    {\includegraphics[width=.18\textwidth]{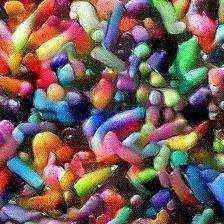}};
\node[inner sep=0pt] (name4) at (12, 0) 
{\includegraphics[width=.18\textwidth]{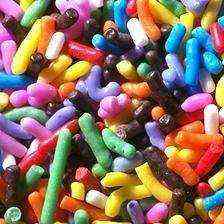}};

\node[inner sep=0pt] (name1) at (0, -3)
    {\includegraphics[width=.18\textwidth]{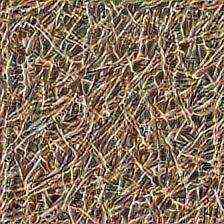}};
\node[inner sep=0pt] (name2) at (3.5, -3) 
{\includegraphics[width=.18\textwidth]{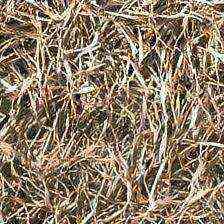}};
\node[inner sep=0pt] (name3) at (7, -3) 
    {\includegraphics[width=.18\textwidth]{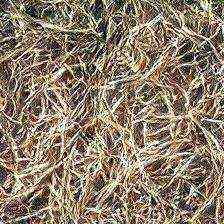}};
\node[inner sep=0pt] (name4) at (12, -3) 
    {\includegraphics[width=.18\textwidth]{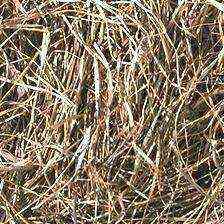}};

\node[label=below:\small Lu-Zhu-Wu , inner sep=0pt] (name1) at (0, -6)
    {\includegraphics[width=.18\textwidth]{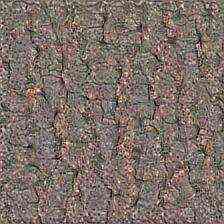}};
\node[label=below:\small Gatys, inner sep=0pt] (name2) at (3.5, -6) 
{\includegraphics[width=.18\textwidth]{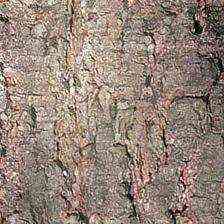}};
\node[label=below:\small ours , inner sep=0pt] (name3) at (7, -6) 
    {\includegraphics[width=.18\textwidth]{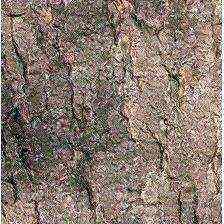}};
\node[label=below:\small exemplar image $x_0$, inner sep=0pt] (name4) at (12, -6) 
{\includegraphics[width=.18\textwidth]{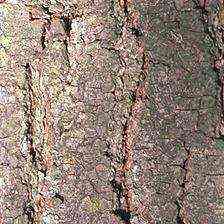}};

\draw [dashed, thick] (9.5,1.25) -- (9.5,-7.25); 
\end{tikzpicture}
}
\label{fig:deepframe}
\caption{\figuretitle{Comparison with \cite{lu2015learning}} The images presented in the column ``Lu-Zhu-Wu'' are synthesized with the algorithm introduced in \cite{lu2015learning}, the ones presented in the column ``Gatys'' are generated with \cite{gatys2015texture} and the third column contains our results. }
\end{figure}

\subsubsection{Texture style transfer}
\label{par:style_transfer}

We conclude this experimental part by considering other applications than texture synthesis and assess that the proposed algorithm can be used for the task of style transfer. Indeed
given one content image $\xcont$, a style image $\xstyle$, not necessarily of the same size, and $\Jcont \subset \calJ$
we consider the same CNN feature as before but $x_0$ is replaced by $\xcont$ for  $j \in \Jcont$ in \eqref{eq:neural_network}. In the rest of the neural network features, $x_0$ is
replaced by $\xstyle$ in \eqref{eq:neural_network}, \ie
\begin{equation}
    F(x) = \left( \overline{\scrG}_{j}^k(x)-\overline{\scrG}_{j}^k(x_0^j)\right)_{j \in \calJ, k \in \{1, \dots, c_{j}\}} \eqsp ,
  \end{equation}
  with $x_0^j = \xcont$ if $j \in \Jcont$ and $\xstyle$ otherwise.
These new features are well-suited to perform a style transfer task as illustrated in 
\Cref{fig:style_transfer} with $\Jcont = \{ 1, 3, 6, 8, 11 \}$ and $\calJ = \{1, 3, 6, 8, 11, 13, 15, 24, 26, 31\}$.
\captionsetup[subfigure]{labelformat=empty}
\begin{figure}[h]
  \centering
  \subfloat[\small exemplar image $x_0$]{\includegraphics[width=0.2\linewidth]{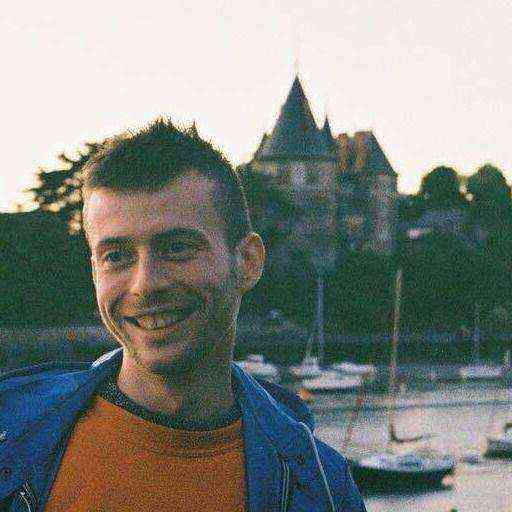}}
  \hfill
  \subfloat{\begin{tabular}[b]{c}
  \subfloat{ \subfloat{\includegraphics[width=0.2\linewidth]{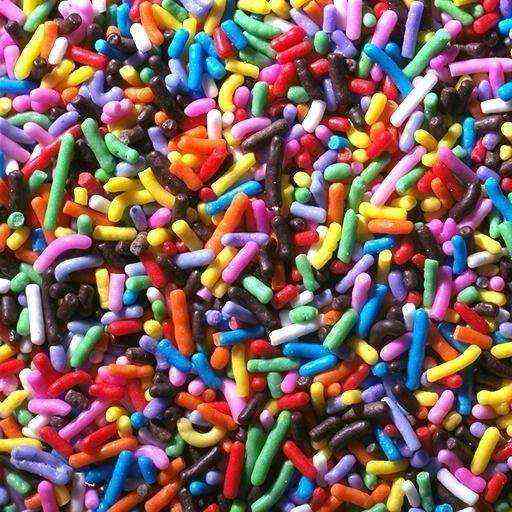}} \quad
             \subfloat{\includegraphics[width=0.2\linewidth]{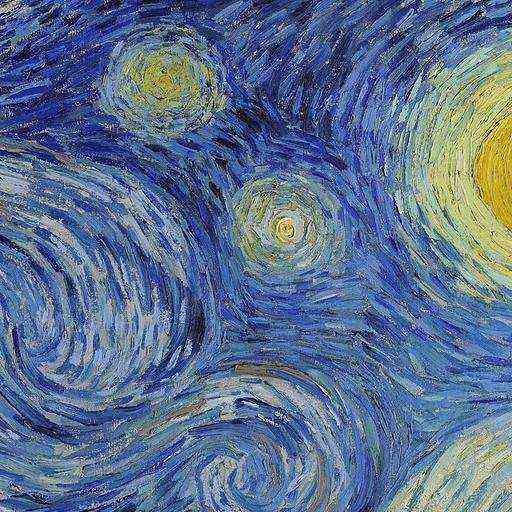}} \quad
             \subfloat{\includegraphics[width=0.2\linewidth]{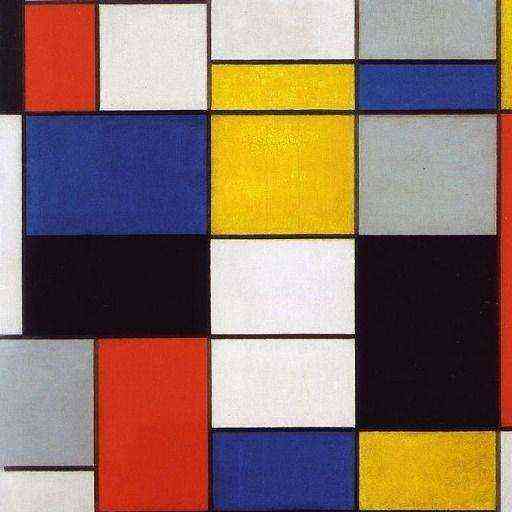}} \quad
              }
  \\\subfloat{\subfloat[(a)]{\includegraphics[width=0.2\linewidth]{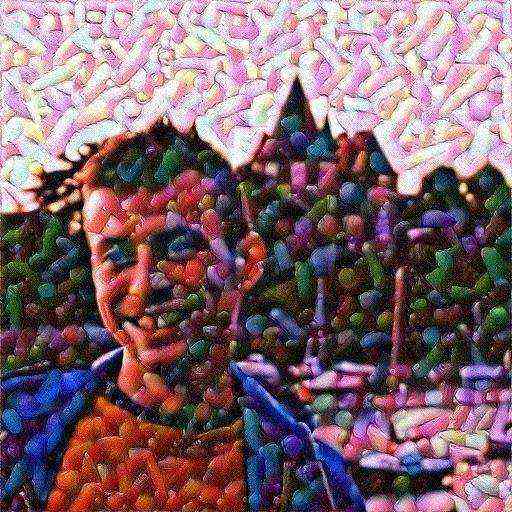}} \quad
              \subfloat[(b)]{\includegraphics[width=0.2\linewidth]{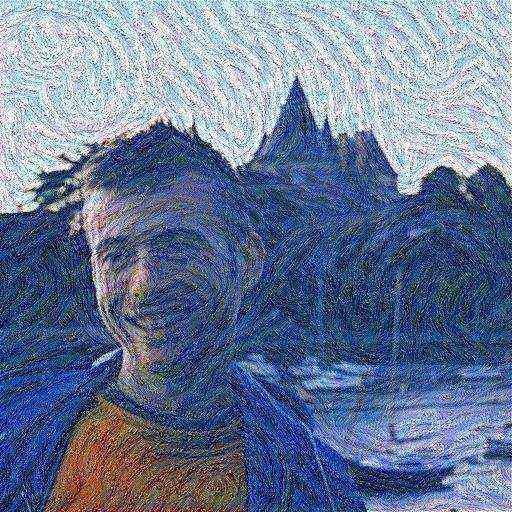}} \quad
              \subfloat[(c)]{\includegraphics[width=0.2\linewidth]{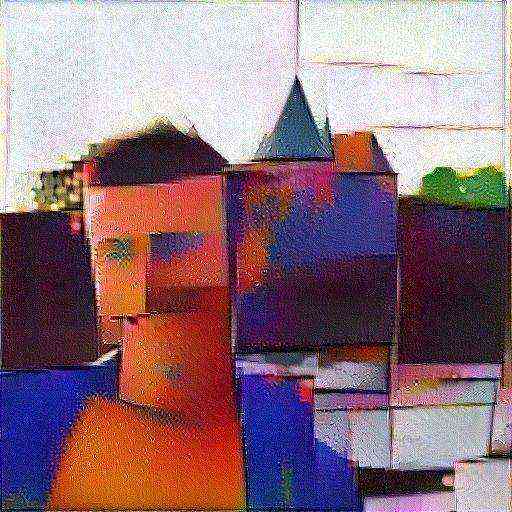}} \quad
              }
    \end{tabular}
    }
  \caption{\figuretitle{Style transfer} In (a), (b) and (c) we present the outputs of the SOUL algorithm with an exemplar content given in the leftmost column and exemplar style given by the first row. See \Cref{par:style_transfer} for more details.}
  \label{fig:style_transfer}
\end{figure}
\captionsetup[subfigure]{labelformat=parens}






  \clearpage
  \bibliographystyle{plainnat}
\bibliography{main_simods.bbl}
\appendix
\section{Proofs of \Cref{sec:maxim-entr-meas}}
\label{sec:inform-geom-proof}
We have the following variational formula which is an extension of \cite[Proposition 1.4.2]{dupuis1997weak} to the case where $F$ is not bounded. More precisely, allowing some growth on $F$, controlled by a parameter $\upalpha$, and restricting the set of probability measures we consider to $\Palpha$ we obtain the same equality. The proof is almost identical but is given for completeness.

\begin{proposition}
  \label{prop:var_dual}
  Assume \tup{\Cref{assum:sub_holder}}($\upalpha$) with $\upalpha >0$.  Then, for any $\theta \in \Theta_F$, with $\Theta_F$ defined by \eqref{eq:def_theta_F},
  \begin{equation} \inf_{\pi \in \Palpha}\defEnsLigne{ \KL{\pi}{\mu} + \langle \theta, \pi(F) \rangle } = -\log\defEns{\int_{\rset^d} \exp\parentheseDeux{-\langle \theta, F(x) \rangle} \rmd \mu(x)} \eqsp . \end{equation}
\end{proposition}

\begin{proof}
  Let $\theta \in \Theta_F$ and $\pi \in \Palpha$.
  Note that under \Cref{assum:sub_holder}($\upalpha$), $\pi(\norm{F}) < \plusinfty$ and $\pi(F)$ is well defined. 

  If $\KL{\pi}{\mu} = + \infty$, then $\KL{\pi}{\mu} + \langle \theta, \pi(F) \rangle = +\infty$.
  Consider now the case $\KL{\pi}{\mu} < + \infty$.
  By definition of $\Theta_F$, we can therefore consider $\pi_{\theta}$, the probability measure with density with respect to $\mu$ given for any $x \in \rset^d$ by
  \begin{equation}
    \frac{\rmd \pi_{\theta}}{\rmd \mu}(x) = \left . \exp [- \langle \theta, F(x) \rangle] \middle/ \int_{\rset^d} \exp [- \langle \theta, F(y) \rangle] \rmd \mu(y) \right . \eqsp.
  \end{equation}
  Note that since $\mu$-almost everywhere, ${(\rmd \pi_{\theta})}/{(\rmd \mu)}(x) >0$,   $\mu$ and $\pi_{\theta}$ are equivalent.
  Since $\KL{\pi}{\mu} < + \infty$,  $\pi \ll \mu$ which implies in turn $\pi \ll \pi_{\theta}$ and we have
  \begin{align}
    \KL{\pi}{\mu} &= \KL{\pi}{\pi_{\theta}} + \int_{\rset^d} \log\parenthese{\frac{\rmd \pi_{\theta}}{\rmd \mu}(x)}\rmd \pi(x) \\
    &= \KL{\pi}{\pi_{\theta}} - \langle \theta, \pi(F) \rangle - \log \defEns{\int_{\rset^d} \exp\parentheseDeux{-\langle \theta, F(x) \rangle} \rmd \mu(x)} \eqsp ,
  \end{align}
which concludes the proof, since $\KL{\pi}{\pi_{\theta}} \geq 0$.
\end{proof}

\subsection{Proof of \Cref{prop:primal_parametric}}
\label{prop:primal_parametric_proof}

The proof is divided in two parts:
\begin{enumerate}[wide, labelwidth=!, labelindent=0pt, label= (\alph*)]
\item Assume that there exists $\pi^{\star}$, solution of \primal . Let   $\convex$ be the convex set defined by $\ensembleLigne{\pi \in \PalphaF}{ \frac{\rmd \pi}{\rmd \pi^{\star}}(x) \leq 2 \ \text{for $\pi^{\star}$ almost every $x$}}$. For any $\pi_1 \in \convex$, consider $\pi_2$ with density with respect to $\pi^{\star}$, $\frac{\rmd \pi_2}{\rmd \pi^{\star}} = 2-  \frac{\rmd \pi_1}{\rmd \pi^{\star}}$ which by definition is an element of $\convex$ and $\pi^{\star} = (\pi_1 + \pi_2)/2$. Hence $\pi^{\star}$ is an algebraic inner point of $\convex$. Therefore using the equality case in \cite[Theorem 2.2]{csiszar1975divergence} we obtain that for any $\pi \in \convex$, $\KL{\pi}{\mu} = \KL{\pi}{\pi^{\star}} + \KL{\pi^{\star}}{\mu}$. 
  Using that for any $\pi \in \convex$, we have $\KL{\pi}{\pi^{\star}}$ and $\KL{\pi^{\star}}{\mu} < +\infty$, we get that 
  \begin{align}
    \label{eq:ortho}
    0 &= \int_{\rset^{\dim}} \log\parenthese{\frac{\rmd \pi}{\rmd \mu}(x)}\rmd \pi(x) - \int_{\rset^{\dim}} \log\parenthese{\frac{\rmd \pi}{\rmd \pi^{\star}}(x)}\rmd \pi(x)  - \int_{\rset^{\dim}} \log\parenthese{\frac{\rmd \pi^{\star}}{\rmd \mu}(x)}\rmd \pi^{\star}(x) 
        \\ &= \int_{\rset^d }\log\parenthese{\frac{\rmd \pi^{\star}}{\rmd \mu}(x)} \parentheseDeux{\frac{\rmd \pi}{\rmd \pi^{\star}}(x)  -1}\rmd \pi^{\star}(x)  \eqsp .
  \end{align}
  Since, $\KL{\pi^{\star}}{\mu} < +\infty$ we have that $\log(\frac{\rmd \pi^{\star}}{\rmd \mu}) \in \mathrm{L}^1(\rset^d, \pi^{\star})$. Let $V = \Span \ensembleLigne{1, \langle F, e_i \rangle}{ i=1, \dots, p}$ where $(e_i)_{i \in \lbrace 1, \dots, p \rbrace}$ is the canonical basis of $\rset^p$. $V$ is a finite dimensional (hence closed) subspace of $\mathrm{L}^1(\rset^{\dim}, \pi^{\star})$. Hence, in order to show that $\log(\frac{\rmd \pi^{\star}}{\rmd \mu}) \in V$ it suffices to show that $\log(\frac{\rmd \pi^{\star}}{\rmd \mu})\in (V^{\perp})^{\perp}$ by \cite[Proposition II.12]{brezis2011functional}.

  We identify the topological dual space of $\mathrm{L}^1(\rset^{\dim}, \pi^{\star})$ and $\mathrm{L}^{\infty}(\rset^{\dim}, \pi^{\star})$, see \cite[Theorem 6.16]{rudin2006real}. Let $h \in \mathrm{L}^{\infty}(\rset^{\dim}, \pi^{\star}) \cap V^{\perp}$. Then by definition, $\int_{\rset^d} F(x) h(x) \rmd \pi^{\star}(x) = 0$ and $\int_{\rset^d} h(x) \rmd \pi^{\star}(x) = 0$. The same goes for $\tilde{h} = h / \| h \|_{\infty}$, and we have that $\pi_h$ defined by $\frac{\rmd \pi_h}{\rmd \pi^{\star}} = 1 + \tilde{h}$ is an element of $\convex$. Therefore, by \eqref{eq:ortho}, we get that $\int_{\rset^{\dim}} \log(\frac{\rmd \pi^{\star}}{\rmd \mu}(x))h(x) = 0$ and $\log(\frac{\rmd \pi^{\star}}{\rmd \mu}) \in V$. More precisely, there exists $\theta^{\star} \in \rset^p$, $C \in \rset$ and $\msn \in \mcb{\rset^{\dim}}$ with $\pi^{\star}(\msn) = 0$ such that for $\mu$ almost any $x \in \rset^{\dim} \backslash \msn$,
\begin{equation}
  \log\parenthese{\frac{\rmd \pi^{\star}}{\rmd \mu}(x)} = \langle \theta^{\star}, F(x) \rangle + C \eqsp .
\end{equation}

We also have that $\pi^{\star}(\msn) = \int_{\msn} \frac{\rmd \pi^{\star}}{\rmd \mu}(y) \rmd \mu(y)$ and therefore for $\mu$ almost any $x \in \msn$, $\frac{\rmd \pi^{\star}}{\rmd \mu}(x) = 0$. 
Using \cite[Remark 2.14]{csiszar1975divergence}, for any $\pi \in \PalphaF$ such that $\KL{\pi}{\mu} < +\infty$ we have $\pi \ll \pi^{\star}$ and therefore 
\begin{equation} \label{eq:pi_n} \pi(\msn) = 0 \eqsp , \end{equation}
which concludes the proof for \eqref{eq:primal_solution}.

Finally, if there exists $\pi \in \PalphaF$ with $\KL{\pi}{\mu} < +\infty$ such that $\mu \ll \pi$ then by \eqref{eq:pi_n}, $\mu(\msn) = \pi(\msn) =0$ and we get that $\frac{\rmd \pi^{\star}}{\rmd \mu}(x) =  \exp\parentheseDeux{- \langle \theta^{\star}, F(x) \rangle} / \int_{\rset^d} \exp\parentheseDeux{- \langle \theta^{\star}, F(y) \rangle} \rmd \mu(y)$ for $\mu$ almost every $x \in \rset^d$. Then, using \Cref{prop:var_dual} and that $\theta^{\star} \in \Theta_F$, we have by definition of \dual , see \eqref{eq:lagrangian},
  \begin{multline}
\val(\mathrm{Q}) \leq \val(\mathrm{P}) = \KL{\pi^{\star}}{\mu} = -\log \parenthese{\int_{\rset^d}\exp \parentheseDeux{-\langle \theta^{\star}, F(x) \rangle} \rmd \mu(x)} \\ = \inf_{\pi \in \Palpha} \{ \KL{\pi}{\mu} - \langle \theta^{\star}, \pi(F) \rangle \} \leq \val(\mathrm{Q}) \eqsp ,
\end{multline}
which concludes the first part of the proof.
\item If there exists $\pi^{\star}$ solution of \primal \ with
  $\val(\mathrm{P})< +\infty$ then $\KL{\pi^{\star}}{\mu} < +\infty$ and
  $\pi^{\star} \in \PalphaF$. Now, assume that there exists $\pi \in \PalphaF$
  such that $\KL{\pi}{\mu} < +\infty$. Let $(\pi_n)_{n \in \nset}$ be a sequence
  of probability measures such that for any $n \in \nset$, $\pi_n \in \PalphaF$,
  $\KL{\pi_n}{\mu} < +\infty$ and
  $\inf_{\PalphaF} \KL{\pi}{\mu} = \lim_{n \to +\infty} \KL{\pi_n}{\mu}$. Using
  \cite[Lemma 5.1]{donsker1976asymptoticIII} we get that $(\pi_n)_{n \in \nset}$
  is tight. Therefore we can assume, up to extraction, that
  $(\pi_n)_{n \in \nset}$ converges to some probability measure $\pi^{\star}$
  for the weak topology. Since $\pi \mapsto \KL{\pi}{\mu}$ is lower
  semi-continuous \cite[Lemma 1.4.3 (b)]{dupuis1997weak} we obtain that
  $\KL{\pi^{\star}}{\mu} = \inf_{\PalphaF} \KL{\pi}{\mu}$.  We recall the
  Donsker-Varadhan variational formula \cite[Lemma 2.1]{donsker1976asymptoticI}
  stating that for any continuous, real-valued and bounded mapping $\phi$ we
  have for any $n \in \nset$
  \begin{equation}
    \label{eq:donsker}
    \int_{\rset^d} \phi(x) \rmd \pi_n(x) \leq \KL{\pi_n}{\mu} + \log \parenthese{\int_{\rset^d}\rme^{\phi(x)} \rmd \mu(x)} \eqsp .
  \end{equation}
  Let $\varphi_M: \ \rset^{\dim} \to \rset$ defined for any $M \geq 0$ and $x \in \rset^{\dim}$ by $\varphi_{M}(x) = \eta  \max(\norm{x}^{\upalpha'}, M)$, with $\eta$ defined in \Cref{assum:weak}($\upalpha'$). Using \eqref{eq:donsker}, \Cref{assum:weak}($\upalpha'$) and that $\varphi_M$ is continuous and bounded we get that for any $n \in \nset$ and $M \geq 0$
  \begin{align}
     \int_{\rset^{\dim}} \varphi_M(x) \rmd \pi_n(x) \leq \sup_{n \in \nset} \KL{\pi_n}{\mu} + \log \parenthese{\int_{\rset^{\dim}} \exp[ \eta \norm{x}^{\upalpha'}] \rmd \mu(x)} < +\infty \eqsp .
  \end{align}
Using the monotone convergence theorem we get that $\sup_{n \in \nset} \int_{\rset^{\dim}} \norm{x}^{\upalpha'} \rmd \pi_n(x) < +\infty$. By \cite[Theorem 5.16]{kallenberg2006foundations} there exist $(X_n)_{n \in \nset}$ a sequence of $\rset^{\dim}$ random variables and $X$ a $\rset^{\dim}$ random variable such that for any $n \in \nset$, $X_n$ is distributed according to $\pi_n$ and $X$ is distributed according to $\pi^{\star}$. Since $(\pi_n)_{n \in \nset}$ converges weakly to $\pi^{\star}$, $(X_n)_{n  \in \nset}$ converges in distribution towards $X$. Therefore, we get that $(\norm{X_n}^{\upalpha})_{n \in \nset}$ converges in distribution to $\norm{X}^{\upalpha}$ and $\sup_{n \in \nset} \expeLigne{\norm{X}^{\upalpha'}} < \infty$. By \cite[Lemma 3.11]{kallenberg2006foundations}, we get that $\expeLigne{\norm{X}^{\upalpha'}} =\int_{\rset^{\dim}} \norm{x}^{\upalpha'} \rmd \pi^{\star}(x) < +\infty$. Hence, $\pi^{\star} \in \Palpha$. In addition, since $F$ is continuous by \Cref{assum:sub_holder}, we have that $(F(X_n))_{n \in \nset}$ converges in distribution to $F(X)$. By \cite[13.3]{williams1991probability} and \Cref{assum:sub_holder}($\upalpha$), we have that $(F(X_n))_{n \in \nset}$ is uniformly integrable. Using \cite[Lemma 3.11]{kallenberg2006foundations} and that for any $n \in \nset$, $\pi_n(F) = 0$, we get $\lim_{n \to +\infty} \pi_n(F) = \pi^{\star}(F) = 0$ and $\pi^{\star} \in \PalphaF$, which concludes the proof.
\end{enumerate}
\subsection{Proof of \Cref{prop:existence_P}}
\label{prop:existence_P_proof}

  Let $L : \ \interior(\Theta_F) \to \rset$ be the function defined for any $\theta \in \interior(\Theta_F)$ by
  \begin{equation}L(\theta) = \log\defEns{\int_{\rset^d} \exp\parentheseDeux{-\langle \theta, F(x) \rangle} \rmd \mu(x)} \eqsp .\end{equation}
 We have $L \in \rmc^{\infty}(\interior(\Theta_F))$. The proposition is trivial if $\interior(\Theta_F) = \emptyset$. Therefore we suppose that $\interior(\Theta_F) \neq \emptyset$ and let $\theta_0 \in \interior(\Theta_F)$. Since $\interior(\Theta_F)$ is open, there exists $a_1  > 1$ such that $a_1\theta_0 \in \interior(\Theta_F)$. Let $a_2 > 1$ such that $1/a_1+1/a_2=1$. Let $R = \eta/(2a_2)$ with $\eta$ given in \Cref{assum:weak}($\upalpha$). For any $\theta \in \cball{\theta_0}{R}$, using that $t^j \rme^{-t} \leq j^j$ for $t \geq 0$ and $j \in \nset$, we have for any $x \in \rset^d$ and $k \in \nset$
  \begin{align}
    \norm{F(x)}^k \exp\parentheseDeux{-\langle \theta, F(x) \rangle} &\leq (k/R)^k \exp\parentheseDeux{R\norm{F(x)}} \exp\parentheseDeux{-\langle \theta, F(x) \rangle} \\
    &\leq (k/R)^k \exp\parentheseDeux{2R\norm{F(x)}} \exp\parentheseDeux{-\langle \theta_0 , F(x) \rangle} \eqsp ,
  \end{align}
  The last quantity is independent of $\theta$ and $\mu$-integrable using Hölder's inequality, since
  \begin{multline}
    \int_{\rset^d} \exp\parentheseDeux{2R\norm{F(x)}} \exp\parentheseDeux{-\langle \theta_0 , F(x) \rangle} \rmd \mu(x) \\ \leq \parenthese{\int_{\rset^d} \exp\parentheseDeux{\eta \norm{x}^{\upalpha}}\rmd \mu(x)}^{1/a_2} \parenthese{\int_{\rset^d} \exp\parentheseDeux{- \langle a_1 \theta_0, F(x) \rangle }\rmd \mu(x)}^{1/a_1} < +\infty \eqsp .
  \end{multline}
  This result implies that $L \in \rmc^{\infty}(\interior(\Theta_F))$. Therefore, if $\theta^{\star}$ is a stationary point, we have
\begin{equation}
  \pi_{\theta^{\star}}(F) \propto \int_{\rset^d} F(x) \exp\parentheseDeux{-\langle \theta^{\star}, F(x) \rangle} \rmd \mu(x) = 0 \eqsp .
  \end{equation}
Since $\pi_{\theta^{\star}} \in \Palpha$ we have $\pi_{\theta^{\star}} \in \PalphaF$. Since $\mu \ll \pi_{\theta^{\star}}$ we have $\pi \ll \pi_{\theta^{\star}}$ for any $\pi \ll \mu$. Therefore for any $\pi \in \PalphaF$ with $\pi \ll \mu$ we have
\begin{equation}
  \KL{\pi}{\mu} = \int_{\rset^d} \log \parenthese{\frac{\rmd \pi}{\rmd \mu}(x)} \rmd \pi(x)  = \KL{\pi}{\pi_{\theta^{\star}}} - L(\theta^{\star}) = \KL{\pi}{\pi_{\theta^{\star}}} + \KL{\pi_{\theta^{\star}}}{\mu} \geq \KL{\pi_{\theta^{\star}}}{\mu} \eqsp .
\end{equation}
If $\pi$ is not absolutely continuous with respect to $\mu$ then $\KL{\pi}{\mu} = +\infty$. Therefore we have that $\pi_{\theta^{\star}}$ solves \primal .

\subsection{Proof of \Cref{prop:existence_Q}}
\label{prop:existence_Q:proof}

We preface the proof with the following lemma

\begin{lemma}
  \label{lemma:ray_coercive}
  Let $h : \ \rset^p \to \rset$ be a convex function such that $h$ is ray-coercive, \ie \ for any $\theta \in \rset^p$, with $\norm{\theta}= 1$ we have $\lim_{t \to +\infty} h(t\theta) = +\infty$. Then $h$ is coercive, \ie \ $\lim_{\norm{\theta} \to +\infty} h(\theta) = +\infty$.
\end{lemma}

\begin{proof}
  Assume that $h$ is not coercive. Then there exists a sequence $(\theta_n)_{n \in \nset} \in (\rset^{d} \without{0})^{\nset}$ such that $\lim_{n \to +\infty} \norm{\theta_n} = +\infty$ and the sequence $(h(\theta_n))_{n \in \nset}$ is bounded. Upon extraction we can assume that $\lim_{n \to + \infty} \theta_n / \norm{\theta_n} = \tilde{\theta}$. We have for any $t \in \rset$
  \begin{equation}
    \label{eq:decompo_h}
    h(t \tilde{\theta}) = h(t \tilde{\theta}) - h(t \theta_n / \norm{\theta_n}) + h(t \theta_n / \norm{\theta_n}) \eqsp .
  \end{equation}
Let $t \geq 0$ and $\vareps >0$. Since $h$ is convex, $h$ is continuous and there exists $n_0 \in \nset$ such that for any $n \in \nset$ with $n \geq n_0$,
\begin{equation}
  \label{eq:majo_cont}
  \norm{h(t \tilde{\theta}) - h(t \theta_n / \norm{\theta_n})} \leq \vareps \eqsp , \qquad h(t \theta_n / \norm{\theta_n}) \leq t / \norm{\theta_n} h(\theta_n) + (1 - t / \norm{\theta_n}) h(0) \eqsp .
\end{equation}
Combining \eqref{eq:decompo_h} and \eqref{eq:majo_cont} we obtain that for any $t \geq 0$,
\begin{equation}
  h(t \tilde{\theta}) \leq \vareps + \norm{h(0)} + \sup_{n \in \nset} \norm{h(\theta_n)} < +\infty \eqsp ,
\end{equation}
which is absurd. Hence, $h$ is coercive.
\end{proof}

We now turn to the proof of \Cref{prop:existence_Q}.
We divide the proof in two parts.
\begin{enumerate}[wide, labelwidth=!, labelindent=0pt, label= (\alph*)]
\item Using that $\Theta_F = \rset^p$  and the first part of \Cref{prop:existence_P} we have that $L : \ \Theta_F \to \rset$ is continuously differentiable over $\rset^p$. In addition, for any $\theta_1, \theta_2 \in \rset^p$ and $\uptheta \in \ooint{0,1}$ we have
\begin{multline}
  \int_{\rset^d} \exp \parentheseDeux{-\langle \uptheta \theta_1 + (1- \uptheta)\theta_2, F(x) \rangle} \rmd \mu(x) \\ = \int_{\rset^d} \exp \parentheseDeux{-\uptheta \langle  \theta_1 , F(x) \rangle} \exp \parentheseDeux{ (1- \uptheta) \langle \theta_2, F(x) \rangle} \rmd \mu(x) \eqsp .
\end{multline}
Applying Hölder's inequality we get that
\begin{multline}
  L(\uptheta \theta_1 + (1-\uptheta)\theta_2) = \log \defEns{ \int_{\rset^d} \exp \parentheseDeux{-\langle \uptheta \theta_1 + (1- \uptheta)\theta_2, F(x) \rangle} \rmd \mu(x)} \\ \leq \uptheta \log \defEns{ \int_{\rset^d} \exp \parentheseDeux{- \langle  \theta_1 , F(x) \rangle}  \rmd \mu(x)} + (1- \uptheta) \log \defEns{ \int_{\rset^d} \exp \parentheseDeux{  \langle \theta_2, F(x) \rangle} \rmd \mu(x) }\eqsp ,
\end{multline}
hence $L$ is convex. Using the monotone convergence theorem and \eqref{eq:identifiability} we have that for any $\theta \in \rset^p$ with $\norm{\theta} = 1$,
\begin{multline}
  \lim_{t \to +\infty} L(t \theta) = \lim_{t \to +\infty} \log \defEns{\int_{\rset^d} \exp \parentheseDeux{- t \langle \theta, F(x) \rangle} \rmd \mu(x) }  \\ \geq \lim_{t \to +\infty} \log \defEns{\int_{\ensemble{x \in \rset^d}{ \langle \theta, F(x) \rangle < 0}} \exp \parentheseDeux{- t \langle \theta, F(x) \rangle} \rmd \mu(x) } =\infty \eqsp ,
\end{multline}
which implies that $L$ is ray-coercive. Combining this result, the fact that $L$ is convex and \Cref{lemma:ray_coercive} we get that $L$ is coercive. Since $L$ is continuous and coercive it admits a minimizer $\theta^{\star}$ and therefore $\nabla L(\theta^{\star}) = 0$. Applying the second part of \Cref{prop:existence_P} concludes the first part of the proof.
\item Let $x \in F^{-1}(\{0\})$ such that $\det(\rmD F(x) \rmD F(x)^{\transpose}) > 0$. We obtain that $\kernellin{\rmD F(x)^{\transpose}} = \kernellin{\rmD F(x) \rmD F(x)^{\transpose}} = \{ 0 \}$. Hence, $\rank(\rmd F(x)) = \rank(\rmd F(x)^{\transpose}) = p$ and $\rmd F(x)$ is surjective. Using the submersion theorem, there exists $G : \msu \to \rset^d$ with $\msu$ an open neighborhood of $0 \in \rset^{p}$ such that for any $\zeta \in \msu$, $F(G(\zeta)) = \zeta$. Therefore, for any $\theta \in \rset^p$ with $\norm{\theta} = 1$ there exists $\zeta_{\theta}$ such that $\langle \theta, F(G(\zeta_{\theta})) \rangle =  \langle \theta, \zeta_{\theta} \rangle <0$. Hence, since $F$ is continuous, there exists an open set $\msv$ in $\rset^d$ such that for any $y \in \msv$, $\langle F(y), \theta \rangle < 0$. Combining this result with the fact that for any $\msa$ open and $\msa \neq \emptyset$, $\mu(\msa) > 0$ we conclude the proof.

\end{enumerate}
\subsection{Proof of \Cref{prop:necessary_Q}}
\label{prop:necessary_Q:proof}

Let $\theta \in \rset^{p}$ such that $\norm{\theta} = 1$ and such that $\mu\parentheseLigne{\ensembleLigne{x \in \rset^d}{\langle F(x), \theta\rangle \leq 0}} =0$. Then, we have using the dominated convergence theorem
\begin{equation}
  \label{eq:zero_lim}
  \lim_{t \to +\infty} \int_{\rset^{\dim}} \exp[- t \langle \theta, F(x) \rangle ] \rmd \mu(x) = \lim_{t \to +\infty} \int_{\ensembleLigne{x \in \rset^d}{\langle F(x), \theta\rangle > 0}} \exp[- t \langle \theta, F(x) \rangle ] \rmd \mu(x) = 0 \eqsp .
\end{equation}
Recall that \Cref{assum:sub_holder}($\upalpha$) and \Cref{assum:weak}($\upalpha'$) imply that $\Theta_F = \rset^{p}$.
Therefore, using \eqref{eq:zero_lim}, we have $v(\mathrm{Q}) = - \inf_{\theta \in \rset^{p}} L(\theta) = + \infty$. Since $v(\mathrm{P}) \geq v(\mathrm{Q})$, see \eqref{eq:lagrangian}, we have $v(\mathrm{P}) = +\infty$. Hence, for any $\pi \in \PalphaF$, $\KL{\pi}{\mu} = +\infty$ and any $\pi \in \PalphaF$ solves \primal . We conclude upon remarking that $\updelta_{x_0} \in \PalphaF$.

\subsection{Proof of \Cref{prop:neural_network}}
\label{prop:neural_network:proof}

We start to show that for any $j \in \{1, \dots, M\}$ there exists $\vareps_j >0$ such that for any $x \in \cball{x_0}{\vareps_j}$
\begin{equation}
  \label{eq:local_lin}
  \scrG_j(x) = A_{j,+} \dots A_{1,+} (x) \eqsp .
\end{equation}
Namely, for any $j \in \{1, \dots, M\}$, $\scrG_j$ is locally affine around $x_0$.
We have that  $x_0 \in \rset^d \backslash \msn_1$ with $\msn_1 = \bigcup_{k = 1}^{c_1 \times n_1} \kernellin{e_{k,1}^{\transpose} A_1}$, where $(e_{k,1})_{k \in \{1, \dots, c_1 \times n_1\}}$ is the canonical basis of $\rset^{c_1 \times n_1}$ by \eqref{eq:condition_x0}. For any $k \in \{1, \dots, n_1 \times c_1 \}$, $A_1(x_0)(k) > 0$ or $A_1(x_0)(k) < 0$. Therefore, since $A_1$ is continuous, there exists $\vareps_1 >0$ such that for any $x \in \cball{x_0}{\vareps_1}$, $A_1(x)(k) A_1(x_0)(k) >0$. Therefore, for any $k \in \{1, \dots, n_1 \times c_1 \}$, $\sign(\scrG_1(x)(k)) = \sign(\scrG_1(x_0)(k))$ and 
\begin{equation}
  \scrG_1(x) = \varphi(A_1(x)) = D_1(x) A_1(x) = D_1(x_0) A_1(x) \eqsp ,
\end{equation}
where $D_1$ is given in \eqref{eq:linear_plus}. Now assume that \eqref{eq:local_lin} is true for $j \in \{1, \dots, \ell\}$ with $\ell \in \{1, \dots, M-1\}$.
There exists $\vareps_{\ell} >0 $ such that for any $x \in \cball{x_0}{\vareps_{\ell}}$
\begin{equation}
  \scrG_{\ell +1}(x) = \varphi(A_{\ell +1} \scrG_{\ell}(x)) = \varphi(A_{\ell +1} A_{\ell,+} \dots A_{1,+} (x)) \eqsp .
\end{equation}
Since $x_0 \in \rset^d \backslash (\bigcup_{j=1}^{\ell+1} \msn_{j})$ with $\msn_j = \bigcup_{k = 1}^{c_{j} \times n_{j}} \kernellin{e_{k,j}^{\transpose} A_jA_{j-1, +} \dots A_{1, +}}$ with $(e_{k,j})_{k \in \{1, \dots, c_j \times n_j \}}$ the canonical basis of $\rset^{c_j \times n_j}$, for any $k \in \{1, \dots, c_{\ell} \times n_{\ell} \}$, we get
$A_{\ell +1} A_{\ell,+} \dots A_{1,+} (x_0)(k) >0$ or we have $A_{\ell +1} A_{\ell,+} \dots A_{1,+} (x_0)(k) <0$. Therefore, since $A_{\ell +1} A_{\ell,+} \dots A_{1,+}$ is continuous there exists $\vareps_{\ell} > \vareps_{\ell +1} >0$ such that for any $x \in \cball{x_0}{\vareps_{\ell +1}}$ and $k \in \{1, \dots, c_{\ell + 1} \times n_{\ell +1} \}$,
\begin{equation}A_{\ell +1} A_{\ell,+} \dots A_{1,+} (x)(k) A_{\ell +1} A_{\ell,+} \dots A_{1,+} (x_0)(k) >0 \eqsp . \end{equation} Therefore, for any $x \in \cball{x_0}{\vareps_{\ell +1}}$ and $k \in \{1, \dots, c_{\ell +1} \times n_{\ell+1} \}$ we have $\sign(\scrG_{\ell+1}(x)(k)) = \sign(\scrG_{\ell +1}(x_0)(k))$ for any $x \in \cball{x_0}{\vareps_{\ell +1}}$, and
\begin{multline}
  \scrG_{\ell+1}(x) = \varphi(A_{\ell+1}  A_{\ell,+} \dots A_{1,+} (x)) = D_{\ell+1}(x) A_{\ell+1}  A_{\ell,+} \dots A_{1,+} (x) \\ = D_{\ell+1}(x_0) A_{\ell+1}  A_{\ell,+} \dots A_{1,+} (x) = A_{\ell+1, +}  A_{\ell,+} \dots A_{1,+} (x)\eqsp ,
\end{multline}
where $D_{\ell +1} $ is given in \eqref{eq:linear_plus}, which concludes the recursion. Let $\theta \in \rset^p$ with $\norm{\theta } = 1$. We have for any $x \in \cball{x_0}{\vareps_M}$
\begin{align}
  \langle \theta, F(x) \rangle &= \sum_{j \in \calJ} \sum_{k=1}^{c_j} \theta_{j,k} \tilde{v}_{j,k}^{\transpose} (\scrG_{j}(x) - \scrG_{j}(x_0))    \\
                               &= \sum_{j \in \calJ} \sum_{k=1}^{c_j}  \theta_{j,k} \tilde{v}_{j,k}^{\transpose} \defEns{A_{j,+} \dots A_{1,+} (x ) - A_{j,+} \dots A_{1,+} (x_0)}  \\
                               &= \sum_{j \in \calJ} \sum_{k=1}^{c_j} \theta_{j,k} \tilde{v}_{j,k}^{\transpose} \tilde{A}_{j,+} \dots \tilde{A}_{1,+} (x - x_0)    = \left\langle \sum_{j \in \calJ} \sum_{k=1}^{c_j}  \theta_{j,k} v_{j,k}, x - x_0 \right\rangle \eqsp .
\end{align}
Since, $(v_{j,c})_{j \in \calJ, c \in\{1, \dots, c_j \}}$ is assumed to be linearly independent we have that $v = \sum_{j \in \calJ} \sum_{c=1}^{c_j}  \theta_{j,c} v_{j,c}$ is non zero and therefore setting $x = x_0 - \vareps_M v / \norm{v}$ we get that $\langle \theta, F(x) \rangle < 0$. Since $F$ is continuous and $\mu(\msa) >0$ for every non-empty open set $\msa$ we have that for $\theta \in \rset^p$ with $\norm{\theta} = 1$, $\mu(\ensembleLigne{x \in \rset^d}{\langle F(x), \theta\rangle < 0}) >0$, which concludes the proof upon using \Cref{prop:existence_Q}-\ref{item:item_a_info}.


\section{Proofs of \Cref{sec:sampl-from-macr}}

We start by introducing some notations. Let $V : \rset^{d} \to \coint{1,+\infty}$ be a measurable function. For $f \in \functionspace[]{\rset^d}$, the $V$-norm of $f$ is given by $\Vnorm[V]{f}= \| f / V \|_{\infty}$. Let $\xi$ be a finite signed measure on $(\rset^d,\mcbb(\rset^d))$. The $V$-total variation norm of $\xi$ is defined as
\begin{equation}
\Vnorm[V]{\xi} = \sup_{f \in \functionspace[]{\rset^d}, \Vnorm[V]{f} \leq 1}  \abs{\int_{\rset^d } f(x) \rmd \xi (x)} \eqsp.
\end{equation}
If $V \equiv 1$, then $\Vnorm[V]{\cdot}$ is the total variation norm denoted by $\tvnorm{\cdot}$. 

Let $c : \rset^d \times \rset^d \to \ocint{0,+\infty}$ be defined for any $x,y \in \rset^d$ by $c(x,y) = \1_{\Delta_{\rset^d}}(x,y)W(x,y)$ where $W : \rset^d \times \rset^d \to \coint{0,+\infty}$ is a lower semi-continuous function such that for any $x,y,  z \in \rset^d$,  $W(x,y) = W(y,x)$ and $W(x,z) \leq W(x,y) + W(y,z)$. Then for any probability measures $\mu$ and $\nu$ such that there exist $x_{\mu}, x_{\nu} \in \rset^{\dim}$ satisfying $\mu(W(\cdot, x_{\mu})) < \infty$ and $\nu(W(\cdot, x_{\nu})) < \infty$, we define the Wasserstein extended distance associated with cost $c$ between $\mu$ and $\nu$ by
\begin{equation}
  \label{eq:def_wass}
  \dw{\mu, \nu} = \sup_{g \in \mathbb{G}_{\mu, W}} \abs{\int_{\rset^d} g(x) \rmd \mu(x) - \int_{\rset^d} g(y) \rmd \nu(y)} \eqsp ,
\end{equation}
with $\mathbb{G}_{\mu, W} = \ensembleLigne{g \in \mathbb{F}(\rset^{\dim})}{ \norm{g(x) - g(y)} \leq  W(x,y) \ , \text{for all $x,y \in \rset^d$}}$.

\subsection{Proof of \Cref{thm:cvx}}
\label{thm:cvx:proof}
This proof is an application of \cite[Theorem 2, Theorem 4]{debortoli2018souk}.
Therefore, we are reduced to checking that \cite[H1, H2]{debortoli2018souk} hold.
More precisely, we study the geometric ergodicity of the Langevin Markov chain under \tup{\Cref{assum:weak}($\upalpha$)}, \tup{\Cref{assum:equi_meas}}, \tup{\Cref{assum:existence_compact}($\upalpha$)} and \tup{\Cref{assum:curv_reg}} with $U_2 =0$ and $\upalpha \geq 1$ as well as its discretization error. Foster-Lyapunov conditions are derived in \Cref{lemma:drift} and we check that \cite[H1-(a)]{debortoli2018souk} holds in  \Cref{lemma:bornitude}. In \Cref{thm:ergo_cv} we show that \cite[H1-(b)]{debortoli2018souk} is satisfied. We check that \cite[H1-(c)]{debortoli2018souk} is satisfied in \Cref{lemma:drift_continuous} and \Cref{prop:error_disc}.

We denote by $\Kker_{\gamma, \theta}$ the Markov kernel associated with the Langevin recursion \eqref{eq:ula}.
This kernel is given for any $x \in \rset^{\dim}$ and $\msa \in \mcb{\rset^{\dim}}$
\begin{equation}
  \Kker_{\gamma, \theta}(x, \msa) = (2 \uppi \gamma)^{-d/2} \int_{\msa}  \exp\parentheseDeux{-(2\gamma)^{-1} \| y - x + \gamma \nabla_x U(\theta, x)  \|^2} \rmd y \eqsp , \label{eq:langevin_kernel}
\end{equation}
with $U$ given by \eqref{eq:potential}. Note that \eqref{eq:langevin_kernel} is well-defined under \Cref{assum:equi_meas} and \Cref{assum:existence_compact}($\upalpha$) with $\upalpha \geq 1$.
We say that a Markov kernel $\Kker$ on $\rset^d\times \mcb{\rset^d}$ satisfies a discrete Foster-Lyapunov drift condition
\hypertarget{assum:drift_discrete}{$\bfDd(V,\lambda,b)$} if there exist $\lambda \in (0,1)$, $b\geq0$ and a measurable function $V: \rset^d \to \coint{1,+\infty}$ such that for all $x \in \rset^d$
\begin{equation}
  \label{eq:discrete_drift}
  \Kker V(x) \leq \lambda V(x) + b \eqsp.
\end{equation}

First, we state the following technical lemma.

\begin{lemma}
  \label{lemma:ineq_reel}
  Let $\pow \in \nsets$. Then for any $u, v >0$ and $t\geq 0$,
  \begin{equation}
    u(1+t)^{2 \pow -1} - vt^{2\pow} \leq \Upsilon_{\pow}(u,v)
  \end{equation}
  with
  \begin{equation}\Upsilon_{\pow}(u,v) = 2^{(4\pow -2)\pow} \max\defEns{u , u^{2\pow} / v^{2\pow - 1}} \eqsp . \end{equation}
\end{lemma}

\begin{proof}
  Let $\pow \in \nsets$, $\tilde{u}, \tilde{v} > 0$ and $\tilde{f}(t) = \tilde{u}t^{2 \pow -1} - \tilde{v}t^{2\pow}$. We have for any $t \in \rset$, $\tilde{f}'(t) = (2\pow - 1) \tilde{u} t^{2 \pow - 2} - 2 \pow \tilde{v}  t^{2\pow -1}$. Since $\lim_{ \abs{t} \to +\infty} \tilde{f}(t) = -\infty$ and $\tilde{f}$ is continuous, the maximum is attained at some point $t_0$ which satisfies
  \begin{equation}
    \tilde{f}'(t_0) = (2\pow-1)\tilde{u}t_0^{2\pow -2} - 2\pow \tilde{v} t_0^{2\pow-1} = 0 \eqsp ,
  \end{equation}
and therefore $t_0 = (2\pow-1)\tilde{u} /(2\pow \tilde{v})$. We have for any $t \geq 0$
  \begin{equation}
    \label{eq:result_inte}
    \tilde{u}t^{2 \pow -1} - \tilde{v}t^{2\pow} \leq \tilde{u} t_0^{2\pow - 1} \leq \tilde{u} (\tilde{u}/\tilde{v})^{2\pow -1} \leq \tilde{u}^{2\pow } / \tilde{v}^{2\pow -1} \eqsp .
  \end{equation}
  If $t \geq 1$ then $u(1+t)^{2 \pow -1} - vt^{2\pow} \leq 2^{2\pow -1}ut^{2 \pow -1} - vt^{2 \pow }$ and using \eqref{eq:result_inte} we have $u(1+t)^{2 \pow -1} - vt^{2\pow} \leq 2^{(4 \pow -2)\pow}u^{2 \pow}/v^{2\pow-1}$.
  If $t \leq 1$ then $u(1+t)^{2 \pow -1} - vt^{2\pow} \leq 2^{2 \pow - 1}u$, which concludes the proof.
\end{proof}

\begin{lemma}
  \label{lemma:drift}
  Assume \tup{\Cref{assum:equi_meas}}, \tup{\Cref{assum:existence_compact}($\upalpha$)} and \tup{\Cref{assum:curv_reg}} with $U_2 = 0$ and $\upalpha \geq 1$. Let $\pow \in \nsets$, $\theta \in \msk$ and $\gamma \in \ocint{0, \bgamma}$ with $\bgamma < \min(\mtt_1 / \Lip^2, 1/2)$. Then $\Kker_{\gamma, \theta}$ satisfies $\bfDd(V,\lambda^{\gamma},\tilde{b}\gamma)$ with
  \begin{equation}
    \label{eq:drift}
    V(x) = 1 + \norm{x - x^{\star}}^{2\pow} \eqsp , \qquad \lambda = \exp[ -\mtt_1 + \bgamma \Lip^2] \eqsp , \qquad \tilde{b}_{\pow} = \Upsilon_{\pow}(2^{2\pow +1}d^{\pow} \Gammabf(\pow + 1/2), \mtt_1) + \mtt_1 \eqsp ,
  \end{equation}
  where for any $t \geq 0$, $\Gammabf(t) = \int_0^{+\infty} u^{t-1} \rme^{-u} \rmd u$ and $\Upsilon_{\pow}$ is given in \Cref{lemma:ineq_reel}. In addition, $\Kker_{\gamma, \theta}$ satisfies $\bfDd(V,\lambda^{\gamma},b_{\pow}(1 + d^{\varpi_{0, \pow}})\gamma)$ with $\lambda$ given in \eqref{eq:drift} and $b_{\pow}, \varpi_{0, \pow} \geq 0$ independent of the dimension $d$.
\end{lemma}

\begin{proof}
  Let $x \in \rset^d$, $\pow \in \nsets$, $\theta \in \msk$, $\gamma \in \ocint{0, \bgamma}$ and  $\Z$ a $d$-dimensional Gaussian random variable with zero mean and identity covariance matrix.
First, denoting $\Z = (z_1, \dots, z_d)$ we have using Holder's inequality 
\begin{equation}
  \label{eq:majo_Z}
  \expe{\norm{\Z}^{2\pow}} = \sum_{i_1=1}^d \dots \sum_{i_{\pow}=1}^d \expe{\prod_{j=1}^{\pow} z_{i_j}^2} \leq  \sum_{i_1=1}^d \dots \sum_{i_{\pow}=1}^d \expe{z_1^{2\pow}} \leq (2d)^{\pow} \Gammabf(\pow + 1/2) \eqsp .
\end{equation}
Let $\Tg(x) = x - x^{\star} - \gamma \nabla_x U(\theta, x)$. Using \Cref{assum:curv_reg} we get
\begin{align}
  \norm{\Tg(x)}^2 &\leq \norm{x - x^{\star}}^2 - 2 \gamma \langle \nabla_x U(\theta,x) - \nabla_x U(\theta, x^{\star}), x - x^{\star} \rangle + \norm{\nabla_x U(\theta,x) - \nabla_x U(\theta, x^{\star})}^2
  \\ &\leq (1 - 2 \gamma \mtt_1 + \gamma^2 \Lip^2) \norm{x - x^{\star}}^2 \eqsp .
\end{align}
Hence, we obtain 
  \begin{align}
    \expe{\norm{\X -x^{\star}}^{2 \pow}} &= \expe{\sum_{k=0}^{\pow} \sum_{j=0}^k {\pow \choose k} {k \choose j} \norm{\Tg(x)}^{2(\pow - k)}(2\gamma)^{j/2}\langle \Tg(x), \Z \rangle^{j}(2\gamma)^{k-j}\norm{\Z}^{2(k-j)}} \\
    &\leq (1- \gamma (\mtt_1 - 2\Lip^2 \bgamma)) \norm{x - x^{\star}}^{2\pow} + \gamma C_{\pow}(x - x^{\star}) \eqsp , \label{eq:majo_1}
  \end{align}
  where we have, using that $\norm{\Tg(x)} \leq \norm{x - x^{\star}}$, \eqref{eq:majo_Z}, $2\gamma \leq 1$, the Isserlis' formula \cite{isserlis1918formula} and the Cauchy-Schwarz inequality
  \begin{align}
    \gamma C_{\pow}(x -x^{\star}) &=  \sum_{k=1}^{\pow} \sum_{j=0}^k {\pow \choose k} {k \choose j} \norm{x - x^{\star}}^{2(\pow - k)} (2\gamma)^{k-j/2 }\expe{\langle \Tg(x), \Z \rangle^{j}\norm{\Z}^{2(k-j)}} \\
                &=  \sum_{k=1}^{\pow} \sum_{j=0}^{\floor{k/2}} {\pow \choose k} {k \choose 2 j} \norm{x - x^{\star}}^{2(\pow - k)} (2\gamma)^{k-j/2 } \expe{\langle \Tg(x), \Z \rangle^{2j}\norm{\Z}^{2(k-2j)}} \\
                &\leq 2 \gamma \sum_{k=1}^{\pow} \sum_{j=0}^{\floor{k/2}} {\pow \choose k} {k \choose 2 j} \norm{x - x^{\star}}^{2(\pow - k + j)} \expe{\norm{\Z}^{2(k-j)}} \\
                           &\leq 2\gamma (1 + \norm{x - x^{\star}})^{2\pow-1} (2d)^{\pow} \Gammabf(\pow + 1/2) \sum_{k=1}^{\pow} \sum_{j=0}^{\floor{k/2}} {\pow \choose k} {k \choose 2j} \\
    &\leq 2^{2\pow+1} \gamma d^{\pow} \Gammabf(\pow + 1/2) (1 + \norm{x - x^{\star}})^{2\pow-1} \eqsp . \label{eq:majo_2}
  \end{align}
Combining \eqref{eq:majo_1} and \eqref{eq:majo_2} we get that
\begin{align}
  \Kker_{\gamma, \theta}(\norm{x - x^{\star}}^{2\pow}) &\leq (1 -\gamma (\mtt_1 -\Lip^2 \bgamma)) \norm{x - x^{\star}}^{2\pow} \\ & \qquad   + 2^{2\pow +1} \gamma d^{\pow} \Gammabf(\pow + 1/2) (1 + \norm{x - x^{\star}})^{2\pow - 1} - \gamma \mtt_1  \norm{x - x^{\star}}^{2\pow} \eqsp .   \label{eq:ineq_x_l}                                                                                                          \end{align}
Using \Cref{lemma:ineq_reel}, we have
\begin{equation}
  2^{2\pow +1}d^{\pow} \Gammabf(\pow + 1/2) (1+\norm{x - x^{\star}})^{2\pow-1} - \mtt_1\norm{x - x^{\star}}^{2\pow} \leq \Upsilon_{\pow}(2^{2\pow +1}d^{\pow}\Gammabf(\pow + 1/2), \mtt_1) \eqsp .
\end{equation}
Combining this result with \eqref{eq:ineq_x_l} we get,
\begin{equation}
    \Kker_{\gamma, \theta}(\norm{x - x^{\star}}^{2\pow}) \leq (1 - \gamma (\mtt_1 - \bgamma \Lip^2)) \norm{x - x^{\star}}^{2\pow} + \gamma \Upsilon_{\pow}(2^{2\pow +1}d^{\pow} \Gammabf(\pow + 1/2), \mtt_1) \eqsp .
\end{equation}
Therefore we obtain
\begin{align}
    \Kker_{\gamma, \theta}(1 + \norm{x - x^{\star}}^{2\pow}) &\leq (1 - \gamma(\mtt_1 - \bgamma \Lip^2)) (1 + \norm{x - x^{\star}}^{2\pow}) \\ & \qquad + \gamma \defEns{\Upsilon_{\pow}(2^{2\pow +1}d^{\pow} \Gammabf(\pow + 1/2), \mtt_1) + \mtt_1} \eqsp ,
\end{align}
which concludes the proof upon noting that $\tilde{b}_{\pow}$ is a polynomial in the dimension $d$.
\end{proof}

\begin{lemma}
  \label{lemma:bornitude}
  Assume \tup{\Cref{assum:equi_meas}},  \tup{\Cref{assum:existence_compact}($\upalpha$)} and \tup{\Cref{assum:curv_reg}} with $U_2 = 0$, $\upalpha \geq 1$ and let $(\X_k^n)_{n \in \nset, k \in \{0, \dots, m_n}$ be given by \eqref{eq:souk} with $\bgamma < \min(\mttun / \Lip^2, 1/2)$.
  Let $\pow \in \nsets$, then there exist $A_{1, \pow} \geq 1$ and $\varpi_{1, \pow} \geq 0$ such that for any $n, p \in \nset$ and $k \in \{0, \dots, m_n\}$
  \begin{equation}
    \CPE{\Kker_{\gamma_n, \theta_n}^p V(\X_k^n)}{\X_0^0} \leq A_{1, \pow} (1 + d^{\varpi_{1, \pow}}) V(\X_0^0) \eqsp , \qquad \expe{V(\X_0^0)} < +\infty\eqsp ,
  \end{equation}
  with $V(x) = 1 + \norm{x - x^{\star}}^{2\pow}$ and $A_{1, \pow}, \varpi_{1, \pow} \geq 0$ which do not depend on the dimension $d$.
\end{lemma}

\begin{proof}
  Combining \cite[Lemma S15]{debortoli2018souk} and \Cref{lemma:drift} conclude the proof.
\end{proof}

\begin{theorem}
  \label{thm:ergo_cv}
    Assume \tup{\Cref{assum:equi_meas}},  \tup{\Cref{assum:existence_compact}($\upalpha$)} and \tup{\Cref{assum:curv_reg}} with $U_2 = 0$ and $\upalpha \geq 1$. Then for any $\pow \in \nsets$ there exist $A_{2, \pow}, \varpi_{2, \pow} \geq 0$ and $\rho_{\pow} \in \ooint{0,1}$ such that for any $\theta \in \msk$ and $\gamma \in \ocint{0, \bgamma}$ with $\bgamma < \min(\mtt_1 / \Lip^2, 1/2)$, $\Kker_{\gamma, \theta}$ admits an invariant probability measure $\pi_{\gamma, \theta}$ and for any $x,y \in \rset^d$ and $n \in \nset$
    \begin{equation}
      \label{eq:strong_convex_ergo}
      \begin{aligned}
      \Vnorm{\updelta_x \Kker_{\gamma, \theta}^n - \pi_{\gamma, \theta}} &\leq A_{2, \pow} (1 + d^{\varpi_{2, \pow}}) \exp[-n \kappa_{\pow} \gamma  / \log^2(1 + d^{\varpi_{2, \pow}})]V(x) \eqsp , \\
      \Vnorm{\updelta_x \Kker_{\gamma, \theta}^n - \updelta_y \Kker_{\gamma, \theta}^n} &\leq A_{2, \pow} (1 + d^{\varpi_{2, \pow}}) \exp[-n \kappa_{\pow} \gamma / \log^2(1 + d^{\varpi_{2, \pow}})] \defEns{V(x) + V(y)} \eqsp ,
      \end{aligned}
    \end{equation}
with $V(x) = 1 + \norm{x - x^{\star}}^{2\pow}$ and $A_{2, \pow}, \varpi_{2, \pow} \geq 0$ and $\kappa_{\pow} > 0$ which do not depend on the dimension $d$.
\end{theorem}

\begin{proof}
  For any $\gamma \in \ocint{0,\bgamma}$ and $\theta \in \msk$, $\Kker_{\gamma, \theta}$  has the Feller property and satisfies \hyperlink{ass:drift_discrete}{$\bfDd(V,\lambda^{\gamma},b\gamma)$} then \cite[Theorem 12.3.3]{douc2018markov}  applies and $\Kker_{\gamma, \theta}$ admits an invariant probability measure $\pi_{\gamma, \theta}$.

  Let $\pow \in \nsets$, $\theta \in \msk$ and $\gamma \in \ocint{0, \bgamma}$. Using \cite[Proposition 3]{debortoli2018back} with $\msa \leftarrow \rset^{\dim}$, we have for any $x,y \in \rset^d$ and $n \in \nset$
  \begin{equation}
    \label{eq:uno}
    \updelta_{(x,y)} \tilde{\Kker}_{\gamma, \theta}^{n \step}(\Delta_{\rset^{\dim}}^{\complementary}) \leq 1 - 2\Phibf\defEns{- \alpha^{-1/2}( n) \norm{x-y}/(2 \sqrt{2})} \eqsp ,
  \end{equation}
  where for any $x,y \in \rset^{\dim}$, $\tilde{\Kker}_{\gamma, \theta}((x,y), \cdot)$ is the reflexive coupling between $\Kker_{\gamma, \theta}(x, \cdot)$ and $\Kker_{\gamma, \theta}(y, \cdot)$, see \cite[Equation 5]{debortoli2018back}. In addition, we have for any $n \in \nset$
  \begin{equation}
    \label{eq:dos}
    \alpha^{-1}(n) = \left. (2 \mtt_1 - \Lip^2 \bgamma) \middle/ \defEns{\exp((2\mtt_1 - \Lip^2 \bgamma)n) - 1} \leq 2\mtt_1 / (\lambda^{-2n} - 1) \right. \leq 2\mtt_1 \lambda^{2n} /(1 - \lambda) \eqsp .
  \end{equation}
  and $\lambda$ is given by \eqref{eq:drift}. Therefore, we get that for any $x,y \in \rset^d$ and $n \in \nset$
  \begin{equation}
    \label{eq:13}
    \updelta_{(x,y)} \tilde{\Kker}_{\gamma, \theta}^{n \step}(\Delta_{\rset^{\dim}}^{\complementary}) \leq 1 - 2\Phibf\defEns{- \lambda^{n}\mtt_1^{1/2} \norm{x-y} / (1 - \lambda)^{1/2}} \eqsp .
  \end{equation}
  For any $x,y \in \rset^d$ let
  \begin{equation}
    \begin{aligned}
      W(x,y) &= 1 + (\norm{x - x^{\star}}^{2\pow} + \norm{y - x^{\star}}^{2\pow}) /2 \eqsp , \\
      K_{\discrete} &= 2b_{\pow}(1 + d^{\varpi_{0, m}})(1+\bgamma)(1 + \log^{-1}(1/\lambda)) \eqsp , \qquad M_{\discrete} = 2 K_{\discrete}^{1/\pow} \eqsp .
      \end{aligned}
  \end{equation}
  Note that for any $x, y \in \rset^d$ such that $\norm{x-y} \geq M_{\discrete}$, $W(x,y) \geq K_{\discrete}$. In addition, let $\no = \max\defEnsLigne{\ceilLigne{-\log(\mtt_1^{1/2}M_{\discrete}/(1-\lambda)^{1/2}) \log^{-1}(1/\lambda)},0}$. We have for any $x,y \in \rset^d$ with $\norm{x-y} \leq M_{\discrete}$,
  \begin{equation}
    \updelta_{(x,y)} \tilde{\Kker}_{\gamma, \theta}^{n \step}(\Delta_{\rset^{\dim}}^{\complementary}) \leq 1 - 2\Phibf(-1) \eqsp .
  \end{equation}
    Then, applying \cite[Theorem 6b]{debortoli2018back}, with $\tilde{\Kker} \leftarrow \tilde{\Kker}_{\gamma, \theta}$,  $\pow \leftarrow \no$, $\vareps_{1, \discrete} \leftarrow 2\Phibf(-1)$, $\bar{C}_1 \leftarrow C$, $\bar{\rho}_1 \leftarrow \rho$, $\bar{A}_1 \leftarrow A$ and $\bar{c}_1 \leftarrow c$, we obtain that for any $x,y \in \rset^d$ and $n \in \nset$
    \begin{equation}
      \Vnorm{\updelta_x \Kker_{\gamma, \theta}^n - \updelta_y \Kker_{\gamma, \theta}^n} \leq C \rho^{\floor{k / (\no \ceil{1/\gamma})}} \defEns{V(x) + V(y)} /2 \eqsp ,
    \end{equation}
    where
    \begin{equation}
  \begin{aligned}
     C &= 2 \parentheseDeux{1 + A}\parentheseDeux{1 + (A + K_{\discrete})/ \defEns{\Phibf(-1)(1-\lambda)}} \eqsp ,\\
     A &= b_{\pow}(1 + d^{\varpi_{0, \pow}})(1+\bgamma)(1 + \log^{-1}(1/\lambda))) \eqsp , \\
              \log^{-1}(1/\rho)& = \log^{-1}(1/(2\Phibf(-1))) + \log^{-1}(2/(1 + \lambda)) \\
    & \qquad \qquad + \log(A + K_{\discrete}) \log^{-1}(1 / (2\Phibf(-1))) \log^{-1}(2/(1 + \lambda)) \eqsp .
  \end{aligned}
\end{equation}
Since $\floor{n / (\no \ceil{1/\gamma})} \geq n\gamma / (\no(1+ \bgamma)) - 1$, setting $\tilde{A}_{2, \pow} = C\rho^{-1}/2$ and $\rho_{\pow} = \rho^{1/(\no(1+\bgamma))}$, we get that for any $x,y \in \rset^d$ and $n \in \nset$
    \begin{equation}
      \label{eq:majo_xy}
      \Vnorm{\updelta_x \Kker_{\gamma, \theta}^n - \updelta_y \Kker_{\gamma, \theta}} \leq \tilde{A}_{2, \pow} \rho_{\pow}^{ \gamma n} \defEns{V(x) + V(y)} \eqsp .      
    \end{equation}
    Using \Cref{lemma:drift} we have that $\pi_{\gamma, \theta}(V) \leq b_{\pow}(1+d^{\varpi_{0,m}})(1 + \log^{-1}(1/\lambda))V(x)$. Combining this result with \eqref{eq:majo_xy} we get that for any $x \in \rset^d$ and $n \in \nset$
    \begin{equation}
      \Vnorm{\updelta_x \Kker_{\gamma, \theta}^n - \pi_{\gamma, \theta}} \leq \tilde{A}_{2, \pow} \defEns{1 + b_{\pow}(1+d^{\varpi_{0,m}}) (1 + \log^{-1}(1/\lambda))} \tilde{\rho}_{\pow}^{ \gamma n} V(x)  \eqsp.
    \end{equation}
    Since in \Cref{lemma:drift}, $\lambda$ and $b_{\pow}$ do not depend on the dimension $d$ we get that $K_{\discrete}$ is upper-bounded by a polynomial in the dimension $d$. Hence, there exists $\varpi_{2, \pow}^{(a)}>0$ which does not depend on the dimension such that $\tilde{A}_{2, \pow} \defEnsLigne{1 + b_{\pow}(1+d^{\varpi_{0,m}})(1 + \log^{-1}(1/\lambda))} \leq A_{2, \pow} (1 + d^{\varpi_{2, \pow}^{(a)}})$ with $A_{2, \pow} \geq 0$ which does not depend on the dimension $d$. Similarly, there exists $\varpi_{2, \pow}^{(b)} > 0 $ independent of $d$ such that $\sup_{d \in \nset}[\{ \log^{-1}(\rho) + \no \}/\log(1 + d^{\varpi_{2, \pow}^{(b)}})^{-1}] < +\infty$ which implies that $\log^{-1}(1/\rho_m) \leq \kappa_{\pow}^{-1} \log^2(1 + d^{\varpi_{2, \pow}^{(b)}})$ with  $\kappa_{\pow} > 0$ which do not depend on the dimension $d$. We conclude the proof upon setting $\varpi_{2, \pow} = \max(\varpi_{2, \pow}^{(a)}, \varpi_{2, \pow}^{(b)})$.
  \end{proof}
  Similarly to the discrete setting, we say that a Markov semi-group $(\Pker_t)_{t \geq 0}$ on $\rset^d \times \mcb{\rset^d}$ with extended infinitesimal generator $(\mathcal{A},\domain(\generator))$ (see \eg~\cite{meyn1993criteria_iii} for the definition of $(\generator,\domain(\generator))$) satisfies a continuous drift condition \hypertarget{assum:drift_continuous}{$\bfDc(V,\zeta,\beta)$} if there exist $\zeta >0$, $\beta \geq 0$ and a measurable function $V: \rset^d \to \coint{1,+\infty}$ with $V \in \domain(\generator)$ such that for all $x \in \rset^d$
\begin{equation}
  \label{eq:continuous_drift}
  \mathcal{A}V(x) \leq - \zeta V(x) + \beta \eqsp .
\end{equation}
 Let $\theta \in \msk$ and  $(\Pker_{t, \theta})_{t \geq 0}$ be the Markov semi-group associated with the Langevin diffusion
  \begin{equation}
    \rmd \tbfX_t =  - \parenthese{\sum_{i=1}^p \theta(i) \nabla F_i(\tbfX_t) + \nabla \Reg(\tbfX_t)} + \rmd \tbfB_t \eqsp ,
  \end{equation}
  where $(\tbfB_t)_{t \geq 0}$ is a $d$-dimensional Brownian motion.
  Consider now the generator $\generator_{\theta}$ of
$(\Pker_{t, \theta})_{t \geq 0}$ for any $\theta \in \msk$, defined for
any $f \in  \rmc^2(\rset^d)$ and $x \in \rset^d$ by
\begin{equation}
\label{eq:def_generator_theta}
  \generator_{\theta} f(x) = -\ps{\nabla f(x)}{\sum_{i=1}^p \theta(i) \nabla F_i(x) + \nabla \Reg(x)} + \Delta f(x) \eqsp.
\end{equation}
  Using \cite[Lemma S14]{debortoli2018souk} we have that
  $\pi_{\theta}$ is an invariant probability measure for $(\Pker_{t, \theta})_{t \geq 0}$.
  \begin{lemma}
    \label{lemma:drift_continuous}
    Assume \tup{\Cref{assum:equi_meas}},  \tup{\Cref{assum:existence_compact}($\upalpha$)} and \tup{\Cref{assum:curv_reg}} with $U_2 = 0$ and $\upalpha \geq 1$. Then for any $\pow \in \nsets$ there exist $\zeta > 0$ and $\beta \geq 0$ such that for any $\theta \in \msk$, $\generator_{\theta}$ satisfies \hyperlink{ass:drift_continuous}{$\bfDc(V, \zeta, \tilde{\beta}_{\pow})$} with
    \begin{equation}
      \label{eq:drift_c}
      V(x) = 1 + \norm{x - x^{\star}}^{2 \pow} \eqsp , \qquad \zeta = -\mtt_1 \pow \eqsp , \qquad \tilde{\beta}_{\pow} = 2 \pow \Upsilon_{\pow}(2(\pow -1) + d, \mttun/2) + 2 \mtt_1 \pow \eqsp ,
    \end{equation}
    with $\Upsilon_m$ given in \Cref{lemma:ineq_reel}.  In addition, $\generator_{\theta}$ satisfies $\bfDc(V,\zeta,\beta_{\pow}(1 + d^{\varpi_{0, \pow}'})\gamma)$ with $\zeta$ given in \eqref{eq:drift_c} and $\beta_{\pow}, \varpi_{0, \pow}' \geq 0$ independent of the dimension $d$.
  \end{lemma}

  \begin{proof}
    Let $\theta \in \msk$ and $\pow \in \nsets$. Then, we have for any $x \in \rset^d$
    \begin{equation}
      \label{eq:grad_l}
      \begin{aligned}
      & V(x) = 1 + \norm{x - x^{\star}}^{2\pow} \eqsp , \\
      &\nabla V(x) = 2 \pow \norm{x - x^{\star}}^{2(\pow -1)} (x - x^{\star}) \eqsp , \\
      &\nabla^2 V(x) = 4\pow (\pow - 1) \norm{x - x^{\star}}^{2(\pow -2)} (x - x^{\star}) (x - x^{\star})^{\transpose} + 2\pow \norm{x - x^{\star}}^{2(\pow -1)} \Id \eqsp .
      \end{aligned}
    \end{equation}
    Hence, for any $x \in \rset^d$, $\Delta V(x) = 4\pow (\pow - 1) \norm{x - x^{\star}}^{2(\pow -1)} + 2\pow d\norm{x - x^{\star}}^{2(\pow -1)}$.
    Using \Cref{assum:curv_reg}, \eqref{eq:grad_l} and \Cref{lemma:ineq_reel} we get that for any $x \in \rset^d$
    \begin{align}
      &\generator_{\theta}V(x) = - 2 \pow \norm{x - x^{\star}}^{2(\pow -1)} \langle \nabla_x U(\theta, x), x - x^{\star} \rangle +2\pow \parenthese{2(\pow-1) + d} \norm{x - x^{\star}}^{2(\pow-1)} \\
                              &= - 2 \pow \norm{x - x^{\star}}^{2(\pow -1)} \langle \nabla_x U(\theta, x) - \nabla_x U(\theta, x^{\star}), x - x^{\star} \rangle +2\pow \parenthese{2(\pow-1) + d} \norm{x - x^{\star}}^{2(\pow-1)} \\ 
                              &\leq - 2\mtt_1 \pow V(x) +2\pow \parenthese{2(\pow-1) + d} \norm{x - x^{\star}}^{2(\pow -1)} + 2\mtt_1 \pow \\
                              &\leq - \mtt_1 \pow V(x) +2\pow \parenthese{2(\pow-1) + d - \mtt_1 \norm{x - x^{\star}}^2 /2} \norm{x - x^{\star}}^{2(\pow -1)} + 2\mtt_1 \pow \\
      &\le - \mtt_1 \pow V(x) + 2 \pow \Upsilon_{\pow}(2(\pow -1) + d, \mttun/2) + 2\mtt_1 \pow \eqsp ,
    \end{align}
    which concludes the proof.
  \end{proof}

  \begin{proposition}
    \label{prop:error_disc}
  Assume \tup{\Cref{assum:equi_meas}},  \tup{\Cref{assum:existence_compact}($\upalpha$)} and \tup{\Cref{assum:curv_reg}} with $U_2 = 0$ and $\upalpha \geq 1$. Then for any $\pow \in \nsets$, there exist $A_{3, \pow}, \varpi_{3, \pow} \geq 0$ such that for any $\theta \in \msk$, $\gamma \in \ocint{0, \bgamma}$ with $\bgamma < \min(\mtt_1 / \Lip^2, 1/2)$,
  \begin{equation}
    \Vnorm[V^{1/2}]{\pi_{\gamma, \theta} - \pi_{\theta}} \leq A_{3, \pow}(1 + \dim^{\varpi_{3, \pow}}) \gamma^{1/2}\eqsp ,
  \end{equation}
with $V(x) = 1 + \norm{x - x^{\star}}^{2\pow}$ and $A_{3, \pow}, \varpi_{3, \pow} \geq 0$ which do not depend on the dimension $d$.
\end{proposition}

The proof is similar to the one of \cite[Proposition S17]{debortoli2018souk} except that in this
presentation we explicit the constants appearing in the proof and track the dependency of the constants with
respect to the dimension $d$.

\begin{proof}
  Let $\pow \in \nsets$, $\theta \in \msk$ and $\gamma \in \ocint{0, \bgamma}$.
 Since $\pi_{\theta}$ is an invariant probability measure for $(\Pker_{t,\theta})_{ t\geq 0}$ we have using \Cref{thm:ergo_cv} that
  \begin{equation} \underset{k  \rightarrow  +\infty}{\lim} \Vnorm[V^{1/2}]{\pi_{\theta} \Kker_{\gamma,\theta}^k - \pi_{\theta} \Pker_{\gamma k, \theta} } =\Vnorm[V^{1/2}]{\pi_{\gamma,\theta} - \pi_{\theta}} \eqsp . \label{eq:limite}\end{equation}
  We now give an upper bound on $\Vnorm[V^{1/2}]{\pi_{\theta} \Kker_{\gamma,\theta}^k - \pi_{\theta} \Pker_{\gamma k, \theta} }$ for $k = q_{\gamma} m_{\gamma} $ with $m_{\gamma} = \ceil{1/\gamma}$ and $q_{\gamma}\in \N$.
  We obtain using \Cref{thm:ergo_cv}
  \begin{align}
    &\Vnorm[V^{1/2}]{\pi_{\theta} \Kker_{\gamma,\theta}^k - \pi_{\theta} \Pker_{\gamma k ,\theta} } \\ &\leq \sum_{j=0}^{q_{\gamma}-1} \Vnorm[V^{1/2}]{\pi_{\theta} \Pker_{\gamma (j+1) m_{\gamma}, \theta} \Kker_{\gamma,\theta}^{(q_{\gamma} -(j+1))m_{\gamma} } - \pi_{\theta} \Pker_{\gamma j m_{\gamma}, \theta} \Kker_{\gamma,\theta}^{(q_{\gamma} -j)m_{\gamma} }} \\
    &\leq \sum_{j=0}^{q_{\gamma} - 1} \left\lbrace A_{2, \pow}(1 + d^{\varpi_{2, \pow}}) \exp[- \kappa_{\pow} \gamma m_{\gamma} (q_{\gamma} - (j+1)) / \log^2(1 + d^{\varpi_{2, \pow}})] \right. \\
    & \qquad \qquad \left. \times \Vnorm[V^{1/2}]{\pi_{\theta} \Pker_{\gamma j m_{\gamma}, \theta} \Pker_{m_{\gamma} \gamma, \theta}- \pi_{\theta} \Pker_{\gamma j m_{\gamma}, \theta} \Kker_{\gamma,\theta}^{m_{\gamma}}} \right\rbrace\\
    & \leq \Vnorm[V^{1/2}]{\pi_{\theta}  \Pker_{m_{\gamma} \gamma, \theta}- \pi_{\theta} \Kker_{\gamma,\theta}^{m_{\gamma}}} \sum_{j=1}^{q_{\gamma}} A_{2, \pow} (1 + d^{\varpi_{2, \pow}}) \exp[ - \kappa_{\pow} \gamma j m_{\gamma} / \log^2(1 + d^{\varpi_{2, \pow}})]  \\
    &\leq A_{2, \pow} (1 + d^{\varpi_{2, \pow}}) (1 + \log^2(1  + d^{\varpi_{2, \pow}}) / \kappa_{\pow}) \Vnorm[V^{1/2}]{\pi_{\theta}  \Pker_{m_{\gamma} \gamma, \theta}- \pi_{\theta} \Kker_{\gamma,\theta}^{m_{\gamma}}}
      \eqsp ,  \label{eq:sum_decompo}
  \end{align}
We now give an upper bound on $\Vnorm[V^{1/2}]{\pi_{\theta}  \Pker_{m_{\gamma} \gamma, \theta}- \pi_{\theta} \Kker_{\gamma,\theta}^{m_{\gamma}}}$. Indeed, since $\A_{\theta}$ satisfies \hyperlink{ass:drift_continuous}{$\bfDc(V, \zeta, \beta)$} by \Cref{lemma:drift_continuous} and
  $\Kker_{\gamma, \theta}$ satisfies \hyperlink{ass:drift_discrete}{$\bfDd(V,\lambda^{\gamma},b\gamma)$} for any $\theta \in \msk$ and $\gamma \in \ocint{0, \bgamma}$ by \Cref{lemma:drift} and, we obtain that 
  \begin{equation}
    \label{eq:V_bound_continuous}
    \begin{aligned}
      &\pi_{\theta}  \Pker_{\gamma m_{\gamma}, \theta}(V) \leq D_0 \eqsp , \qquad \pi_{\theta}  \Kker_{\gamma,\theta}^{m_{\gamma}}(V) \leq D_1 \eqsp , \\
      &D_0 = \beta_{\pow}(1 + d^{\varpi_{0, \pow}'}) / \zeta \eqsp , \qquad D_1 = D_0 + b_{\pow}(1 + d^{\varpi_{0, \pow}})(\bgamma + \log^{-1}(1/\lambda)) \eqsp .
      \end{aligned}
  \end{equation}
Combining this result and \cite[Lemma S16]{debortoli2018souk} we have for any $\theta \in \msk$ and $\gamma \in \ocint{0, \bgamma}$
  \begin{equation}
    \label{eq:bound_vdemi}
    \Vnorm[V^{1/2}]{\pi_{\theta}  \Pker_{\gamma m_{\gamma}, \theta}- \pi_{\theta}  \Kker_{\gamma,\theta}^{m_{\gamma}}} \leq D_2 \gamma^{1/2} \eqsp , 
  \end{equation}
  with 
  \begin{equation}
    \label{eq:constant_Csecond}
    D_2 = 2D_1^{1/2}(1 + \bgamma)^{1/2}\defEns{d + 2 \bgamma (2\Lip^2 + \sup_{\theta \in \msk} \| \nabla_x U_1(\theta, 0) \|^2 )  D_1}^{1/2}\Lip \eqsp .
  \end{equation}
  Combining \eqref{eq:sum_decompo} and \eqref{eq:bound_vdemi} we get
  \begin{equation}
    \Vnorm[V^{1/2}]{\pi_{\theta} \Kker_{\gamma,\theta}^k - \pi_{\theta} \Pker_{\gamma k, \theta} } \leq D_2 A_{2, \pow} (1 + d^{\varpi_{2, \pow}}) (1 +\log^2(1 + d^{\varpi_{2, \pow}})/\kappa_{\pow}) \gamma^{1/2}\eqsp .
  \end{equation}
  Combining \eqref{eq:constant_Csecond}, \eqref{eq:V_bound_continuous} and \Cref{lemma:drift}, we get that $D_0$, $D_1$, $D_2$ are upper-bounded by polynomials in the dimension $d$, which concludes the proof.
\end{proof}

We now turn to the proof of \Cref{thm:cvx}. 
\begin{proof}[Proof of \Cref{thm:cvx}]
  Let $\pow = \ceil{2\upalpha}$ and $V(x)  = 1 + \norm{x - x^{\star}}^{2\pow}$.
  \Cref{lemma:bornitude} implies \cite[H1a]{debortoli2018souk} with $A_{1, \pow} \leftarrow A_{1, \pow}(1 + d^{\varpi_{1,\pow}}) $ , \Cref{thm:ergo_cv} implies \cite[H1b]{debortoli2018souk} with $A_{2, \pow} \leftarrow A_{2, \pow} (1 + d^{\varpi_{2,\pow}})$ and $\rho \leftarrow \exp[-\kappa_{\pow} / \log^2(1 + d^{\varpi_{2,\pow}})]$. In addition, \Cref{prop:error_disc} implies \cite[H1c]{debortoli2018souk} with $\Psibf(\gamma) = A_{3, \pow} (1 + d^{\varpi_{3,\pow}}) \gamma^{1/2}$. Using \Cref{prop:existence_P} we have that \cite[A1, A2, A3]{debortoli2018souk} hold. Since $H_{\theta} \leftarrow F$ in \eqref{eq:souk} we get that for any $\theta \in \msk$ and $x\in \rset^d$, $\norm{H_{\theta}(x)} \leq V^{1/2}(x)$. We can apply \cite[Theorem 2]{debortoli2018souk} and we get that for any $n \in \nset$
  \begin{equation}
    \expe{  \defEns{\left. \sum_{k=1}^n \delta_k L(\theta_k) \middle/ \sum_{k=1}^n \delta_k \right. } - \min_{\msk} L  }\leq  \left. E_n \middle/  \left( \sum_{k=1}^n \delta_k \right) \right. \eqsp ,
  \end{equation}
  with,
  \begin{align}
&E_n =  2\Rtheta^2 + 2B_{1, \pow} \Rtheta\expe{V^{1/2}(\X_0^0)} \sum_{k=0}^{n-1} \delta_{k+1}/ (m_k \gamma_k) \\ & \qquad + 2 \Rtheta A_{3,\pow}(1 + d^{\varpi_{3, \pow}}) \sum_{k=0}^{n-1} \delta_{k+1}\gamma_k^{1/2}  +  4B_{1, \pow}^2 \expe{V(\X_0^0)} \sum_{k=0}^{n-1} \delta_{k+1}^2/ (m_k \gamma_k)^2   \\ & \qquad \qquad +   4 A_{3, \pow}^2 (1 + d^{\varpi_{3, \pow}})^2 \sum_{k=0}^{n-1} \delta_{k+1}^2 \gamma_k + B_{2, \pow} \sum_{k=0}^{n-1} \delta_{k+1}^2 / (m_k \gamma_k)^2  \eqsp , \numberthis     
\end{align}
and
\begin{equation}
  \begin{aligned}
    &B_{1, \pow} =  2(1 + d^{\varpi_{\pow}})^2\log^2(1 + d^{\varpi_{\pow}}) A_{1, \pow} A_{2, \pow}  \exp[-\bgamma\kappa_{\pow} / \log^2(1 + d^{\varpi_{\pow}})] / \kappa_{\pow} \eqsp ; \\
    &B_{2, \pow} = 2(1 + \bgamma)^2 \max (B_{2, \pow}^{(a)}, B_{2, \pow}^{(b)}) \eqsp ; \\
     &B_{2, \pow}^{(a)} = 24 (1 + d^{\varpi_{\pow}})^3 A_{2, \pow}^2(1-\exp[-\kappa_{\pow} / (2\log^2(1 + d^{\varpi_{\pow}}))])^{-2}A_{3, \pow} \eqsp ; \\    
     &B_{2, \pow}^{(b)} = 4 (1 + d^{\varpi_{\pow}})^3 A_{1, \pow} \left[ 1 + 6A_{2, \pow}^2(1-\exp[-\kappa_{\pow} / (2\log^2(1 + d^{\varpi_{\pow}}))])^{-2} \right. \\ & \qquad \qquad \left. \times \defEns{A_{2, \pow}(1-\exp[-\kappa_{\pow} / \log^2(1 + d^{\varpi_{\pow}})])^{-1} + 2} + A_{2, \pow}^2 \log^4(1 + d^{\varpi_{\pow}}) / \kappa_{\pow}^2 + A_{3, \pow}^2  \right] \eqsp ,
  \end{aligned}
\end{equation}
which concludes the proof of \Cref{thm:cvx}-\ref{item:b}, upon setting $\varpi_{\pow} = \max_{i \in \{1, 2, 3\}}(\varpi_{i, \pow})$.

  We have that for any $x \in \rset^d$, $\norm{H_{\theta}(x)} = \norm{F(x)} \leq V^{1/4}(x)$. Since $H_{\theta}$ does not depend on $\theta$ we get that \cite[A4]{debortoli2018souk} is satisfied. In addition \cite[Proposition 24]{debortoli2018souk} implies that \cite[H2]{debortoli2018souk} is satisfied with
  \begin{equation}
    \Lambdabf_1(\gamma_1, \gamma_2) = A_{4, \pow} (1 + d^{\varpi_{4, \pow}}) \gamma_2^{-1/2} \abs{\gamma_1 - \gamma_2 } \eqsp , \quad \Lambdabf_2(\gamma_1, \gamma_2) = A_{4, \pow} (1 + d^{\varpi_{4, \pow}}) \gamma_2^{1/2} \eqsp ,
  \end{equation}
  with $A_{4, \pow} \geq 0$ which does not depend on the dimension $d$. As a consequence we can apply \cite[Theorem 4]{debortoli2018souk} and we get that for any $n \in \nsets$
      \begin{equation}
    \expe{  \defEns{\left. \sum_{k=1}^n \delta_k L(\theta_k) \middle/ \sum_{k=1}^n \delta_k \right. } - \min_{\msk} L  }\leq  \left. \tE_n \middle/  \left( \sum_{k=1}^n \delta_k \right) \right. \eqsp ,
  \end{equation}
  with,
  \begin{align}
    \label{eq:horrible_bound_not}
    \tE_n &= 2 \Rtheta +2\Rtheta \sum_{k=0}^n \delta_{k+1} \Psibf(\gamma_k) + C_{3, \pow} \sum_{k=0}^n \abs{\delta_{k+1} - \delta_k}\gamma_k^{-1} \\ &+2 \Rtheta C_{2, \pow} \sum_{k=0}^n \delta_{k+1} \gamma_{k+1}^{-2} \parentheseDeux{ \Lambdabf_1(\gamma_{k}, \gamma_{k+1})  + \Lambdabf_2(\gamma_{k}, \gamma_{k+1})\delta_{k+1} + \delta_{k+1}\gamma_{k+1}}  \\ &+ C_{3, \pow} \sum_{k=0}^n \delta_{k+1}^2\gamma_{k+1}^{-1} + C_{3, \pow} (\delta_{n+1} / \gamma_n - \delta_0 / \gamma_0) + C_{1, \pow} \sum_{k=0}^n \delta_{k+1}^2 \eqsp ,
  \end{align}
\end{proof}
with
\begin{align}
  &C_{1, \pow} = 2 A_{1, \pow} (1 + d^{\varpi_{\pow}}) \expe{V(\X_0^0)} + 2 \sup_{\theta \in \msk} \norm{\nabla L(\theta)}^2  \eqsp ,\\
  &C_{2, \pow} =  8 (1 + d^{\varpi_{\pow}})^4 \log^{4}(1 + d^{\varpi_{\pow}})A_{1, \pow} A_{2, \pow}^2 \\
  &\qquad \qquad \times \exp[- 2\bgamma \kappa_{\pow} / \log^2(1 + d^{\varpi_{\pow}})](1 + 2 A_{1, \pow} \expe{V(\X_0^0)})/\kappa_{\pow} \eqsp ,\\
  &C_{3, \pow} =  (1 + d^{\varpi_{\pow}}) A_{1, \pow} C_H (4 \Rtheta + \sup_{\theta \in \msk} \norm{\nabla L(\theta)}  + 1 + \delta_1 \LipgradF) \expe{V(\X_0^0)^{1/4}}\eqsp ,\\
  &C_H = 8 (1 + d^{\varpi_{\pow}}) \log^{2}(1 + d^{\varpi_{\pow}}) A_{2, \pow} \exp[- \bgamma \kappa_{\pow} / (4\log^2(1 + d^{\varpi_{\pow}}))]  / \kappa_{\pow} \eqsp .
\end{align}
Similarly to \Cref{thm:cvx}-\ref{item:b} since $A_{1, \pow}$, $A_{2, \pow}$, $A_{3, \pow}$, $A_{4, \pow}$ and $\kappa_{\pow}$ are independent of the dimension $d$. Setting $\varpi_{\pow} = \max_{i \in \{1, 2, 3, 4\}}(\varpi_{i,\pow})$ concludes the proof of \Cref{thm:cvx}-\ref{item:a} 


\subsection{Proof of \Cref{thm:non_cvx}}
\label{thm:non_cvx:proof}
In this section, we give alternative results to  \cite[Theorem 2, Theorem 4]{debortoli2018souk}. The main results of \cite{debortoli2018souk} are stated in $V$-norm or total variation. However, our particular framework allows us to use a Wasserstein distance with an appropriate cost function which implies that the constants appearing in our results  scale polynomially in the dimension $d$ even if the potential is non convex. The increasing batch size case is considered in \Cref{sec:incr-batch-size} and the fixed batch size case in \Cref{sec:fixed-batch-size}. We check that the main assumptions \Cref{assum:condition_majo_V} and \Cref{assum:condition_kernel_fix} below are satisfied in the setting of \Cref{thm:non_cvx} in \Cref{sec:check-cref-assumpt} and conclude in \Cref{sec:proof_of_thm_non_cvx}.

\subsubsection{Increasing batch size}
\label{sec:incr-batch-size}
In this section, we give an alternative result of \cite[Theorem 2]{debortoli2018souk} in the case where
\begin{enumerate*}[label=(\alph*)]
\item the controls on the family of Markov kernels $\ensembleLigne{\Kker_{\gamma, \theta}}{\gamma \in \ocint{0, \bgamma}, \ \theta \in \msk}$ are obtained with respect to an appropriate Wasserstein distance
\item the stochastic gradient does not depend on $\theta$. 
\end{enumerate*}
First, we show that under \tup{\Cref{assum:existence_compact}(1)}, $\mu \mapsto \mu(F)$ is Lipschitz with respect to the considered Wasserstein distance in \Cref{lemma:majo_d_w}. Then, we control the error in the perturbed gradient scheme in \Cref{lemma:error_bound} and \Cref{lemma:cv_sq_fix}. Our main result  is stated in \Cref{thm:cv_soul_non_cvx_1}.

Let $c: \ \rset^d \times \rset^d \to \coint{0,+\infty}$, defined for any $x,y \in \rset^d$ by $c(x,y) = \1_{\diag}(x,y) (1 + \norm{x-y}/R)$ where $R \geq 0$. Consider also the function $\WR : \rset^{2d} \to \rset_+$, and $\VR : \rset^d \to \rset_+$ given for $x,y \in \rset^d$ by
\begin{equation}
  \label{eq:def_V_W}
  \WR(x,y) = 1 + \norm{x-y}/R \eqsp, \qquad \VR(x) = 1 + \norm{x}/R \eqsp.
\end{equation}
We also define for any $\pow \in \nset$, $\Vpow : \ \rset^{\dim} \to \coint{1, +\infty}$ given for any $x \in \rset^{\dim}$ by
\begin{equation}
  \label{eq:def_V_pow}
  \Vpow(x) = 1 + \norm{x}^{2\pow} \eqsp .
\end{equation}
We recall that $\Kker_{\gamma, \theta}$ is the Markov kernel associated with the Langevin recursion \eqref{eq:ula} and expression given by \eqref{eq:langevin_kernel}. This kernel is well-defined under \Cref{assum:equi_meas} and \Cref{assum:existence_compact}($\upalpha$) with $\upalpha \geq 1$. Consider the following assumption.
 
\begin{assumptionH}
  \label{assum:condition_majo_V}
  \begin{enumerate}[label=(\roman*), leftmargin=0.5cm]
  \item  \label{assum:condition_majo_V_i}  There exists $A_1 \geq 1$ such that for any $a \in \ccint{1, 3}$, $n,p \in \N$, $k \in \{0, \dots, m_n \}$
       \begin{equation}
     \label{eq:condition_majo_V}
     \CPE{\Kker_{\gamma_n, \theta_n}^p \VR^a(\X_{k}^n)}{\X_0^0} \leq A_1 \VR^a(\X_0^0) \eqsp , \qquad \expe{\VR^a(\X_0^0)} < + \infty \eqsp .
   \end{equation}
   with $\{(\X_k^{\ell})_{k \in \{0,\ldots,m_\ell\}} \, : \, \ell \in \iint{0}{n}\}$ given by \eqref{eq:souk}.
 \item   \label{assum:condition_majo_V_ii}    There exist $A_2, A_3\geq 1$, $\rho \in \coint{0,1}$ such that for any $\gamma \in\ocint{0,\bgamma}$, $\theta \in \msk$, $x, y\in \rset^d$ and  $n \in \N$, $\Kker_{\gamma,\theta}$ has a stationary distribution $\pi_{\gamma,\theta}$ and 
   \begin{equation}
     \label{eq:condition_V}
     \dw{\updelta_x \Kker_{\gamma, \theta}^n, \updelta_y \Kker_{\gamma, \theta}^n} \leq A_2 \rho^{\gamma n} \WR(x,y) \eqsp, \quad \dw{\updelta_x \Kker_{\gamma, \theta}^n, \pi_{\gamma, \theta}} \leq A_2 \rho^{\gamma n} \VR(x) \eqsp , \quad \pi_{\gamma, \theta}(\VR) \leq A_3 \eqsp .
   \end{equation}  
 \item \label{assum:condition_majo_V_iii} There exists $\Psibf: \ \rset_+^{\star} \to \rset_+$ such that for any $\gamma \in \ocint{0, \bgamma}$ and $\theta \in \msk$, $\dw{\pi_{\gamma, \theta}, \pi_{\theta}} \leq \Psibf(\gamma)$.
  \end{enumerate}
\end{assumptionH}
  \begin{theorem}
    \label{thm:cv_soul_non_cvx_1}
Assume \tup{\Cref{assum:weak}(1)}, \tup{\Cref{assum:equi_meas}}, \tup{\Cref{assum:existence_compact}(1)} and \tup{\Cref{assum:condition_majo_V}}. Let $(\gamma_n)_{n \in \nset}$, $(\delta_n)_{n \in \nset}$ be sequences of non-increasing positive real numbers and $(m_n)_{n \in \nset}$ a sequence of positive integers satisfying $\delta_n < \parenthese{\sup_{\theta \in \msk} \normLigne{\nabla^2 L(\theta)}}^{-1}$ and $\gamma_n < \bgamma$. Then, there exists $(E_n)_{n \in \N}$ such that for any $n \in \nsets$
    \begin{equation}
    \expe{  \defEns{\left. \sum_{k=1}^n \delta_k L(\theta_k) \middle/ \sum_{k=1}^n \delta_k \right. } - \min_{\msk} L  }\leq  \left. E_n \middle/  \left( \sum_{k=1}^n \delta_k \right) \right. \eqsp ,
  \end{equation}
  with $(\theta_k)_{k \in \nset}$ and $L$ are defined in \eqref{eq:souk} and \eqref{eq:log_partition} respectively,  and  for any $n \in \nsets$
  \begin{multline}
    E_n = 2 \Rtheta^2 + 6 \Rtheta R \LipF p A_1 A_2 (\rho^{-\bgamma} / \log(1/\rho) +1) \sum_{k=0}^{n-1} \delta_{k+1} \defEns{1/(m_k \gamma_k) + \Psibf(\gamma_k)} \\ +\parenthese{2A_1 (\norm{F(0)} + 3R\LipF)^2 \expe{\VR^2(\X_0)}  + 2 \sup_{\theta \in \msk} L(\theta)^2 }\sum_{k=0}^{n-1} \delta_{k+1}^2  \eqsp .
  \end{multline}
\end{theorem}
The proof of this result is a simple adaptation to the one of \cite[Theorem 2]{debortoli2018souk}. However it is given for completeness. 

Let $(\eta_n)_{n \in \N}$ be defined for any $n \in \nset$ by 
\begin{equation}
  \eta_n = m_n^{-1} \sum_{k=1}^{m_n} \defEns{F(\X_k^n) - \pi_{\theta_n}(F)} \eqsp.
  \label{eq:error_term}
\end{equation}
We consider the following decomposition for any $n \in \nset$, 
\begin{equation}
  \label{eq:decompo_1}
  \eta_n = \eta_n^{(1)} + \eta_n^{(2)} \eqsp, \qquad \eta_n^{(1)} = \CPE{\eta_n}{\mathcal{F}_{n-1}} \eqsp , \qquad \eta_n^{(2)} = \eta_n - \CPE{\eta_n}{\mathcal{F}_{n-1}}\eqsp ,
\end{equation}
and $(\mathcal{F}_n)_{n \in \nset \cup \{ -1 \}}$ is defined for all $n \in \nset$ by
\begin{equation}
 \label{eq:def_F_n}
   \mathcal{F}_n = \sigma \left( \theta_0, \{(\X_k^{\ell})_{k \in \{0,\ldots,m_\ell\}} \, : \, \ell \in \iint{0}{n}\} \right) \eqsp ,  \qquad \mcf_{-1} = \sigma(\theta_0)
 \end{equation}

 We start with the following technical lemmas.

 \begin{lemma}
  \label{lemma:majo_d_w}
  Assume \tup{\Cref{assum:existence_compact}(1)} and  \tup{\Cref{assum:condition_majo_V}}. Then for any probability measures $\mu, \nu$ on $\mcb{\rset^d}$ such that $\mu(\norm{\cdot}) + \nu(\norm{\cdot}) < +\infty$,
  \begin{equation}
    \norm{\mu(F) - \nu(F)} \leq \RLipFp \dw{\mu, \nu} \eqsp ,
  \end{equation}
  with $\WR$ given in \eqref{eq:def_V_W}.
\end{lemma}

\begin{proof}
  Using \Cref{assum:existence_compact}(1) we have that for any $i \in \{1, \dots, p \}$ and $x,y \in \rset^d$, $\abs{F_i(x) - F_i(y)} \leq 3 \LipF \norm{x -y} \leq 3 R \LipF \WR(x,y)$.
  Let $\mu$ and $\nu$ on $\mcb{\rset^d}$ such that $\mu(\norm{\cdot}) + \nu(\norm{\cdot}) < +\infty$. Using the definition of the Wasserstein distance \eqref{eq:def_wass}, we have $\norm{\mu(F) - \nu(F)} \leq \sum_{i=1}^p \abs{\mu(F_i) - \nu(F_i)} \leq  \RLipFp \dw{\mu, \nu}$.
\end{proof}
  \begin{lemma}
    \label{lemma:error_bound}
 Assume \tup{\Cref{assum:equi_meas}}, \tup{\Cref{assum:existence_compact}(1)} and  \tup{\Cref{assum:condition_majo_V}}. Then we have for any $n \in \nset$
  \begin{equation}
      \expe{\normLigne{\eta_n^{(1)}} }   \leq  B_1 \defEns{\expe{\VR(\X_0^0)}/(m_n \gamma_n) +   \Psibf(\gamma_{n})}  \eqsp , \end{equation}
with $B_1 =  \RLipFp A_1 A_2 (\rho^{-\bgamma} / \log(1/\rho) + 1)$.
\end{lemma}
\begin{proof}
  Using the definition of $(\mcf_n)_{n \in \nset}$, see \eqref{eq:def_F_n}, the Markov property, \Cref{assum:condition_majo_V}-\ref{assum:condition_majo_V_ii}-\ref{assum:condition_majo_V_iii}, \Cref{lemma:majo_d_w} and that for any $\theta \in \msk$, we have for any $n \in \nsets$
  \begin{align}
    &    \| \CPE{ \eta_{n} }{\mathcal{F}_{n-1} } \| \leq m_{n}^{-1} \sum_{k=1}^{m_{n}} \norm{ \Kker_{\gamma_{n},\theta_{n}}^k F(\X_{0}^{n}) - \pi_{\theta_{n}} \left(F\right) } \\
    &\qquad \leq \RLipFp m_{n}^{-1} \sum_{k=1}^{m_{n}} \defEns{\dw{\updelta_{\X_0^n}\Kker_{\gamma_n, \theta_n}^k, \pi_{\gamma_n, \theta_n}}}  + \RLipFp \dw{\pi_{\gamma_{n}, \theta_{n}}, \pi_{\theta_{n}} } \\
                                   &\qquad  \leq \RLipFp m_{n}^{-1} \sum_{k=1}^{m_{n}} \defEns{A_2\rho^{\gamma_{n} k}\VR(\X_{m_n}^n)} + \RLipFp \Psibf(\gamma_{n})  \leq \frac{\RLipFp A_2\rho^{-\bgamma}\VR(\X^{n}_{m_{n}})}{\log(1/\rho)\gamma_{n}m_{n}} +  \RLipFp \Psibf(\gamma_{n}) \eqsp . \label{eq:error_cond} 
  \end{align}
  In a similar manner, we have
  \begin{equation}
    \norm{\CPE{ \eta_0 }{\X_0^0}}  \leq \frac{\RLipFp A_2\rho^{-\bgamma}\VR(\X^{0}_{0})}{\log(1/\rho)\gamma_{0}m_{0}} +  \RLipFp \Psibf(\gamma_{0}) \eqsp .
  \end{equation}
  We conclude using \Cref{assum:condition_majo_V}-\ref{assum:condition_majo_V_i}.
\end{proof}

  \begin{lemma}
    \label{lemma:cv_sq_fix}
 Assume \tup{\Cref{assum:weak}(1)}, \tup{\Cref{assum:equi_meas}}, \tup{\Cref{assum:existence_compact}(1)} and \tup{\Cref{assum:condition_majo_V}}. Then we have for any $n \in \nset$, 
 $   \expe{ \| \eta_n\|^2}  \leq B_2$, 
  with
  \begin{equation}
    \label{eq:C1}
    B_2 = 2A_1(\norm{F(0)} + 3R\LipF)^2\expe{\VR^2(\X_0^0)}  +  2 \sup_{\theta \in \msk} \| \nabla L(\theta) \|^2  \eqsp .
  \end{equation}
\end{lemma}

  \begin{proof}
    Using that $\norm[2]{x+y} \leq 2 (\norm[2]{x} + \norm[2]{y})$ for any $x,y \in \rset^d$, the Cauchy-Schwarz inequality \Cref{assum:condition_majo_V}-\ref{assum:condition_majo_V_i} and \Cref{prop:existence_P}, we get for any $ n \in \N$,
    \begin{align}
      \expe{\| \eta_n \|^2} &\leq 2m_n^{-1}\sum_{k=1}^{m_n} \normLigne{F(\X_k^n)}^2 + 2\normLigne{\nabla L(\theta_n)}^2 \eqsp .
    \end{align}
    We conclude using that for any $x \in \rset^d$, $\normLigne{F(x)} \leq (\norm{F(0)} + 3R\LipF) \VR(x)$ and the fact that $\sup_{\theta \in \msk} \normLigne{\nabla L(\theta)} < +\infty$.
  \end{proof}

\begin{proof}[Proof of \Cref{thm:cv_soul_non_cvx_1}]
Taking the expectation in  \cite[Theorem 3, Equation (8)]{atchade2017perturbed}, using the Cauchy-Schwarz inequality and the fact that $(\eta_n^{(2)})_{n \in \nset}$ is a martingale increment with respect to $(\mathcal{F}_n)_{n \in \nset}$, we get that for every $n \in \N$
    \begin{align}
      &\expe{ \sum_{k=1}^{n} \delta_{k} \defEns{L(\theta_k) - \min_{\msk} L}} \\
      & \qquad \leq \expe{2 \Rtheta^2 - \sum_{k=0}^{n-1} \delta_{k+1} \langle \Pi_{\msk}(\theta_k - \delta_{k+1} \nabla L (\theta_k)) - \theta^{\star}, \eta_k \rangle + \sum_{k=0}^{n-1} \delta_{k+1}^2 \norm{\eta_k}^2} \\ 
      & \qquad \leq 2 \Rtheta^2 + 2 \Rtheta \sum_{k=0}^{n-1} \delta_{k+1} \expe{\normLigne{\eta_k^{(1)}}} +2 \sum_{k=0}^{n-1} \delta_{k+1}^2 \expe{\norm{\eta_k}^2} \eqsp .
              \label{eq:majo_L1}
    \end{align}
    Combining this result, \Cref{lemma:error_bound} and \Cref{lemma:cv_sq_fix} completes the proof.
  \end{proof}

\subsubsection{Fixed batch size}
\label{sec:fixed-batch-size}
In this section, we give an alternative result of \cite[Theorem 4]{debortoli2018souk}, in the case where $m_n =1$ and $\gamma_n = \gamma_0$ for all $n \in \nset$. We consider the following additional assumption on the family of kernels $\ensemble{\Kker_{\gamma, \theta}}{\theta \in \msk, \gamma \in \ocint{0, \bgamma}}$.

\begin{assumptionH}
  \label{assum:condition_kernel_fix}
  There exists
    $\Lambdabf: \rset_+^* \to \rset_+$ such that for any $\gamma \in\ocint{0,\bgamma}$, $\theta_1,\theta_2 \in \msk$, $x \in \rset^d$
        \begin{equation}     
     \Vnorm[\VR]{\updelta_x \Kker_{\gamma, \theta_1} - \updelta_x \Kker_{\gamma, \theta_2} } \leq \Lambdabf(\gamma)\| \theta_1 - \theta_2 \| \VR^{2}(x) \eqsp .
   \end{equation}
\end{assumptionH}

\begin{theorem}
  \label{thm:cv_soul_non_cvx_2}
  Assume \tup{\Cref{assum:weak}(1)}, \tup{\Cref{assum:equi_meas}}, \tup{\Cref{assum:existence_compact}(1)}, \tup{\Cref{assum:curv_reg}}, \tup{\Cref{assum:f_grad_lip}}, \tup{\Cref{assum:condition_majo_V}} and \tup{\Cref{assum:condition_kernel_fix}}. Let $(\gamma_n)_{n \in \nset}$, $(\delta_n)_{n \in \nset}$ be sequences of non-increasing positive real numbers and $(m_n)_{n \in \nset}$ a sequence of positive integers satisfying $\delta_n < 1/(\sup_{\theta \in \msk} \normLigne{\nabla^2 L(\theta)})$ and for any $n \in \nset$, $\gamma_n = \gamma < \bgamma$, $m_n = m_0$ and $\sup_{n \in \nset} \abs{\delta_{n+1} - \delta_n} \delta_n^{-2} < +\infty$. Then, there exists $(\tE_n)_{n \in \N}$ such that for any $n \in \nsets$
    \begin{equation}
    \expe{  \defEns{\left. \sum_{k=1}^n \delta_k L(\theta_k) \middle/ \sum_{k=1}^n \delta_k \right. } - \min_{\msk} L  }\leq  \left. \tE_n \middle/  \left( \sum_{k=1}^n \delta_k \right) \right. \eqsp ,
  \end{equation}
   with $(\theta_k)_{k \in \nset}$ and $L$ are defined in \eqref{eq:souk} and \eqref{eq:log_partition} respectively, and
  \begin{equation}
    E_n = D\defEns{1 + \sum_{k=0}^{n-1} \delta_{k+1}^2 / \gamma + \sum_{k=0}^{n-1} \delta_{k+1}^2 \Lambdabf(\gamma)/ \gamma^2 + \sum_{k=0}^{n-1} \delta_{k+1} \Psibf(\gamma) + \delta_{n+1}/\gamma}  \eqsp ,
  \end{equation}
    and $D = 2 \Rtheta^2 + 6 \Rtheta R \LipF p + 2\Rtheta \tilde{B}_2 + 3 \tilde{B}_1 + B_2$ where $\tilde{B}_2$ is given in \Cref{lemma:cv_sq_fix}, $\tilde{B}_1$ in \Cref{lemma:cv_norm_b} and $\tilde{B}_2$ in \Cref{lemma:cv_norm_c_d}.
  \end{theorem}

  The proof of this result is an adaptation to the one of \cite[Theorem 4]{debortoli2018souk}. The main difference in the proof consists in a refinement of \cite[Lemma 16]{debortoli2018souk} which can be established in our setting and is given \Cref{lemma:cv_norm_c_d}.

Similarly to the proof of \Cref{thm:cv_soul_non_cvx_1}, we need to analyze the error $\eta_n$ for $n \in \nset$ defined by \eqref{eq:error_term}, but the decomposition \eqref{eq:decompo_1} has to be improved.  
For that purpose, we introduce \textit{Poisson solutions} associated with $F$.
Under \Cref{assum:condition_majo_V} for any
  $\theta \in \msk$ and $ \gamma \in \ocint{0,\bgamma}$, consider  $\hat{F}_{\gamma,\theta}: \rset^d \to \rset^{p}$ solution of
  the \textit{Poisson equation},
  \begin{equation}
  \label{eq:poisson_eq}
(\Id-\Kker_{\gamma,\theta}) \hat{F}_{\gamma,\theta} = F -
  \pi_{\gamma, \theta}(F)    \eqsp.
  \end{equation}
Note that by \Cref{assum:condition_majo_V}-\ref{assum:condition_majo_V_ii}, $\hat{F}_{\gamma,\theta}$ is well defined and is given for any  $x \in \rset^d$ by 
  \begin{equation}
    \label{eq:def_poisson}
       \hat{F}_{\gamma,\theta}(x) = \sum_{j \in \N} \{
  \Kker_{\gamma,\theta}^jF(x) - \pi_{\gamma,
    \theta}(F)\} \eqsp .
\end{equation}
In addition, by \Cref{lemma:majo_d_w} and \Cref{assum:condition_majo_V}-\ref{assum:condition_majo_V_ii}, we have for any $\theta \in \msk$ and $x \in \rset^d$
\begin{align}
  \norm{\hat{F}_{\gamma, \theta}(x)} &\leq \norm{\sum_{j \in \nset} \Kker_{\gamma, \theta}^jF(x) - \pi_{\gamma, \theta}(F)} \leq \RLipFp \sum_{j \in \nset} \dw{\Kker_{\gamma, \theta}^j, \pi_{\gamma, \theta}} \\
                                     &\leq \RLipFp A_2 \sum_{j \in \nset} \rho^{\gamma j} \VR(x) \leq \RLipFp A_2 \log^{-1}(1/\rho) \rho^{-\bgamma} \gamma^{-1} \VR(x) \leq C_F \gamma^{-1} \VR(x) \eqsp ,
                                         \label{eq:majo_poisson}
\end{align}
with $C_F = \RLipFp A_2 \log^{-1}(1/\rho) \rho^{-\bgamma}$. We now denote for any $n \in \nset$, $\tX_{n+1} = \X_1^n$ and therefore 
 $\eta_n$ defined by \eqref{eq:error_term} is given for any $n \in \nset$ by $\eta_n = F(\tX_{n+1}) - \pi_{\theta_n}(F)$.
Using \eqref{eq:poisson_eq} an alternative expression of $(\eta_n)_{n \in \nset}$ is given for any $n \in \N$ by 
\begin{equation}
  \eta_{n} =    \tpoissfix{n}(\tX_{n+1})-\tkernelfix{n} \tpoissfix{n}(\tX_{n+1}) + \pi_{\gamma, \theta_n}(F) - \pi_{\theta_{n}}(F)  = \etaa{n} + \etab{n} + \etac{n} + \etad{n} \eqsp,
\label{eq:proof_salem_cv_fix_decomp_eta}
\end{equation}
where 
\begin{equation}
\label{eq:tetan_def_2}
  \begin{aligned}
  \etaa{n} &=  \tpoissfix{n}(\tX_{n+1}) - \tkernelfix{n}\tpoissfix{n}(\tX_n) \eqsp,
  \\
  \etab{n} &=  \tkernelfix{n}\tpoissfix{n}(\tX_{n}) - \tkernelfix{n+1}\tpoissfix{n+1}(\tX_{n+1})  \eqsp ,
  \\
  \etac{n} &= \tkernelfix{n+1}\tpoissfix{n+1}(\tX_{n+1}) -  \tkernelfix{n}\tpoissfix{n}(\tX_{n+1})
  \eqsp ,
  \\
       \etad{n} &= \pi_{\gamma, \theta_n}(F) - \pi_{\theta_n}(F) \eqsp.
    \end{aligned}
  \end{equation}
  In the next results, we analyze  each term in this decomposition separately, except for $(\etaa{n})_{n \in \nset}$ which is a
  sequence of martingale increments with respect to $(\mathcal{F}_n)_{n \in \nset}$.
\begin{lemma}
  \label{lemma:cv_norm_b}
   Assume \tup{\Cref{assum:weak}(1)}, \tup{\Cref{assum:equi_meas}}, \tup{\Cref{assum:existence_compact}(1)} and \tup{\Cref{assum:condition_majo_V}}. Then, 
   for any $n \in \nset$
      \begin{multline}
        \expe{\norm{\sum_{k=0}^n \delta_{k+1} \langle a_{k+1}, \etab{k} \rangle}} \\ \leq \tilde{B}_1 \parentheseDeux{\sum_{k=0}^n \abs{\delta_{k+1} - \delta_k} \gamma^{-1} +  \sum_{k=0}^n \delta_{k}^2\gamma^{-1} +  (\delta_{n+1} / \gamma - \delta_1 / \gamma)} \eqsp .
      \end{multline}
      with $(\etab{n})_{n \in \nset}$ defined in \eqref{eq:tetan_def_2}, $a_{k+1} = \Pi_{\msk} \parentheseDeux{\theta_k - \delta_{k+1} \nabla L(\theta_k)} - \theta^{\star}$, $\theta^{\star} \in \argmin_{\msk} L$  and 
      \begin{multline}
        \label{eq:C3}
        \tilde{B}_1 =  A_1 C_F(2\Rtheta + \sup_{\theta \in \msk} \norm{\nabla L(\theta)}) \expe{\VR(\tX_0)} \\ + A_1 C_F (1 + \delta_1 \sup_{\theta \in \msk} \normLigne{\nabla^2 L(\theta)}) (\norm{F(0)} + R\LipF) \expe{\VR^2(\tX_0)}  + 4 A_1C_F \Rtheta \expe{\VR(\tX_0)} \eqsp .
      \end{multline} 
  \end{lemma}

  \begin{proof}
    By \eqref{eq:tetan_def_2} we have for any $n \in \nset$ and $\theta \in \msk$ 
    \begin{align}
      &\sum_{k=0}^n  \delta_{k+1} \langle a_{k+1}, \eta_k^{(b)}\rangle
      =  \sum_{k=1}^{n} \langle \delta_{k+1}a_{k+1} - \delta_{k}a_{k}, \tkernelfix{k} \tpoissfix{k} (\tX_{k}) \rangle  \\
      & \qquad \qquad - \langle \delta_{n+1}a_{n+1}, \tkernelfix{n+1} \tpoissfix{n+1} (\tX_{n+1}) \rangle + \langle \delta_1 a_{1}, \tkernelfix{0}\tpoissfix{0} (\tX_0) \rangle  \eqsp .\label{eq:ipp}
    \end{align}
    In addition, using \Cref{prop:existence_P}, we have for any $n \in \nset$ and $\theta \in \msk$,
    \begin{align}
      &\norm{\delta_{n+1} a_{n+1} - \delta_n a_{n}} \leq 2 \Rtheta \abs{\delta_{n+1} - \delta_n} + \delta_{n+1} \norm{a_{n+1} - a_{n}}\\
&   \leq 2 \Rtheta \abs{\delta_{n+1} - \delta_n} + \delta_{n+1} (1 + \delta_{n} \sup_{\theta \in \msk} \normLigne{\nabla^2 L(\theta)}) \norm{\theta_{n} - \theta_{n-1}}  + \delta_{n+1} \abs{\delta_{n+1} - \delta_n} \norm{\nabla L(\theta_{n})} \\
      &  \leq (2 \Rtheta + \delta_1 \sup_{\theta \in \msk} \norm{\nabla L(\theta)}) \abs{\delta_{n+1} - \delta_n}  \\
      \vspace{-1cm}
      & \qquad \qquad + \delta_{n}^2 (1+ \delta_{n} \sup_{\theta \in \msk} \normLigne{\nabla^2 L(\theta)})(\norm{F(0)} + 3 R\LipF) \VR(\tX_{n+1}) \eqsp, \label{eq:inter_1}
    \end{align}
    where we have used in the last inequality that $\Pi_{\msk}$ is non-expansive, \Cref{assum:existence_compact}(1) and  \Cref{assum:condition_majo_V}-\ref{assum:condition_majo_V_i} and \Cref{prop:existence_P} again.
Combining \eqref{eq:ipp}, \eqref{eq:inter_1}, \eqref{eq:majo_poisson}, the Cauchy-Schwarz inequality and \Cref{assum:condition_majo_V}-\ref{assum:condition_majo_V_i} we get that
  \begin{align}
    \expe{\norm{\sum_{k=0}^n  \delta_{k+1} \langle a_{k}, \etab{k}\rangle}} &\leq (2\Rtheta + \delta_1 \sup_{\theta \in \msk} \norm{\nabla L(\theta)})A_1 C_F \expe{\VR(\tX_0)} \sum_{k=0}^n \abs{\delta_{k+1} - \delta_k} \gamma^{-1} \\ &+ A_1 C_F(\norm{F(0)} + 3 R \LipF) (1+\delta_1 \sup_{\theta \in \msk} \normLigne{\nabla^2 L(\theta)})\expe{\VR^2(\tX_0)} \sum_{k=0}^n \delta_{k}^2\gamma^{-1} \\ &+ 2 A_1\Rtheta C_F \expe{\VR(\tX_0)} \defEns{\delta_{n+1}/\gamma + \delta_1 / \gamma} \eqsp ,
  \end{align}
  which concludes the proof of \Cref{lemma:cv_norm_b}.
\end{proof}

We now upper bound $\expe{\normLigne{\etac{n}}}$ for $n \in \nset$ using the two following lemmas.

\begin{lemma}
  \label{lemma:error_theta_1}
  Assume \tup{\Cref{assum:equi_meas}}, \tup{\Cref{assum:existence_compact}(1)}, \tup{\Cref{assum:condition_majo_V}} and \tup{\Cref{assum:condition_kernel_fix}}. Then, for any $\gamma \in \ocint{0, \bgamma}$, $\theta_1, \theta_2 \in \msk$ and $x \in \rset^d$
  \begin{equation}
    \dw{\pi_{\gamma, \theta_1}, \pi_{\gamma, \theta_2}} \leq A_1 A_2\rho^{-\bgamma}\log^{-1}(1/\rho) \Lambdabf(\gamma)\| \theta_1 - \theta_2 \| \VR^2(x) \gamma^{-1} \eqsp .
  \end{equation}
\end{lemma}

\begin{proof}
  Let $\gamma \in \ocint{0, \bgamma}$, $\theta_1, \theta_2 \in \msk$, $\ell \in \nsets$, $j \in \nset$ with $\ell \geq j+1$ and $g : \ \rset^d \to \rset$ measurable such that for any $y, z \in \rset^d$, $\abs{g(y) - g(z)}\leq \WR(y,z)$. Using \Cref{assum:condition_majo_V}-\ref{assum:condition_majo_V_ii} we have
  \begin{equation}
    \label{eq:ergo_geo}
    \abs{\Kker_{\gamma, \theta_2}^{\ell - 1 - j} g(x) - \pi_{\gamma,\theta_2}(g)} \leq A_2 \rho^{(\ell -1 -j)\gamma} \VR(x) \eqsp .
  \end{equation}
Combining this result and \Cref{assum:condition_kernel_fix} we have that
\begin{equation}
  \label{eq:inter_ergo}
  \abs{(\Kker_{\gamma, \theta_1} - \Kker_{\gamma, \theta_2})\Kker_{\gamma, \theta_2}^{\ell-1-j}g(x)} \leq A_2\rho^{\gamma(\ell - 1 -j)} \Lambdabf(\gamma)\| \theta_1 - \theta_2 \| \VR^2(x) \eqsp .
\end{equation}
Using \Cref{assum:condition_majo_V}-\ref{assum:condition_majo_V_i} in \eqref{eq:inter_ergo}, we get
\begin{equation}
  \abs{\Kker_{\gamma, \theta_1}^j(\Kker_{\gamma, \theta_1} - \Kker_{\gamma, \theta_2})\Kker_{\gamma, \theta_2}^{\ell-1-j}g(x)} \leq A_1 A_2\rho^{\gamma(\ell - 1 -j)}  \Lambdabf(\gamma)\| \theta_1 - \theta_2 \| \VR^2(x) \eqsp .
\end{equation}
Combining this result and the triangular inequality we obtain
\begin{align}
  \abs{\Kker_{\gamma, \theta_1}^{\ell} g(x) - \Kker_{\gamma, \theta_2}^{\ell} g(x)} &\leq \sum_{j=0}^{\ell-1} \abs{\Kker_{\gamma, \theta_1}^{j+1}\Kker_{\gamma, \theta_2}^{\ell - j-1}g(x) - \Kker_{\gamma, \theta_1}^{j}\Kker_{\gamma, \theta_2}^{\ell - j}g(x)}  \\
                                                                                     &\leq \sum_{j=0}^{\ell- 1} \abs{\Kker_{\gamma, \theta_1}^j(\Kker_{\gamma, \theta_1} - \Kker_{\gamma, \theta_2})\Kker_{\gamma, \theta_2}^{\ell-1-j}g(x)} \\
                                                                                     &\leq A_1 A_2  \Lambdabf(\gamma)\| \theta_1 - \theta_2 \| \VR^2(x) \sum_{j=0}^{\ell-1}\rho^{\gamma(\ell - 1 -j)} \\
  &\leq A_1 A_2\rho^{-\bgamma}\log^{-1}(1/\rho) \Lambdabf(\gamma)\| \theta_1 - \theta_2 \| \VR^2(x) \gamma^{-1} \eqsp .
\end{align}
Taking the limit $\ell \to +\infty$ and using \Cref{assum:condition_majo_V}-\ref{assum:condition_majo_V_ii} concludes the proof.
\end{proof}

\begin{lemma}
  \label{lemma:error_theta_2}
  Assume \tup{\Cref{assum:equi_meas}}, \tup{\Cref{assum:existence_compact}(1)},   \tup{\Cref{assum:f_grad_lip}} and \tup{\Cref{assum:condition_majo_V}}. Then, for any $\gamma \in \ocint{0, \bgamma}$, $\theta_1, \theta_2 \in \msk$, $\ell \in \nsets$, $j \in \nset$ with $\ell \geq j+1$ and  $x \in \rset^d$
  \begin{equation}
    \norm{\defEns{\Kker_{\gamma, \theta_1}^j - \pi_{\gamma, \theta_1}} (\Kker_{\gamma, \theta_1} - \Kker_{\gamma, \theta_2})\defEns{\Kker_{\gamma, \theta_2}^{\ell-1-j}F(x) - \pi_{\gamma, \theta_2}(F)}} \leq D_F   \norm{\theta_1 - \theta_2} \VR(x) \gamma \rho^{\gamma \ell} \eqsp ,
  \end{equation}
  with $D_F = A_2^2 \LipgradF \LipF (1 + 2 \Rtheta \LipgradF \bgamma)R + 2 \LipF^2$.
\end{lemma}

\begin{proof}
  Let $\gamma \in \ocint{0, \bgamma}$, $\theta_1, \theta_2 \in \msk$, $\ell \in \nsets$, $j \in \nset$ with $\ell \geq j+1$ and  $x, y \in \rset^d$. First, we have 
  \begin{align}
    &(\Kker_{\gamma, \theta_1} - \Kker_{\gamma, \theta_2})\Kker_{\gamma, \theta_2}^{\ell-1-j}\defEns{F(x) - \pi_{\gamma, \theta_2}(F)} - (\Kker_{\gamma, \theta_1} - \Kker_{\gamma, \theta_2})\Kker_{\gamma, \theta_2}^{\ell-1-j}\defEns{F(x) - \pi_{\gamma, \theta_2}(F)} \\
    & \qquad \qquad  = (\Kker_{\gamma, \theta_1} - \Kker_{\gamma, \theta_2})\Kker_{\gamma, \theta_2}^{\ell-1-j}F(x) - (\Kker_{\gamma, \theta_1} - \Kker_{\gamma, \theta_2})\Kker_{\gamma, \theta_2}^{\ell-1-j}F(y)  \\ & \qquad \qquad  = \Kker_{\gamma, \theta_1}\Kker_{\gamma, \theta_2}^{\ell - 1 - j}F(x) - \Kker_{\gamma, \theta_2}^{\ell - j}F(x) - \Kker_{\gamma, \theta_1}\Kker_{\gamma, \theta_2}^{\ell - 1 - j}F(y) + \Kker_{\gamma, \theta_2}^{\ell - j}F(y) \\
    & \qquad \qquad  = \Kker_{\gamma, \theta_2}^{\ell -j}(F(x + \Delta_{\gamma}(x)) - F(x) - F(y + \Delta_{\gamma}(y)) + F(y))  \\
    & \qquad \qquad  = \Kker_{\gamma, \theta_2}^{\ell -j}G(x) - \Kker_{\gamma, \theta_2}^{\ell -j}G(y) \eqsp , \label{eq:majo_F_1}
  \end{align}
  with $\Delta_{\gamma}(x) = \gamma(\nabla_x U(\theta_1, x) - \nabla_x U(\theta_2, x)) = \gamma \sum_{i=1}^p (\theta_1^i - \theta_2^i) \nabla F_i(x)$ and $G: \ \rset^{\dim} \to \rset^{p}$ defined for any $z \in \rset^{\dim}$ by
  \begin{equation}
    \label{eq:G_def}
    G(z) = F(z+\Delta_{\gamma}(z)) - F(z) \eqsp .
  \end{equation}
  Using \Cref{assum:existence_compact}(1) and \Cref{assum:f_grad_lip} we have that for any $x, y \in \rset^d$,
  \begin{equation}
    \label{eq:majo_delta}
    \norm{\Delta_{\gamma}(x)} \leq \LipF \gamma \norm{\theta_1 - \theta_2} \eqsp , \qquad \norm{\Delta_{\gamma}(x) - \Delta_{\gamma}(y)} \leq \Rtheta \LipgradF \gamma  \norm{x -y} \eqsp  .
  \end{equation}
Using \eqref{eq:G_def}, \eqref{eq:majo_delta}, we have for any $x,y \in \rset^{\dim}$ with $x \neq y$
  \begin{align}
    \norm{G(x) - G(y)} & = \norm{ \int_0^1 \defEns{\rmd F(x + t\Delta_{\gamma}(x))(\Delta_{\gamma}(x)) - \rmd F(y + t\Delta_{\gamma}(y))(\Delta_{\gamma}(y)) } \rmd t} \\
                       &\leq \int_0^1 \norm{\rmd F(x + t\Delta_{\gamma}(x)) - \rmd F(y + t\Delta_{\gamma}(y))}\norm{\Delta_{\gamma}(x)} \rmd t \\ & \qquad + \int_0^1 \norm{\rmd F(y + t\Delta_{\gamma}(y))} (\norm{\Delta_{\gamma}(x)} + \norm{\Delta_{\gamma}(y)}) \rmd t \\
                       &\leq \LipgradF (\norm{x-y}+ \norm{\Delta_{\gamma}(x) - \Delta_{\gamma}(y)}) \norm{\Delta_{\gamma}(x)} + 2 \LipF^2 \gamma \norm{\theta_1 - \theta_2} \\
                       &\leq \LipgradF \LipF (1 + 2 \Rtheta \LipgradF \bgamma) \gamma  \norm{\theta_1 - \theta_2}\norm{x-y} + 2 \LipF^2 \gamma \norm{\theta_1 - \theta_2} \\
    &\leq \tilde{D}_F \gamma \norm{\theta_1 - \theta_2} (1 + \norm{x-y}/R) \eqsp ,                         
  \end{align}
  with $\tilde{D}_F = \LipgradF \LipF (1 + 2 \Rtheta \LipgradF \bgamma)R + 2 \LipF^2$.
  Therefore, for any $x, y \in \rset^{\dim}$,
  \begin{equation}
    \normLigne{G(x) - G(y)} \leq  \tilde{D}_F \gamma \norm{\theta_1 - \theta_2} \WR(x,y)  \eqsp .\label{eq:majo_F_2} \end{equation}
Combining \eqref{eq:majo_F_1}, \eqref{eq:majo_F_2} and \Cref{assum:condition_majo_V}-\ref{assum:condition_majo_V_ii} we obtain that
\begin{align}
  &\norm{(\Kker_{\gamma, \theta_1} - \Kker_{\gamma, \theta_2})\Kker_{\gamma, \theta_2}^{\ell-1-j}\defEns{F(x) - \pi_{\gamma, \theta_2}(F)} - (\Kker_{\gamma, \theta_1} - \Kker_{\gamma, \theta_2})\Kker_{\gamma, \theta_2}^{\ell-1-j}\defEns{F(x) - \pi_{\gamma, \theta_2}(F)}} \\ & \qquad \qquad \qquad \qquad \leq A_2 \tilde{D}_F \norm{\theta_1 - \theta_2} \gamma \rho^{\gamma(\ell -j)} \WR(x,y) \eqsp .
\end{align}
Therefore, using \Cref{assum:condition_majo_V}-\ref{assum:condition_majo_V_ii} we get
\begin{equation}
      \norm{\defEns{\Kker_{\gamma, \theta_1}^j - \pi_{\gamma, \theta_1}} (\Kker_{\gamma, \theta_1} - \Kker_{\gamma, \theta_2})\defEns{\Kker_{\gamma, \theta_2}^{\ell-1-j}F(x) - \pi_{\gamma, \theta_2}(F)}} \leq A_2^2 \tilde{D}_F   \norm{\theta_1 - \theta_2} \VR(x) \gamma \rho^{\gamma \ell} \eqsp ,
\end{equation}
which concludes the proof.
\end{proof}

\begin{lemma}
  \label{lemma:cv_norm_c_d}
   Assume \tup{\Cref{assum:equi_meas}}, \tup{\Cref{assum:existence_compact}(1)},  \tup{\Cref{assum:f_grad_lip}}, \tup{\Cref{assum:condition_majo_V}} and \tup{\Cref{assum:condition_kernel_fix}}. Then we have for any $n \in \nset$
    \begin{equation}
      \expe{\norm{\etac{n}}} \leq \tilde{B}_2 \delta_{n+1} \gamma^{-2}\parenthese{\Lambdabf( \gamma) + \gamma} \eqsp ,
  \end{equation}
  with
  \begin{equation}
    \label{eq:C2}
    \tilde{B}_2 = A_1 (\norm{F(0)} +3R\LipF) \rho^{-2\bgamma} \log^{-2}(1/\rho) \max\defEns{ D_F, E_F}\eqsp ,
  \end{equation}
  with $(\etac{n})_{n \in \nset}$ defined in \eqref{eq:tetan_def_2}, $E_F = \RLipFp A_1 A_2^2$ and $D_F$ in \Cref{lemma:error_theta_2}.
\end{lemma}
\begin{proof}
We first give an upper bound on  $\norm{\Kker_{\gamma, \theta_1} \poissp{\gamma}{\theta_1}(x) - \Kker_{\gamma, \theta_2} \poissp{\gamma}{\theta_2}(x)}$ for any $\theta_1, \theta_2 \in \msk$, $\gamma\in \ocint{0,\bgamma}$  and $x \in \rset^d$.
By \eqref{eq:def_poisson}  we have  for any $\theta_1, \theta_2 \in \msk$, $\gamma \in \ocint{0,\bgamma}$ and $x \in \rset^d$,
\begin{align}
&\norm{\Kker_{\gamma, \theta_1} \poissp{\gamma}{\theta_1}(x) - \Kker_{\gamma, \theta_2} \poissp{\gamma}{\theta_2}(x)} \\
& = \norm{ \sum_{\ell \in \nsets} \defEns{\Kker_{\gamma,\theta_1}^\ell F (x) - \pi_{\gamma,\theta_1}(F)} -\sum_{\ell \in \nsets} \defEns{\Kker_{\gamma,\theta_2}^\ell F(x) - \pi_{\gamma,\theta_2}(F)} } \\
  &  \leq \sum_{\ell \in \nsets}  \norm{  \Kker_{\gamma,\theta_1}^\ell F(x) - \pi_{\gamma,\theta_1}(F) -\Kker_{\gamma,\theta_2}^\ell F(x) - \pi_{\gamma,\theta_2}(F) }  \eqsp . \label{eq:bound_eta_c_0}
\end{align}
We now bound each term of the series in the right hand side.
For any measurable functions $g_1,g_2$ with $g_i: \rset^d \to \rset^{p}$ and such that $\sup_{x \in \rset^d} \norm{g_i(x)}/\VR(x) < + \infty$ with $i \in \{ 1, 2 \}$,  
 $\theta_1, \theta_2 \in \msk$, $\gamma\in \ocint{0,\bgamma}$,  $x \in \rset^d$  and $\ell \in \nsets$, using that $\pi_{\gamma, \theta_1}$ is invariant for $\Kker_{\gamma, \theta_1}$, it holds that 
\begin{align}
  \Kker_{\gamma,\theta_1}^\ell g_1(x) - \Kker_{\gamma,\theta_2}^\ell  g_2(x) &= \Kker_{\gamma,\theta_1}^\ell g_1(x) - \Kker_{\gamma,\theta_2}^\ell  g_1(x) + \Kker_{\gamma,\theta_2}^\ell  (g_1(x) - g_2(x)) \\
  &= \sum_{j=0}^{\ell -1}\defEns{\Kker_{\gamma,\theta_1}^j - \pi_{\gamma,\theta_1}}(\Kker_{\gamma,\theta_1} - \Kker_{\gamma,\theta_2}) \defEns{\Kker_{\gamma,\theta_2}^{\ell -1-j}g_1(x) - \pi_{\gamma,\theta_2}(g_1)} \\
                                                             & \phantom{aa} + \sum_{j=0}^{\ell -1}\pi_{\gamma,\theta_1}\defEns{\Kker_{\gamma,\theta_2}^{\ell -1-j}g_1(x) - \Kker_{\gamma,\theta_2}^{\ell -j}g_1(x)} + \Kker_{\gamma,\theta_2}^\ell  (g_1(x) - g_2(x)) \\
  &=\sum_{j=0}^{\ell -1}\defEns{\Kker_{\gamma,\theta_1}^j - \pi_{\gamma,\theta_1}}(\Kker_{\gamma,\theta_1} - \Kker_{\gamma,\theta_2}) \defEns{\Kker_{\gamma,\theta_2}^{\ell -1-j}g_1(x) - \pi_{\gamma,\theta_2}(g_1)} \\
  & \phantom{aa} - \pi_{\gamma,\theta_1}(\Kker_{\gamma,\theta_2}^\ell g_1(x) - g_1(x)) + \Kker_{\gamma,\theta_2}^\ell  (g_1(x) - g_2(x)) \eqsp.   \label{eq:error_all}
    \end{align}
Setting $g_1 = F - \pi_{\gamma,\theta_1}(F)$ and $g_2 = F - \pi_{\gamma,\theta_2}(F)$, we obtain that
\begin{align}
  &\Kker_{\gamma,\theta_1}^\ell F(x) - \pi_{\gamma,\theta_1}(F) -\Kker_{\gamma,\theta_2}^\ell F(x) - \pi_{\gamma,\theta_2}(F) \\
  & \qquad \qquad =  \sum_{j=0}^{\ell -1}\defEns{\Kker_{\gamma,\theta_1}^j - \pi_{\gamma,\theta_1}}(\Kker_{\gamma,\theta_1} - \Kker_{\gamma,\theta_2}) \defEns{\Kker_{\gamma,\theta_2}^{\ell -1-j}F(x) - \pi_{\gamma,\theta_2}(F)}
                                                                                                              + \Xi_{\ell} \eqsp, \label{eq:decompo_K_gamma_ell}
\end{align}
where, using that $\pi_{\gamma,\theta_2}$ is invariant for $\Kker_{\gamma, \theta_2}$, we have
\begin{align}
  \Xi_{\ell} &= - \pi_{\gamma,\theta_1}(\Kker_{\gamma,\theta_2}^\ell F(x) - F(x)) + \Kker_{\gamma,\theta_2}^{\ell}\big[ \pi_{\gamma,\theta_2}(F) - \pi_{\gamma,\theta_1}(F)\big] \\
      & = (\pi_{\gamma,\theta_2} - \pi_{\gamma,\theta_1})\Kker_{\gamma,\theta_2}^\ell F(x) \eqsp . 
\end{align}
Using \Cref{lemma:error_theta_2}
we obtain for any $\theta_1, \theta_2 \in \msk$, $\gamma \in \ocint{0,\bgamma}$, $x \in \rset^d$ and $\ell \in \nsets$
\begin{align}
  &\norm{\sum_{j=0}^{\ell -1}\defEns{\Kker_{\gamma,\theta_1}^j - \pi_{\gamma,\theta_1}}(\Kker_{\gamma,\theta_1} - \Kker_{\gamma,\theta_2}) \defEns{\Kker_{\gamma,\theta_2}^{\ell -1-j}F(x) - \pi_{\gamma,\theta_2}(F)}} \\
  &\qquad \leq \sum_{j=0}^{\ell-1} \norm{\defEns{\Kker_{\gamma, \theta_1}^j - \pi_{\gamma, \theta_1}}(\Kker_{\gamma, \theta_1} - \Kker_{\gamma, \theta_2})\defEns{\Kker_{\gamma, \theta_2}^{\ell-1-j}(F) - \pi_{\gamma, \theta_2}(F)}} \\
  &\qquad \leq \sum_{j=0}^{\ell-1} D_F \gamma \norm{\theta_1 - \theta_2} \rho^{\gamma \ell} \VR(x)  \leq D_F \gamma  \VR(x) \norm{\theta_1 - \theta_2} \ell \rho^{\gamma \ell} \eqsp .
    \label{eq:majo_1_poiss}
\end{align}
Using \Cref{assum:condition_majo_V}-\ref{assum:condition_majo_V_ii}, \Cref{lemma:majo_d_w} and \Cref{lemma:error_theta_1}, we obtain for any $\theta_1, \theta_2 \in \msk$, $\gamma \in \ocint{0,\bgamma}$, $x \in \rset^d$ and $\ell \in \nsets$
\begin{align}
  \norm{(\pi_{\gamma,\theta_1} - \pi_{\gamma,\theta_2})\Kker_{\gamma,\theta_2}^\ell F(x)} &\leq \RLipFp A_2   \rho^{\gamma \ell} \dw{\pi_{\gamma,\theta_1},\pi_{\gamma,\theta_2} } \\ &\leq E_F \rho^{-\bgamma}\log^{-1}(1/\rho)  \Lambdabf(\gamma)\| \theta_1 - \theta_2 \| \VR^2(x) \gamma^{-1}\rho^{\gamma \ell} \eqsp. \label{eq:majo_2_poiss}
\end{align}
with $E_F = \RLipFp A_1 A_2^2$.
Combining \eqref{eq:majo_1_poiss} and \eqref{eq:majo_2_poiss} in \eqref{eq:decompo_K_gamma_ell}, we obtain that  for any $\theta_1, \theta_2 \in \msk$, $\gamma \in \ocint{0,\bgamma}$ and $x \in \rset^d$ that
\begin{align}
  &\Kker_{\gamma,\theta_1}^\ell F(x) - \pi_{\gamma,\theta_1}(F) -\Kker_{\gamma,\theta_2}^\ell F(x) - \pi_{\gamma,\theta_2}(F) \\ & \qquad \leq  D_F  \norm{\theta_1 - \theta_2} \VR(x) \gamma \ell \rho^{\gamma \ell} + E_F \rho^{-\bgamma}\log^{-1}(1/\rho)  \| \theta_1 - \theta_2 \| \VR^2(x) \Lambdabf(\gamma) \gamma^{-1}\rho^{\gamma \ell}  \eqsp .
\end{align}
Using this result in \eqref{eq:bound_eta_c_0} and that for any $t \in \ooint{-1,1}$ and $a > 0$, $\sum_{k \in \nset} k t^{ak} = t (1 - t^a)^{-2} \leq a^{-2} t^{-a} \log^{-2}(1/t) $, we get that 
\begin{align}
  \norm{\Kker_{\gamma,\theta_1}\poissp{\gamma}{\theta_1}(x) - \Kker_{\gamma,\theta_2}\poissp{\gamma}{\theta_2}(x)}
  &\leq D_F \rho^{-2\bgamma} \log^{-2}(1/\rho) \norm{\theta_1 - \theta_2} \VR(x)   \gamma^{-1} \\
  & \qquad + E_F \rho^{-2\bgamma}\log^{-2}(1/\rho)  \| \theta_1 - \theta_2 \| \VR^2(x) \Lambdabf(\gamma)\|  \gamma^{-2}   \\ 
  &\leq C_c\, \gamma^{-2} \parenthese{  \Lambdabf(\gamma)\| \theta_1 - \theta_2 \|  + \gamma \| \theta_1 - \theta_2\|}\VR^{2}(x) \eqsp , \label{eq:finaaaal}
\end{align}
with $C_c\, = \rho^{-2\bgamma} \log^{-2}(1/\rho) \max\defEns{D_F, E_F}$.
Note that for any $k \in \nset$, by \Cref{assum:existence_compact}(1) and the fact that $\Pi_{\msk}$ is non-expansive  we have $\norm{\theta_{k+1}-\theta_k} \leq \delta_{k+1} (\norm{F(0)} + 3 R \LipF) \VR(\tX_{k+1})$. Therefore, plugging this result in \eqref{eq:finaaaal}, we get for any $k \in \nset$,
\begin{align}
&  \norm{\Kker_{\gamma,\theta_k}\poissp{\gamma}{\theta_k}(\tX_{k+1}) - \Kker_{\gamma,\theta_{k+1}}\poissp{\gamma}{\theta_{k+1}}(\tX_{k+1})} \\  &\phantom{aaaa} \leq C_{c,2} (\norm{F(0)} + 3R\LipF)  \delta_{k+1}  \gamma^{-2} \parenthese{ \Lambdabf(\gamma)+ \gamma}\VR^3(\tX_{k+1}) \eqsp . \label{eq:derder}
\end{align}
 Therefore by definition of $(\eta_k^{(c)})_{k \in \nset}$, see \eqref{eq:tetan_def_2}, and using \Cref{assum:condition_majo_V}-\ref{assum:condition_majo_V_i} in \eqref{eq:derder} we get that for any $k \in \nset$
\begin{equation}\expe{\norm{\eta_k^{(c)}}} \leq \tilde{B}_2 \delta_{k+1} \gamma^{-2} \parenthese{\Lambdabf(\gamma) + \gamma}\eqsp , \end{equation} with $\tilde{B}_2$ given by \eqref{eq:C2}.
\end{proof}

\begin{lemma}
  \label{lemma:cv_norm_d}
  Assume \tup{\Cref{assum:equi_meas}}, \tup{\Cref{assum:existence_compact}(1)} and \tup{\Cref{assum:condition_majo_V}}. Then we have for any $n \in \nset$
    \begin{equation}
      \expe{\norm{\etad{n}}} \leq \RLipFp \Psibf(\gamma) \eqsp ,
    \end{equation}
    with $(\etad{n})_{n \in \nset}$ defined in \eqref{eq:tetan_def_2}.
\end{lemma}

\begin{proof}
  The proof is a direct consequence of \Cref{lemma:majo_d_w} and \Cref{assum:condition_majo_V}-\ref{assum:condition_majo_V_iii}.
\end{proof}

\begin{proof}[Proof of \Cref{thm:cv_soul_non_cvx_2}]
  Taking the expectation in  \cite[Theorem 3, Equation (8)]{atchade2017perturbed}, using the Cauchy-Schwarz inequality, the decomposition of the error \eqref{eq:tetan_def_2} and the fact that $(\eta_n^{(a)})_{n \in \nset}$ is a martingale increment with respect to $(\mathcal{F}_n)_{n \in \nset}$, we get that for every $n \in \N$
    \begin{align}
      &\expe{ \sum_{k=1}^{n} \delta_{k} \defEns{L(\theta_k) - \min_{\msk} L}} \\ &\leq \expe{2\Rtheta^2 - \sum_{k=0}^{n-1} \delta_{k+1} \langle\Pi_{\msk} ( \theta_k - \delta_{k+1} \nabla L(\theta_k)) -\theta^{\star}, \eta_k \rangle + \sum_{k=0}^{n-1} \delta_{k+1}^2 \| \eta_k \|^2 } \\
      &\leq 2 \Rtheta^2 + 2 \Rtheta \sum_{k=0}^{n-1} \delta_{k+1} \expe{\norm{\eta_k^{(c)}} + \norm{\eta_k^{(d)}}} + \expe{\sum_{k=0}^{n-1} \delta_{k+1} \langle a_{k+1}, \etab{k} \rangle} +\sum_{k=0}^{n-1} \delta_{k+1}^2 \expe{\norm{\eta_k}^2} \eqsp .
              \label{eq:majo_L1}
    \end{align}
    Combining this result, \Cref{lemma:cv_norm_b}, \Cref{lemma:cv_norm_c_d}, \Cref{lemma:cv_norm_d} and \Cref{lemma:cv_sq_fix} completes the proof.
\end{proof}


\subsubsection{Proof of \Cref{thm:non_cvx}}
\label{sec:check-cref-assumpt}
\label{sec:proof_of_thm_non_cvx}
In this section, we check that \Cref{assum:condition_majo_V} and \Cref{assum:condition_kernel_fix} are satisfied in order to apply \Cref{thm:cv_soul_non_cvx_2}. More precisely, we study the geometric ergodicity of the Langevin Markov chain under \tup{\Cref{assum:equi_meas}}, \tup{\Cref{assum:existence_compact}($1$)} and \tup{\Cref{assum:curv_reg}} as  well as its discretization error. We begin with the following technical lemma

\begin{lemma}
    \label{lemma:str_cvx_outside}
    Assume \tup{\Cref{assum:equi_meas}}, \tup{\Cref{assum:existence_compact}(1)} and \tup{\Cref{assum:curv_reg}}. Let $\mtt = \mtt_1 / 2$, $\tilde{\Lip} = 2 \Lip$, $R = 4 \borne /\mtt_1$ and $\upupsilon = \sup_{\theta \in \msk} \norm{\nabla_x U(\theta, 0)}$. In addition, for any $\theta \in \msk$, $\gamma > 0$ and $x \in \rset^{\dim}$, let 
    \begin{equation}
      \label{eq:tg_def}
      \Tg(x) = \norm{x - \gamma \nabla_x U(\theta, x)}^2 \eqsp .
    \end{equation}
Then for any $\theta \in \msk$ and $x, y \in \rset ^ {\dim}$
    \begin{enumerate}[label = (\alph *), leftmargin = 1.5cm]
    \item  $\norm{\nabla_x U(\theta, x) - \nabla_x U(\theta, y)} \leq \tilde{\Lip} \norm{x - y}$; \label{item:grad_lip}
    \item if $\norm{x - y} \geq R$, $\langle \nabla_x U(\theta, x) - \nabla U_x(\theta, y), x - y \rangle \geq \mtt \norm{x - y} ^ 2$, \label{item:strong_cvx_outside}
    \item \label{item:borne_Tg_global} we have 
      \begin{equation}
        \norm{\Tg(x)} \leq (1 + \gamma\tilde{\Lip}) \norm{x} + \gamma \upupsilon \eqsp ,
      \end{equation}      
    \item \label{item:borne_Tg} if $\norm{x} \geq \max(R, 2 \upupsilon / \mtt)$ and $\gamma \leq \mtt / (2 \tilde{N}^2)$
      \begin{equation}
\norm{\Tg(x)} \leq (1 - \gamma \mtt /2 + \gamma^2 \tilde{\Lip}^2 /2) \norm{x} \eqsp .
      \end{equation}
    \end{enumerate}
\end{lemma}

\begin{proof}
        Let $\theta \in \msk$. The proof of \ref{item:grad_lip} is straightforward. Let $\mtt = \mttun / 2$ and $x, y \in \rset ^ d$ such that $\norm{x - y} \geq R$ with $R = 4 \borne /\mttun$. Using \Cref{assum:curv_reg}-\ref{item:strong_convex}-\ref{item:bounded_grad} we have
    \begin{align}
        &\langle \nabla_x U(\theta, x) - \nabla_x U(\theta, y), x - y \rangle  \\
            &\qquad = \langle \nabla_x U_1(\theta, x) - \nabla_x U_1(\theta, y), x - y \rangle + \langle \nabla_x U_2(\theta, x) - \nabla_x U_2(\theta, y), x - y \rangle \\
            &\qquad \geq \mttun\norm{x - y} ^ 2 - 2 \borne \norm{x - y} \geq(\mttun - 2 \borne / R) \norm{x - y} ^ 2 \\
            &\qquad \geq(\mttun / 2) \norm{x - y} ^ 2 \eqsp,
    \end{align}
    which concludes the proof of \ref{item:strong_cvx_outside}.
    For any $x \in \rset^{\dim}$,
    \begin{equation}
      \norm{\Tg(x)} \leq \norm{\Tg(x) - \Tg(0)} + \norm{\Tg(0)} \leq (1 + \gamma \tilde{\Lip}) \norm{x} + \gamma \upupsilon \leq (1 + \gamma (\tilde{\Lip} + \upupsilon)) (1 + \norm{x}) - 1 \eqsp ,
    \end{equation}
    and therefore \ref{item:borne_Tg_global} holds.    
Finally, let $\norm{x} \geq \max(R, 2 \upupsilon / \mtt)$ and $\gamma \leq \mtt / (2 \tilde{N}^2)$. Using that for any $t \geq 0$, $\sqrt{1 + t} \leq 1 + t/2$,
    \begin{align}
      \norm{\Tg(x)} &\leq \norm{\Tg(x) - \Tg(0)} + \gamma \norm{\nabla_x U(\theta, 0)} \\
                    &\leq (1 - \gamma \mtt + \gamma^2 \tilde{\Lip}^2/2) \norm{x} + \gamma \upupsilon \leq (1 - \gamma \mtt /2 + \gamma^2 \tilde{\Lip}^2 /2) \norm{x} \eqsp , 
    \end{align}
    which concludes the proof of \ref{item:borne_Tg}.
\end{proof}

\begin{lemma}
  \label{lemma:drift_norm_noncvx}
  Assume \tup{\Cref{assum:equi_meas}}, \tup{\Cref{assum:existence_compact}(1)} and \tup{\Cref{assum:curv_reg}}.  Let $\mtt, \tilde{\Lip}$ and $R$ be given by \Cref{lemma:str_cvx_outside}. Then for any $\pow \in \nsets$, $\theta \in \msk$ and $\gamma \in \ocint{0, \bgamma}$ with $\bgamma < \min(\mtt / (2\tilde{\Lip}^2), 1/2)$, $\Kker_{\gamma, \theta}$ satisfies $\bfDd(\Vpow,\lambda^{\gamma},\tilde{b}_{\pow}\gamma)$ with $\Vpow$ given in \eqref{eq:def_V_pow} and 
  \begin{equation}
    \label{eq:drift_norm_noncvx}
    \begin{aligned}
      &\lambda = \exp[ -\mtt/4 + \bgamma \tilde{\Lip}^2/2] \eqsp , \\
      &\tilde{b}_{\pow}=  \Upsilon_{\pow}(2^{2\pow +1} d^{\pow} (1+ \bgamma \tilde{\Lip})^{2\pow-1}\Gammabf(\pow + 1/2), \mtt / 4) + \mtt / 4 + \rme^{\upkappa \bgamma}(\upkappa + \log(1/\lambda))\Vpow(\tilde{R}) + C_{\pow}(\tilde{R}) \eqsp , \\
      &\upkappa = \defEns{(1 + \bgamma\tilde{\Lip})\max(\tilde{R}, 1) + \upupsilon}^{2\pow} \eqsp ,\\
      &\upupsilon = \sup_{\theta \in \msk} \norm{\nabla_x U(\theta, 0)} \eqsp , \qquad \tilde{R} = \max(R, 2 \upupsilon / \mtt) \eqsp , \\
      &C_{\pow}(\tilde{R}) = 2^{2\pow +1} d^{\pow} \defEns{1+ \bgamma (\tilde{\Lip} + \upupsilon) }^{2\pow-1}\Gammabf(\pow + 1/2) (1 + \tilde{R})^{2\pow -1} \eqsp ,
      \end{aligned}
    \end{equation}
  where for any $t \geq 0$, $\Gammabf(t) = \int_0^{+\infty} u^{t-1} \rme^{-u} \rmd u$ and $\Upsilon_{\pow}$ is given in \Cref{lemma:ineq_reel}. In addition, $\Kker_{\gamma, \theta}$ satisfies $\bfDd(V,\lambda^{\gamma},b_{\pow}(1 + d^{\varpi_{0, \pow}})\gamma)$ with $\lambda$ given in \eqref{eq:drift_norm_noncvx} and $b_{\pow}, \varpi_{0, \pow} \geq 0$ independent of the dimension $d$.
\end{lemma}

\begin{proof}
  First, note that using \Cref{assum:existence_compact}(1)-\ref{item:a_compact} we get
  \begin{equation}
    \label{eq:mmsk}
    \upupsilon \leq \norm{\sum_{i=1}^p \theta_i \nabla F_i(0) + \nabla \Reg(0)} \leq \norm{\nabla \Reg (0)} + p \Rtheta \sup_{i \in \{1, \dots, p\}}  \norm{\nabla F_i(0)}  < +\infty \eqsp .
  \end{equation}
  Let $\pow \in \nsets$, $\theta \in \msk$, $\gamma \in \ocint{0, \bgamma}$ and $x \in \rset^{\dim}$. 
  Similarly to \Cref{lemma:drift}, we obtain that
  \begin{align}
    \label{eq:origin}
    \int_{\rset^{\dim}}\norm{y}^{2 \pow}  \Kker_{\gamma, \theta}(x, \rmd y)  & \leq \norm{\Tg(x)}^{2\pow} +\gamma 2^{2\pow +1}\gamma d^{\pow} \Gammabf(\pow + 1/2) (1 + \norm{\Tg(x)})^{2\pow -1} \\ &\leq  \norm{\Tg(x)}^{2\pow} +\gamma 2^{2\pow +1}\gamma C_{\pow}(x) \eqsp ,
  \end{align}
with $\Tg$ defined in \eqref{eq:tg_def} and 
  \begin{equation}
    \label{eq:C_pow}
    C_{\pow}(x) = 2^{2\pow +1} d^{\pow} \defEns{1+ \bgamma (\tilde{\Lip} + \upupsilon)}^{2\pow-1}\Gammabf(\pow + 1/2) (1 + \norm{x})^{2\pow -1} \eqsp ,
  \end{equation}
  where we have used \Cref{lemma:str_cvx_outside}-\ref{item:borne_Tg_global}.
  Let $\tilde{R} = \max(R, 2 \upupsilon / \mtt)$.
  We divide the rest of the proof in two parts:
  \begin{enumerate}[label=(\alph*), wide, labelwidth=!, labelindent=0pt]    
  \item Let $\norm{x} \geq \tilde{R}$. We have using \Cref{lemma:str_cvx_outside}-\ref{item:borne_Tg},
    \begin{equation}
      \norm{\Tg(x)} \leq (1 - \gamma \mtt /2 + \gamma^2 \tilde{\Lip}^2 /2) \norm{x} \eqsp .
    \end{equation}
Hence, $\norm{\Tg(x)}^{2\pow} \leq (1 - \gamma \mtt /4 + \gamma^2 \tilde{\Lip}^2 /2) \norm{x}^{2\pow} - \gamma \mtt \norm{x}^{2\pow}/4$ and we have using \Cref{lemma:ineq_reel} and \eqref{eq:C_pow} in \eqref{eq:origin}
\begin{align}
  &\Kker_{\gamma, \theta}\Vpow(x) \leq (1 - \gamma \mtt /4 + \gamma^2 \tilde{\Lip}^2 /2) (1 + \norm{x}^{2\pow}) \\
                               &\qquad \qquad +  \gamma \parentheseDeux{ 2^{2\pow +1}  d^{\pow} \defEns{1+ \bgamma (\tilde{\Lip} + \upupsilon)}^{2\pow-1}\Gammabf(\pow + 1/2) (1 + \norm{x})^{2\pow -1} -  \mtt \norm{x}^{2\pow}/4 +  \mtt /4} \\
                               &\qquad \leq (1 - \gamma \mtt /4 + \gamma^2 \tilde{\Lip}^2 /2) (1+\norm{x}^{2\pow}) \\
                               &\qquad \qquad + \gamma \parentheseDeux{ \Upsilon_{\pow}(2^{2\pow +1} d^{\pow} \defEns{1+ \bgamma (\tilde{\Lip} + \upupsilon)}^{2\pow-1}\Gammabf(\pow + 1/2), \mtt / 4) +  \mtt / 4} \eqsp .
\end{align}
\item Now assume that $\norm{x} \leq \tilde{R}$.    Let $\upkappa = \defEnsLigne{(1 + \tilde{\Lip}) \max(1, \tilde{R}) + \upupsilon}^{2 \pow}$. We have, using that $\gamma \leq 1$,
  \begin{equation}(1 + \gamma \tilde{\Lip})^{2\pow} \leq 1 + \gamma \sum_{k=1}^{2\pow} {2\pow \choose k} \Lip^{k} \leq 1 + \gamma (1 + \tilde{\Lip})^{2 \pow} \leq 1 + \gamma \upkappa \eqsp . \label{eq:ineq_util}\end{equation}
  Combining this result with \Cref{lemma:str_cvx_outside}-\ref{item:borne_Tg_global} and the fact that $\bgamma \leq 1$, we get 
  \begin{align}
    1 + \norm{\Tg(x)}^{2\pow} &\leq 1 + \parentheseDeux{(1 + \gamma \tilde{\Lip})\norm{x} + \gamma \upupsilon}^{2\pow} \\
                  &\leq 1 + (1 + \gamma \tilde{\Lip})^{2 \pow}\norm{x}^{2 \pow} + \gamma \sum_{k=1}^{2 \pow} {2\pow \choose k} (1 + \bgamma \tilde{\Lip})^{2 \pow - k }\tilde{R}^{2 \pow -k} \upupsilon^{k}\\
    &\leq 1 + (1 + \gamma \upkappa)\norm{x}^{2 \pow} + \gamma \upkappa \leq (1 + \gamma \upkappa) (1 + \norm{x}^{2\pow} ) \eqsp . \label{eq:ineq_gros}
  \end{align}
  Let
  \begin{equation}
    C_{\pow}(\tilde{R}) = 2^{2\pow +1} d^{\pow} \defEns{1+ \bgamma (\tilde{\Lip} + \upupsilon) }^{2\pow-1}\Gammabf(\pow + 1/2) (1 + \tilde{R})^{2\pow -1} \eqsp .
  \end{equation}
  Using \eqref{eq:ineq_gros} in \eqref{eq:origin} and that for any $a \geq b$, $\rme^a - \rme^b \leq (a-b)\rme^a$ we have
\begin{align}
  \Kker_{\gamma, \theta}(1 + \norm{x}^{2\pow}) &\leq 1 + \norm{\Tg(x)}^{2\pow} + \gamma C_m(\tilde{R})  \\
  &\leq \lambda^{\gamma} \Vpow(x) + \gamma \rme^{\upkappa \bgamma}(\kappa + \log(1/\lambda)) \Vpow(\tilde{R}) + \gamma C_m(\tilde{R}) \eqsp ,
\end{align}
which concludes the proof upon noting that $\tilde{b}_{\pow}$ is a polynomial in the dimension $d$.
  \end{enumerate}
\end{proof}
\begin{lemma}
  \label{lemma:bornitude_noncvx}
  Assume \tup{\Cref{assum:equi_meas}}, \tup{\Cref{assum:existence_compact}(1)}, \tup{\Cref{assum:curv_reg}} and let $(\X_k^n)_{n \in \nset, k \in \{0, \dots, m_n}$ be given by \eqref{eq:souk} with $\bgamma < \min(\mtt / (2 \tilde{\Lip}^2), 1/2)$.
  There exist $A_{1} \geq 1$ and $\varpi_{1} \geq 0$ such that for any $a \in \ccint{1, 3}$, $n, p \in \nset$ and $k \in \{0, \dots, m_n\}$
  \begin{equation}
    \CPE{\Kker_{\gamma_n, \theta_n}^p \VR^{a}(\X_k^n)}{\X_0^0} \leq A_{1} \VR^{a}(\X_0^0) \eqsp , \qquad \expe{\VR^{a}(\X_0^0)} < +\infty\eqsp ,
  \end{equation}
  with $\VR$ given in \eqref{eq:def_V_W} and $A_1, \varpi_{1}$ which do not depend on the dimension $d$.
\end{lemma}

\begin{proof}
  Using Jensen's inequality it suffices to prove the result for $a =3$.
  Using \Cref{lemma:drift_norm_noncvx}, there exist $\lambda \in \ooint{0,1}$ and $b \geq 0$ such that for any $\pow \in \{1, 2 , 3\}$, $\theta \in \msk$ and $\gamma \in \ocint{0, \bgamma}$, $\Kker_{\gamma, \theta}$ satisfies $\bfDd(\Vpow,\lambda^{\gamma},b_{\pow}(1 + d^{\varpi_{0, \pow}}) \gamma)$ with $\Vpow$ given in \eqref{eq:def_V_pow}. Hence, since $\lambda$ and $b_{\pow}$ do not depend on the dimension $d$, using \cite[Lemma S15]{debortoli2018souk}, there exists $\tilde{A}_1, \varpi_{1} \geq 0$ such that for any $\pow \in \{1, 2, 3\}$
  \begin{equation}
    \CPE{\Kker_{\gamma_n, \theta_n}^p \Vpow(\X_k^n)}{\X_0^0} \leq \tilde{A}_{1} (1 + d^{\varpi_{1}}) \Vpow(\X_0^0) \eqsp , \qquad \expe{\Vpow(\X_0^0)} < +\infty\eqsp ,
  \end{equation}
  with $\tilde{A}_1$ and $\varpi_{1}$ which do not depend on the dimension $d$.
Combining this result and Jensen's inequality we obtain that
\begin{align}
  \CPE{\Kker_{\gamma_n, \theta_n}^p \VR^3(\X_k^n)}{\X_0^0} &\leq  R^{-\pow}  \sum_{\pow=0}^3 {3 \choose \pow} \CPE{\Kker_{\gamma_n, \theta_n}^p \Vpow(\X_k^n)}{\X_0^0}^{1/2} \\
                                                         &\leq R^{-\pow}  \sum_{\pow=0}^3 {3 \choose \pow} \tilde{A}_1^{1/2}(1 + d^{\varpi_{1}})^{1/2}(1 + \normLigne{\X_0^0}^{2\pow})^{1/2}  \\
  &\leq R^{-\pow}  \tilde{A}_1^{1/2}(1 + d^{\varpi_{1}})^{1/2} \sum_{\pow=0}^3 {3 \choose \pow} (1 + \normLigne{\X_0^0}^{\pow})  \\ &\leq 9\tilde{A}_1^{1/2}(1 + d^{\varpi_{1}})^{1/2}  (1 + \normLigne{\X_0^0}/R)^3 \eqsp ,
\end{align}
which concludes the proof.
\end{proof}

\begin{theorem}
  \label{thm:ergo_cv_noncvx}
    Assume \tup{\Cref{assum:equi_meas}}, \tup{\Cref{assum:existence_compact}(1)} and \tup{\Cref{assum:curv_reg}}. Then  there exist $A_{2}, \varpi_{2} \geq 0$ and $\rho \in \ooint{0,1}$ such that for any $\theta \in \msk$ and $\gamma \in \ocint{0, \bgamma}$ with $\bgamma < \min(\mtt / (2\tilde{\Lip}^2), 1/2)$, $\Kker_{\gamma, \theta}$ admits an invariant probability measure $\pi_{\gamma, \theta}$ and for any $n \in \nset$ and $x \in \rset^d$
    \begin{equation}
      \label{eq:ergo_cv_noncvx}
      \dw{\updelta_x \Kker_{\gamma, \theta}^n, \pi_{\gamma, \theta}} \leq A_{2} (1 + d^{\varpi_{2}}) \rho^{\gamma n}\VR(x) \eqsp , \qquad 
      \dw{\updelta_x \Kker_{\gamma, \theta}^n, \updelta_y \Kker_{\gamma, \theta}^n} \leq A_{2} (1 + d^{\varpi_{2}}) \rho^{\gamma n}\WR(x,y) \eqsp ,
    \end{equation}
with $\VR, \WR$ given in \eqref{eq:def_V_W} and $A_2, \varpi_{2} \geq 0$ and $\rho \in \ooint{0, 1}$ which do not depend on the dimension $d$.
\end{theorem}

\begin{proof}
  Let $\theta \in \msk$, $\gamma \in \ocint{0, \bgamma}$, $n \in \nset$ and $x,y \in \rset^d$
  Applying \cite[Theorem 10]{debortoli2018back}, we obtain that there exist $\tA_2 \geq 0$ and $\rho \in \ooint{0,1}$ independent of the dimension $d$ such that 
  \begin{equation}
    \dw{\updelta_x \Kker_{\gamma, \theta}^n, \updelta_y \Kker_{\gamma, \theta}^n} \leq \tA_{2}  \rho^{\gamma n}\WR(x,y) \eqsp .
  \end{equation}
  In addition,  $\Kker_{\gamma, \theta}$ admits an invariant probability measure $\pi_{\gamma, \theta}$. Hence, we have
  \begin{equation}
    \dw{\updelta_x \Kker_{\gamma, \theta}^n, \pi_{\gamma, \theta}} \leq \tA_{2} \rho^{\gamma n}\int_{\rset^d} \WR(x,y) \rmd \pi_{\gamma, \theta}(y) \eqsp . \label{eq:ineq_inter_1}
  \end{equation}
  By \Cref{lemma:drift_norm_noncvx} we get that  $\Kker_{\gamma, \theta}$ satisfies $\bfDd(\Vpowdeux,\lambda^{\gamma},b_{2}(1+ d^{\varpi_{0, 2}})\gamma)$ where $\Vpowdeux(x)$ is given by \eqref{eq:def_V_pow} with $m=2$, $\lambda, b_{2}$ and $\varpi_{0, 2}$ are independent of the dimension $d$. Hence, using Jensen's inequality and that $\sqrt{1 + t} \leq 1 + t/2$ for $t \geq 0$ we get 
  \begin{equation}
    \pi_{\gamma, \theta}(\Vpowdeux^{1/2}) \leq \lambda^{\gamma/2} \pi_{\gamma, \theta}(\Vpowdeux^{1/2}) + b_2(1+d^{\varpi_{0,2}}) \lambda^{-\bgamma /2} \gamma /2 \eqsp . \end{equation}
  Therefore we obtain that
    \begin{equation}
    \label{eq:ineq_inter_2}
    \pi_{\gamma, \theta}(\Vpowdeux^{1/2}) \leq  b_2(1+d^{\varpi_{0,2}}) \lambda^{-\bgamma /2}\gamma /(2 - 2\lambda^{\gamma /2} ) \leq b_2(1+d^{\varpi_{0,2}}) \lambda^{-\bgamma} \log^{-1}(1/\lambda) \eqsp . \end{equation}
  In addition, for any $y \in \rset^d$,
  \begin{equation}
    \label{eq:ineq_inter_3}
    1 + \norm{y}/R \leq 2^{1/2}(1 + \norm{y}^2 /R^2)^{1/2}\leq 2^{1/2}(1 + 1/R^2)^{1/2}\Vpowdeux^{1/2}(y) \eqsp .
    \end{equation}
  Combining \eqref{eq:ineq_inter_1}, \eqref{eq:ineq_inter_2}, \eqref{eq:ineq_inter_3} and Jensen's inequality we get that
  \begin{align}
    \dw{\updelta_x \Kker_{\gamma, \theta}^n, \pi_{\gamma, \theta}} &\leq \tA_{2} \rho^{\gamma n}\int_{\rset^d} \WR(x,y) \rmd \pi_{\gamma, \theta}(y) \\
                                                                   &\leq \tA_{2} \rho^{\gamma n} \defEns{\norm{x}/R + \pi_{\gamma, \theta}(\VR)} \\
                                                                   &\leq \tA_2 \rho^{\gamma n} \defEns{\norm{x}/R + 2^{1/2}(1 + 1/R^2)^{1/2}b_2(1+d^{\varpi_{0,2}})\lambda^{-\bgamma}/\log(1/\lambda)} \\
    &\leq \tA_2 \defEns{1 + 2^{1/2}(1 + 1/R^2)^{1/2}b_2(1+d^{\varpi_{0,2}})\lambda^{-\bgamma}/\log(1/\lambda)}\VR(x) \eqsp ,
  \end{align}
which concludes the proof.
\end{proof}

\begin{lemma}
  \label{lemma:control_error_v_norm}
  Assume \tup{\Cref{assum:equi_meas}}, \tup{\Cref{assum:existence_compact}(1)} and \tup{\Cref{assum:curv_reg}}. Then there exists $\tilde{A}_3, \varpi_{3}' \geq 0$ such that for any $\theta \in \msk$, $\gamma \in \ocint{0, \bgamma}$ with $\bgamma < \min(\mtt / (2\tilde{\Lip}^2), 1/2)$ and $k \in \nset$
  \begin{equation}
    \Vnorm[\VR]{\pi_{\theta} \Pker_{k \gamma \step, \theta} \Pker_{\gamma \step, \theta} - \pi_{\theta} \Pker_{ k \gamma \step, \theta} \Kker_{\gamma, \theta}^{\step}} \leq \tilde{A}_3 (1 + d^{\varpi_{3}'}) \gamma^{1/2} \eqsp , \end{equation}
 with $\VR$ given in \eqref{eq:def_V_W} and $\tilde{A}_3, \varpi_{3}' \geq 0$ which do not depend on the dimension $d$.
\end{lemma}

\begin{proof}
  Let $\theta \in \msk$, $\gamma \in \ocint{0, \bgamma}$.
  First, we show that $(\Pker_{t, \theta})_{t \geq 0}$ satisfies a drift condition $\bfDc(\Vpowdeux,\zeta,\beta_{\pow}(1 + d^{\varpi_{0, \pow}'}))$, with $\Vpowdeux(x) = 1 + \norm{x}^2$, $\zeta >0$ and $\beta_{\pow}, \varpi_{0, \pow}' \geq 0$ independent of the dimension $d$.
  We have that for any $x \in \rset^d$, $\nabla \Vpowdeux(x) = 2 x $ and $\Delta \Vpowdeux(x) = 2 d $.
  Hence, for any $x \in \rset^d$
  \begin{equation}
    \generator_{\theta} \Vpowdeux(x) = -\langle \nabla_x U(\theta, x), \nabla \Vpowdeux(x) \rangle + \Delta \Vpowdeux(x) = -2 \langle \nabla_x U(\theta, x), x \rangle  + 2d \eqsp .
  \end{equation}
  We now distinguish two cases.
  \begin{enumerate}[label=(\alph*), wide, labelwidth=!, labelindent=0pt]
  \item If $\norm{x} \geq R$, using \Cref{lemma:str_cvx_outside}-\ref{item:strong_cvx_outside} we have 
    \begin{align}
      \generator_{\theta} \Vpowdeux(x) &\leq - 2 \mtt \norm{x}^2 + 2d + 2  \sup_{\theta \in \msk} \norm{\nabla_x U(\theta, 0)}\norm{x} \\
      &\leq - \mtt \Vpowdeux(x) + 2\defEnsLigne{d +  \sup_{\theta \in \msk} \norm{\nabla_x U(\theta, 0)}\norm{x} - \mtt \norm{x}^2 /2 + \mtt} \\
      &\leq - \mtt \Vpowdeux(x) + 2\defEnsLigne{d + \sup_{\theta \in \msk} \norm{\nabla_x U(\theta, 0)}^2 /(2 \mtt) + \mtt} \eqsp .
    \end{align}
  \item If $\norm{x} \leq R$, using \Cref{lemma:str_cvx_outside}-\ref{item:grad_lip} we have
    \begin{align}
      \generator_{\theta} \Vpowdeux(x) &\leq 2(\tilde{\Lip} \norm{x} + \sup_{\theta} \norm{\nabla_x U(\theta, x)}) \norm{x} + 2d \\
      &\leq -\mtt \Vpowdeux(x) + 2(\tilde{\Lip} R + \sup_{\theta \in \msk} \norm{\nabla_x U(\theta, x)}) R + 2d + \mtt \Vpowdeux(R) \eqsp .
    \end{align}
  \end{enumerate}
    Hence, there exists $\zeta >0$ and $\beta_{\pow}, \varpi_{0, \pow}' \geq 0$ such that $(\Pker_{t, \theta})_{t \geq 0}$ satisfies $\bfDc(\Vpowdeux,\zeta,\beta_{\pow}(1+d^{\varpi_{0, \pow}}))$, with $\zeta, \beta_{\pow}$ and $\varpi_{0, \pow}$ independent of the dimension $d$. This implies by \cite[Theorem 4.5]{meyn1993criteria_iii}
    \begin{equation}
      \label{eq:drift_c_disc}
       \pi_{\theta}(\Vpowdeux) \leq \beta_{\pow}(1 +d^{\varpi_{0, \pow}}) / \zeta \eqsp .
    \end{equation}
    Using a generalized Pinsker inequality \cite[Lemma 24]{durmus2017unadjusted}, \cite[Equation 15]{durmus2017unadjusted} and that for any $y \in \rset^{\dim}$, $\VR(y) \leq (1 + 1/R^2)^{1/2} \Vpowdeux(y)$, we get that
    \begin{align}
      &\Vnorm[\VR]{\pi_{\theta} \Pker_{k \gamma \step, \theta} \Pker_{\gamma \step, \theta} - \pi_{\theta} \Pker_{ k \gamma \step, \theta} \Kker_{\gamma, \theta}^{\step}} \\ &\leq 2(1 +1/R^2)^{1/2} (\pi_{\theta} \Pker_{(k+1) \gamma \step} \Vpowdeux + \pi_{\theta} \Pker_{k \gamma \step} \Kker_{\gamma, \theta}^{\step}\Vpowdeux)^{1/2} \\ & \qquad \qquad \qquad \times \KL{\pi_{\theta} \Pker_{ k \gamma \step, \theta} \Kker_{\gamma, \theta}^{\step}}{\pi_{\theta} \Pker_{k \gamma \step, \theta} \Pker_{\gamma \step, \theta}}^{1/2} \\
      &\leq  (1+1/R^2)^{1/2}(\pi_{\theta}( \Vpowdeux) + \pi_{\theta} \Kker_{\gamma, \theta}^{\step}\Vpowdeux)^{1/2} \\ & \qquad \qquad \qquad \times \tilde{\Lip} \parenthese{2 \tilde{\Lip} \bgamma \sup_{j \in \nset} \defEns{\pi_{\theta}  \Kker_{\gamma, \theta}^{\step j} \Vpowdeux} + 2 \bgamma \sup_{\theta \in \msk} \norm{\nabla_x U(\theta, 0)}^2 + d}^{1/2} \eqsp .
    \end{align}
    Combining this result, \eqref{eq:drift_c_disc} and \Cref{lemma:drift_norm_noncvx} completes the proof.
\end{proof}
  \begin{proposition}
    \label{prop:error_disc_noncvx}
  Assume \tup{\Cref{assum:equi_meas}}, \tup{\Cref{assum:existence_compact}(1)} and \tup{\Cref{assum:curv_reg}}. Then there exist $A_{3}, \varpi_{3} \geq 0$ such that for any $\theta \in \msk$, $\gamma \in \ocint{0, \bgamma}$ with $\bgamma < \min(\mtt / (2\tilde{\Lip}^2), 1/2)$,
  \begin{equation}
    \dw{\pi_{\gamma, \theta},\pi_{\theta}} \leq A_{3} (1 + d^{\varpi_{3}}) \gamma^{1/2}\eqsp ,
  \end{equation}
with $\WR$ given in \eqref{eq:def_V_W} and $A_{3}, \varpi_{3} \geq 0$ which do not depend on the dimension $d$.
\end{proposition}

\begin{proof}
  Let $\theta \in \msk$, $\gamma \in \ocint{0, \bgamma}$ and $x \in \rset^d$.
  Using \cite[Theorem 10]{debortoli2018back}, we get that
  \begin{equation}
    \dw{\pi_{\gamma, \theta}, \pi_{\theta}} = \lim_{\n \to +\infty} \dw{\pi_{\theta} \Kker_{\gamma, \theta}^{n \step}, \pi_{\theta}} \eqsp .
  \end{equation}
By \Cref{thm:ergo_cv_noncvx}, \Cref{lemma:control_error_v_norm} and that for any $\theta \in \msk$, $\pi_{\theta}$ is an invariant probability measure for $(\Pker_{t, \theta})_{t \geq 0}$, we get for any $n \in \nset$
\begin{align}
  &\dw{\pi_{\theta} \Kker_{\gamma, \theta}^{n \step}, \pi_{\theta} \Pker_{n \gamma \step, \theta}} \leq \sum_{k=0}^{n-1} \dw{\pi_{\theta} \Pker_{(k+1) \gamma \step, \theta} \Kker_{\gamma, \theta}^{(n-k-1)\step}, \pi_{\theta} \Pker_{k \gamma \step, \theta} \Kker_{\gamma, \theta}^{(n-k)\step}} \\
                                                                &\qquad \qquad \leq A_2(1+d^{\varpi_2}) \sum_{k=0}^{n-1}\rho^{n-k-1} \dw{\pi_{\theta} \Pker_{k \gamma \step, \theta} \Pker_{\gamma \step, \theta}, \pi_{\theta} \Pker_{ k \gamma \step, \theta} \Kker_{\gamma, \theta}^{\step}} \\
                                                                &\qquad \qquad \leq A_2(1+d^{\varpi_2}) \sum_{k=0}^{n-1} \rho^{n-k-1} \Vnorm[\VR]{\pi_{\theta} \Pker_{k \gamma \step, \theta} \Pker_{\gamma \step, \theta} - \pi_{\theta} \Pker_{ k \gamma \step} \Kker_{\gamma, \theta}^{\step}} \\
  &\qquad \qquad \leq \gamma^{1/2} A_2 \tilde{A}_3 (1+d^{\varpi_2})^2 / \log(1/\rho)  \eqsp ,                                                                  
\end{align}
which concludes the proof since $A_2$, $\tilde{A}_3$ and $\rho$ do not depend on the dimension $d$.
\end{proof}

\begin{lemma}
  \label{lemma:error_kernel_noncvx}
  There exist $A_{4}, \varpi_4 \geq 0$ such that for any $\theta \in \msk$, $\gamma \in \ocint{0, \bgamma}$ with $\bgamma <  \min(\mtt / (2\tilde{\Lip}^2), 1/2)$,
  \begin{equation}
    \Vnorm[\VR]{\updelta_x \Kker_{\gamma_1, \theta_1} - \updelta_x \Kker_{\gamma_2, \theta_2}} \leq A_4 (1 + d^{\varpi_4}) \parentheseDeux{\gamma_2^{-1/2} \abs{\gamma_1 - \gamma_2} + \gamma_2^{1/2} \norm{\theta_1 - \theta_2}}\VR^2(x) \eqsp ,
  \end{equation}
   with $\VR$ given in \eqref{eq:def_V_W} and $A_4, \varpi_4 \geq 0$ which do not depend on the dimension $d$.
\end{lemma}

\begin{proof}
  The proof is similar to the one of \cite[Proposition S18]{debortoli2018souk}.
\end{proof}


\begin{proof}[Proof of \Cref{thm:non_cvx}]
Combining \Cref{lemma:bornitude_noncvx}, \Cref{thm:ergo_cv_noncvx} and \Cref{prop:error_disc_noncvx} we obtain that
\Cref{assum:condition_majo_V} is satisfied. \Cref{lemma:error_kernel_noncvx} implies that \Cref{assum:condition_kernel_fix} holds. Therefore \Cref{thm:cv_soul_non_cvx_1} and \Cref{thm:cv_soul_non_cvx_2} can be applied.
\end{proof}


\subsection{Proof of \Cref{prop:limit}}
\label{prop:limit_proof}
We recall that for any $\vareps>0$ and $x \in \rset^{\dim}$ we define $f_{\vareps}(x) = \norm{F(x)}^2 - \vareps$.

\begin{proposition}
  \label{propo:micro_limit}
  Assume \tup{\Cref{assum:sub_holder}($\upalpha$)} and \tup{\Cref{assum:weak}($2\upalpha$)} with $\upalpha >0$. In addition, assume that $F$ is continuous, $F^{-1}(\{ 0 \}) \neq \emptyset$, $F^{-1}(\{0\}^{\complementary})\neq \emptyset$ and that for every open set $\msa \neq \emptyset$, $\mu(\msa) > 0$. Then there exists $\vareps_0 >0$ such that for any $\vareps \in \ocint{0, \vareps_0}$, the macrocanonical model $\pi_{\vareps}$  associated with  $f_{\vareps}$ and the reference measure $\mu$, solution of \primal, exists and is given by $(\rmd \pi_{\vareps} /\rmd \mu )(x) \propto \exp \parentheseDeux{- \vartheta_{\vareps} f_{\vareps}(x)} \eqsp$, with $\vartheta_{\vareps} >0$.
  In addition $\lim_{\vareps \to 0} \vartheta_{\vareps} = +\infty$. 
\end{proposition}

\begin{proof}
    Let $\vareps_0 = \mu (\normLigne{F}^2) /2 >0$, since $\mu(F^{-1}(\{ 0 \}^{\complementary}))>0$. Let $\vareps \in \ocint{0, \vareps_0}$. Since \Cref{assum:weak}($ 2\upalpha$) holds, for any $\vartheta > -\eta / C_{\upalpha}^2$, with $C_{\upalpha}$ given in \Cref{assum:sub_holder}($\upalpha$) and $\eta$ given in \Cref{assum:weak}($2\upalpha$), we have $\int_{\rset^d} \exp [ - \vartheta f_{\vareps}(x) ] \rmd \mu (x) < +\infty$. Let $\msi = \ooint{- \eta / C_{\upalpha}^2,  +\infty}$ and
  $L_{\vareps} : \ \msi \to \rset$ such that for any $\vartheta \in \msi$, $L(\vartheta) = \log \defEns{\int_{\rset^d} \exp \parentheseDeux{- \vartheta f_{\vareps}(x)}  \rmd \mu (x)}$.
 By  \Cref{prop:existence_P}, we have that $L$ is continuously differentiable on $\msi$. Since $F^{-1}(\{ 0 \}) \neq \emptyset$ we have that there exists a non-empty open set $\msi_{\vareps}$ such that for any $x \in \msi_{\vareps}$, $f_{\vareps}(x) < 0$.
  Therefore
  \begin{equation}
    \lim_{\vartheta \to +\infty} L(\vartheta) \geq \lim_{\vartheta \to +\infty} \log \defEns{ \int_{\msi_{\vareps}} \exp \parentheseDeux{- \vartheta f_{\vareps}(x)}  \rmd \mu (x)}  = +\infty \eqsp ,
  \end{equation}
  where we used the monotone convergence theorem in the last inequality.
  Since $L$ is continuous we obtain that there exists $\vartheta_{\vareps} \in \coint{0,+\infty}$ such that $L(\vartheta_{\vareps}) = \min_{\coint{0,+\infty}} L(\vartheta)$. We have that $L'(0) \leq \vareps -\mu(\norm{F}^2) < 0$, therefore $\vartheta_{\vareps} \in \ooint{0,+\infty}$ and $L'(\vartheta_{\vareps}) = 0$. Applying \Cref{prop:existence_P}, we obtain that $\pi_{\vartheta_{\vareps}}$ is a solution of \primal . We denote $\pi_{\vareps}$ this solution.

  Assume that there exists a sequence $(\vareps_n)_{n \in \nset}$ with $\vareps_n >0$ such that $(\vartheta_{\vareps_n})_{n \in \nset}$ is bounded. Then, up to extraction, there exists $\vartheta^{\star} \geq 0$ such that $\lim_n \vartheta_{\vareps_n} = \vartheta^{\star}$. Using the dominated convergence theorem we obtain that
  \begin{equation}
    0 = \lim_{n} \vareps_n = \lim_{n} \pi_{\vareps_n}(f_{\vareps_n}) = \pi_{\vartheta^{\star}}(\norm{F}^2) > 0 \eqsp ,
  \end{equation}
  which is a contradiction. Therefore, $\lim_{\vareps \to 0} \vartheta_{\vareps} = +\infty$. 
\end{proof}

We now turn to the study of the tightness of the sequence $(\pi_{\vareps})_{\vareps >0}$ in 
the special case where $F$ is given by \eqref{eq:neural_network}.
Under the assumptions of \Cref{thm:weak_neural_network}, for each sequence $(\vareps_n)_{n \in \nset}$ such that
$\lim_n \vareps_n = 0$, up to extraction, we have that $(\pi_{\vareps_n})_{n \in \nset}$ converges to a probability measure
$\pi_{\infty}$ which concentrates on $F^{-1}(\{0 \})$.

\begin{proposition}
  \label{thm:weak_neural_network}
  Assume \tup{\Cref{assum:weak}($2$)} and that for any non-empty open set $\msa \subset \rset^d$, $\mu(\msa) >0$. Let $F$ be given by \eqref{eq:neural_network} assume that $1 \in \calJ$
  and that there exists $k \in \{1, \dots, c_1\}$ such that for any $x \in \rset^{\dim}$ with $x \neq 0$, there exists $\ell \in \{1, \dots, n_1 \}$ with $e_{\ell}^{\transpose} \tilde{A}_1^k x >0$. Then for any sequence $(\vareps_n)_{n \in \nset}$ with $\lim_n \vareps_n = 0$,  $(\pi_{\vareps_n})_{n \in \nset}$ is tight.
\end{proposition}

\begin{proof}
          Let $F$ be given by \eqref{eq:neural_network}. Then for any $x \in \rset^{\dim}$, $f_{\vareps}(x) = \normLigne{F(x)}^2 - \vareps$. We show that the level sets of $x \mapsto \normLigne{F(x)}^2$ are compact. Let $\mathbb{S}^{d-1}$ be the sphere in $\rset^d$ and define $\mathrm{f}: \ \mathbb{S}^{d-1} \to \ooint{0,+\infty}$ for any $x \in \mathbb{S}^{d-1}$ by
        \begin{equation}
          \mathrm{f}(x) = \max_{\ell \in \{1, \dots, c_1\}} \defEns{e_{\ell}^{\transpose} A_1^k x} \eqsp .
        \end{equation}
        $\mathrm{f}$ is continuous and since $\mathbb{S}^{d-1}$ is compact, $\mathrm{f}$ reaches its minimum $\mathrm{f}_0$ and therefore $\mathrm{f}_0 > 0$. Let $x \in \rset^d$, using that $\varphi$ is non-increasing
        we have for any $k \in \{1, \dots, c_1 \}$
        \begin{align}
          n_1^{-1}\sum_{\ell=1}^{n_{1}} \scrG_{1}^k(x)(\ell) &= n_1^{-1} \sum_{\ell=1}^{n_{1}} \varphi(e_\ell^{\transpose} A_1^k x) \\
                                                           &\geq  n_1^{-1} \sum_{\ell=1}^{n_{1}} \varphi(e_\ell^{\transpose} \tilde{A}_1^k x + e_\ell^{\transpose} A_1^k0) \\
          &\geq n_{1}^{-1} \varphi(\mathrm{f}_0 \norm{x} + \min_{\ell \in \{1, \dots, n_1\}} e_\ell^{\transpose} b_1^k ) 
            \eqsp .
        \end{align}
This result combined with the fact that $\lim_{t \to +\infty}\varphi(t) = +\infty$ implies that 
  \begin{equation}
    \lim_{ \norm{x} \to +\infty} \norm{F(x)}^2 = +\infty \eqsp .
  \end{equation}
Therefore, $F^{-1}(\{0\}^{\complementary}) \neq \emptyset$. $F^{-1}(\{0\}) \neq \emptyset$ since $F(x_0) = 0$. $F$ is continuous and \Cref{assum:sub_holder}($1$) is satisfied. Therefore, \Cref{propo:micro_limit} applies and $(\pi_{\vareps_n})_{n \in \nset}$ is well-defined for any sequence $\lim_{n \to +\infty} \vareps_n = 0$. In addition, $\lim_{n \to +\infty} \vartheta_{\vareps_n} = +\infty$.
  
  Since $x \mapsto \norm{F(x)}^2$ is continuous and coercive, the level sets of $f_{\vareps}$ are compact for any $\vareps >0$. We conclude using \cite[Proposition 2.3]{hwang1980laplace}.
\end{proof}

\begin{proof}[proof of \Cref{prop:limit}]
The proof is then a direct consequence of the tightness of
any sequence $(\pi_{\vareps_n})_{n \in \nset}$ and that
$\lim_{n \to +\infty} \vartheta_{\vareps_n} = +\infty$ combined with
\cite[Proposition 2.2, Theorem 2.1, Theorem 3.1]{hwang1980laplace}. 
\end{proof}

It should be noted that
\Cref{prop:limit} is merely a strengthening of \Cref{thm:weak_neural_network}
under additional assumptions on the form of $F^{-1}(\{ 0\})$.


\section{Additional experiments}

\subsection{Accelerations and noiseless versions of SOUL}
\label{sec:accel-nois-vers}

\label{sec:addit-exper}
First we investigate the following discrete dynamics
\begin{equation}
  \label{eq:souk_no_noise}
  \begin{cases}
    &\tilde{\X}_{k+1}^n = \tilde{\X}_k^n  - \gamma_n \parenthese{\sum_{i=1}^p \tilde{\theta}_n^i \nabla F_i(\tilde{\X}_k^n) + \nabla \Reg(\tilde{\X}_k^n)}  \quad \text{and } \tilde{\X}_0^n = \tilde{\X}_{n-1}^{m_{n-1}} \eqsp ;     \\
    &\tilde{\theta}_{n+1} = \Pi_{\msk} \parentheseDeux{ \tilde{\theta}_n - \delta_{n+1} m_{n}^{-1} \sum_{j=1}^{m_n} F(\tilde{\X}_j^n)} \eqsp , 
  \end{cases}
\end{equation}
which corresponds to the one of SOUL \eqref{eq:souk} without the Gaussian noise term in the Langevin update. We refer to this algorithm as noiseless SOUL. Note that the families $\ensembleLigne{\tilde{\theta}_n}{n \in \nset}$ and $\ensembleLigne{\tilde{\X}_k^n}{n \in \nset, k \in \lbrace 0, \dots, m_n  \rbrace}$ are deterministic up to their initialization. In the setting \eqref{eq:souk_no_noise}, the sequence $(\tilde{\X}_0^n)_{n \in \nset}$ seems to converge to one of the configurations presented in \Cref{fig:target_gaussian}, whereas the sequence $(\theta_n)_{n \in \nset}$ does not converge towards the optimal parameters, see \Cref{fig:no_noise}. This experiment highlights that the use of a Markov kernel in the SOUL dynamics cannot be avoided in order to obtain the convergence of $(\theta_n)_{n \in \nset}$ towards $\theta^{\star}$.

\begin{figure}[h]
  \centering
  \subfloat[]{\includegraphics[width=0.3\linewidth]{./data/ADSN_init.jpg}} \hfill
  \subfloat[]{\includegraphics[width=0.3\linewidth]{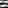}} \hfill
  \subfloat[]{\begin{tikzpicture}[scale = 0.7]
  \begin{axis}[grid=major,no markers,domain=-5:5,enlargelimits=false,ymin=0, legend style={at={(0.65,.9)},anchor=north west}]
    \input{./data/ADSN_no_noise_weight.tex}
    \input{./data/ADSN_no_noise_weight_avg.tex}
    \addlegendentry{$(\theta_n)_{n \in \nset}$}
    \addlegendentry{$(\bar{\theta}_n)_{n \in \nset}$}    
    \end{axis}
  \end{tikzpicture}}
\caption{\figuretitle{Noiseless SOUL} The original target image is recalled in (a) and the limiting configuration obtained with the noiseless SOUL algorithm \eqref{eq:souk_no_noise} is given in (b), whereas the non-convergence of the error towards $0$ can be observed in (c). The blue curve is the \NRMSE \ of the sequence $(\tilde{\theta}_n)_{n \in \nset}$ and the red curve is the \NRMSE \ of the associated averaged sequence.}
\label{fig:no_noise}
\end{figure}

Another modification of the SOUL algorithm can be considered replacing the gradient descent
step in \eqref{eq:souk} by another optimization methodology. Here, we focus on a popular extrapolation technique:
the Nesterov acceleration. The accelerated SOUL algorithm is then given by the following recursion
\begin{equation}
  \label{eq:souk_nesterov}
  \begin{cases}
    &\tilde{\X}_{k+1}^n = \tilde{\X}_k^n  - \gamma_n \parenthese{\sum_{i=1}^p \tilde{\theta}_n^i \nabla F_i(\tilde{\X}_k^n) + \nabla \Reg(\tilde{\X}_k^n)} + \sqrt{2 \gamma_n} \Z_{k+1}^n  \quad \text{and } \tilde{\X}_0^n = \tilde{\X}_{n-1}^{m_{n-1}} 
    \eqsp ;     \\
    &\tilde{\theta}_{n+1/2} = \Pi_{\msk} \parentheseDeux{ \tilde{\theta}_n - \delta_{n+1} m_{n}^{-1} \sum_{j=1}^{m_n} F(\tilde{\X}_j^n)} \eqsp ; \\
    &\tilde{\theta}_{n+1} = \tilde{\theta}_{n+1} + \frac{n-2}{n+1} \defEns{\tilde{\theta}_{n+1/2} - \tilde{\theta}_{n-1/2}} \eqsp ,
  \end{cases}
\end{equation}
where the sequence  $(\noise_k^n)_{n \in \nset, k \in \lbrace 1, \dots, m_n \rbrace}$ is a sequence of independent $d$-dimensional zero mean Gaussian random variables with covariance identity. 
This algorithm is not a descent algorithm but reaches the optimal convergence rate $\bigO(1/n^2)$ for convex functions in a deterministic setting, see \cite{nesterov2013introductory,nesterov1983nesterov}. 
The perturbed gradient case is treated in \cite{attouch2015fast} in a general framework and in \cite{fort2018stochastic, aujol2019rates} when the perturbation is given by a Monte Carlo approximation of the gradient. Recall that for any $n \in \nset$ we define $\eta_n = \nabla L(\theta_n) - m_n^{-1} \sum_{k=1}^{m_n} F(\X_k^n)$. The assumption on the summability of the sequence of perturbations $(\eta_n)_{n \in \nset}$ is of the form $\sum_{n \in \nset} n \norm{\eta_n} < +\infty$ in \cite[Theorem 5.1]{attouch2015fast}. This is a stronger requirement than $\sum_{n \in \nset} \norm{\eta_n} < +\infty$ which is a common assumption for the convergence of the perturbed gradient descent, see \cite[Section 5.2.1]{kushner2003stochastic}
. In this accelerated setting \eqref{eq:souk_nesterov}, letting $m_n = 1$, generates oscillatory sequences $(\theta_n)_{n\in \nset}$ which do not reduce the \NRMSE . However this oscillatory effect can be counterbalanced with the use of a larger batch size, \eg \ $m_n = 10$, see \Cref{fig:nesterov}. 

\begin{figure}[h]
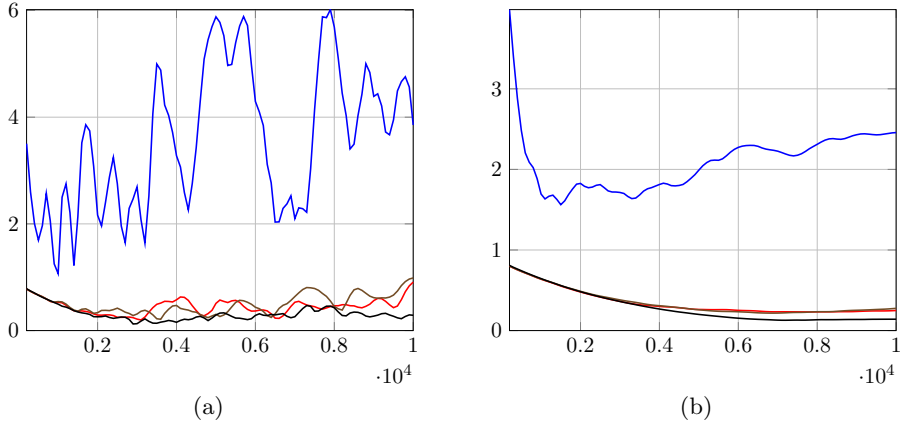

  \centering
  \subfloat[]{
\begin{tikzpicture}[scale = 0.75]
  \begin{axis}[grid=major,no markers,domain=-5:5,enlargelimits=false,ymin=0]
    \input{./data/ADSN_nesterov_brutal_weight.tex}
    \input{./data/ADSN_nesterov_small_weight.tex}
    \input{./data/ADSN_nesterov_small_decrease_weight.tex}
    \input{./data/ADSN_nesterov_small_decrease_increase_weight.tex}
  \end{axis}
  \end{tikzpicture}} \qquad
\subfloat[]{\begin{tikzpicture}[scale = 0.75]
  \begin{axis}[grid=major,no markers,domain=-5:5,enlargelimits=false,ymin=0]
    \input{./data/ADSN_nesterov_brutal_weight_avg.tex}
    \input{./data/ADSN_nesterov_small_weight_avg.tex}
    \input{./data/ADSN_nesterov_small_decrease_weight_avg.tex}
    \input{./data/ADSN_nesterov_small_decrease_increase_weight_avg.tex}           
    \end{axis}
  \end{tikzpicture}}
\caption{\figuretitle{Nesterov acceleration} The Nesterov accelerated version of SOUL \eqref{eq:souk_nesterov} does not yield a satisfactory sequence $(\theta_n)_{n \in \nset}$ in terms of \NRMSE \ (blue curve in (a)) nor a satisfactory averaged sequence $(\bar{\theta}_n)_{n \in \nset}$ (blue curve in (b)) with parameters $\delta_n = 10^{-1}$, $\gamma_n = 10^{-4}$ and $m_n = 1$. If $\delta_n =10^{-2}$ then the results are improved (red curves) or $\delta_n = 10^{-2}\times n^{-0.5}$ (brown curve). The best results are obtained if  $\delta_n = 10^{-2}\times n^{-0.5}$ and $m_n = \ceil{n^{0.5}}$ (black curve). }
\label{fig:nesterov}
\end{figure}


\subsection{Highly regular textures}
\label{sec:highly-regul-text}
In \Cref{fig:regular} we perform the comparison on highly regular textures. On these textures our algorithm and the one of \cite{gatys2015texture} fail at reproducing visually satisfying images (with the notable exception of the brick image for which \cite{gatys2015texture} yields excellent results). Adding spectral constraints, as in \cite{liu2016texture}, yields more regular images although the results are still not satisfactory. A solution is proposed in \cite{gonthier2019high} where  autocorrelation features are considered at each layers. This method yields the best visual results but the parameter space is much larger than  the initial image space.

\begin{figure}[h]
  \centering
  \begin{tikzpicture}
    \node[inner sep=0pt] (name1) at (0, 0)
    {\includegraphics[width=.13\textwidth]{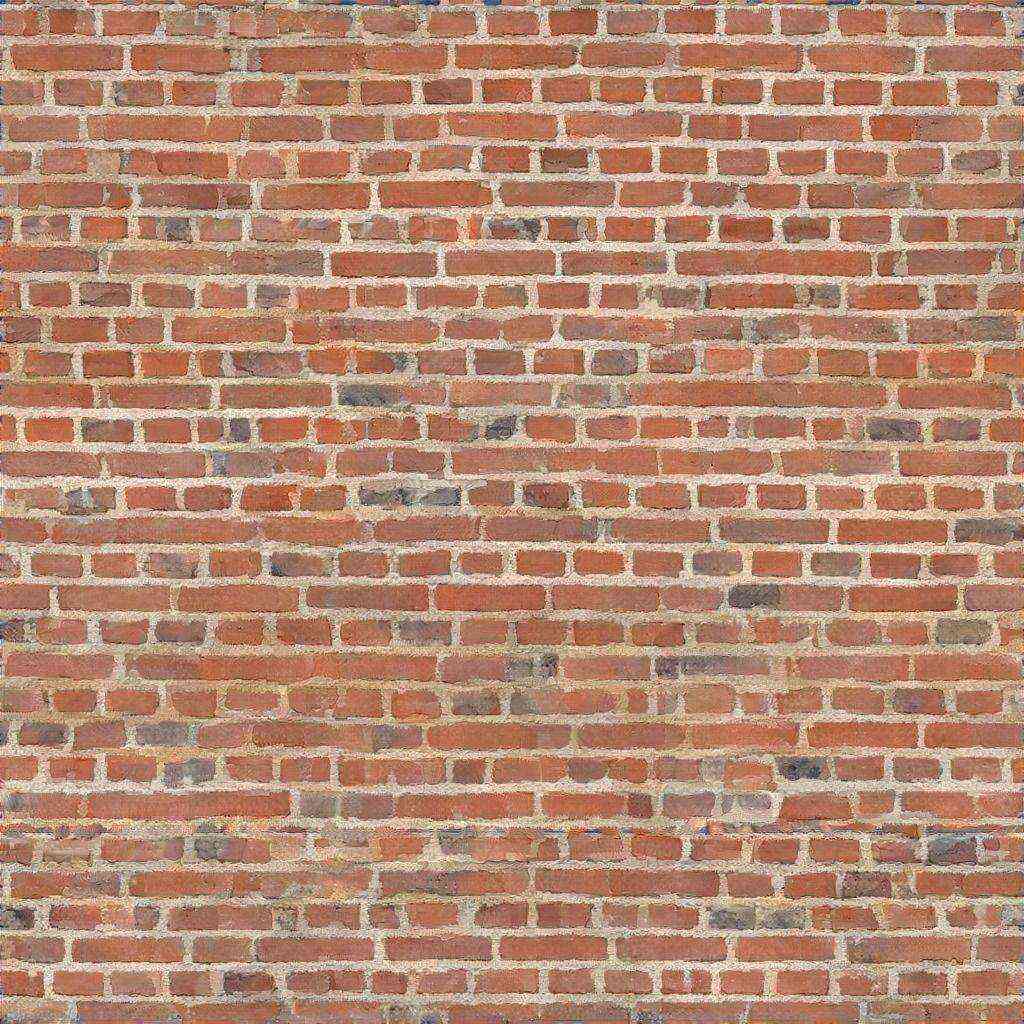}};
\node[inner sep=0pt] (name2) at (2.5, 0) 
{\includegraphics[width=.13\textwidth]{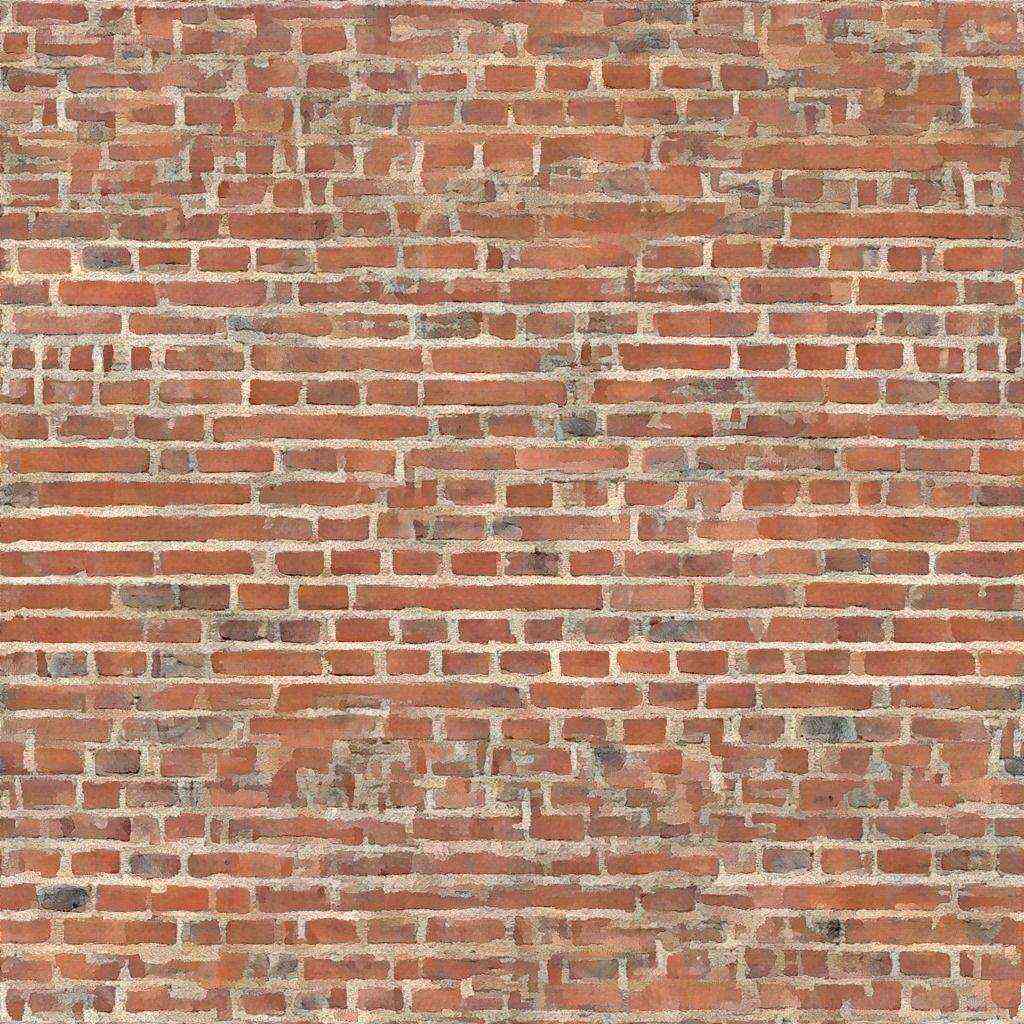}};
\node[inner sep=0pt] (name3) at (5, 0) 
{\includegraphics[width=.13\textwidth]{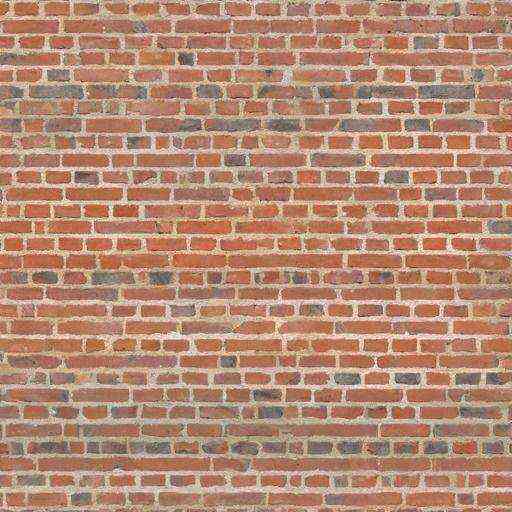}};
\node[inner sep=0pt] (name3) at (7.5, 0) 
    {\includegraphics[width=.13\textwidth]{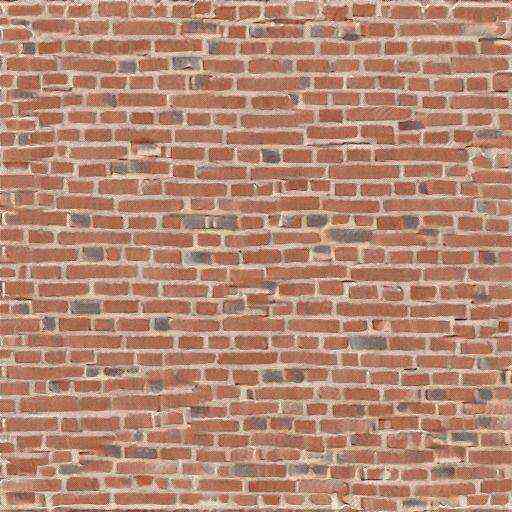}};
\node[inner sep=0pt] (name4) at (12, 0) 
{\includegraphics[width=.13\textwidth]{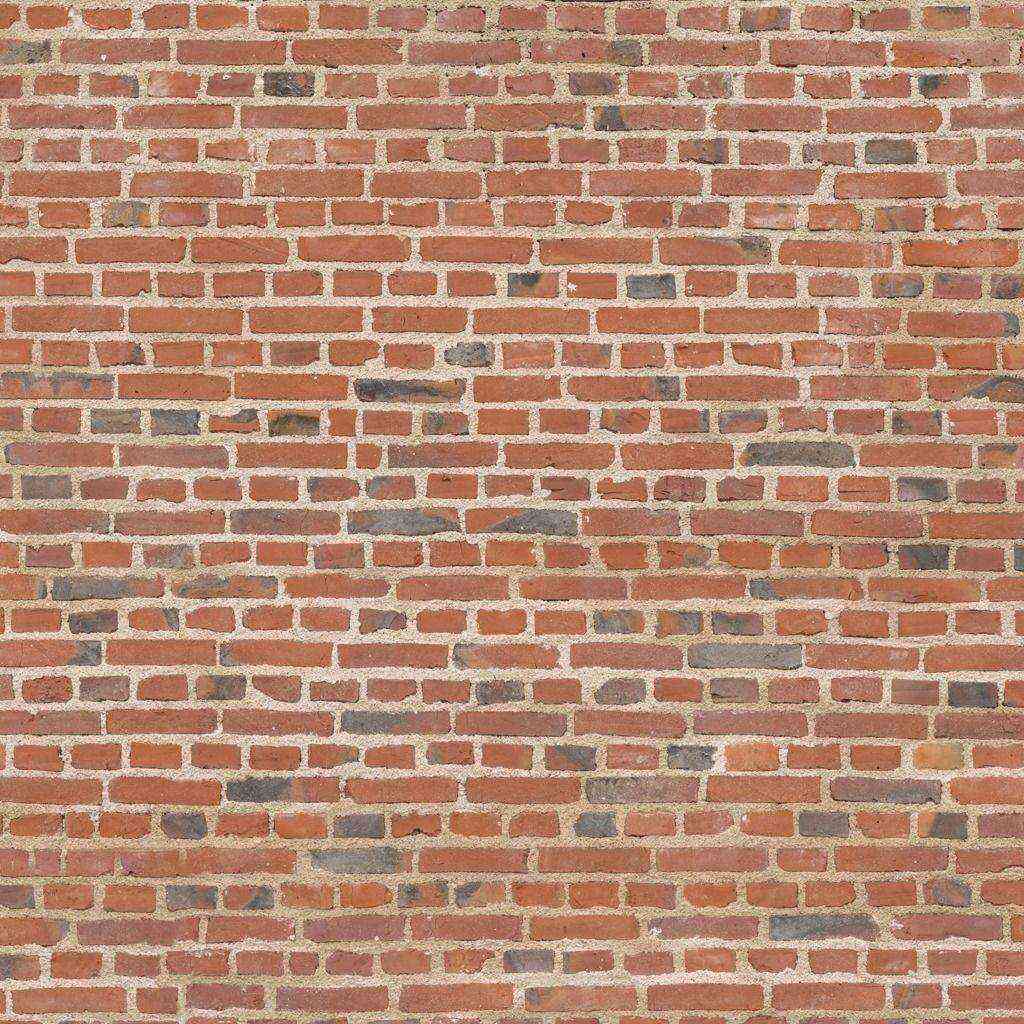}};


    \node[label=below:\scriptsize Gonthier-Gousseau, inner sep=0pt] (name1) at (0, -6)
    {\includegraphics[width=.13\textwidth]{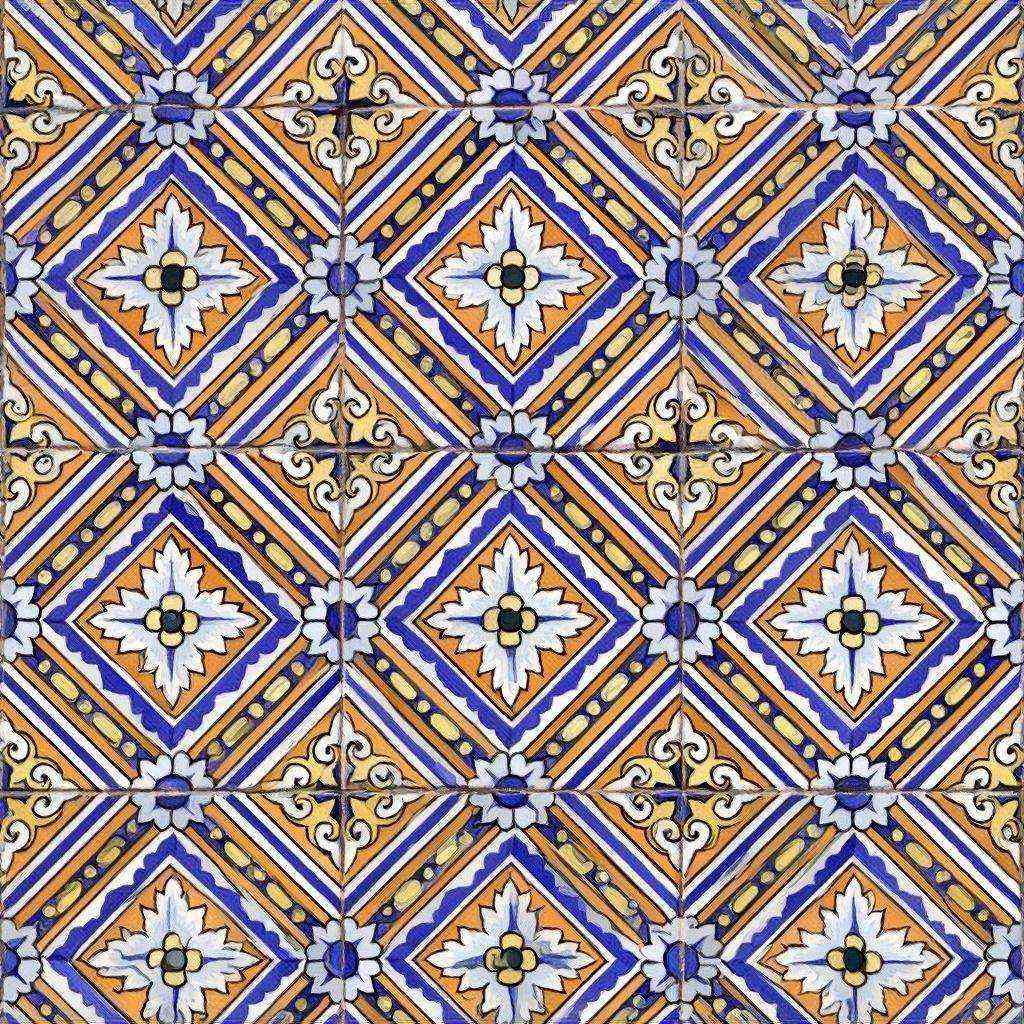}};
\node[label=below:\scriptsize Liu-Gousseau-Xia, inner sep=0pt] (name2) at (2.5, -6) 
{\includegraphics[width=.13\textwidth]{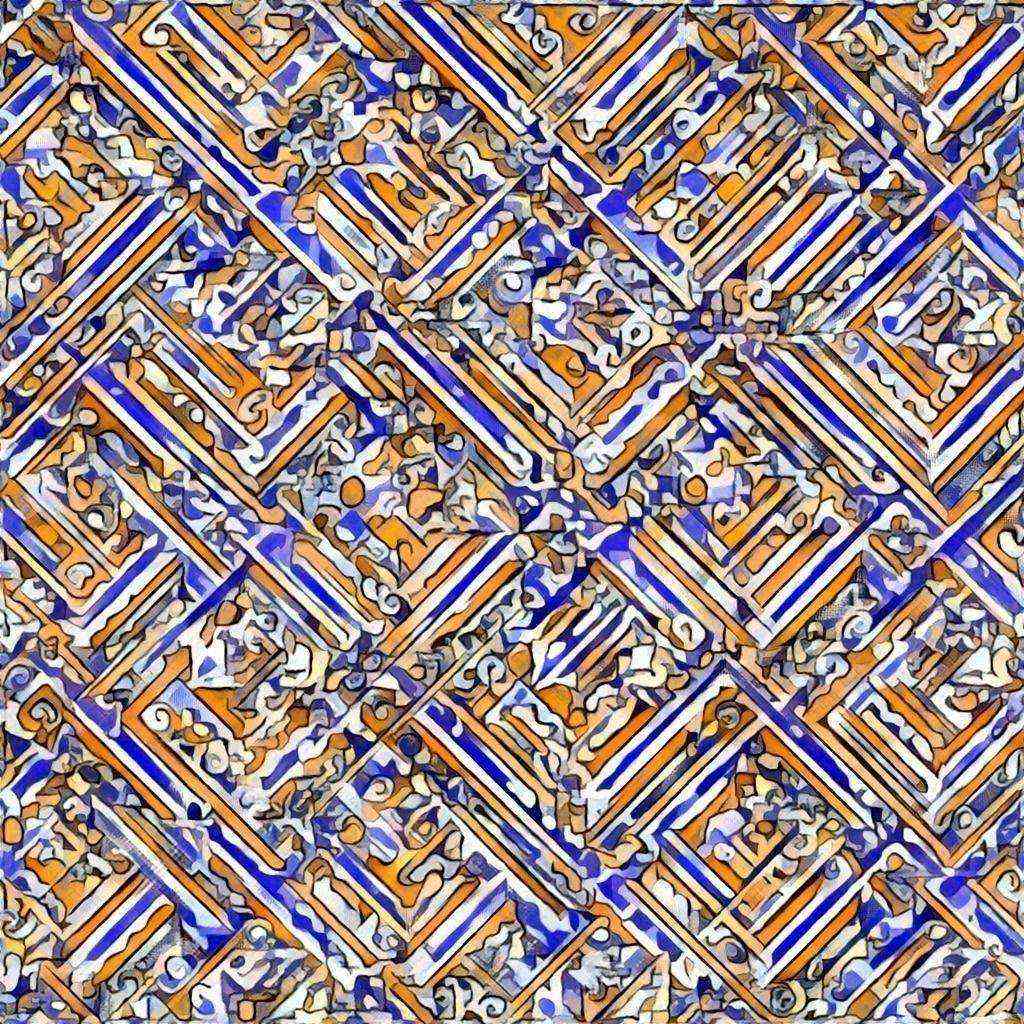}};
\node[label=below:\scriptsize Gatys, inner sep=0pt] (name3) at (5, -6) 
{\includegraphics[width=.13\textwidth]{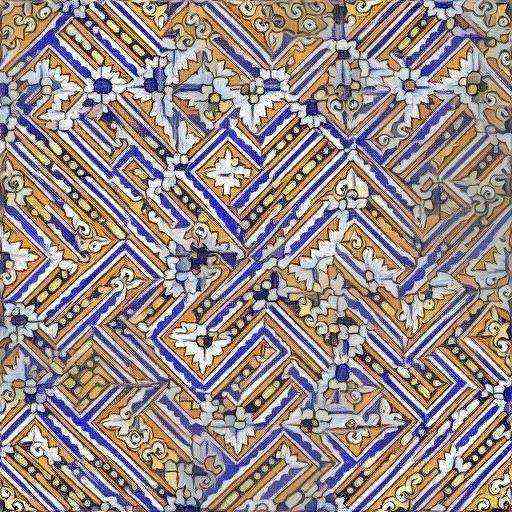}};
\node[label=below:\scriptsize ours, inner sep=0pt] (name3) at (7.5, -6) 
    {\includegraphics[width=.13\textwidth]{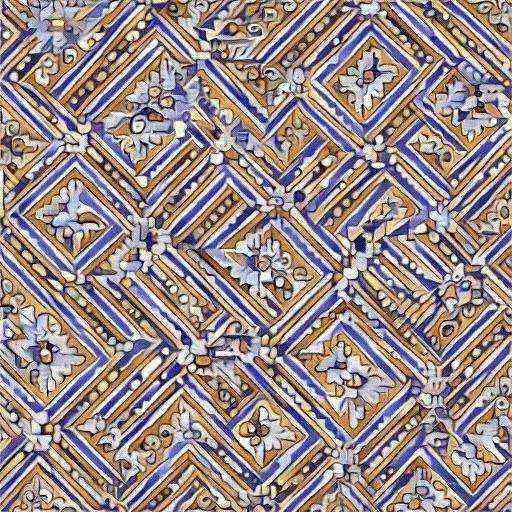}};
\node[label=below:\scriptsize exemplar image $x_0$, inner sep=0pt] (name4) at (12, -6) 
{\includegraphics[width=.13\textwidth]{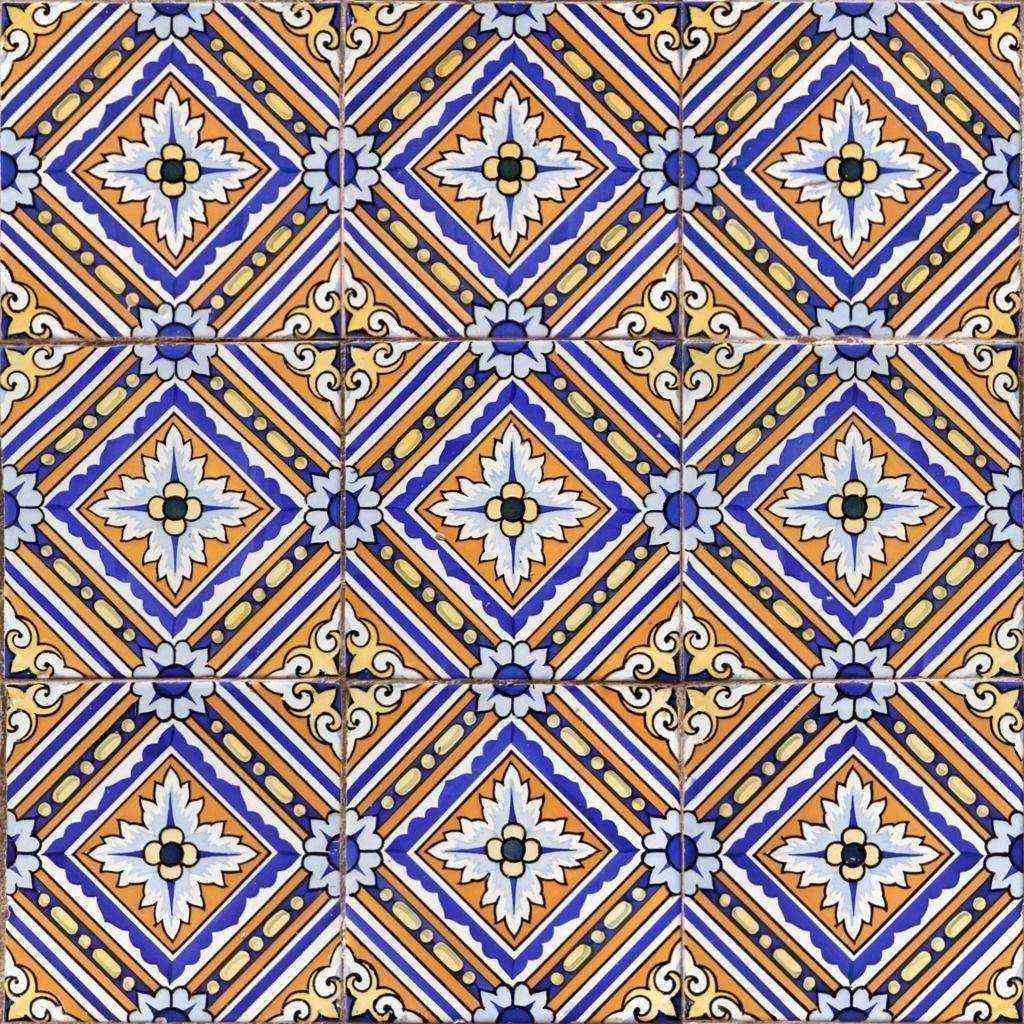}};

\draw [dashed, thick] (9.75,1.25) -- (9.75,-7.25); 
\end{tikzpicture}
\label{fig:regular}
\caption{\figuretitle{Comparison with \cite{liu2016texture}} The images presented in the column ``Gonthier-Gousseau'' corresponds to the features described in \cite{gonthier2019high} ``Liu-Gousseau-Xia'' are synthesized with the features considered in \cite{liu2016texture}, the ones presented in the column ``Gatys'' are generated with \cite{gatys2015texture} and the fourth column contains our results.}
\end{figure}

\section{Structure of \vgg19}
\label{sec:structure_of_vgg19}

The layers of the \vgg19 \ network \cite{simonyan2014vgg} are given as follows (for each convolutional layer we indicate $(c_j, n_j) \to (c_{j+1}, n_{j+1})$):
\newline
\scriptsize
\begin{multicols}{2}
\begin{enumerate}
  \setcounter{enumi}{-1}
\item Convolutional layer, $(3, n_0) \to (64, n_0)$
\item ReLU layer
\item Convolutional layer, $(64, n_0) \to (64, n_0)$
\item ReLU layer  
\item Max-pooling layer
\item Convolutional layer, $(64, n_0/2) \to (128, n_0/2)$
\item ReLU layer
\item Convolutional layer, $(128, n_0/2) \to (128, n_0/2)$
\item ReLU layer
\item Max-pooling layer
\item Convolutional layer,  $(128, n_0/4) \to (256, n_0/4)$
\item ReLU layer
\item Convolutional layer, $(256, n_0/4) \to (256, n_0/4)$
\item ReLU layer
\item Convolutional layer, $(256, n_0/4) \to (256, n_0/4)$
\item ReLU layer
\item Convolutional layer, $(256, n_0/4) \to (256, n_0/4)$
\item ReLU layer
\item Max-pooling layer
\item Convolutional layer, $(256, n_0/8) \to (512, n_0/8)$
\item ReLU layer
\item Convolutional layer, $(512, n_0/8) \to (512, n_0/8)$
\item ReLU layer
\item Convolutional layer, $(512, n_0/8) \to (512, n_0/8)$
\item ReLU layer
\item Convolutional layer, $(512, n_0/8) \to (512, n_0/8)$
\item ReLU layer
\item Max-pooling layer
\item Convolutional layer, $(512, n_0/16) \to (512, n_0/16)$
\item ReLU layer
\item Convolutional layer, $(512, n_0/16) \to (512, n_0/16)$
\item ReLU layer
\item Convolutional layer, $(512, n_0/16) \to (512, n_0/16)$
\item ReLU layer
\item Convolutional layer, $(512, n_0/16) \to (512, n_0/16)$
\item ReLU layer  
\item Max-pooling layer
\end{enumerate}
\end{multicols}


\end{document}